\documentclass[11pt,reqno,makeidx]{amsart}

\usepackage{amssymb,color}
\usepackage{amsmath}
\usepackage{amsthm}
\usepackage{mathrsfs}
\usepackage{enumerate}
\usepackage{epsfig}
\usepackage{hyperref}
\usepackage{refcount} 

\usepackage{array}
\usepackage{longtable}
\usepackage{booktabs}

\usepackage[letterpaper,left=1.25in,right=1.25in,top=1in,bottom=1in]{geometry}

\hypersetup{colorlinks=true, linkcolor=blue, citecolor=blue, urlcolor=blue}

\newtheorem{Proposition}{Proposition}[section]
\newtheorem{Remark}[Proposition]{Remark}
\newtheorem{Corollary}[Proposition]{Corollary}
\newtheorem{Lemma}[Proposition]{Lemma}

\newtheorem{Theorem}[Proposition]{Theorem}
\newtheorem{Definition}[Proposition]{Definition}

\numberwithin{equation}{section}

\def\RR{\mathbb R}

\def\ZZ{\mathbb Z}

\def\be{\begin{equation}}
\def\ee{\end{equation}}
\def\bel{\begin{equation}\label}
\def\la{\langle}
\def\ra{\rangle}

\def\del{\partial}

\def\Le{\leqslant}
\def\Ge{\geqslant}

\def\calA{\mathcal A}

\def\calF{\mathcal F}
\def\calG{\mathcal G}

\def\calM{\mathcal M}

\def\calQ{\mathcal Q}
\def\calR{\mathcal R}

\def\calT{\mathcal T}
\def\calV{\mathcal V}
\def\calW{\mathcal W}

\def\K{{\rm K}}
\def\I{{\rm I}}
\def\dd{\,{\rm d}}

\def\wt{\widetilde}

\newcommand{\Hdef}{\mathrm{H}}

\newcommand{\abs}[1]{\left|#1\right|}
\newcommand{\norm}[1]{\left\|#1\right\|}
\newcommand{\red}[1]{#1}
\newcommand{\blue}[1]{#1}

\newcommand{\green}[1]{#1}

\title{Rigorous analysis of giant magnetic vortex strings}

\author{O. Avadanei}
\address{Department of Mathematics, California Institute of Technology, 1200 E.\ California Boulevard, Mail Code 253-37, Pasadena, CA 91125, USA}
\email{avadanei@caltech.edu}

\author{W. Schlag}
\address{Department of Mathematics, Yale University, Kline Tower, 219 Prospect Street, New Haven, Connecticut, USA}
\email{wilhelm.schlag@yale.edu}

\date{}

\begin{document}

\begin{abstract}
\red{In this paper we are concerned with} axially symmetric Abrikosov--Nielsen--Olesen (ANO) vortices in the quartic Abelian Higgs model as the magnetic flux
$n\to\infty$.  The emphasis is on developing the rigorous asymptotic analysis behind the formal large-flux
expansions of Hakim, Lema\^itre, and Mallick '01 and Dumitrescu, Gaikwad~'25. For every fixed $0<\beta<4$ we prove that the radius of the
magnetic core is of size
$\sqrt2\,\beta^{-1/4} n^{1/2}$ and the string tension has the form
\[     \calT_n=\sqrt\beta\,n+\sqrt2\,s_\beta\beta^{-1/4}n^{1/2}+O(1),
\]
which confirms the formal predictions of the physics literature.
The constant $s_\beta$ is the renormalized energy of a one-dimensional interface problem.  We identify this interface with
the classical normal-to-superconducting transition layer in Ginzburg--Landau theory. The theorem of
Chapman--Howison--McLeod--Ockendon~'91 proves its existence, uniqueness and monotonicity.   The connection of the Gaussian tail of the scalar field on the left with the interface is made directly with the Weber equation, rather than by an Airy normal form.
A key step is to show that the matching map between the admissible
Cauchy data at the two ends of the linearized interface equations has
maximal rank.  Monotonicity of the interface gives a positive quadratic-form
identity which proves this for every $0<\beta<4$.  The estimates are uniform
when $\beta$ ranges in a compact subinterval of $(0,4)$.
\end{abstract}

\thanks{
The first author thanks Yale for its hospitality and support during the 2025--2026 academic year.
The second author is partially supported by the United States NSF through grant DMS-2350356. He thanks Zohar Komargodski for pointing him to the 2025 preprint by Dumitrescu and Gaikwad, and he is grateful to Nick Manton for his interest in this topic and for several helpful discussions. \red{The authors used OpenAI Codex during the preparation of the manuscript in order to carry out mathematical consistency checks, verify intermediate steps, and for the purpose of editorial proofreading. They reviewed all of the suggestions that they received in this process independently, and they take full responsibility for the mathematical content of the final text.} 
}

\maketitle

\tableofcontents

\section{Introduction}

Magnetic vortices with large quantized flux appear naturally in the
Ginzburg--Landau theory of superconductivity, which was introduced by Ginzburg and Landau in \cite{Ginzburg-Landau}, as well as in the relativistic Abelian Higgs model. This in turn enabled Abrikosov in~\cite{Abrikosov} to identify the existence of flux-carrying vortices in the type-II regime, and also motivated Nielsen--Olesen~\cite{NielsenOlesen} to introduce the corresponding relativistic vortex strings. \red{In ~\cite{HakimLemaitreMallick}, Hakim, Lema\^itre, and Mallick carried out a systematic analysis of the quartic model in the case of large winding numbers, in which they identified the normal magnetic core, the thin transition layer, and formulated the leading large-flux energy law.}
\red{Later on, Bolognesi developed in~\cite{BolognesiWallsTubes} the closely-related wall-vortex interpretation, and together with Gudnason in \cite{Bolognesi-Gudnason}, they carried out numerical simulations for this scenario in the case of very large winding numbers.} More recent large-winding analyses
include~\cite{PeninWellerGiant,PeninWeller,DG}. It also turns out that a specific scenario, known as the Bogomolny or the BPS regime, exhibits a special first-order coupling structure, and its discovery goes back to Bogomolny \cite{Bogomolny}.  For mathematical treatments of vortices and general background on solitons see the classic texts by Jaffe-Taubes and Manton-Sutcliffe~\cite{JaffeTaubes,MantonSutcliffe}. Motivated in particular by~\cite{HakimLemaitreMallick} and the
large-flux program of Dumitrescu and Gaikwad~\cite{DG}, we consider the quartic
Abelian Higgs Lagrangian
\bel{eq:Lagrangian}
 {\mathcal L}=-{\frac1{4e^2}}F_{\mu\nu}F^{\mu\nu}
       -|(\del_\mu-iA_\mu)\varphi|^2
       -\frac{\lambda}{2}(|\varphi|^2-v^2)^2 \red{,}
\ee
\red{where $F=dA$, }\blue{and $A=A_0\, dt+ A_1\, dx+ A_2\, dy+A_z\, dz$}. 

We restrict attention to straight, static strings oriented along the $z$-axis, so that the fields are independent of
$t$ and $z$ and $A_0=A_z=0$.  The energy then factors as
\red{\[
        T=-\int_{\RR^2}{\mathcal L}(x,y)\dd x\dd  y,
\]}
so the (string) tension $T$ is the energy per unit length in the $z$-direction.
Finite energy forces $|\varphi|\to v$ as $r\to\infty$ and implies flux quantization: the phase winds an integer
$n\in\ZZ$ and equivalently
\[
        n=\frac{1}{2\pi}\int_{\RR^2}F_{12}\,\dd x \dd y.
\]
The dimensionless parameter is
\bel{eq:beta}
       \beta=\frac{m_H^2}{m_V^2}=\frac{\lambda}{e^2}.
\ee
We restrict ourselves throughout to the quartic potential.  The starting point is the formal large-flux analysis of
Dumitrescu--Gaikwad\footnote{The sextic model considered there belongs to a different large-flux regime and is not
discussed in this work.} \cite{DG}. 

We make the standard axially symmetric ansatz
\bel{eq:ansatz}
 \varphi(r,\theta)=v\,\phi(r)e^{in\theta},\qquad
 A_r=0,\qquad A_\theta(r)=n(1-a(r)).
\ee
Here $A=A_\theta(r)\,d\theta$. Thus $A_\theta$ is the covariant polar
component, not the component in the orthonormal angular frame.  We refer
to $a$ as the \emph{gauge potential profile}, since it parametrizes
$A_\theta=n(1-a)$.  The magnetic field is $F_{12}=-\frac{n}{r}a'(r)$ (and in
dimensionless variables $u=evr$, it is $-\frac{n}{u}a'(u)$).  In this
gauge
\[
        F_{12}=\frac{1}{r}\partial_r A_\theta(r)
              =-\frac{n}{r}\partial_r a(r),
\]
so $a'<0$ corresponds to a positive magnetic field and the boundary values
$a(0)=1$, $a(\infty)=0$ give flux $2\pi n$.
\red{Indeed, since $A=A_1\,dx+A_2\,dy$ we have
\begin{align*}
  F&=dA=(\partial_xA_2-\partial_yA_1)\,dx\wedge dy=F_{12}\,dx\wedge dy=rF_{12}\,dr\wedge d\theta   
\end{align*}
Here we have used the polar coordinates $(r,\theta)$, for which we write $x=r\cos(\theta)$ and $y=r\sin(\theta)$, with $r\in(0,\infty)$ and $\theta\in[0,2\pi)$.

On the other hand, if we write $A=A_r\,dr+A_\theta \,d\theta$, since $A_r=0$, we get that $A=A_\theta(r) \,d\theta$. It follows that
\begin{align*}
F=dA=\partial_rA_\theta(r)\,dr\wedge d\theta
\end{align*}
Equating the two expressions for $F$, and using the ansatz $A_\theta=n(1-a(r))$, we get that
\begin{align*}
  F_{12}=\frac{1}{r}\partial_rA_\theta(r)=-\frac{na'(r)}{r},  
\end{align*}
as desired.
}
With $u=e v r$, the
dimensionless string tension is
\bel{eq:energy}
 \calT_n[\phi,a]=\frac{T_n}{2\pi v^2}
 =
 \int_0^\infty \left[
 \frac{n^2}{2u}(a')^2+u(\phi')^2+\frac{n^2}{u}a^2\phi^2
       +\frac{\beta}{2}u(1-\phi^2)^2
 \right]\dd u .
\ee
\red{To see this, we begin by noting the identities
\begin{align*}
\partial_x&=\cos(\theta)\partial_r-\frac{\sin(\theta)}{r}\partial_\theta\\
\partial_y&=\sin(\theta)\partial_r+\frac{\cos(\theta)}{r}\partial_\theta\\
A_1&=\cos(\theta)A_r-\frac{\sin(\theta)}{r}A_\theta=-\frac{\sin(\theta)}{r}A_\theta=-\frac{n\sin(\theta)(1-a(r))}{r}\\
A_2&=\sin(\theta)A_r+\frac{\cos(\theta)}{r}A_\theta=\frac{\cos(\theta)}{r}A_\theta=\frac{n\cos(\theta)(1-a(r))}{r}
\end{align*}
Together with the ansatz $\varphi(r,\theta)=v\phi(r)e^{in\theta}$, these lead to
\begin{align*}
    |(\partial_\mu-iA_\mu)&\varphi|^2=|(\partial_x-iA_1)\varphi|^2+|(\partial_y-iA_2)\varphi|^2\\
    &=v^2\left|\cos(\theta)\phi'(r)e^{in\theta}-in\frac{\sin(\theta)}{r}\phi(r)e^{in\theta}+in\frac{\sin(\theta)(1-a(r))}{r}\phi(r)e^{in\theta}\right|^2\\
    &+v^2\left|\sin(\theta)\phi'(r)e^{in\theta}+in\frac{\cos(\theta)}{r}\phi(r)e^{in\theta}-in\frac{\cos(\theta)(1-a(r))}{r}\phi(r)e^{in\theta}\right|^2\\
    &=v^2\left((\phi'(r))^2+\frac{n^2}{r^2}(a(r)\phi(r))^2\right)
\end{align*}
For the other terms of $\mathcal{L}$, our ansatz for $A$ and $\phi$ leads to
\begin{align*}
    \frac{1}{4e^2}F_{\mu\nu}F^{\mu\nu}&=\frac{n^2(a'(r))^2}{2e^2r^2}\\
    \frac{\lambda}{2}(|\varphi|^2-v^2)^2&=\frac{\lambda}{2}(v^2(\phi(r))^2-v^2)^2=\frac{\lambda v^4}{2}((\phi(r))^2-1)^2
\end{align*}
The expression of the string tension now becomes
\begin{align*}
    T&=\int_{\mathbb{R}^2}\frac{1}{2e^2}F_{12}^2+|(\partial_\mu-iA_\mu)\varphi|^2+\frac{\lambda}{2}(|\varphi|^2-v^2)^2\,dx dy\\
    &=2\pi v^2\int_0^\infty r\left(\frac{n^2a'^2(r)}{2e^2v^2r^2}+\left(\phi'^2(r)+\frac{n^2}{r^2}a^2(r)\phi^2(r)\right)+\frac{\lambda v^2}{2}(\phi^2(r)-1)^2\right)\,dr
\end{align*}
Setting $\displaystyle u=evr$ and carrying out the other constant normalizations leads to the desired expression \eqref{eq:energy} for the dimensionless string tension $\displaystyle\calT_n[\phi,a]$.
}

The Euler--Lagrange equations are
\bel{eq:vortex}
\begin{aligned}
    \phi''(u)+\frac{1}{u}\phi'(u)&=\frac{n^2}{u^2}a^2\phi(u)+\beta\phi(u)(\phi(u)^2-1),\\
 a''(u)-\frac{1}{u}a'(u)&=2a(u)\phi(u)^2,
\end{aligned} 
\ee
with boundary conditions
\bel{eq:bc}
       \phi(0)=0,\quad a(0)=1,\qquad \phi(\infty)=1,\quad a(\infty)=0.
\ee
The presence of the large parameter in~\eqref{eq:vortex} means that the asymptotic analysis of the minimizers of~\eqref{eq:energy} requires methods from singular perturbation theory, cf.~\cite{Olver, CST, CDST, CPS}. 
We refer to \red{the} first equation in~\eqref{eq:vortex} as {\em Higgs equation}, while the second one is the {\em magnetic equation}. 
To identify the basic scale of the problem, we substitute
\[
        \phi(u)=0,\qquad a(u)= 1-\frac{u^2}{R^2},\qquad 0<u<R.
\]
into \eqref{eq:energy}. This gives
\bel{eq:bulkenergy}
        E_{\rm bulk}(R)=\frac{n^2}{R^2}+\frac{\beta}{4}R^2,
\ee
and hence
\bel{eq:Rbulk}
        R_n^{(0)}= \sqrt2\,\beta^{-1/4}n^{1/2},\qquad
        E_{\rm bulk}(R_n^{(0)})=\sqrt\beta\, n,
\ee
which is well-known from the physics literature.

There are three analytic regions, cf.~\cite[Figure 4]{DG}.  In the core defined as $u\ll R=R_n^{(0)}$, the Higgs field is exponentially small and the gauge potential profile is close to
$1-u^2/R^2$, so the Higgs equation is a scalar Sturm--Liouville equation with a large parameter.  In the $O(1)$ boundary
variable $y=u-R$, the limiting equations are autonomous and nonlinear.  On the right the system linearizes about the Higgs
vacuum and the homogeneous modes are modified Bessel functions.  The left tail of the boundary problem is the point at
which the Weber normal form enters, cf.~Costin--Park--Schlag \cite{CPS}: the equation for $\Phi$ becomes, to leading order,
\[
        \Phi''(y)=(\beta y^2-\beta)\Phi(y) .
\]
The WKB asymptotics of this Weber equation are precisely given by the Gaussian tail which matches the core.

The global
interface orbit (the ``domain wall") formally identified in~\cite{DG} is exactly the normal-to-superconducting transition
layer treated by Chapman--Howison--McLeod--Ockendon~\cite{CHMO} (which we refer to as CHMO throughout). A numerical and asymptotic analysis of the aforementioned transition was carried out by Osborn and Dorsey in \cite{Osborn-Dorsey}. The normalized interface linearization is
non-degenerate throughout $0<\beta<4$.  After changing the sign of the
magnetic component, the derivative of the monotone interface is a
componentwise positive solution of a symmetric linear system.  The
quadratic-form identity in \eqref{eq:interfacePositiveRepresentation}
shows that every bounded homogeneous solution is a multiple of the
translation solution.  The normalization at $-\infty$ excludes that
multiple.  Together with the endpoint constructions, this gives the
maximal rank property in Proposition~\ref{prop:nearBPS}, uniformly on
compact parameter intervals. In the resonant or forced regime $\beta\ge4$ one would moreover need
an additional radial right-tail construction, as well as extend the validity of Lemma~2.4 in~\cite{CHMO} to $\kappa\ge \sqrt{2}$. The corresponding change in the radial asymptotics at the superconducting end at and beyond the threshold $\kappa=\sqrt{2}$ was studied by Plohr \cite{Plohr} and Perivolaropoulos \cite{Perivolaropoulos}. In particular, when $\kappa=\sqrt{2}$, the tail of the scalar field and the quadratic contribution of the gauge-field tail from forcing term in its corresponding equation are resonant, while above this threshold the aforementioned quadratic contribution provides the dominant mode.

The two global problem-specific inputs are the classical radial existence and regularity theory for the
Ginzburg--Landau vortex and the CHMO theorem~\cite{CHMO} for the one-dimensional transition layer.    Radial uniqueness
away from the BPS point is a separate delicate shooting question, which we neither address nor rely on here.  

\subsection{Interface constants and main results}

The flux asymptotics for large flux are expressed in terms of a one-dimensional
normal-to-superconducting transition layer.  For $0<\beta<4$,
\cite{CHMO} establishes a unique (affinely
normalized) interface orbit $(\Phi_\beta,\Gamma_\beta)$ solving
\bel{eq:interface}
       \Phi''=\Gamma^2\Phi+\beta\Phi(\Phi^2-1),\qquad
       \Gamma''=2\Gamma\Phi^2 ,
\ee
with endpoint conditions
\bel{eq:interfaceleft}
        \Phi(y)\to0,\qquad \Gamma(y)+\sqrt\beta\,y\to0
        \qquad (y\to-\infty),
\ee
and
\bel{eq:interfaceright}
        \Phi(y)\to1,\qquad \Gamma(y)\to0
        \qquad (y\to+\infty).
\ee
Define the interface energy density
\bel{eq:edens}
       e_\beta(y)=
       (\Phi_\beta')^2+\frac{1}{2}(\Gamma_\beta')^2
       +\Gamma_\beta^2\Phi_\beta^2+\frac{\beta}{2}(1-\Phi_\beta^2)^2 .
\ee
Since $e_\beta(y)\to\beta$ as $y\to-\infty$ and $e_\beta(y)\to0$ as
$y\to+\infty$, the renormalized surface energy is
\bel{eq:sbeta}
       s_\beta=\lim_{L\to\infty}\left[\int_{-L}^{\infty}e_\beta(y)\dd y-\beta L\right],
\ee
see \eqref{eq:uniformInterfaceLeft}, \eqref{eq:uniformInterfaceRight} for the existence of this limit.
We also set
\bel{eq:sigma}
       \sigma_\beta:=\sqrt2\,s_\beta .
\ee
We can now state the two main theorems. The first result confirms the predictions of~\cite{HakimLemaitreMallick, DG} for $0<\beta<4$.

\begin{Theorem}\label{thm:main}
For every $0<\beta<4$, and for every choice of radial minimizers
$(\phi_n,a_n)$ as in Proposition~\ref{prop:static}, one has, as $n\to\infty$,
\bel{eq:Rasymp}
       R_n=\sqrt2\,\beta^{-1/4}n^{1/2}-\frac{s_\beta}{2\beta}+O(n^{-1/2}),
\ee
and
\bel{eq:Tasymp}
       \calT_n=\sqrt\beta\,n+\sqrt2\,s_\beta\beta^{-1/4}n^{1/2}+O(1)
              =\sqrt\beta\,n+\sigma_\beta\beta^{-1/4}n^{1/2}+O(1).
\ee
The remainders are uniform for $\beta$ in a fixed compact interval
$J\Subset(0,4)$. At the BPS point $\beta=1$, $s_1=0$ and $\calT_n=n$ for all $n$.
\end{Theorem}

The second main theorem describes the asymptotic shape of any vortex profiles $(\phi_n,a_n)$ defined as minimizers of the string tension; see the next section for the precise constructions. The bulk of the paper is devoted to proving the assertions in this theorem, which will then be used in the proof of Theorem~\ref{thm:main}.  For the centered variables see Definition~\ref{def:PhiGamma}.

\begin{Theorem}\label{thm:asymptoticShape}
Let $0<\beta<4$, and let $(\phi_n,a_n)$ be vortex profiles defined to be any choice of radial
minimizer of~\eqref{eq:energy} as in Proposition~\ref{prop:static}.  Define the normalized
profiles $(\Phi_n,\Gamma_n)$ by \eqref{eq:boundaryvars}.  Then the following
statements hold.
For every fixed $M<\infty$,
\bel{eq:shapeH2}
 \norm{(\Phi_n,\Gamma_n)-(\Phi_\beta,\Gamma_\beta)}_{H^2([-M,M])}
 \Le C_{M,\beta}R_n^{-1}.
\ee
In particular, the transition takes place on the fixed $y$-scale and has
the limiting shape $(\Phi_\beta,\Gamma_\beta)$ uniquely determined in~\cite{CHMO}.
Equivalently, given $0<\theta<1$, let $y_{\beta,\theta}$ be the unique number
such that
\[
        \Phi_\beta(y_{\beta,\theta})=\theta,
\]
and let $u_{n,\theta}$ be the unique number such that
$\phi_n(u_{n,\theta})=\theta$.  For every
$0<\varepsilon<1/2$,
\bel{eq:shapeLevels}
 \sup_{\theta\in[\varepsilon,1-\varepsilon]}
 \abs{u_{n,\theta}-R_n-y_{\beta,\theta}}
 \Le C_{\varepsilon,\beta}R_n^{-1}.
\ee
Consequently, if
$\varepsilon\le\theta_1<\theta_2\le1-\varepsilon$, then
\bel{eq:shapeWidth}
 u_{n,\theta_2}-u_{n,\theta_1}
 =
 y_{\beta,\theta_2}-y_{\beta,\theta_1}+O_{\varepsilon,\beta}(R_n^{-1}).
\ee
Thus every fixed pair of scalar level sets has a limiting separation of
order one, whereas $R_n$ is of order $n^{1/2}$.

There is a constant $C_\beta>0$ such that, for every
$0<\alpha<\frac{1}{3}$, the scalar profile has the 
asymptotic shape\footnote{This is an overlap statement on the normal side of the transition (normal refers to the region that is not superconducting, in which the scalar field is negligible).  It
does not assert that the Gaussian formula remains valid down to $u=0$,
where the regular solution instead has the Frobenius behavior in
\eqref{eq:frob}.}
\bel{eq:shapeLeftOverlap}
 \lim_{L\to\infty}\limsup_{n\to\infty}
 \sup_{L\le -y\le R_n^\alpha}
 \left|
 \frac{\phi_n(R_n+y)}
 {C_\beta |y|^{(\sqrt\beta-1)/2}
  \exp(-\sqrt\beta\,y^2/2)}
 -1
 \right|=0.
\ee
On the superconducting side there are $K_0<\infty$, $c>0$, and, for each
$K\ge K_0$, a constant $C_{K,\beta}$ such that
\bel{eq:shapeRightComparison}
 \begin{split}
 &\sum_{j=0}^2
   |\partial_y^j(\Gamma_n(y)-\Gamma_\beta(y))|\\
   &+
   \sum_{j=0}^2
   \left|\partial_y^j((1-\Phi_n(y))-(1-\Phi_\beta(y)))\right|
 \Le C_{K,\beta}R_n^{-1}e^{-cy},
 \qquad y\ge K.
 \end{split}
\ee
The constants $C_\beta,G_\beta,H_\beta$ are positive, and the limiting
interface has the endpoint expansions
\bel{eq:shapeInterfaceTails}
 \begin{split}
 \Phi_\beta(y)
 &=C_\beta |y|^{(\sqrt\beta-1)/2}
   e^{-\sqrt\beta\,y^2/2}(1+O_\beta(|y|^{-2})),
 \qquad y\to-\infty,\\
 \Gamma_\beta(y)
 &=G_\beta e^{-\sqrt2y}(1+O_\beta(e^{-\delta_\beta y})),
 \qquad y\to+\infty,\\
 1-\Phi_\beta(y)
 &=H_\beta e^{-\sqrt{2\beta}y}
   (1+O_\beta(e^{-\delta_\beta y})),
 \qquad y\to+\infty,
 \end{split}
\ee
for some $\delta_\beta>0$.
\end{Theorem}

\subsection{Overview of the proof}

We describe the proof in the order in which its ingredients are used.  We
work throughout with an arbitrary radial minimizer.  Uniqueness of the
non-Bogomolny radial boundary-value problem is neither assumed nor needed.
Minimality supplies the variational bounds which locate the transition and,
through stationarity under radial dilations, the exact virial identity used
to determine its position more accurately.  After centering at the magnetic
radius, compactness identifies the limiting transition with the normalized
CHMO interface.  To complete the proof, we further obtain quantitative rates on this convergence, then connect the interface with the core and exterior asymptotics,
and finally we evaluate the energy and virial identities with bounded errors.

The organization after the local analysis is as follows.  Sections
\ref{sec:trueoverlap}--\ref{sec:exterior} establish the Weber and modified
Bessel descriptions needed at the two ends of the interface.  Section
\ref{sec:preliminaryLocalization} then uses only qualitative convergence and
the exact virial identity to prove the preliminary localization required by
the quantitative argument.  Section~\ref{sec:quantitativeBoundary} proves the
$O(R_n^{-1})$ estimate on fixed intervals, and
Section~\ref{sec:leftMatching} extends it into the two tails.  Only then
does Section~\ref{sec:refinedEnergy} return to the energy and virial identities
to obtain the constant correction to the radius and the
asymptotic expansion of the tension.  

\medskip
\noindent\emph{Radial minimizers and coarse localization.}
Proposition~\ref{prop:static} supplies a radial minimizer and the monotonicity
properties used throughout the paper.  Its magnetic radius $R_n$ is defined
by the regular expansion
\[
        a_n(u)=1-\frac{u^2}{R_n^2}+O(u^4)
        \qquad (u\downarrow0).
\]
Evaluating the energy on  suitable test functions \red{gives} the required upper bound for the tension.  The
exact magnetic identity \eqref{eq:magneticCouplingIdentity}, together with
the potential energy, provides the corresponding lower bound.  Comparing these
bounds proves
\[
        \frac{R_n}{R_n^{(0)}}\longrightarrow1,
        \qquad
        R_n^{(0)}=\sqrt2\,\beta^{-1/4}n^{1/2},
\]
in Proposition~\ref{prop:coarseloc}.  The same variational estimates show
that, in the core at distances to the left of $R_n$ which grow with $n$, the
Higgs field has Gaussian decay and the magnetic profile is close to its
quadratic approximation.  These estimates give compactness in every fixed
interval after translation by $R_n$.

\medskip
\noindent\emph{The limiting interface and its linearization.}
Set $y=u-R_n$.  Coarse localization and compactness show that every
subsequential limit of $(\Phi_n,\Gamma_n)$ solves the planar interface
equations.  The expansion defining $R_n$ fixes the affine normalization at
the left end (which we also call the {\em normal}, i.e., not superconducting end). \red{The uniqueness result of Chapman--Howison--McLeod--Ockendon implies that the limit coincides with the interface $(\Phi_\beta,\Gamma_\beta)$ and allows us to infer the full qualitative convergence in Lemma~\ref{lem:qualitativeBoundary}.}   Thus\red{,} all radial minimizers have the same
centered limiting profile\red{,} even though uniqueness of the finite-$n$ minimizer
is not known. \red{Lemma~\ref{lem:interfaceUniformTails} proves that the solution exhibits Gaussian decay at the normal end and decaying exponential behavior at the superconducting one.} \red{This justifies the definition of $s_\beta$ and also subsequently allows us to control truncation errors in energy integrals.} 

The linearized interface equation has a one-dimensional space
$E_\beta^-(-K)$ of Cauchy data at $-K$ satisfying the normalized recessive
condition at $-\infty$, and a two-dimensional space $E_\beta^+(K)$ of data at
$K$ which generate solutions decaying at $+\infty$.  The
positive quadratic-form identity \eqref{eq:interfacePositiveRepresentation}
and a cutoff argument show that every bounded solution of the linearized
equations is a multiple of the translation mode.  Its magnetic component
has the nonzero limit $-\sqrt\beta$ at $-\infty$, so the normalization
excludes it.  Proposition~\ref{prop:nearBPS} therefore gives rank three
for the matching map \eqref{eq:EvansMatchingMap} for every $0<\beta<4$.
Section~\ref{sec:quantitativeBoundary} uses this property to prove the
finite-interval estimate \eqref{eq:finiteWindowOverdet}.

\medskip
\noindent\emph{Local models at the two ends of the interface.}
On the normal side, write $\eta=R_n-u$.  In the core--interface overlap the
Higgs equation is a perturbation of
\[
        D''(\eta)=(\beta\eta^2-\beta)D(\eta).
\]
\red{In Section~\ref{sec:trueoverlap} we construct the radial Weber solution and also prove estimates for the resulting errors and their derivatives in the corresponding asymptotic expansions. The regularity condition at the origin then allows us to infer that only the recessive branch appears in our solution. In Lemma~\ref{lem:webermatching}, we prove using a Volterra argument that the corresponding interface coefficient is $C_\beta$. In particular, this shows that no additional connection coefficient appears when passing from the core to the interface. 

As for the superconducting side, Proposition~\ref{prop:righttail} shows that the interface exhibits two exponential tails.} The finite radial equations
are compared instead with the recessive modified Bessel functions in
Proposition~\ref{prop:exteriorBessel}.  Lemma
\ref{lem:saturatingBesselComparison} compares the normalized Bessel and
exponential solutions from a fixed point $y=K$ to infinity.  Lemma
\ref{lem:rightTailManifold} expresses the consequence at $y=K$: the Cauchy
data of the decaying radial tails form a two-parameter $C^2$ family which is
$O(R_n^{-1})$-close to the corresponding planar family, and its tangent
plane at the CHMO data is $E_\beta^+(K)$.

\medskip
\noindent\emph{First use of the virial identity.}
The quantitative comparison requires
\[
        R_n=R_n^{(0)}+O_\beta(1),
        \qquad
        \frac{2n}{R_n^2}=\sqrt\beta+O_\beta(R_n^{-1}).
\]
These estimates are proved before the quantitative boundary argument.  For
$\lambda>0$, consider the admissible dilation
\[
        \phi_{n,\lambda}(u)=\phi_n(u/\lambda),
        \qquad
        a_{n,\lambda}(u)=a_n(u/\lambda).
\]
The magnetic and potential contributions to the radial energy become
$\lambda^{-2}\calM_n$ and $\lambda^2\calV_n$, \red{where $\calM_n$ and $\calV_n$ are the magnetic and potential terms in \eqref{eq:energy},} while the two remaining
contributions are invariant.  Since the original profiles minimize the
energy, differentiation at $\lambda=1$ gives the virial, or Pohozaev,
identity
\[
        \calM_n=\calV_n
\]
of Lemma~\ref{lem:scalingVirial}.  Combining it with
\eqref{eq:magneticCouplingIdentity} gives the exact reduced identity
\[
 0=-\frac{2n^2}{R_n^2}
   +2\int_{-R_n}^{\infty}(R_n+y)j_n(y)\,\dd y,
\]
which is \eqref{eq:exactReducedVirialDensity}.

For the first evaluation of this identity, the splitting point $-M$ is
fixed.  The core estimate, qualitative convergence on fixed intervals, and
uniform right-tail decay give
\[
        F_n(R_n)\longrightarrow0,
        \qquad
        F_n(R)=-\frac{2n^2}{R^3}+\frac{\beta}{2}R+s_\beta.
\]
Since
\[
        F_n'(R)=\frac{6n^2}{R^4}+\frac{\beta}{2}
        \ge \frac{\beta}{2},
\]
comparison with the unique zero of $F_n$ proves
\eqref{eq:SLhyp}, and this yields
Lemma~\ref{lem:preliminaryLocalization}.  No quantitative boundary or tail
estimate is used in this first calculation.

\medskip
\noindent\emph{Quantitative convergence and propagation into the tails.}
With \eqref{eq:SLhyp} available, set
\[
        V_n=(\Phi_n-\Phi_\beta,\Gamma_n-\Gamma_\beta).
\]
On $[-K,K]$, subtraction of the radial and interface equations gives
\[
        \calA_\beta V_n
        =\calR_n+\calQ_\beta^{\rm int}(V_n),
        \qquad
        \|\calR_n\|_{H^1([-K,K])}=O_{K,\beta}(R_n^{-1}),
\]
where $\calQ_\beta^{\rm int}(V)$ is at least quadratic in $V$.
The rank-three property of Proposition~\ref{prop:nearBPS} first gives the
linear estimate \eqref{eq:finiteWindowOverdet}: a function on $[-K,K]$ is
controlled by its linearized equation and by the distances of its two
Cauchy traces from $E_\beta^-(-K)$ and $E_\beta^+(K)$.  The proof is a
compactness contradiction, in the sense that a failure of the estimate would produce a
nonzero homogeneous solution satisfying both endpoint conditions.

It remains to estimate those two distances for $V_n$.  Regularity at $u=0$
determines a one-parameter family of Cauchy data at $y=-K$.
Section~\ref{sec:quantitativeBoundary} constructs this family on a
logarithmically long interval and proves that its trace map is uniformly
$O(R_n^{-1})$-close to the corresponding limiting trace map.  The latter is
differentiable at the CHMO data, and its tangent line is
$E_\beta^-(-K)$.  The construction of the tail to the right in
Lemma~\ref{lem:rightTailManifold} provides the analogous uniform comparison at
$y=K$, and the derivative of the limiting trace map has image
$E_\beta^+(K)$.  Qualitative convergence implies that the actual vortex  parameters lie
near the base points of these limiting maps.  Differentiability therefore
leads to sequences $\varepsilon_{n,K}^\pm\to0$ such that the two projected
endpoint traces are bounded by $O(R_n^{-1})$ plus
$\varepsilon_{n,K}^\pm$ times the corresponding full trace.  Applying
\eqref{eq:finiteWindowOverdet} to $V_n$ yields
\[
        \|V_n\|_{H^2([-K,K])}
        \le C_{K,\beta}R_n^{-1}
        +\varepsilon_{n,K}\|V_n\|_{H^2([-K,K])}
        +C_{K,\beta}\|V_n\|_{H^2([-K,K])}^2.
\]
\red{The q}ualitative convergence \red{of $(V_n)_{n\geq 1}$ to $0$ in $H^2([-K,K])$} \red{allows us to absorb} the last two terms and \red{prove}
Theorem~\ref{thm:boundaryconvergence}. \red{In Section~\ref{sec:leftMatching} we propagate this estimate into the logarithmically-sized region located towards the normal endpoint, where the solution is characterized by a decaying Gaussian, and towards the superconducting end, where the solution instead exhibits two decaying exponential tails. In the same section we obtain the uniform core--interface matching estimate that we shall use in the proof of Theorem \ref{thm:asymptoticShape}.}

\medskip
\noindent\emph{Refined energy and virial calculation.}
The second calculation uses the quantitative estimates just obtained.  Both
the exact energy identity \eqref{eq:exactReducedEnergy} and the exact virial
identity \eqref{eq:exactReducedVirialDensity} are split at
\[
        y=-C_{\log}\sqrt{\log R_n}.
\]
Corollary~\ref{cor:bulkExpansion} evaluates the core contribution, and
Lemma~\ref{lem:boundaryExteriorEnergy} replaces the remaining radial
densities by their interface limits with bounded error.  Lemma
\ref{lem:cutoffcancel} displays the cancellation of all terms depending on
the artificial splitting point.  Applied to the energy identity, these
estimates give
\[
        \calT_n
        =\frac{n^2}{R_n^2}
         +\frac{\beta}{4}R_n^2+s_\beta R_n+O(1).
\]
Applied independently to the virial identity, they give
\[
        R_n\left(
        -\frac{2n^2}{R_n^3}+\frac{\beta}{2}R_n+s_\beta
        \right)=O(1).
\]
These are the two conclusions of Proposition~\ref{prop:energyreduction}.
The second identity, rather than formal minimization of the displayed
approximate energy, determines the constant correction to $R_n$.
Expanding around $R_n^{(0)}$ gives \eqref{eq:Rasymp}, which in turn yields \eqref{eq:Tasymp} when plugged into the
energy formula.  Finally,
Theorem~\ref{thm:asymptoticShape} follows by combining the fixed-interval
estimate with the Weber description on the left and the modified-Bessel
description on the right.

\section{Static radial theory}

We use one classical existence theorem and prove the qualitative properties needed below directly from the radial
equations.  The variational construction of Berger--Chen
\cite{BergerChen}, in the precise form recalled in
Chen--Spirn \cite[Lemma~7.1 and Theorem~7.2]{ChenSpirn}, shows that the infimum of the tension is attained 
over radial equivariant profiles.  We do not rely on a general uniqueness theorem, which seems to be unknown away from the BPS point $\beta=1$ (see however~\cite{ABG} for uniqueness for large~$\beta$).

To be precise, we denote the admissible class for the minimization of the tension~\eqref{eq:energy}  by $\mathscr A_n$. It is defined as
\begin{align}
    \red{\mathscr A_n} &:= \{ (\phi,a)\in H^1_{\rm loc}((0,\infty))^2 \text{\ \ non-negative pairs for which the four limits in \eqref{eq:bc} } \nonumber \\
   &\text{ hold and each of the
nonnegative terms in \eqref{eq:energy} is integrable} \} \label{eq:AnClass}
\end{align}

In the notation
of Chen--Spirn, $f$ is the scalar profile and $q$ is the magnetic
profile.  The exact conversion is
\[
 r=\sqrt2\,u,\qquad f(r)=\phi(u),\qquad q(r)=1-a(u),
 \qquad \lambda_{\rm CS}^2=\frac{\beta}{4},
\]
under which $G_{\rm gl}^r(f,q)/\pi=\mathcal T_n[\phi,a]$.

\begin{Proposition}\label{prop:static}
Let $n\ge1$ and $\beta>0$.  The radial functional \eqref{eq:energy} attains its infimum in $\mathscr A_n$
subject to \eqref{eq:bc}.  Every radial minimizer $(\phi_n,a_n)$ is smooth on $(0,\infty)$, satisfies
\eqref{eq:vortex}, and obeys
\bel{eq:monotone}
 0<\phi_n(u)<1,\qquad 0<a_n(u)<1,\qquad
 \phi_n'(u)>0,\qquad a_n'(u)<0,\qquad u>0.
\ee
There exist unique $c_n>0$ and $R_n>0$ such that
\bel{eq:frob}
 \phi_n(u)=c_nu^n+O(u^{n+2}),\qquad
 a_n(u)=1-\frac{u^2}{R_n^2}+O(u^4),\qquad u\downarrow0.
\ee
For the fixed minimizer there are constants $c,C>0$ (depending on $\beta,n$) such that
\bel{eq:coarseRadialTail}
 \sum_{j=0}^2\left(
  |\partial_u^ja_n(u)|+|\partial_u^j(1-\phi_n(u))|
 \right)\le Ce^{-cu},\qquad u\ge1.
\ee
\end{Proposition}

\begin{proof}
For the solution of the minimization problem in the class~\eqref{eq:AnClass} see~\cite{BergerChen} and
\cite[Theorem~7.2]{ChenSpirn}.  Lemma~7.1 of the latter paper establishes continuity and the limits at the origin.
Since $1-f\in H^1_{\rm rad}(\mathbb R^2)$, the Strauss radial bound implies that $f(r)\to1$ as $r\to\infty$.  The magnetic equation and finite
energy imply $q(r)\to1$, while Frobenius theory 
gives the smooth equivariant extension through the origin used below. We recall why every minimizer
has the properties~\eqref{eq:monotone}.  
Capping
the resulting nonnegative profiles by~$f\mapsto\min(f,1)$ cannot increase any derivative term and decreases each zeroth-order term on
the set where the cap by~$1$ decreases the profiles.  
Equality excludes a set of positive measure on
which either absolute-value profile exceeds~$1$, and continuity gives the
pointwise upper bounds.  
Standard ODE arguments imply that in fact $0<\phi_n,a_n<1$; we omit the routine details.
Put $g_n=a_n'/u$.  The magnetic equation gives
\[
 g_n'(u)=\frac{2a_n(u)\phi_n(u)^2}{u}>0.
\]
If $g_n(u_0)\ge0$ at some point, then $g_n>0$ for all $u>u_0$ and hence $a_n'>0$, contradicting
$a_n>0$ and $a_n(u)\to0$.  Therefore $a_n'<0$.
Frobenius theory at $u=0$, applied after smooth equivariant
extension through the origin, gives
\[
 \phi_n=c_nu^n+O(u^{n+2}),\qquad a_n=1-\alpha_nu^2+O(u^4).
\]
If $c_n=0$, uniqueness for the regular singular initial-value problem
would give $\phi_n\equiv0$, contrary to its limit at infinity.  Hence
$c_n>0$ and $\phi_n'>0$ for all sufficiently small positive $u$.  Since
$g_n<0$ and $g_n'>0$, its finite limit $g_n(0)=-2\alpha_n$ is strictly
negative.  Thus $\alpha_n>0$ and $R_n=\alpha_n^{-1/2}$.
To prove monotonicity define
\[
 Q_n(u)=\frac{n^2a_n(u)^2}{u^2}+\beta(\phi_n(u)^2-1).
\]
Then $(u\phi_n')'=uQ_n\phi_n$, and the expansion just proved gives $\phi_n'>0$ for small $u$.  Suppose that $u_0$
is its first zero.  The scalar equation gives $Q_n(u_0)\le0$.  On any interval immediately to the right on which
$\phi_n'<0$,
\[
 Q_n'=\frac{2n^2a_na_n'}{u^2}-\frac{2n^2a_n^2}{u^3}+2\beta\phi_n\phi_n'<0.
\]
Consequently $Q_n<0$ and $u\phi_n'$ is strictly decreasing there.  It can never return to zero, so $\phi_n$ would
remain strictly decreasing and could not tend to $1>\phi_n(u_0)$.  If $Q_n(u_0)=0$, the same reasoning gives
$Q_n'(u_0)<0$ and the same conclusion.  Therefore $\phi_n'>0$ everywhere.

Choose $u_0$ so large that $\phi_n\ge1/2$ on $[u_0,\infty)$.  Fix
$0<\mu<\widetilde\mu<1/\sqrt2$.  Let $w(u)=uK_1(\widetilde\mu u)$, where $K_1$ is the modified Bessel function of the second kind of order $1$.  Then $w$ satisfies
\[
        w''-u^{-1}w'-\widetilde\mu^2w=0
\]
and is a positive supersolution for the magnetic equation on this half-line.
Comparison gives $a_n(u)\le Cw(u)$.  The standard Bessel asymptotic and
its differentiated form yield
\[
        w(u)+|w'(u)|\le C u^{1/2}e^{-\widetilde\mu u}\le Ce^{-\mu u}.
\]
Thus $|a_n|+|a_n'|\le Ce^{-\mu u}$.  For
$\eta_n=1-\phi_n$ one has
\[
 \eta_n''+\frac{1}{u}\eta_n'-\beta(1-\eta_n)(2-\eta_n)\eta_n
 =-\frac{n^2a_n^2}{u^2}(1-\eta_n).
\]
Set $c_{\beta,n}(u)=\beta(1-\eta_n(u))(2-\eta_n(u))$.  After increasing $u_0$,
$c_{\beta,n}(u)\ge3\beta/4$ on $[u_0,\infty)$.  Choose
$0<\nu<\min(2\mu,\sqrt{3\beta/4})$ and denote by $G_\nu$ the positive
decaying Green kernel, normalized with respect to the Lebesgue measure $ds$, of
$-\partial_u^2-u^{-1}\partial_u+\nu^2$ on $[u_0,\infty)$ with zero
Dirichlet trace at $u_0$.
\red{To be more specific, let $K_0$ and $I_0$ be the modified Bessel functions of the second and first kinds of order $0$, so that $K_0$ has exponential decay at $+\infty$ and $I_0$ has exponential increase at $+\infty$. They are the fundamental solutions of the equation 
\begin{align*}\psi''(u)+\frac{1}{u}\psi'(u)-\psi(u)=0.
\end{align*}
Then, we have the relation
\[ G_\nu(u,v)=\begin{cases}vK_0(\nu v)\left(I_0(\nu u)-K_0(\nu u)\frac{I_0(\nu u_0)}{K_0(\nu u_0)}\right),\text{ if }u<v\\
vK_0(\nu u)\left(I_0(\nu v)-K_0(\nu v)\frac{I_0(\nu u_0)}{K_0(\nu u_0)}\right),\text{ if }v<u.\end{cases}\]
}

A sufficiently large multiple of the decaying
homogeneous solution, plus
\[
 \int_{u_0}^\infty G_\nu(u,s)
 \frac{n^2a_n(s)^2}{s^2}(1-\eta_n(s))\,ds,
\]
is a supersolution for $\eta_n$, since $c_{\beta,n}-\nu^2>0$.  The maximum
principle gives $0\le\eta_n(u)\le Ce^{-\nu u}$.  To obtain the derivative
bound without differentiating this supersolution, put
$p_n=-\eta_n'=\phi_n'\ge0$ and
$f_n=n^2a_n^2(1-\eta_n)/u^2$.  The exact defect equation gives
\[
        (u p_n)'=u(f_n-c_{\beta,n}\eta_n).
\]
The right side is exponentially integrable.  Hence $u p_n$ has a finite
limit at infinity which must be zero, since a positive limit would
give $p_n(u)\gtrsim u^{-1}$ and contradict $\eta_n(u)\to0$.  Therefore
\[
 u p_n(u)=\int_u^\infty s\red{(}c_{\beta,n}(s)\eta_n(s)-f_n(s)\red{)}\,ds.
\]
The already established bounds for $a_n$ and $\eta_n$ imply
$|\eta_n'(u)|\le Ce^{-cu}$ after decreasing $c>0$ if necessary.  Substituting this estimate into the two differential equations gives the second-derivative bounds and
proves \eqref{eq:coarseRadialTail}. \end{proof}

\begin{Definition}
\blue{The positive number $R_n$ in \eqref{eq:frob} is the \emph{magnetic core radius} (sometimes called the critical radius) used throughout the paper.  It is defined from the quadratic coefficient of the gauge potential profile $a_n$ at the origin; later we prove that translating by $R_n$ centers the transition (boundary) layer.}
\end{Definition}
\red{For the rest of the paper, we shall use the previous proposition in the asymptotic analysis of arbitrary radial minimizers. In the special case $\beta=1$, which is also known as the BPS regime, we can reduce this second order system to a first order one in \eqref{eq:BPS} via what is also known as a Bogomolny square completion. This is exactly Taubes' vortex setting, for which uniqueness is already known; see \cite[Ch. III]{JaffeTaubes}, \cite{Taubes}. We must also point out that the aforementioned reduction can only be carried out in the case $\beta=1$. }
  The radial ansatz leads instead to the singular second-order boundary-value problem
\[
 \phi''+\frac{1}{u}\phi'
 =
 \frac{n^2}{u^2}a^2\phi+\beta\phi(\phi^2-1),\qquad
 a''-\frac{1}{u}a'=2a\phi^2,
\]
with local data at the origin
\[
        \phi(u)=c u^n+O(u^{n+2}),\qquad
        a(u)=1-\alpha u^2+O(u^4),\qquad c,\alpha>0,
\]
and two conditions at infinity, $\phi(\infty)=1$, $a(\infty)=0$.  Thus uniqueness away from the BPS point would mean that
this two-parameter shooting map has at most one admissible zero.  In general this is unknown.

\blue{Consequently, throughout the rest of the paper, $(\phi_n,a_n)$ denotes an arbitrary   minimizer as in Proposition~\ref{prop:static}.  Our arguments apply to every such minimizer, so the asymptotic formulae in
Theorem~\ref{thm:main} hold for all of them.  In particular, we do not rely on any global uniqueness theorem for solutions of the Euler-Lagrange equations with the boundary conditions~\eqref{eq:bc} if $\beta\ne1$.  The uniqueness result of Chapman--Howison--McLeod--Ockendon~\cite{CHMO} is \red{only needed to conclude that the interfaces arising as limit points of the sequence of profiles $(\Phi_n,\Gamma_n)$, obtained upon centering at the magnetic radii $R_n$, coincide.}}

At the BPS point $\beta=1$ we have the following Bogomolny identity, 
\bel{eq:BPScompletion}
\calT_n[\phi,a]
=
\int_0^\infty u\left[
 \frac{1}{2}\left(\frac{n}{u}a'-(\phi^2-1)\right)^2
 +\left(\phi'-\frac{n}{u}a\phi\right)^2
\right]\dd u+n .
\ee
Thus $\calT_n\ge n$, with equality precisely for
\bel{eq:BPS}
       \frac{n}{u}a'=\phi^2-1,\qquad
       \phi'=\frac{n}{u}a\phi .
\ee
The Taubes solution  is rotationally symmetric, belongs
to $\mathscr A_n$, and satisfies \eqref{eq:BPS}. Hence the radial infimum is
exactly $n$.  Every radial minimizer has this energy and therefore makes both
squares in \eqref{eq:BPScompletion} vanish.  Consequently, the coefficient of
$n^{1/2}$ in the large-flux expansion vanishes at $\beta=1$.  At this BPS point several arguments can be simplified, but we do not pursue this here.

\section{The planar interface}

\begin{Definition}\label{def:PhiGamma}
    Let $R_n$ be the magnetic radius from \eqref{eq:frob} and introduce the centered variables
\bel{eq:boundaryvars}
       y=u-R_n,\qquad
       \Phi_n(y)=\phi_n(R_n+y),\qquad
       \Gamma_n(y)=\frac{n}{R_n+y}a_n(R_n+y).
\ee
\end{Definition}

A direct computation gives the exact equations
\begin{align}
 \Phi_n''\red{(y)}+\frac{1}{R_n+y}\Phi_n'\red{(y)}
 &=\Gamma_n^2\red{(y)}\Phi_n\red{(y)}+\beta\Phi_n\red{(y)}(\Phi_n^2\red{(y)}-1),\label{eq:exactPhi}\\
 \Gamma_n''\red{(y)}+\frac{1}{R_n+y}\Gamma_n'\red{(y)}-\frac{1}{(R_n+y)^2}\Gamma_n\red{(y)}
 &=2\Gamma_n\red{(y)}\Phi_n^2\red{(y)} .\label{eq:exactGamma}
\end{align}
\blue{Since $R_n\to\infty$, the limiting interface system is \eqref{eq:interface}, viz.
\[ 
       \Phi''=\Gamma^2\Phi+\beta\Phi(\Phi^2-1),\qquad
       \Gamma''=2\Gamma\Phi^2 ,
\]
with endpoint conditions \eqref{eq:interfaceleft}--\eqref{eq:interfaceright}. We will prove in the next section that in fact $\Phi_n\to\Phi$ and $\Gamma_n \to \Gamma$.}
The condition at $-\infty$ fixes translation: translating by $y_0$ changes the limiting affine
constant of $\Gamma+\sqrt\beta\,y$ by $-\sqrt\beta\,y_0$.

\subsection{Identification with the classical normal-to-superconducting transition}

 In the notation of
Chapman--Howison--McLeod--Ockendon \cite{CHMO}, the transition equations are
\bel{eq:CHMOeq}
 \Psi_{XX}=\kappa^2\Psi(\Psi^2-1+A^2),\qquad
 A_{XX}=\Psi^2A,
\ee
and 
\bel{eq:CHMObc}
 (\Psi,A)\rightarrow(1,0)\qquad(X\to-\infty),
 \qquad
 \Psi\rightarrow0,\qquad A_X\rightarrow2^{-1/2}\quad (X\to+\infty).
\ee
Their independent variable therefore points from the superconducting phase
toward the normal phase, whereas $y$ in this paper points in the opposite
direction.  We  use the change of variables 
\bel{eq:CHMOscaling}
 X=-\sqrt2\,y,\qquad \kappa^2=\frac{\beta}{2},\qquad
 \Phi_\beta(y)=\Psi(-\sqrt2 y),\qquad
 \Gamma_\beta(y)=\sqrt\beta\,A(-\sqrt2 y), 
\ee
leading to \eqref{eq:interface}.  If
$A(X)=X/\sqrt2+c+o(1)$ at the normal end, then
\[
 \Gamma_\beta(y)=-\sqrt\beta\,y+\sqrt\beta\,c+o(1)
 \qquad(y\to-\infty).
\]
The translation of the CHMO orbit is chosen so that $c=0$.  Notice in
particular that the applied field in the CHMO normalization is $A_X=1/\sqrt2$. The factor depending on $\kappa$ appears only after the rescaling
$q=\kappa A$ sometimes used in the literature.
We use the following form\footnote{The validity of Lemma~2.4 in~\cite{CHMO} for $\kappa\ge\sqrt{2}$ is not clear to the authors.} of the main CHMO transition-layer theorem.

\begin{Theorem}[\cite{CHMO}]\label{thm:CHMO}
For every $\sqrt{2}>\kappa>0$, the boundary value problem
\eqref{eq:CHMOeq}, \eqref{eq:CHMObc} has a solution which is unique up to
translation.  After fixing the affine constant of $A-X/\sqrt2$ at the normal
end, the solution is unique.  It is monotone:
\[
 0<\Psi<1,\qquad \Psi_X<0,\qquad A>0,\qquad A_X>0,\qquad A_{XX}>0 .
\]
\end{Theorem}

After undoing the scaling, the CHMO result gives the interface orbit in the normalization used in this paper.

\begin{Theorem}\label{thm:interface}
For every $0<\beta<4$ there exists a unique solution $(\Phi_\beta,\Gamma_\beta)$ of
\eqref{eq:interface} satisfying \eqref{eq:interfaceleft}, \eqref{eq:interfaceright}.  It obeys
\bel{eq:interface-monotone}
        0<\Phi_\beta<1,\qquad \Phi_\beta'>0,\qquad
        \Gamma_\beta>0,\qquad -\sqrt\beta<\Gamma_\beta'<0,\qquad \Gamma_\beta''>0 .
\ee
The Hamiltonian identity
\bel{eq:hamzero}
        \Phi_\beta'^2+\frac{1}{2}\Gamma_\beta'^2
        =
        \Gamma_\beta^2\Phi_\beta^2+\frac{\beta}{2}(1-\Phi_\beta^2)^2
\ee
holds for all $y\in\RR$.
\end{Theorem}

\begin{proof}  Existence, uniqueness and monotonicity follow from Theorem~\ref{thm:CHMO} with  the
scaling \eqref{eq:CHMOscaling}.  It remains to verify~\eqref{eq:hamzero}.  From
\eqref{eq:interface} we see that
\[
        H=\Phi'^2+\frac{1}{2}\Gamma'^2
        -\Gamma^2\Phi^2-\frac{\beta}{2}(1-\Phi^2)^2
\]
is constant.  
Because $\Gamma$ is positive, decreasing, and tends to zero, while
$\Gamma''>0$, one has $\Gamma'(y)\to0$ as $y\to+\infty$.  Also
$\Phi'\ge0$ is integrable on $[0,\infty)$ and $\Phi''$ is bounded there by
the ODE.  Hence $\Phi'$ is uniformly continuous and
$\Phi'(y)\to0$.  Together with $(\Phi,\Gamma)\to(1,0)$ this gives $H=0$ as claimed. \end{proof}  

The endpoint estimates needed below are uniform when the coupling remains
in a compact subset of $(0,4)$.  The following lemma suffices for our purposes. 

\begin{Lemma}\label{lem:interfaceUniformTails}
Let $J\Subset(0,4)$ and let $k\ge0$ be fixed.  There are
$c_J,C_{J,k}>0$ and an integer $N_k$ such that, for every $\beta\in J$,
\bel{eq:uniformInterfaceLeft}
 \abs{\partial_y^k\Phi_\beta(y)}
 +\abs{\partial_y^k(\Gamma_\beta(y)+\sqrt\beta\,y)}
 \le C_{J,k}\la y\ra^{N_k}e^{-c_Jy^2},\qquad y\le-1,
\ee
and
\bel{eq:uniformInterfaceRight}
 \abs{\partial_y^k(1-\Phi_\beta(y))}
 +\abs{\partial_y^k\Gamma_\beta(y)}
 \le C_{J,k}e^{-c_Jy},\qquad y\ge1.
\ee
After slightly reducing the constants $c_J$, the corresponding weighted quantities in these estimates depend continuously on $\beta\in J$ in the associated weighted $C^k$ norms.
\end{Lemma}

\begin{proof}
Fix the translate by imposing $\Phi_\beta(0)=1/2$.  Let $\beta_-:=\min J>0$ and
$\beta_+:=\max J$.  By monotonicity,
\[
-\sqrt{\beta_+}\le \Gamma_\beta'(0)<0.
\]
We set  $G:=\Gamma_\beta(0)$.  If $G\to\infty$, then for any fixed $A>0$ and
$0\le y\le A/G$ one has $\Gamma_\beta(y)\ge G/2$ for $G$ large.  Since
$\Phi_\beta\ge1/2$ on this interval, the scalar equation yields
$\Phi_\beta''\ge G^2/16$, and integrating twice gives $\Phi_\beta(A/G)>1$ for
$A$ large, a contradiction.  Hence $\Gamma_\beta(0)$ is uniformly bounded on
$J$.  Using the Hamiltonian identity at $y=0$,
\[
(\Phi_\beta'(0))^2+\tfrac12(\Gamma_\beta'(0))^2
=\tfrac14\Gamma_\beta(0)^2+\tfrac{9\beta}{32},
\]
we also obtain a uniform bound for $\Phi_\beta'(0)$ and \red{$\Gamma_\beta'(0)$}.

Let $\beta_j\to\beta_*\in J$.  By the uniform  bounds on the Cauchy data, a
subsequence \newline \red{$(\Phi_{\beta_j}(0),\Gamma_{\beta_j}(0))$ of initial data} converges, hence $(\Phi_{\beta_j},\Gamma_{\beta_j})$
converges in $C^{k+2}_{\mathrm{loc}}$ to a solution $(\Phi_*,\Gamma_*)$ of
\eqref{eq:interface} with parameter $\beta_*$ and $\Phi_*(0)=1/2$.
Monotonicity gives limits $ \ell_+:=\lim_{y\to+\infty}\Phi_*$ and
$g_+:=\lim_{y\to+\infty}\Gamma_*$.  Since $\ell_+\ge1/2$, the equation
$\Gamma_*''=2\Phi_*^2\Gamma_*$ rules out $g_+>0$, and the scalar equation then
forces $\ell_+=1$.  Thus, $(\Phi_*,\Gamma_*)\to(1,0)$ as $y\to+\infty$.

On $\mathbb R$, the maximum principle and convexity imply $\Gamma_*>0$ and
$\Gamma_*'<0$ at every finite point: an interior zero would mean that
$\Gamma_*\equiv0$, and an interior critical point would force $\Gamma_*'$ to
vanish on $[y_0,\infty)$, both cases leading to contradictions.

As $y\to-\infty$, the limit $d_-:=\lim\Gamma_*'(y)$ exists since
$\Gamma_*''\ge0$, and $d_-<0$ because $\Gamma_*'<0$ on $\mathbb R$.
\red{For every $\beta_j$ and $y\in\mathbb{R}$, we have
\begin{align*}
  -\sqrt{\beta_j}<\Gamma'_{\beta_j}(y)<0,  
\end{align*}
and letting $j\rightarrow\infty$,
\begin{align*}
  -\sqrt{\beta_\ast}\leq\Gamma'_{\ast}(y)<0 . 
\end{align*}
}
Hence, \red{$d_-$ is finite, and}  $\Gamma_*(y)$ grows linearly to $+\infty$ as $y\to-\infty$.
If $\ell_-:=\lim_{y\to-\infty}\Phi_*(y)$ were positive, then
$\Gamma_*''=2\Phi_*^2\Gamma_*$ would not \red{be} integrable on
$(-\infty,0]$, \red{and $\Gamma'_\ast$ would not have finite limit at $-\infty$.} \red{This means that} $\ell_-=0$.  A Gaussian comparison for the scalar equation
then establishes that 
\[
\Phi_*(y)+|\Phi_*'(y)|+|\Gamma_*(y)\Phi_*(y)|\to0\qquad (y\to-\infty).
\]
Passing to the limit in $H=0$ yields $d_-^2=\beta_*$, hence
$d_-=-\sqrt{\beta_*}$.  Integrating
$$(\Gamma_*(y)+\sqrt{\beta_*}\,y)''=2\Gamma_*(y)\Phi_*^2(y)$$ twice shows that
$\Gamma_*(y)+\sqrt{\beta_*}\,y$ has a finite limit as $y\to-\infty$.  Therefore, 
$(\Phi_*,\Gamma_*)$ is a translate of the CHMO orbit at $\beta_*$, and the
normalization $\Phi_*(0)=1/2$ fixes it uniquely.  It follows that
we do not need to pass to subsequences, and the normalized orbit
depends continuously on $\beta$.

For the tails, note that $\Phi_\beta\ge1/2$ on $[0,\infty)$, so
$\Gamma_\beta''\ge \Gamma_\beta/2$ there.  A standard comparison argument then yields the
exponential bound in \eqref{eq:uniformInterfaceRight}, and the equation for
$1-\Phi_\beta$ gives the remaining term, uniformly for $\beta\in J$.
On the left, fix $\beta_*\in J$ and choose $y_*<0$ with
$\Gamma_{\beta_*}'(y_*)<-\sqrt{\beta_*}/2$.  By  continuity
and a finite cover of $J$, there exist $Y_J,c_J>0$ such that for all
$\beta\in J$,
\[
\Gamma_\beta'(y)\le -c_J,\qquad \Gamma_\beta(y)\ge c_J|y|,\qquad y\le -Y_J.
\]
A Gaussian supersolution for the scalar equation gives the first term in
\eqref{eq:uniformInterfaceLeft}.  Define
\[
 b_\beta:=\lim_{y\to-\infty}\bigl(\Gamma_\beta(y)+\sqrt\beta\,y\bigr),
 \qquad g_\beta(y):=\Gamma_\beta(y)+\sqrt\beta\,y-b_\beta.
\]
Then $g_\beta''=2\Gamma_\beta\Phi_\beta^2$ and
$g_\beta(-\infty)=g_\beta'(-\infty)=0$, so two integrations give the
corresponding bounds for $g_\beta$.  Differentiating the system yields the
bounds for $\partial_y^k$.
In addition,
\[
 b_\beta=\Gamma_\beta(0)-\int_{-\infty}^0\int_{-\infty}^s
 2\Gamma_\beta(t)\Phi_\beta(t)^2\,\dd t\,\dd s,
\]
so $\beta\mapsto b_\beta$ is continuous by dominated convergence.  The
normalized orbit is obtained by translating by $b_\beta/\sqrt\beta$. 
This translation is continuous and uniformly bounded on $J$, so the uniform
estimates apply to the affine normalization.

Finally, fix weights strictly smaller than the Gaussian and exponential rates.
Given $\varepsilon>0$, choose $M$ so that the weighted $C^k$ norms on
$(-\infty,-M]\cup[M,\infty)$ are at most $\varepsilon$, uniformly in $\beta$.
On $[-M,M]$, continuous dependence yields convergence in $C^k$.  This proves
continuity in the weighted $C^k$ norms.
\end{proof}

Recall the interface energy density $e_\beta$, the renormalized surface
energy $s_\beta$, and $\sigma_\beta$ from \eqref{eq:edens}--\eqref{eq:sigma}.
The existence of the limit in \eqref{eq:sbeta} follows already from
Lemma~\ref{lem:interfaceUniformTails}. 

\subsection{The normalized linearization}

The linearized system at the interface orbit is
\bel{eq:linode}
\begin{split}
        \xi''&=(\Gamma_\beta^2+\beta(3\Phi_\beta^2-1))\xi
             +2\Gamma_\beta\Phi_\beta\psi,\\
        \psi''&=4\Gamma_\beta\Phi_\beta\xi+2\Phi_\beta^2\psi .
\end{split}
\ee
We write
\bel{eq:Alinearization}
        \calA_\beta
        \begin{pmatrix}\xi\\ \psi\end{pmatrix}
        =
        \begin{pmatrix}
        \xi''-(\Gamma_\beta^2+\beta(3\Phi_\beta^2-1))\xi-2\Gamma_\beta\Phi_\beta\psi\\
        \psi''-4\Gamma_\beta\Phi_\beta\xi-2\Phi_\beta^2\psi
        \end{pmatrix}.
\ee
Translation \blue{invariance} gives \red{
$\calA_\beta(\Phi_\beta',\Gamma_\beta')^T=0$.}
This bounded mode does not satisfy the following normalization at $-\infty$,
because \red{$\lim_{\substack{y\rightarrow-\infty}}\Gamma_\beta'(y)=-\sqrt\beta$}.  Specifically, we impose the vanishing  
\bel{eq:normalized}
        \psi(-\infty)=0,\qquad \psi'(-\infty)=0 .
\ee
We shall now see that the left admissible data form a line and the right admissible data form a
plane.  Non-degeneracy is therefore equivalent to a maximal rank condition for
a $4\times3$ linear matching map between these spaces, as we shall see below. 

The following lemma constructs the two-dimensional stable subspace (solutions
which decay as $y\to+\infty$) by the Lyapunov--Perron method.  We present the
construction uniformly on compact subintervals of $(0,4)$.

\begin{Lemma}\label{lem:rightstable}
Let $J\Subset(0,4)$ be a compact interval.  There exist numbers
$Y_1<\infty$ and $c>0$ such that the following statements hold.
For every $\beta\in J$ and
$Y\ge Y_1$, the solutions of \eqref{eq:linode} which decay as
$y\to+\infty$ form a two-dimensional plane
\[
        E_\beta^+(Y)\subset\RR^4
\]
of Cauchy data at $y=Y$.  This plane depends continuously on
$\beta$.  It has a continuously chosen basis
$Z_{H,\beta},Z_{G,\beta}$ satisfying
\[
\begin{split}
Z_{H,\beta}(y)
&=e^{-\sqrt{2\beta}y}
  \left((1,0,-\sqrt{2\beta},0)^t+O(e^{-cy})\right),\\
Z_{G,\beta}(y)
&=e^{-\sqrt2 y}
  \left((0,1,0,-\sqrt2)^t+O(e^{-cy})\right),
\end{split}
\qquad y\ge Y.
\]
The estimates are uniform for $\beta\in J$, and remain valid after
differentiating these asymptotics with respect to~$y$.
\end{Lemma}

\begin{proof}
Write \eqref{eq:linode} as
\[
        Z'=A_\beta(y)Z,
        \qquad Z=(\xi,\psi,\xi',\psi')^t.
\]
The uniform right-tail estimates for the interface orbit give
\[
        A_\beta(y)=A_\beta^+ +R_\beta^+(y),
        \qquad
        \|\partial_y^jR_\beta^+(y)\|
        \le C_j e^{-\gamma y},
        \quad y\ge Y_1,\quad j=0,1,
\]
for some $\gamma>0$, uniformly for $\beta\in J$, by
Lemma~\ref{lem:interfaceUniformTails}.

Set
\[
        a_\beta=\sqrt{2\beta},
        \qquad b=\sqrt2.
\]
The stable eigenvalues of $A_\beta^+$ are $-a_\beta$ and $-b$,
while its unstable eigenvalues are $a_\beta$ and $b$.  The stable
and unstable spectral subspaces are uniformly separated for
$\beta$ in a compact subset of $(0,\infty)$.  At $\beta=1$ the
two stable eigenvalues coincide, but the geometric multiplicity remains~$2$. 

Let $T_\beta$ be the matrix whose columns are the  
stable eigenvectors followed by the corresponding unstable
eigenvectors.  This matrix and its inverse depend continuously on
$\beta$ and are uniformly bounded on $J$.  In the coordinates
\[
        T_\beta^{-1}Z=\begin{pmatrix}U\\V\end{pmatrix},
        \qquad U,V\in\RR^2,
\]
the equation becomes
\[
 \begin{pmatrix}U\\V\end{pmatrix}'
 =
 \begin{pmatrix}D_s&0\\0&D_u\end{pmatrix}
 \begin{pmatrix}U\\V\end{pmatrix}
 +\widetilde R_\beta(y)
 \begin{pmatrix}U\\V\end{pmatrix},
\]
where
\[
        D_s=\operatorname{diag}(-a_\beta,-b),
        \qquad
        D_u=\operatorname{diag}(a_\beta,b),
\]
and
\[
        \|\partial_y^j\widetilde R_\beta(y)\|
        \le C_j e^{-\gamma y},
        \qquad j=0,1.
\]
Choose
\[
0<\mu<
m_*:=\min_{\beta\in J}\min\{a_\beta,b\}.
\]
For prescribed $U(Y)=U_0$, a solution having no growing component
at $+\infty$ must satisfy
\[
\begin{split}
U(y)
&=e^{D_s(y-Y)}U_0
  +\int_Y^y e^{D_s(y-s)}
       \pi_s\widetilde R_\beta(s)(U,V)(s)\,\dd s,\\
V(y)
&=-\int_y^\infty e^{D_u(y-s)}
       \pi_u\widetilde R_\beta(s)(U,V)(s)\,\dd s,
\end{split}
\]
where $\pi_s,\pi_u$ denote the two coordinate projections.  On the
space with norm
\[
        \|(U,V)\|_{\mu,Y}
        =
        \sup_{y\ge Y}
        e^{\mu(y-Y)}\bigl|(U,V)(y)\bigr|,
\]
the linear part of this map is bounded, and its contraction constant
is at most
\[
        C\int_Y^\infty e^{-\gamma s}\,\dd s
        \le C e^{-\gamma Y}.
\]
After increasing $Y_1$, this number is strictly less than $1$,
uniformly for $\beta\in J$.  A contraction mapping argument that is uniform in $\beta$
gives a unique fixed point for each $U_0$, depending
continuously on $(\beta,U_0)$.  Since $U_0\in\RR^2$ is arbitrary,
the decaying Cauchy data form a two-dimensional plane.

\green{In order to see why} every decaying solution is obtained \green{in this way, we} first note
that \green{any such solution must be} bounded.  Variation of constants in the unstable coordinates
then gives the second integral equation above \green{while the first one follows similarly, taking into account the Cauchy data $U(Y)=U_0$}. \green{When $Y_1$ is large, this resulting }map is also a contraction in the
unweighted supremum norm.  \green{By uniqueness, it follows that the resulting }
bounded solution \green {must coincide} with the weighted fixed point.  \green{Now, the} decay \green{properties of} $\xi$ \green{and} $\psi$, \green{immediately imply that }their equations give $\xi'',\psi''\to0$. \green{Applying the Mean Value Theorem to each interval} $[y,y+1]$ \green{immediately implies that} $\xi',\psi'\to0$.

We next normalize its two modes.  It is important here to retain the
faster decay of the off-diagonal coefficients.  \green{We s}et
\[
 \Delta_*:=\sup_{\beta\in J}|a_\beta-b|<b,
 \qquad \lambda:=\tfrac12(b+\Delta_*),
\]
so that $\Delta_*<\lambda<b$. \green{Since
$\displaystyle \lim_{\substack{y\rightarrow\infty}}\Phi_\beta=1$} \green{uniformly in $\beta\in J$, we deduce that} there \green{exists} $Y_2\ge Y_1$ such that
$2\Phi_\beta^2\ge\lambda^2$ on $[Y_2,\infty)$ for all $\beta\in J$.
Since $\Gamma_\beta>0$, $\Gamma_\beta\to0$, and
$\Gamma_\beta''=2\Phi_\beta^2\Gamma_\beta$, \green{a direct} comparison yields
\[
 0<\Gamma_\beta(y)
 \le\Gamma_\beta(Y_2)e^{-\lambda(y-Y_2)}
 \le C_J e^{-\lambda y},\qquad y\ge Y_2.
\]
To see the comparison directly, \green{we} subtract the middle expression from
$\Gamma_\beta$.  The difference vanishes at $Y_2$ and at infinity
and satisfies $w''-\lambda^2w\ge0$, so it cannot have a positive
maximum.  \green{We can also immediately see that }
$\Gamma_\beta'(y)=-\int_y^\infty2\Phi_\beta(s)^2\Gamma_\beta(s)\,\dd s$,
which gives the same bound for $|\Gamma_\beta'|$.

With $u=(\xi,\psi)^t$, $m_1=a_\beta$, and $m_2=b$, the equations read
\[
 u_i''-m_i^2u_i=\sum_{k=1}^2q_{ik,\beta}(y)u_k,
 \qquad
 (q_{ik,\beta})=
 \begin{pmatrix}
  \Gamma_\beta^2+3\beta(\Phi_\beta^2-1)&2\Gamma_\beta\Phi_\beta\\
  4\Gamma_\beta\Phi_\beta&2(\Phi_\beta^2-1)
 \end{pmatrix}.
\]
\green{For a sufficiently large constant} $M_0$, the \green{previous} estimates give
\[
 |q_{ii,\beta}(y)|\le M_0e^{-\gamma y},\qquad
 |q_{ik,\beta}(y)|\le M_0e^{-\lambda y}\quad(i\ne k),
 \qquad y\ge Y_2.
\]
\green{We s}et $\sigma:=\min\{\gamma,\lambda-\Delta_*\}>0$ and \green{we choose} $Y_1$
to be at least $Y_2$.  For $j=1,2$, \green{we} define
\[
 r_i^{(j)}:=\max\{m_i,m_j\},\qquad
 \|u\|_{j,\beta,Y}
 :=\max_{i=1,2}\sup_{y\ge Y}e^{r_i^{(j)}y}|u_i(y)|.
\]
\green{We c}onsider the linear integral equations
\[
 u_i(y)=\delta_{ij}e^{-m_jy}
 +\int_y^\infty\frac{\sinh(m_i(s-y))}{m_i}
       \sum_{k=1}^2q_{ik,\beta}(s)u_k(s)\,\dd s.
\]
The homogeneous term has norm one.  Since
$r_i^{(j)}\ge m_i$ and
$|r_i^{(j)}-r_k^{(j)}|\le\Delta_*$, the forcing \green{term in the integrand} satisfies
\[
 \left|\sum_{k=1}^2q_{ik,\beta}(s)u_k(s)\right|
 \le2M_0\|u\|_{j,\beta,Y}e^{-(r_i^{(j)}+\sigma)s}.
\]
For $r\ge m_i$ \green{we have the exact identity}
\[
 \int_y^\infty\frac{\sinh(m_i(s-y))}{m_i}
       e^{-(r+\sigma)s}\,\dd s
 =\frac{e^{-(r+\sigma)y}}{(r+\sigma)^2-m_i^2}.
\]
\green{This implies that} the integral operator has norm at most
\[
 \widehat\kappa_Y
 :=\frac{2M_0e^{-\sigma Y}}{\sigma(2m_*+\sigma)}.
\]
\green{By further increasing} $Y_1$\green{, we may ensure that} $\widehat\kappa_Y\le1/2$.
\green{It follows that there} is a unique fixed point $u^{(j)}$, with
$\|u^{(j)}\|_{j,\beta,Y}\le2$.  \green{If we set}
$v_i=u_i^{(j)}-\delta_{ij}e^{-m_jy}$ and
$f_i=\sum_kq_{ik,\beta}u_k^{(j)}$, we obtain
\[
 |f_i(y)|\le4M_0e^{-(r_i^{(j)}+\sigma)y},\qquad
 |v_i(y)|\le
 \frac{4M_0e^{-(r_i^{(j)}+\sigma)y}}
      {(r_i^{(j)}+\sigma)^2-m_i^2}.
\]
\green{By differentiating under the integral sign in} the \green{previous} absolutely convergent integral\green{, we obtain the relations}
\[
 v_i'(y)=-\int_y^\infty\cosh(m_i(s-y))f_i(s)\,\dd s,
 \qquad v_i''=m_i^2v_i+f_i.
\]
In particular,
\[
 |v_i'(y)|\le
 \frac{4M_0(r_i^{(j)}+\sigma)e^{-(r_i^{(j)}+\sigma)y}}
      {(r_i^{(j)}+\sigma)^2-m_i^2},\qquad
 |v_i''(y)|\le C_J e^{-(r_i^{(j)}+\sigma)y}.
\]
All \green{of the }denominators \green{in the previous relation} are bounded \green{from} below by $\sigma(2m_*+\sigma)$.
Thus
\[
 Z_{H,\beta}=(u_1^{(1)},u_2^{(1)},(u_1^{(1)})',(u_2^{(1)})')^t,
 \qquad
 Z_{G,\beta}=(u_1^{(2)},u_2^{(2)},(u_1^{(2)})',(u_2^{(2)})')^t
\]
solve \eqref{eq:linode} and have the stated asymptotics and derivative
bounds, with $c=\sigma$.  Moreover,
\[
 \lim_{y\to\infty}e^{m_i y}u_i^{(j)}(y)=\delta_{ij}.
\]
These identities \green{readily prove linear }independence.  The two solutions therefore span
the decaying plane \green{that we have }already constructed.  \green{This} also \green{proves the} uniqueness of
this coefficient normalization and \green{the} independence \green{with respect to} the choice of $Y$.

We finally \green{prove that} these normalized modes \green{are continuous in} $\beta$ \green{with respect to their defining norms}.  For fixed $j$ \green{we set}
$F_i(y)=e^{r_i^{(j)}y}u_i^{(j)}(y)$.  This identifies \green{all the} spaces above
with the \green{one} of bounded continuous pairs on $[Y,\infty)$.
\green{If we set} $s=y+t$, the conjugated integral operator \green{appearing in the equations for $F_i$ takes the form}
\[
 K_{ik,\beta}^{(j)}(y,t)
 =\frac{\sinh(m_it)}{m_i}
   e^{(r_i^{(j)}-r_k^{(j)})y-r_k^{(j)}t}
   q_{ik,\beta}(y+t).
\]
The \green{previous} bounds \green{for $q_{ik,\beta}$} and $r_i^{(j)}\ge m_i$ \green{readily }imply \green{that}
\[
 |K_{ik,\beta}^{(j)}(y,t)|
 \le\frac{M_0}{2m_*}e^{-\sigma y}e^{-\sigma t},
 \qquad y\ge Y,\quad t\ge0.
\]
\green{The continuous dependence of the interface and of $r_i^{(j)}$ on $\beta$ implies that these kernels} converge uniformly \green{on bounded rectangles} as
$\beta\to\beta_0$.  \green{Our previous} bound \green{also guarantees that the} integrals over $t\ge L$
and the operator norms over $y\ge L$ \green{are uniformly} $O_J(e^{-\sigma L})$.
\green{This will immediately imply } that the conjugated operators, \green{which we shall denote} by
$\widehat{\mathcal L}_\beta$, are continuous in \green{$\beta$} \green{with respect to the} operator norm.

\green{We now turn back to the fixed point integral equations. Their} homogeneous term\green{s are} the constant vector\green{s} $e_j$, so subtract\green{ing}
\green{the} fixed point equations \green{for any fixed index $\displaystyle j=\overline{1,2}$ cancels the homogeneous terms out and} gives
\[
 \|F_\beta-F_{\beta_0}\|_\infty
 \le4\|\widehat{\mathcal L}_\beta
          -\widehat{\mathcal L}_{\beta_0}\|\longrightarrow0.
\]
 \green{By applying the same dominated integral argument to the previous integral }formula for $v_i'$, \green{we immediately deduce that first derivatives are locally uniformly continuous. By using the equations for the second derivatives and repeating the argument, we reach the same conclusion in their case}.  \green{All of these estimates imply the following bound which is uniform on $J$} 
\[
 |Z_{H,\beta}(y)|+|Z_{H,\beta}'(y)|
 +|Z_{G,\beta}(y)|+|Z_{G,\beta}'(y)|\le C_J e^{-m_*y}.
\]
This proves \green{the} continuity of the Cauchy data of the normalized modes.
\end{proof}

\blue{For $y_0<Y$, we write $E_\beta^+(y_0)$ for the image of
$E_\beta^+(Y)$ under the exact flow of \eqref{eq:linode} from $Y$ to
$y_0$.  We next construct the corresponding one-dimensional space of
admissible Cauchy data at the left endpoint.
For a solution
\[
        Z=(\xi,\psi,\xi',\psi')^t
\]
of \eqref{eq:linode}, we say that its scalar component is
\emph{recessive at $-\infty$} if
\[
 \lim_{y\to-\infty}
 \frac{\xi(y)}
 {|y|^{(\sqrt\beta-1)/2}e^{-\sqrt\beta y^2/2}}
\]
exists and is finite.  Under the change of variables
\[
        x=\sqrt2\,\beta^{1/4}y,\qquad \nu=\sqrt\beta,
\]
this is precisely the recessive class for the perturbed Weber equation
constructed later in Lemma~\ref{lem:webermatching}.
}

\begin{Lemma}\label{lem:leftline}
Let $J\Subset(0,4)$ be compact.  For $Y\ge Y_0(J)$ and
$\beta\in J$, the solutions of
\eqref{eq:linode} which are scalar-recessive at $-\infty$ and satisfy
\eqref{eq:normalized} form a one-dimensional space
$E_\beta^-(-Y)\subset\RR^4$ of Cauchy data at $-Y$.  This line depends
continuously on $\beta$.  After normalizing its first coordinate to one,
it is spanned by
\bel{eq:leftEvansLine}
 \ell_\beta^-(-Y)=
 \left(1,O(e^{-cY^2}),m_\beta(Y)+O(Y^{-3}),O(e^{-cY^2})\right)^t,
\ee
where $m_\beta(Y)=\sqrt\beta Y-\frac{\sqrt\beta-1}{2Y}.$
The estimates are uniform for $\beta\in J$.
\end{Lemma}

\begin{proof}
Set
\[
 V_\beta=\Gamma_\beta^2+\beta(3\Phi_\beta^2-1),\qquad
 c_\beta=2\Gamma_\beta\Phi_\beta,
\]
\[
 d_\beta=4\Gamma_\beta\Phi_\beta,\qquad
 r_\beta=2\Phi_\beta^2.
\]
\blue{By Lemma~\ref{lem:webermatching}, after increasing $Y$ if necessary,
there is a unique solution $D_\beta$ of
\[
        D_\beta''(y)=V_\beta(y)D_\beta(y),
        \qquad y\le -Y,
\]
satisfying
\[
 \lim_{y\to-\infty}
 \frac{D_\beta(y)}
 {|y|^{(\sqrt\beta-1)/2}e^{-\sqrt\beta y^2/2}}=1.
\]}
Moreover, $D_\beta$ is positive on $(-\infty,-Y]$.
This normalization depends continuously on $\beta$ in the weighted
topology below.  Under $x=\sqrt2\,\beta^{1/4}y$, the model coefficient
$\beta y^2-\beta$ becomes
\[
 x^2/4-\sqrt\beta/2,
\]
and
\[
 V_\beta(y)-(\beta y^2-\beta)
 =\Gamma_\beta(y)^2-\beta y^2+3\beta\Phi_\beta(y)^2
\]
is bounded by a polynomial times a Gaussian, uniformly for $\beta\in J$.
Lemma~\ref{lem:webermatching}, with $\nu=\sqrt\beta$, gives
\[
 \frac{D_\beta'(-Y)}{D_\beta(-Y)}
 =m_\beta(Y)+O(Y^{-3})
\]
and weighted continuous dependence on $\beta$.

For functions $f,h$ on the left half-line, define
\[
 (\mathcal S_\beta f)(y)=D_\beta(y)\int_{-\infty}^yD_\beta(s)^{-2}
       \int_{-\infty}^sD_\beta(t)f(t)\,dt\,ds,
 \qquad
 (\mathcal Gh)(y)=\int_{-\infty}^y(y-s)h(s)\,ds .
\]
Then $(\partial_y^2-V_\beta)\mathcal S_\beta f=f$,  and $(\mathcal Gh,\mathcal Gh')$ vanish at
$-\infty$.  Set
\[
 p_\beta=\mathcal G(|d_\beta|D_\beta),\qquad
 \|u\|_D=\sup_{y\le-Y}\frac{|u(y)|}{D_\beta(y)},
\]
\[
 \|\psi\|_p=\sup_{y\le-Y}
 \left\{\frac{|\psi(y)|}{p_\beta(y)}
       +\frac{|\psi'(y)|}{p_\beta'(y)}\right\}.
\]
The denominators are positive.  The two-sided Weber tail
\eqref{eq:lefttail}, proved from Theorem~\ref{thm:interface},
Lemma~\ref{lem:interfaceUniformTails}, and
Lemma~\ref{lem:webermatching}, gives, uniformly for $\beta\in J$\footnote{Throughout, $a\simeq b$ for $a,b>0$ means $C^{-1}\le \frac{a}{b}\le C$ for some constant $C$. This constant might depend on parameters, which will be clear from the context.},
\[
\begin{split}
 D_\beta(y)&\simeq  |y|^{a_\beta}e^{-k_\beta y^2/2},\\
 p_\beta'(y)&\simeq  |y|^{2a_\beta}e^{-k_\beta y^2},\\
 p_\beta(y)&\simeq  |y|^{2a_\beta-1}e^{-k_\beta y^2},
 \qquad y\le-Y,
\end{split}
\]
where $k_\beta=\sqrt\beta$ and $a_\beta=(k_\beta-1)/2$.
Indeed,
\[
|d_\beta(y)|D_\beta(y)\simeq 
|y|^{2a_\beta+1}e^{-k_\beta y^2},
\]
and the last two estimates follow by one and two integrations from
$-\infty$. 
These bounds imply
\[
 \|\mathcal S_\beta(c_\beta\psi)\|_D
 \le\varepsilon_Y\|\psi\|_p,
 \qquad
 \|\mathcal G(r_\beta\psi)\|_p
 \le\varepsilon_Y\|\psi\|_p,
\]
\[
 \|\mathcal G(d_\beta\mathcal S_\beta(c_\beta\psi))\|_p
 \le\varepsilon_Y\|\psi\|_p,\qquad \varepsilon_Y\longrightarrow0.
\]
Each estimate follows from
\[
 \int_{-\infty}^y\langle s\rangle^Ne^{-cs^2}\,ds
 \le C_N\langle y\rangle^{N-1}e^{-cy^2},
\]
uniformly for $\beta\in J$.  For every scalar coefficient $a$, the
equation
\[
 \psi=\mathcal G\red{(}d_\beta(aD_\beta+
        \mathcal S_\beta(c_\beta\psi))+r_\beta\psi\red{)}
\]
is a contraction when $Y$ is large.  Set
\[
\xi=aD_\beta+\mathcal S_\beta(c_\beta\psi).
\]
Variation of constants shows that every normalized recessive solution has
this form.  The admissible space is therefore one-dimensional.  At
$-Y$, the same estimates give
\[
 \frac{|\psi(-Y)|+|\psi'(-Y)|}{|\xi(-Y)|}\le Ce^{-cY^2},
 \qquad
 \left|\frac{\xi'(-Y)}{\xi(-Y)}-m_\beta(Y)\right|
 \le C(Y^{-3}+e^{-cY^2}).
\]
Uniform contraction and weighted continuity prove continuous dependence on
$\beta$.
\end{proof}
\green{It turns out that t}he signs of the derivatives of the interface \green{allow us} to
identify all \green{of the }bounded solutions of the linearized equations.  The \green{following} argument
\green{uses only the }monotonicity \green{and the second-order equations of the interface components.}

\begin{Lemma}\label{lem:boundedInterfaceKernel}
Let $0<\beta<4$.  Every bounded solution $(\xi,\psi)$ of
\eqref{eq:linode} on $\RR$ is a constant multiple of
$(\Phi_\beta',\Gamma_\beta')$.  In particular, a bounded solution
satisfying \eqref{eq:normalized} vanishes identically.
\end{Lemma}

\begin{proof}
\green{We f}ix $\beta$ and abbreviate $\Phi=\Phi_\beta$, $\Gamma=\Gamma_\beta$.
The differential expression
\bel{eq:symmetricInterfaceOperator}
 \mathfrak L_\beta
 :=-\begin{pmatrix}2&0\\0&1\end{pmatrix}\calA_\beta
 =\begin{pmatrix}
 -2\partial_y^2+2(\Gamma^2+\beta(3\Phi^2-1))&4\Gamma\Phi\\
 4\Gamma\Phi&-\partial_y^2+2\Phi^2
 \end{pmatrix}
\ee
\green{describes a} symmetric \green{operator} on \green{the space of} compactly supported functions.  \green{We set}
\[
 S=\operatorname{diag}(1,-1),\qquad
 p=\Phi'>0,\qquad q=-\Gamma'>0,\qquad c=4\Gamma\Phi>0.
\]
Differentiating \green{the interface equations} \eqref{eq:interface} shows that
\[
 \widetilde{\mathfrak L}_\beta:=S\mathfrak L_\beta S
 =\begin{pmatrix}-2\partial_y^2+V_1&-c\\-c&-\partial_y^2+V_2\end{pmatrix},
 \qquad
 \widetilde{\mathfrak L}_\beta\binom{p}{q}=0,
\]
where $V_1=2(\Gamma^2+\beta(3\Phi^2-1))$ and $V_2=2\Phi^2$.
Since $p,q$ are positive at every finite point, it follows that
\bel{eq:positiveSolutionPotentials}
        V_1=2\frac{p''}{p}+c\frac{q}{p},
        \qquad
        V_2=\frac{q''}{q}+c\frac{p}{q}.
\ee
For $v=(v_1,v_2)\in C_c^\infty(\RR;\RR^2)$\green{, we} define
\[
 \mathscr Q_\beta[v]
 =\int_{\RR}
 2(v_1')^2+(v_2')^2+V_1v_1^2+V_2v_2^2-2cv_1v_2\,dy.
\]
The scalar identity
\[
 (v')^2+\frac{p''}{p}v^2
 =p^2\left(\left(\frac{v}{p}\right)'\right)^2
   +\left(\frac{p'}{p}v^2\right)'
\]
and its counterpart for $q$, together with \green{the identity}
\[
 c\frac{q}{p}v_1^2+c\frac{p}{q}v_2^2-2cv_1v_2
 =cpq\left(\frac{v_1}{p}-\frac{v_2}{q}\right)^2,
\]
give the exact representation
\bel{eq:interfacePositiveRepresentation}
\begin{split}
 \mathscr Q_\beta[v]
 ={}&2\int_{\RR}p^2\left(\left(\frac{v_1}{p}\right)'\right)^2dy
 +\int_{\RR}q^2\left(\left(\frac{v_2}{q}\right)'\right)^2dy\\
 &+\int_{\RR}cpq\left(\frac{v_1}{p}-\frac{v_2}{q}\right)^2dy.
\end{split}
\ee
\green{Here, t}he integrated derivatives \green{vanish upon integration} \green{due to their} compact support \green{property}.  In particular,
$\mathscr Q_\beta[v]\ge0$.

\green{Let} $U=(\xi,\psi)$ \green{be} a bounded solution of
\eqref{eq:linode}, and \green{we} set $v=SU$.  \green{We know that}
$\widetilde{\mathfrak L}_\beta v=0$.  \green{We c}hoose
$\chi\in C_c^\infty(\RR)$ such that $0\le\chi\le1$, $\chi=1$ on
$[-1,1]$, and $\operatorname{supp}\chi\subset[-2,2]$, and \green{we set}
$\chi_T(y)=\chi(y/T)$.  Testing the equation against $\chi_T^2v$
and integrating by parts gives
\bel{eq:interfaceKernelCutoff}
 \mathscr Q_\beta[\chi_Tv]
 =\int_{\RR}(\chi_T')^2(2v_1^2+v_2^2)\,dy
 \le \frac{C}{T}\|v\|_{L^\infty(\RR)}^2.
\ee
Indeed, \green{on one hand} the derivative terms \green{resulting from testing $\chi_T^2v$ against the} equation are
\begin{align*}
    2\chi_T^2(v_1')^2+4\chi_T\chi_T'v_1v_1'
 +\chi_T^2(v_2')^2+2\chi_T\chi_T'v_2v_2',
\end{align*}\green{ while}
expanding the derivatives in $\mathscr Q_\beta[\chi_Tv]$ leaves \green{us with} the integral displayed in \eqref{eq:interfaceKernelCutoff}.

\green{We f}ix a bounded interval $I$.  For $T$ large enough, \green{we have }$\chi_T=1$ on $I$.
By \eqref{eq:interfacePositiveRepresentation}, each of the three
nonnegative integrals over $I$ is bounded by the right\green{-hand} side of
\eqref{eq:interfaceKernelCutoff}.  Letting $T\to\infty$ \green{implies that}
\[
 \left(\frac{v_1}{p}\right)'=0,\qquad
 \left(\frac{v_2}{q}\right)'=0,\qquad
 \frac{v_1}{p}=\frac{v_2}{q}
 \quad\hbox{on }I.
\]
As $I$ was arbitrar\green{ily chosen}, \green{it immediately follows that there exists a \green{real} constant $a$ such that} $v=a(p,q)$ on $\RR$,
and \green{that} $U=a(\Phi',\Gamma')$.  The condition
$\psi(-\infty)=0$ and the limit \green{condition} $\Gamma'(-\infty)=-\sqrt\beta$
\green{immediately imply that} $a=0$\green{, hence $U=0$, as desired.}   

\green{We note that no integrability assumptions on the solution $U$ or on the translation mode were needed. In particular, the nonzero limit of $q$ at the left endpoint didn't generate any boundary terms in our argument.}
\end{proof}

At $\beta=1$ the additional first-order equations yield a square
completion and determine the surface energy exactly.  We retain this
calculation for its later use in the energy formula.

\begin{Lemma}\label{lem:BPSkernel}
At $\beta=1$ the normalized operator has trivial kernel in the endpoint class described by
Lemmas~\ref{lem:rightstable} and~\ref{lem:leftline}.
\end{Lemma}

\begin{proof}
At $\beta=1$, the interface satisfies
\bel{eq:BPSinterface}
        \Phi'=\Gamma\Phi,\qquad \Gamma'=\Phi^2-1 .
\ee
To establish this directly, set
\[
 \gamma(\varphi)=\bigl(\varphi^2-1-2\log\varphi\bigr)^{1/2},
 \qquad 0<\varphi<1.
\]
The scalar equation
\[
 \Phi'=\Phi\gamma(\Phi),\qquad \Gamma=\gamma(\Phi)
\]
has a solution increasing from $0$ to $1$ on $\RR$, because
$dy/d\Phi=(\Phi\gamma(\Phi))^{-1}$ has divergent integrals at both
endpoints.  Differentiating yields
\[
 \gamma'(\varphi)=\frac{\varphi^2-1}{\varphi\gamma(\varphi)},
 \qquad
 \Gamma'=\gamma'(\Phi)\Phi'=\Phi^2-1.
\]
Thus the pair solves \eqref{eq:BPSinterface}.  The affine constant at the
normal end exists since
\[
 \int_{-\infty}^{y_0}\Phi(y)^2\,\dd y
 =\int_0^{\Phi(y_0)}\frac{\varphi}{\gamma(\varphi)}\,\dd\varphi<\infty
\]
and $(\Gamma+y)'=\Phi^2$.  The integral is finite at zero because
$\gamma(\varphi)\sim\sqrt{-2\log\varphi}$.  If
$\ell=\lim_{y\to-\infty}\{\Gamma(y)+y\}$, then replacing the orbit by
\[
        (\Phi(y+\ell),\Gamma(y+\ell))
\]
sets the affine constant to zero.  Theorem~\ref{thm:interface} identifies
the resulting solution of \eqref{eq:interface} with
$(\Phi_1,\Gamma_1)$.
The surface energy satisfies
\bel{eq:BPSsurface}
\begin{split}
        &\Phi'^2+\frac{1}{2}\Gamma'^2+\Gamma^2\Phi^2
        +\frac{1}{2}(1-\Phi^2)^2\\
        &\qquad =
        (\Phi'-\Gamma\Phi)^2
        +\frac{1}{2}(\Gamma'-(\Phi^2-1))^2
        +\frac{d}{dy}\{\Gamma(\Phi^2-1)\}.
\end{split}
\ee
Along the BPS orbit both squares vanish.  Since
$\Gamma(y)(\Phi(y)^2-1)=-|y|+o(1)$ at $-\infty$ and tends to zero at
$+\infty$, integration of \eqref{eq:BPSsurface} in
\eqref{eq:sbeta} gives $s_1=0$.
For a homogeneous solution in the endpoint class, set
\[
\calQ_1[\xi,\psi]=\frac12\int_{\RR}
\langle\calA_1(\xi,\psi),(-2\xi,-\psi)\rangle\,dy.
\]
Integration by parts gives
\bel{eq:BPSfactor}
        \calQ_1[\xi,\psi]
        =
        \int_{\RR}\left[
        (\xi'-\Gamma\xi-\Phi\psi)^2
        +\frac{1}{2}(\psi'-2\Phi\xi)^2
        \right]\dd y .
\ee
\red{To see this, we note that on} a finite interval $[a,b]$, the boundary contribution is
\[
 \left[\Gamma\xi^2+2\Phi\xi\psi-\xi\xi'
       -\frac{1}{2}\psi\psi'\right]_a^b.
\]
It vanishes at $+\infty$ by exponential decay.  At $-\infty$, the
scalar component is recessive and satisfies $|y|\xi(y)\to0$.  The
normalized magnetic component and its derivative tend to zero, and the
derivative estimates in
Lemmas~\ref{lem:rightstable} and~\ref{lem:leftline} give the same conclusion.
This proves \eqref{eq:BPSfactor}. 
A homogeneous zero mode therefore satisfies
\bel{eq:BPSlin}
        \xi'=\Gamma\xi+\Phi\psi,\qquad
        \psi'=2\Phi\xi .
\ee
At $+\infty$, \eqref{eq:BPSlin} has eigenvalues $\pm\sqrt2$, so its
stable space is one-dimensional.  Differentiating
\eqref{eq:BPSinterface} shows that the nonzero decaying solution
$(\Phi',\Gamma')$ spans it.  Hence every zero mode in the normalized endpoint class
is a multiple of this translation mode.  Since
$\Gamma'(y)\to-1$ at $-\infty$, the normalization
$\psi(-\infty)=0$ excludes it.
\end{proof}

\begin{Proposition}\label{prop:nearBPS}
Let $J\Subset(0,4)$ be a compact interval.  For every $\beta\in J$,
the normalized homogeneous problem
\eqref{eq:linode}, \eqref{eq:normalized} has no nonzero solution which is
scalar-recessive at $-\infty$ and decays at $+\infty$.
There is $K_0=K_0(J)$ such that, for every fixed $K\ge K_0$,
the linear matching map $\RR^3\longrightarrow\RR^4$ given by
\bel{eq:EvansMatchingMap}
 \mathcal M_{\beta,K}(a,b_1,b_2)
 =T_\beta(K,-K)a\ell_\beta^-(-K)
  -b_1Z_{H,\beta}(K)-b_2Z_{G,\beta}(K)
\ee
has rank three for every $\beta\in J$.  Moreover, there is
$c_{J,K}>0$ such that
\bel{eq:uniformMatchingRank}
        |\mathcal M_{\beta,K}z|\ge c_{J,K}|z|,
        \qquad \beta\in J,\quad z\in\RR^3.
\ee
Here $T_\beta(K,-K)$ is the transfer matrix of
\eqref{eq:linode}.  The same rank condition then holds at every other
matching point by invertibility of the exact flow.
\end{Proposition}

\begin{proof}
\green{We first note that any} solution \green{with the aforementioned endpoint properties must be bounded on} $\RR$.  \green{Indeed, this immediately follows since} its
scalar component decays at the left endpoint, its magnetic component
there tends to zero, and both components decay at the right endpoint.
\green{By} Lemma~\ref{lem:boundedInterfaceKernel}\green{, any such solution must be zero}.
\green{Let} $K_0(J)$ \green{be chosen} so that both endpoint constructions apply uniformly on
$J$.  \green{We assume} that
\[
        \mathcal M_{\beta,K}(a,b_1,b_2)=0
\]
\green{for some coefficients $a$, $b_1$, and $b_2$.}

\green{We can immediately see that} the left solution
with coefficient $a$ and the right solution with coefficients $(b_1,b_2)$
have identical Cauchy data at $K$.  Uniqueness for the initial-value
problem \green{allows us to join them and construct} a bounded normalized solution on $\RR$, which\green{,} by the preceding paragraph\green{, must be identically zero}.  The \green{linear} independence of the chosen
endpoint data \green{immediately implies that} $a=b_1=b_2=0$.  Conversely, \green{any} nonzero solution
in the endpoint class would provide a nonzero vector in this kernel.
This proves \green{the desired }rank\green{-}three \green{property} and the asserted equivalence.

\green{We note that f}or fixed $K$, the endpoint constructions and the transfer matrix
depend continuously on $\beta$.  \green{This means that the function}
$(\beta,z)\rightarrow|\mathcal M_{\beta,K}z|$ is positive and continuous on
\[
        J\times\{z\in\RR^3:|z|=1\}.
\]
Its minimum on this compact set is \green{well-defined and} positive and gives the
\green{desired} constant $c_{J,K}$ in \eqref{eq:uniformMatchingRank}.
\end{proof}

In the remainder of the paper the parameter may range in any fixed compact
interval $J\Subset(0,4)$.  The constants in the estimates and the lower
bounds required on $n$ may depend on this interval.  The non-degeneracy
\green{property we have }just proved is uniform in this sense.  It does not assert a positive
spectral gap for an operator on \green{the} unweighted $L^2(\RR)$ \green{space}. \green{This property will be crucial later on in the proof of the finite-interval estimate \eqref{eq:finiteWindowOverdet}.}

\section[Localizing the critical radius $R_n$ by minimization]{Localizing the critical radius \texorpdfstring{$R_n$}{Rn} by minimization}\label{sec:locRn}

As noted before, we do not rely on the Euler-Lagrange equations~\eqref{eq:vortex} and the boundary conditions~\eqref{eq:bc} alone. Rather, as we shall now see, our analysis will depend crucially on the minimization of the tension. 
The first consequence of minimality is the following coarse localization. Throughout this section, we fix some compact interval $J\Subset(0,\infty)$, and all results will hold uniformly for $\beta\in J$. 

\begin{Proposition}\label{prop:coarseloc}
As $n\to\infty$, 
\bel{eq:coarseloc}
       \frac{R_n}{R_n^{(0)}}\longrightarrow 1,
       \qquad R_n^{(0)}=\sqrt2\,\beta^{-1/4}n^{1/2}.
\ee
Equivalently,
\bel{eq:bncoarse}
       b_n:=\frac{2n}{R_n^2}\longrightarrow \sqrt\beta .
\ee
\end{Proposition}

\begin{proof}  Choose a fixed smooth
cutoff $\chi$ with $\chi(t)=0$ for $t\le0$, $\chi(t)=1$ for $t\ge1$, and $0\le\chi\le1$, and set
\[
        a_R(u)=\left(1-\frac{u^2}{R^2}\right)_+,\qquad
        \phi_R(u)=\chi(u-R).
\]
 These comparison functions belong to the admissible class~$\calA_n$ and give
\bel{eq:trialUpperBound}
        \calT_n\Le \frac{n^2}{R^2}+\frac{\beta}{4}R^2+\red{C_{J}}R .
\ee
Choosing $R=R_n^{(0)}$ implies that $\calT_n\Le\sqrt\beta\,n+C_JR_n^{(0)}$.
For the lower bound we introduce the magnetic field 
\bel{eq:magneticField}
        \mathfrak B_n(u)=-\frac{n}{u}a_n'(u)>0 .
\ee
The magnetic equation gives $\mathfrak B_n'(u)=-(2n/u)a_n(u)\phi_n(u)^2\Le0$ and
\[\int_0^\infty \mathfrak B_n(u)u\,du=n.\]    Since $a_n(0)=1$, $\mathfrak B_n(0)=2n/R_n^2$, and
$a_n\mathfrak B_n\to0$ at infinity,
\[
\begin{split}
        \int_0^\infty \frac{n^2}{u}a_n(u)^2\phi_n(u)^2\,du
        &=-\frac{n}{2}\int_0^\infty a_n(u)\mathfrak B_n'(u)\,du\\
        &=-\frac{n}{2}[a_n\mathfrak B_n]_{0}^{\infty}
          +\frac{n}{2}\int_0^\infty a_n'(u)\mathfrak B_n(u)\,du\\
        &=\frac{n^2}{R_n^2}
          -\int_0^\infty \frac{1}{2}\mathfrak B_n(u)^2u\,du .
\end{split}
\]
Thus \red{we have the magnetic identity}
\bel{eq:magneticCouplingIdentity}
        \int_0^\infty \frac{1}{2}\mathfrak B_n(u)^2u\,du
        +\int_0^\infty \frac{n^2}{u}a_n(u)^2\phi_n(u)^2\,du
        =\frac{n^2}{R_n^2}.
\ee
In particular $n^2/R_n^2\le\calT_n\le C_Jn$, and hence
\bel{eq:RlowerCoarse}
        R_n^2\Ge \red{c_J}n .
\ee
Also,
\bel{eq:anLowerBound}
        a_n(u)=1-\frac{1}{n}\int_0^u\mathfrak B_n(s)s\,ds
        \Ge \left(1-\frac{u^2}{R_n^2}\right)_+ .
\ee
Indeed, Proposition~\ref{prop:static} shows that $u\mapsto a_n'(u)/u$ is increasing on $(0,\infty)$ and that
$$\lim_{u\to0_+} a_n'(u)/u=-2/R_n^2.$$  Hence $a_n'(u)\ge -2u/R_n^2$ for all $u>0$, and integrating from $0$ gives
$a_n(u)\ge 1-u^2/R_n^2$.  Since $a_n\ge0$, this implies \eqref{eq:anLowerBound}.
Fix $L\gg1$ independent of $n$ and put $\theta_L=1-L^{-1}$.  On $0<u<\theta_LR_n$ one has
$a_n(u)\ge c/L$ for some $c>0$ independent of $n$.  The second term in the energy \eqref{eq:magneticCouplingIdentity} therefore yields
\[
        \frac{c n^2}{L^2}\int_0^{\theta_LR_n}\frac{\phi_n(u)^2}{u}\,du\Le \calT_n\Le C_J n
\]
and consequently
\[
        \int_0^{\theta_LR_n}\phi_n(u)^2u\,du
        \Le R_n^2\int_0^{\theta_LR_n}\frac{\phi_n(u)^2}{u}\,du
        \Le C_J L^2 \frac{R_n^2}{n} .
\]
Hence
\[
\begin{split}
        \int_0^\infty \frac{\beta}{2}u(1-\phi_n^2)^2\,du
        &\Ge \int_0^{\theta_LR_n}\frac{\beta}{2}u(1-2\phi_n^2)\,du  \\
        &\Ge \frac{\beta}{4}\theta_L^2R_n^2-C_JL^2\frac{R_n^2}{n}\\
        & \Ge \frac{\beta}{4}R_n^2 - C_J L^{-1} R_n^2 - C_JL^2\frac{R_n^2}{n}.
\end{split}
\]
For fixed $L$, the last term is absorbed into the leading $R_n^2$ term for all sufficiently large~$n$.  Together with
\eqref{eq:magneticCouplingIdentity} and \eqref{eq:trialUpperBound} this implies
\[
        \frac{n^2}{R_n^2}+c_{\red{J},L}R_n^2\Le C_Jn,
\]
and therefore, after choosing one large fixed value of $L$,
\bel{eq:RtwosidedCoarse}
        c_Jn\Le R_n^2\Le C_Jn .
\ee
The constants in \eqref{eq:RtwosidedCoarse} are now uniform in $n$ and depend only on $\red{J}$ (after fixing a single large value of $L$ in the preceding step).  With this coarse bound in hand, we may vary $L$ in the potential estimate above without revisiting \eqref{eq:RtwosidedCoarse}. In fact,  by \eqref{eq:RtwosidedCoarse} the resulting error is $O_J(L^2)$.
Combining \eqref{eq:magneticCouplingIdentity} with this refined bound\red{, along with the previous bounds \eqref{eq:RtwosidedCoarse} and $\displaystyle\mathcal{T}_n\leq \sqrt{\beta}n+C_JR_n^{(0)}$,} gives
\[
        \frac{n^2}{R_n^2}+\frac{\beta}{4}R_n^2
        \Le \sqrt\beta\,n+C_JR_n^{(0)}+C_JL^{-1}R_n^2+C_JL^2 .
\]
Here,  the role of $L$ is only to obtain a relative localization.  The term $C_JL^2$ is harmless because $L$ is fixed while $n\to\infty$, whereas the term $C_JL^{-1}R_n^2$ is the only error on the leading scale and is removed only after this first limit by sending $L\to\infty$.
Define 
\[
 x_n:=\frac{R_n}{R_n^{(0)}},
 \qquad (R_n^{(0)})^2=\frac{2n}{\sqrt\beta}.
\]
Then
\[
 \frac{n^2}{R_n^2}=\frac{\sqrt\beta\,n}{2}\,x_n^{-2},
 \qquad
 \frac{\beta}{4}R_n^2=\frac{\sqrt\beta\,n}{2}\,x_n^{2},
\]
and hence
\[
 \frac{n^2}{R_n^2}+\frac{\beta}{4}R_n^2
 =\frac{\sqrt\beta\,n}{2}\,(x_n^{-2}+x_n^2), 
\]
as well as 
\[
 \frac{n^2}{R_n^2}+\frac{\beta}{4}R_n^2-\sqrt\beta\,n
 =\frac{\sqrt\beta\,n}{2}\,(x_n-x_n^{-1})^2.
\]
Subtract $\sqrt\beta\,n$ from the preceding energy inequality to obtain
\[
 \frac{\sqrt\beta\,n}{2}\,(x_n-x_n^{-1})^2
 \Le C_JR_n^{(0)}+C_JL^{-1}R_n^2+C_JL^2.
\]
\red{We d}ivide by $n$.  Since $\beta\in \red{J}\Subset(0,\infty)$, one has
$R_n^{(0)}/n=O_J(n^{-1/2})$, and \eqref{eq:RtwosidedCoarse} gives
$R_n^2/n\Le C_J$.  Consequently
\[
 (x_n-x_n^{-1})^2\Le C_J\left(n^{-1/2}+L^{-1}+\frac{L^2}{n}\right).
\]
For each fixed $L$, let $n\to\infty$ to get
\[
 \limsup_{n\to\infty}(x_n-x_n^{-1})^2\Le \frac{C_J}{L}.
\]
The constant $C_J$ is independent of the cutoff $L$, so letting $L\to\infty$
yields $x_n-x_n^{-1}\to0$.  Since $x_n>0$ one has the exact identity
\[
 |x_n-1|=|x_n-x_n^{-1}|\,\frac{x_n}{x_n+1}
 \le |x_n-x_n^{-1}|,
\]
and therefore $x_n\to1$, i.e., 
\[
 \frac{R_n}{R_n^{(0)}}\to1.
\]
Finally,
\[
 b_n=\frac{2n}{R_n^2}=\frac{2n}{(R_n^{(0)})^2}\,\frac{(R_n^{(0)})^2}{R_n^2}
     =\sqrt\beta\,x_n^{-2}\to\sqrt\beta,
\]
which is \eqref{eq:bncoarse}. \end{proof}

\subsection{Core estimates and compactness of the transition profile}

Following~\cite[Figure 5]{DG}, we distinguish between the {\em core} $u<R_n$ (which Dumitrescu and Gaikwad split into the {\em deep core} and {\em outer core}), followed by the {\em boundary or transition layer}, which we \blue{later show is localized to an $O(1)$ neighborhood of $u=R_n$} (equivalently $y=u-R_n=O(1)$), and which is then followed to the right by the {\em exterior} domain. 
We now establish the estimates we will repeatedly use when passing between these regions: (i) a Gaussian-type core barrier showing that $\phi_n$ is already exponentially small a distance $\eta\gg1$ to the left of $u=R_n$; (ii) ``moving'' expansions\footnote{``Moving'' refers to working in the shifted coordinate $\eta=R_n-u$, i.e., in a frame centered at the $n$-dependent radius $u=R_n$ \blue{(later shown to center the transition layer)}.}
 for the magnetic quantities $a_n$ and $\Gamma_n^-$ in terms of $\eta=R_n-u$ on the window $L_0\le\eta\le R_n^\alpha$\red{, where $\Gamma^{-}_n$ is defined in \eqref{eq:PhiGamma_def}}; and (iii) a centering or compactness step on fixed $O(1)$ windows around the interface that produces a limiting transition (boundary-layer) profile.  Later matching arguments may use this limit, but its existence comes solely from (i)--(iii).

To analyze the vortex profiles to the left of the interface, throughout this section we define $u=R_n-\eta$, $R_n\ge\eta\ge0$ 
\bel{eq:PhiGamma_def}
        \Phi_n^-(\eta):=\phi_n(R_n-\eta),
        \qquad
        \Gamma_n^-(\eta):=\frac{n}{R_n-\eta}a_n(R_n-\eta).
\ee
Our first core estimate is a barrier argument establishing Gaussian decay of the Higgs field to the left of the interface.

\begin{Lemma}\label{lem:coreGaussianBarrier}
Fix $0<\alpha<1$.  There are constants
$L_0,\red{c_J},\red{C_J}>0$ \red{depending on $J$} such that, for all $n\red{\geq n_{J,\alpha}}$ and all
$L_0\le \eta\le R_n^\alpha$,
\bel{eq:coreGaussianBarrier}
       \Phi_n^-(\eta)
       \Le \red{C_J} e^{-\red{c_J}\eta^2} .
\ee
Consequently, if $\eta_*=R_n^\alpha$, then for every $A>0$,
\bel{eq:innerAgmonEndpoint}
        \Phi_n^-(\eta_*)\le C_{A\red{,J}} R_n^{-A} .
\ee
\end{Lemma}

\begin{proof}   To remove a first order derivative in the Higgs equation we introduce 
\bel{eq:Psin_def}
        \Psi_n(\eta)=(R_n-\eta)^{1/2}\Phi_n^-(\eta).
\ee
The ODE then reads
\bel{eq:coreLangerEquation}
        \Psi_n''=\mathcal V_n\Psi_n,
        \qquad
        \mathcal V_n(\eta)=\Gamma_n^-(\eta)^2-\beta-\frac{1}{4(R_n-\eta)^2}+\beta\Phi_n^-(\eta)^2 .
\ee
\red{We claim that for $L_0$ sufficiently large,}
\[
        \Gamma_n^-(\eta)
        \ge \frac{n}{R_n-\eta}\left(1-\frac{(R_n-\eta)^2}{R_n^2}\right)
        \ge c_J\eta,
        \qquad L_0\le\eta<R_n,
\]
 \red{We distinguish two cases. I}f $\eta\le R_n/2$, then
\[
        1-\frac{u^2}{R_n^2}
        =\frac{(R_n-u)(R_n+u)}{R_n^2}
        \ge \frac{\eta}{R_n},
\]
and therefore\red{, by \eqref{eq:RtwosidedCoarse},}
\[
        \Gamma_n^-(\eta)=\frac{n a_n(u)}{u}
        \ge \red{c_J} \frac{n\eta}{R_n^2}\ge \red{c_J}\eta .
\]
If $\eta\ge R_n/2$, then $u\le R_n/2$ and the same lower bound for $a_n$ gives $a_n(u)\ge3/4$.  \red{Thus, by \eqref{eq:RtwosidedCoarse},}
\[
        \Gamma_n^-(\eta)\ge \frac{3n}{4u}\ge \red{c_J}\frac{n}{R_n}\ge \red{c_J}R_n\ge \red{c_J}\eta .
\]
The negative terms in $\mathcal V_n$ are $-(\beta+1/(4u^2))$ and are dominated by the angular term
$n^2a_n(u)^2/u^2$ as $u\to0+$ \red{, hence 
\bel{eq:coreLangerPotentialLower}
        \mathcal V_n(\eta)\ge c_0\eta^2,
        \qquad L_0\le\eta<R_n .
\ee
To see that, we note that when $\displaystyle\eta\leq\frac{R_n}{2}$, it would immediately follow that 
\begin{align*}
\mathcal{V}_n(\eta)\geq c_J\eta^2-C_J\geq c_J\eta^2,
\end{align*}upon shrinking $c_J>0$, whereas when $\displaystyle \eta>\frac{R_n}{2}$, the inequality $\displaystyle a_n(u)\geq 1-\frac{u^2}{R_n^2}\geq \frac{3}{4}$, together with \eqref{eq:RtwosidedCoarse} readily imply that
\begin{align*}
    \mathcal{V}_n(\eta)\geq\Gamma_n^{-}(\eta)^2-\beta-\frac{1}{4u^2}\geq\frac{n^2a_n^2(u)-\frac{1}{4}}{u^2}-C_J\geq c_J\frac{n^2}{R_n^2}-C_J\geq c_J\eta^2, 
\end{align*}
upon shrinking $c_J>0$ when necessary.} Thus, after increasing $L_0$ if necessary,
\eqref{eq:coreLangerPotentialLower} holds on the whole interval.
Let
\[
        W(\eta)=C_0R_n^{1/2}\exp[-\kappa(\eta^2-L_0^2)],
\]
where $0<\kappa\ll c_0$.  Choose $C_0$ so that $\Psi_n(L_0)\le W(L_0)$.  This is possible uniformly in $n$
since $0\le\Phi_n\le1$.  By \eqref{eq:coreLangerPotentialLower},
\[
        W''-\mathcal V_n W\le0,
        \qquad L_0\le\eta<R_n .
\]
At the other endpoint, $\Psi_n(\eta)\to0$ as $\eta\uparrow R_n$, by the Frobenius behavior $\phi_n(u)=O(u^n)$ at $u=0$.
Hence for all sufficiently small $\delta>0$ one has $\Psi_n(R_n-\delta)\le W(R_n-\delta)$.

We compare on $[L_0,R_n-\delta]$ and then let $\delta\downarrow0$.  For such $\delta$ the positive part
$h=(\Psi_n-W)_+$ has zero trace at the endpoints.  On the set where $h>0$,
$(\partial_\eta^2-\mathcal V_n)(\Psi_n-W)\ge0$.  Hence
\[
        0\le \int h(\partial_\eta^2-\mathcal V_n)h\,d\eta
        =-\int |h'|^2\,d\eta-\int \mathcal V_n h^2\,d\eta\le0 .
\]
Thus $h=0$ \red{for every small $\delta>0$}, and 
\bel{eq:LangerPsiGaussian}
        0\le \Psi_n(\eta)\le \red{C_J} R_n^{1/2}e^{-c\eta^2},
        \qquad L_0\le\eta\red{\le}R_n\red{-\delta} .
\ee
\red{Since the previous estimate is uniform in $\delta>0$, it can immediately be extended to the entire range $[L_0,R_n)$.}

For $\eta\le R_n^\alpha$ with $\alpha<1$, one has $R_n-\eta\simeq R_n$, and
\eqref{eq:LangerPsiGaussian} gives \eqref{eq:coreGaussianBarrier}.  Finally,
$e^{-cR_n^{2\alpha}}\le C_A R_n^{-A}$ for every fixed $A$, which gives
\eqref{eq:innerAgmonEndpoint}.  \end{proof}  

Only a bound on~$\Phi_n^{-}$  is asserted in Lemma~\ref{lem:coreGaussianBarrier}.  The  derivative estimate is established in the following 
Lemma~\ref{lem:clearingout}, after the proof develops a magnetic Volterra estimate that supplies the
matching upper bound on~$\Gamma_n^-$ (see~\eqref{eq:PhiGamma_def} for the definitions). We also define
\[
 q_n(u):=a_n(u)-\left(1-\frac{u^2}{R_n^2}\right)
\]
As the reader can infer, we view $u\in [0,R_n-R_n^\alpha]$ as the deep core (the value of $\alpha$ will be determined later). The inner most region $0\le u\le L_0$ is controlled by Frobenius-type bounds rather than by Frobenius series. The proof of the next lemma exhibits this mechanism.

\begin{Lemma}\label{lem:clearingout}
Fix $0<\alpha<1$. There are constants
$L_0,\red{c_J},\red{C_J}>0$ \red{depending on $J$} such that the following estimates \red{hold} for all $n\red{\geq n_{J}}$. Uniformly for
$L_0\le \eta\le R_n^\alpha$,
\bel{eq:movingclear}
       \red{\Phi_n^-(\eta)+\frac{|\partial_\eta\Phi_n^-(\eta)|}{1+\eta}
       \Le \red{C_J} e^{-\red{c_J}\eta^2}},
\ee
and
\bel{eq:qSharp}
        0\le q_n(R_n-\eta)
        \le \red{C_J\frac{e^{-\red{c_J}\eta^2}}{R_n(1+\eta)}},
\ee
\bel{eq:qSharpDer}
        |q_n'(R_n-\eta)|
        \le \red{C_J\frac{e^{-\red{c_J}\eta^2}}{R_n}} .
\ee
Consequently, with $b_n$ as in \eqref{eq:bncoarse}, 
\bel{eq:movingGamma}
       \red{\Gamma_n^-(\eta)
       =b_n\eta+O_J\!\left(\frac{\eta^2}{R_n}\right)
       +O_J\!\left(\frac{e^{-\red{c_J}\eta^2}}{1+\eta}\right)},
\ee
\bel{eq:movingGammaDer}
       \red{\partial_\eta\Gamma_n^-(\eta)
       =b_n+O_J\!\left(\frac{\eta}{R_n}\right)
       +O_J\!\left(e^{-\red{c_J}\eta^2}\right)},
\ee
and the second derivative satisfies the symbol bound
\bel{eq:movingGammaSecond}
       \red{\partial_\eta^2\Gamma_n^-(\eta)
       =O_J\!\left(\frac{1}{R_n}\right)+O_J\!\left(\frac{\eta}{R_n^2}\right)
       +O_J\!\left((1+\eta)e^{-c\eta^2}\right)
       }.
\ee
\end{Lemma}

\begin{proof}  
By \eqref{eq:anLowerBound},  $q_n\ge0$, $q_n(0)=q_n'(0)=0$, and from the magnetic ODE in~\eqref{eq:vortex}
\bel{eq:qeqforbootstrap}   \left(u^{-1}q_n'(u)\right)'=2u^{-1}a_n(u)\phi_n(u)^2 .
\ee
Hence, for $u=R_n-\eta$,
\bel{eq:qIntegralExact}
        q_n(R_n-\eta)=
        \int_\eta^{R_n}
        \frac{(R_n-\eta)^2-(R_n-\tau)^2}{R_n-\tau}
        a_n(R_n-\tau)(\Phi_n^-(\tau))^2\,d\tau .
\ee
\red{Indeed, since $q_n'(0)=q_n''(0)=0$, we have 
\begin{align*}
   2u\int_0^u&\frac{a_n(\lambda)(\phi_n(\lambda))^2}{\lambda}\,d\lambda=u\int_0^u(\lambda^{-1}q'_n(\lambda))'\,d\lambda=u(u^{-1}q'_n(u))=q_n'(u), 
\end{align*}
for every $u$, which along with $q_n(0)=0$, implies that
\begin{align*}
    q_n(u)=
        \int_0^{u}
        \frac{u^2-\lambda^2}{\lambda}
        a_n(\lambda)(\phi_n(\lambda))^2\,d\lambda,
\end{align*}
as claimed. 
}
The integral \red{in \eqref{eq:qIntegralExact}} is finite because $\Phi_n^-(\tau)=\phi_n(R_n-\tau)=O((R_n-\tau)^n)$ as $\tau\uparrow R_n$.
On $0<u\le1$ one has $a_n(u)\ge1/2$ for all large $n$ (we will assume this from now on).  Set
$m=\lfloor n/3\rfloor$ and define $v(u)=\phi_n(1)u^m$.  A direct computation shows that, for $0<u\le1$,
\[
        v''+u^{-1}v'
        =m^2u^{-2}v
        \le \left(\frac{n^2a_n(u)^2}{u^2}+\beta(\phi_n(u)^2-1)\right)v
\]
for $0<u\le1$.  The inequality holds since 
$a_n(u)\ge1/2$ on $0<u\le1$, and therefore
\[
\begin{split}
 &\frac{n^2a_n(u)^2}{u^2}+\beta(\phi_n(u)^2-1)-\frac{m^2}{u^2}\\
 &\qquad\ge \frac{n^2/4-m^2}{u^2}-\beta
 \ge \frac{5n^2}{36u^2}-\beta>0, 
\end{split}
\]
where $m\le n/3$ was used.  
\red{We a}pply the maximum principle on $[\varepsilon,1]$ to $\phi_n-v$ and then let
$\varepsilon\downarrow0$: at $u=1$ the difference is zero, while near $u=0$ it is negative because
$\phi_n(u)=O(u^n)$ and $m<n$.  We infer that
\bel{eq:normalizedOriginPhi}
        0\le\phi_n(u)\le \phi_n(1)u^m,\qquad 0<u\le1 .
\ee
Consequently
\bel{eq:normalizedOriginIntegral}
        \int_0^1\frac{\phi_n(u)^2}{u}\,du
        \le \frac{\red{C_J}}{n}\phi_n(1)^2 .
\ee
Let $T_n(\eta)$ denote the contribution to \eqref{eq:qIntegralExact} from $\tau\ge R_n^\alpha$. In this range, the original variable satisfies 
$u=R_n-\tau\in [0,R_n-R_n^\alpha]$.  If $1\le u\le R_n-R_n^\alpha$, monotonicity and Lemma~\ref{lem:coreGaussianBarrier} give
\[
        \phi_n(u)\le\phi_n(R_n-R_n^\alpha)
        \le \red{C_J} e^{-\red{c_J}R_n^{2\alpha}} .
\]
For the kernel
\[
        K(\eta,u)=\frac{(R_n-\eta)^2-u^2}{u}
\]
we have $|K(\eta,u)|\le C R_n^2/u$ and $|\partial_\eta K(\eta,u)|\le C R_n/u$.  Hence the part
$1\le u\le R_n-R_n^\alpha$ contributes at most
$C R_n^2\log R_n\,e^{-2cR_n^{2\alpha}}$, and its $\eta$-derivative at most
$\red{C_J} R_n\log R_n\,e^{-2cR_n^{2\alpha}}$.  On $0<u\le1$, \eqref{eq:normalizedOriginIntegral} and
$\phi_n(1)\le\phi_n(R_n-R_n^\alpha)\le \red{C_J} e^{-cR_n^{2\alpha}}$ give
\[
        \int_0^1 |K(\eta,u)|a_n(u)\phi_n(u)^2\,du
        \le \red{C_J} \frac{R_n^2}{n}e^{-2\red{c_J}R_n^{2\alpha}},
\]
and the differentiated kernel gives the same bound with $R_n/n$ in place of $R_n^2/n$.  Since $n\simeq R_n^2$, this yields
$$T_n(\eta)+|\partial_\eta T_n(\eta)|\le C R_n^3 e^{-2cR_n^{2\alpha}},$$ uniformly for $L_0\le\eta\le R_n^\alpha$.
We  rewrite this
\red{bound in the} form
\bel{eq:qTailGaussian}
        T_n(\eta)+(1+\eta)\,|\partial_\eta T_n(\eta)|
        \le \red{C_J \frac{e^{-\red{c_J}\eta^2}}{R_n(1+\eta)}},
        \qquad L_0\le\eta\le R_n^\alpha,
\ee
which is convenient for the remainder of the proof. 
On $\eta\le\tau\le R_n^\alpha$ one has
\[
        \frac{(R_n-\eta)^2-(R_n-\tau)^2}{R_n-\tau}\le C_J(\tau-\eta),
        \qquad
        a_n(R_n-\tau)=\frac{2\tau}{R_n}+O\!\left(\frac{\tau^2}{R_n^2}\right)+q_n(R_n-\tau).
\]
The\red{se, together with \eqref{eq:coreGaussianBarrier}, imply that}
\bel{eq:qVolterraIneq}
        0\le q_n(R_n-\eta)
        \le C\int_\eta^{R_n^\alpha}(\tau-\eta)
        \left(\frac{1+\tau}{R_n}+q_n(R_n-\tau)\right)
        e^{-\red{c_J}\tau^2}\,d\tau
        +\red{C_J \frac{e^{-\red{c_J}\eta^2}}{R_n(1+\eta)}} .
\ee
We will repeatedly use the following Gaussian moment estimate: for each fixed integer $j\ge0$ and  $c>0$,
\bel{eq:gaussianMomentBound}
        \int_\eta^\infty (1+\tau)^j e^{-c\tau^2}\,d\tau
        \le C_{j\red{,c}}(1+\eta)^{j-1}e^{-c\eta^2},\qquad \eta\ge1.
\ee
For $L_0$ large the Volterra inequality~\eqref{eq:qVolterraIneq} is a contraction in the weighted norm
\[
        \|q\|_*=
        \sup_{L_0\le\eta\le R_n^\alpha}
        R_n(1+\eta)e^{c\eta^2/2}|q(R_n-\eta)| .
\]
Indeed, pointwise the inhomogeneous term, including the tail term
\eqref{eq:qTailGaussian}, is bounded by
$\red{C_J} R_n^{-1}(1+\eta)^{-1}e^{-\red{c_J}\eta^2/2}$, and thus its $\|\cdot\|_*$-norm
is bounded by $\red{C_J}$.  Write $\mathcal K$ for the Volterra operator
\[
        (\mathcal K q)(\eta):=\red{C_J}\int_\eta^{R_n^\alpha}(\tau-\eta)q(\tau)e^{-\red{c_J}\tau^2}\,d\tau,
\]
so that the $q$--term in \eqref{eq:qVolterraIneq} is $\mathcal K(q_n(R_n-\cdot))$.  It satisfies
$\|\mathcal Kq\|_*\le \red{C_J} e^{-\red{c_J}L_0^2/4}\|q\|_*$, and choosing $L_0$ large makes the last
coefficient smaller than $1/2$.  This proves \eqref{eq:qSharp}, in fact with the final polynomial term multiplied by
$e^{-\red{c_J}\eta^2}$.  \red{As we have already seen, d}ifferentiating \eqref{eq:qIntegralExact} gives no endpoint term, because the kernel vanishes at
$\tau=\eta$.  More explicitly,
\[
        -q_n'(R_n-\eta)
        =\partial_\eta q_n(R_n-\eta)
        =-2(R_n-\eta)\int_\eta^{R_n}
        \frac{a_n(R_n-\tau)(\Phi_n^-(\tau))^2}{R_n-\tau}\,d\tau .
\]
Split this integral at $R_n^\alpha$.  The part $\tau\ge R_n^\alpha$ is exactly the differentiated ``deep-core tail" already
estimated in \eqref{eq:qTailGaussian}. One easily checks that  its contribution is  bounded by
\[
        \red{C_JR_n^{-1}}(1+\eta)^{-2}e^{-\red{c_J}\eta^2}
\]
  On the remaining interval
$\eta\le\tau\le R_n^\alpha$ we use
\[
        \frac{R_n-\eta}{R_n-\tau}\le C_J,\qquad
        a_n(R_n-\tau)\le \red{C_J}\left(\frac{1+\tau}{R_n}+q_n(R_n-\tau)\right),
\]
and obtain
\[
        |q_n'(R_n-\eta)|
        \le C\int_\eta^{R_n^\alpha}
        \left(\frac{1+\tau}{R_n}+q_n(R_n-\tau)\right)e^{-c\tau^2}\,d\tau
        +\red{C_J R_n^{-1}}(1+\eta)^{-2}e^{-c\eta^2}.
\]
The term without $q$ is bounded by $C R_n^{-1}e^{-c\eta^2}$ using \eqref{eq:gaussianMomentBound}, and the term involving $q$ is absorbed by \eqref{eq:qSharp}.  This yields \eqref{eq:qSharpDer}.
Next, write
\[
        \Gamma_n^-(\eta)=\frac{n}{R_n-\eta}\left(1-\frac{(R_n-\eta)^2}{R_n^2}\right)+\frac{n}{R_n-\eta}\,q_n(R_n-\eta).
\]
The first term satisfies
\[
        \frac{n}{R_n-\eta}\left(1-\frac{(R_n-\eta)^2}{R_n^2}\right)=b_n\eta+O_J(\eta^2/R_n),
\]
and the remainder is controlled by \eqref{eq:qSharp} (respectively \eqref{eq:qSharpDer} after differentiation).  This proves \eqref{eq:movingGamma} and \eqref{eq:movingGammaDer}.

We now turn to the derivative bound of \eqref{eq:movingclear}.  The estimates just obtained imply
\[
        0\le \Gamma_n^-(\eta)\le C_J(1+\eta),
        \qquad L_0\le\eta\le R_n^\alpha .
\]
In the equation~\eqref{eq:coreLangerEquation} for $\Psi_n\red{(\eta)}=(R_n-\eta)^{1/2}\Phi_n^-\red{(\eta)}$ the coefficient is therefore bounded by
$C_J(1+\eta)^2$ on this interval.  Let
\[
        I_\eta=\left[\eta-\frac{c_0}{1+\eta},\eta+\frac{c_0}{1+\eta}\right]
        \cap [L_0,R_n^\alpha]
\]
with some $c_0>0$ fixed and small. Applying  the bound  \eqref{eq:coreGaussianBarrier} on this interval and the equation
$\Psi_n''=\mathcal V_n\Psi_n$ yields 
\[
        \red{\sup_{\red{\lambda\in}I_\eta}|\Psi_n\red{(\lambda)}|
        \le C_J R_n^{1/2}e^{-\red{c_J}\eta^2}},\qquad
        \red{\sup_{\red{\lambda\in}I_\eta}|\Psi_n''\red{(\lambda)}|
        \le C_J R_n^{1/2}(1+\eta)^2e^{-\red{c_J}\eta^2}}.
\]
Using 
\[
        |f'(x)|\le 2\rho^{-1}\sup_{|s-x|\le\rho}|f(s)|
              +\rho\sup_{|s-x|\le\rho}|f''(s)|
\]
with $\rho=c_0(1+\eta)^{-1}$ (and with the corresponding one-sided version when $\eta$ is within $\rho$ of an endpoint
of $[L_0,R_n^\alpha]$), yields
\[
        |\Psi_n'(\eta)|\le
        \red{C_J R_n^{1/2}(1+\eta)e^{-\red{c_J}\eta^2}}.
\]
Since $R_n-\eta\simeq R_n$ for $\eta\le R_n^\alpha$, differentiating
$\Phi_n^-\red{(\eta)}=(R_n-\eta)^{-1/2}\Psi_n\red{(\eta)}$ gives
\[
        \red{\frac{|\partial_\eta\Phi_n^-(\eta)|}{1+\eta}
        \le C_J e^{-\red{c_J}\eta^2}},
\]
as claimed. 
For the second derivative of $\Gamma_n^-$ we use 
\bel{eq:GammaEtaEquation}
        \partial_\eta^2\Gamma_n^-(\eta)
        -\frac{1}{R_n-\eta}\partial_\eta\Gamma_n^-(\eta)
        -\frac{1}{(R_n-\eta)^2}\Gamma_n^-(\eta)
        =2\Gamma_n^-(\eta)\Phi_n^-(\eta)^2 .
\ee
Together with \eqref{eq:movingclear}, \eqref{eq:movingGamma}, and \eqref{eq:movingGammaDer}, this gives
\eqref{eq:movingGammaSecond}.
\end{proof}  
\red{\begin{Remark}
    In what follows, we note that the following arguments are applicable to any branch whose core satisfies the bounds from Lemma~\ref{lem:clearingout}. We then use the minimality property of the solution in order to single out the relevant branch and to coarsely localize its magnetic radius. We later on use the renormalized virial identity (see Proposition~\ref{prop:energyreduction}, in particular \eqref{eq:reducedVirial}, and Corollary~\ref{cor:sharploc}) in order to prove the sharper $O(1)$ localization. 
\end{Remark}
}

The previous lemma develops bounds on an interval $u\in [R_n-R_n^\alpha,R_n-L_0]$. We now extend some of them to the full core interval $u\in [0,R_n-M]$. 

\begin{Corollary}
    In the notation of Lemma~\ref{lem:clearingout}, \red{there exists a positive constant $c$ such that} for every $M$ with $L_0\le M\le R_n^\alpha$,
\bel{eq:fixedclear}
       \sup_{0\le u\le R_n-M}\phi_n(u)
       +\sup_{0\le u\le R_n-M}\frac{|\phi_n'(u)|}{1+R_n-u}
       \Le \red{C_J e^{-cM^2}},
\ee
and
\bel{eq:fixedgauge}
       \sup_{0\le u\le R_n-M}|q_n(u)|
       +R_n\sup_{0\le u\le R_n-M}|q_n'(u)|
       \Le \red{C_Je^{-cM^2}}.
\ee
\end{Corollary}
\begin{proof}
Fix $M$ as above. 
On $R_n-R_n^\alpha\le u\le R_n-M$ we have $M\le\eta\le R_n^\alpha$, and \eqref{eq:fixedclear}, \eqref{eq:fixedgauge} \red{immediately follow from} \eqref{eq:movingclear}, \eqref{eq:qSharp}, and $\eqref{eq:qSharpDer}$ multiplied by~$R_n$,  rewritten in the $u$--variable.
On $0\le u\le R_n-R_n^\alpha$, monotonicity and \eqref{eq:coreGaussianBarrier} at $\eta=R_n^\alpha$ imply 
$\phi_n(u)\le C e^{-cR_n^{2\alpha}}$.  The derivative bound on $1\le u\le R_n-R_n^\alpha$ follows by applying the same local interpolation argument used in the proof of Lemma~\ref{lem:clearingout} to $\psi_n=u^{1/2}\phi_n$, which satisfies $\psi_n''=\mathcal V_n\psi_n$ with $|\mathcal V_n(u)|\lesssim n^2/u^2$. \red{Here, the interpolation argument is applied on intervals of length comparable to $\displaystyle\frac{u}{n}$.} On $0<u\le1$ \red{we} integrate $(u\phi_n')'=uQ_n\phi_n$ and \red{we} use 
\red{the formulas
\begin{align*}
    q_n'(u)&=2u\int_0^u\frac{a_n(s)\phi_n(s)^2}{s}\,ds\\
    q_n(u)&=\int_0^u \frac{u^2-s^2}{s}a_n(s)\phi_n(s)^2\,ds,
\end{align*}
along with }\eqref{eq:normalizedOriginPhi}.

For the error term $q_n(u)=a_n(u)-\bigl(1-u^2/R_n^2\bigr)$ on $0\le u\le R_n-R_n^\alpha$, integrate \eqref{eq:qeqforbootstrap} with $q_n(0)=q_n'(0)=0$ to obtain the exact radial formulae for $q_n$ and $q_n'$.  Using $0<a_n\le1$, \eqref{eq:normalizedOriginIntegral}, and the bound $\phi_n\le Ce^{-cR_n^{2\alpha}}$ on $[1,R_n-R_n^\alpha]$, we obtain
$$\sup_{0\le u\le R_n-M}|q_n(u)|+\red{\sup_{0\le u\le R_n-M}}R_n|q_n'(u)|\le \red{C_Je^{-cM^2}}.$$
Combining the two regions yields \eqref{eq:fixedclear} and \eqref{eq:fixedgauge}.
\end{proof}

In the overlap region $R_n-R_n^\alpha\le u\le R_n-L_0$ we will use the refined estimates \eqref{eq:qSharp}--\eqref{eq:movingGammaSecond}.
We next record a convenient normalization of the interface position using the gauge potential profile: we fix the translation by evaluating $a_n$ (equivalently $\Gamma_n$) at a fixed distance $L$ to the left of $u=R_n$.  This will be used to pin down limits along suitable subsequences of the translated boundary layer. As before, the constants may depend on the fixed compact set in which $\beta$ lives, but not on $n$.

\begin{Lemma}\label{lem:fixedWindowMagnetic}
\red{There exists a positive constant $c$ such that f}or every fixed $L\ge L_0$, \red{and for all sufficiently large $n$, depending on $J$ and the fixed $L$, we have}
\bel{eq:fixedLmagCompact}
\begin{split}
        a_n(R_n-L)&=\red{1-\frac{(R_n-L)^2}{R_n^2}
        +O_J\!\left(\frac{e^{-cL^2}}{R_n(1+L)}\right)},\\
        -a_n'(R_n-L)&=\red{\frac{2(R_n-L)}{R_n^2}
        +O_J\!\left(\frac{e^{-cL^2}}{R_n}\right)}.
\end{split}
\ee
Equivalently,

\begin{equation}\label{eq:fixedGammaWeak}
\begin{aligned}
        \red{\Gamma_n(-L)}&=b_nL+\red{O_J\left(\frac{L^2}{R_n}\right)}+O_J\left(\frac{e^{-cL^2}}{1+L}\right)\\
        \red{\Gamma_n'(-L)}&=-b_n+\red{O_J\left(\frac{L}{R_n}\right)}+O_J(e^{-cL^2}).
        \end{aligned}
        \end{equation}

\end{Lemma}

\begin{proof}  Put $\eta=L$ in
\eqref{eq:qSharp} and \eqref{eq:qSharpDer}.  This gives the two estimates for $a_n$ and $a_n'$ in
\eqref{eq:fixedLmagCompact}.  Multiplication by $n/(R_n-L)$ gives the first estimate in
\eqref{eq:fixedGammaWeak}. Differentiating
\[ 
        \Gamma_n(y)=\frac{n}{R_n+y}a_n(R_n+y)
\]
and using the derivative estimate for $q_n$ gives the second.  \end{proof}  

The error $O(e^{-cL^2})$ is a fixed-$L$ tail error and does not vanish as $n\to\infty$.
Lemma~\ref{lem:rightEndpointCompact}  assumes local compactness of translates of $(\Phi_n,\Gamma_n)$ and that any subsequential limit solves~\eqref{eq:interface} with the left boundary condition. It then shows that the right boundary condition must also hold.  Lemma~\ref{lem:qualitativeBoundary} supplies the required compactness.

\begin{Lemma}\label{lem:rightEndpointCompact}
Let $(\Phi,\Gamma)$ be a $C^2_{\rm loc}$ limit on (all) compact $y$-intervals of a subsequence of the translated radial
minimizers $(\Phi_n,\Gamma_n)$.  Suppose that the limit satisfies the interface equations
\eqref{eq:interface} and the left normalization
\[
        \Phi(y)\to0,\qquad \Gamma(y)+\sqrt\beta\,y\to0
        \qquad (y\to-\infty).
\]
Then the limit inherits
\[
        0\le\Phi\le1,\qquad \Phi'\ge0,\qquad \Gamma\ge0,\qquad \Gamma'\le0 
\]
\red{and}
\[
        \Phi(y)\to1,\qquad \Gamma(y)\to0
        \qquad (y\to+\infty).
\]
\end{Lemma}

\begin{proof}  For each $n$, the radial monotonicity in Proposition~\ref{prop:static} gives
\[
        \Phi_n'(y)=\phi_n'(R_n+y)\ge0 .
\]
Also,
\[
        \Gamma_n'(y)
        =\frac{n}{(R_n+y)^2}\bigl((R_n+y)a_n'(R_n+y)-a_n(R_n+y)\bigr)<0,
\]
because $a_n'<0$ and $a_n>0$.  Passing to the limit, $\Phi$ is nondecreasing and $\Gamma$ is
nonincreasing.  Thus
\[
        0\le\Phi\le1,\qquad \Phi'\ge0,\qquad \Gamma\ge0,\qquad \Gamma'\le0 .
\]
Since $0\le \Phi\le1$ and $\Gamma\ge0$, the limits
\[
        \ell:=\lim_{y\to+\infty}\Phi(y)\in[0,1],
        \qquad
        g:=\lim_{y\to+\infty}\Gamma(y)\in[0,\infty)
\]
exist.
\red{We now prove that $\Phi'$ and $\Gamma'$ have zero limits at $+\infty$. To see that, we first note that the boundedness of $\Phi$ and $\Gamma$, along with the interface equations, imply that $\Phi'$ and $\Gamma'$ are uniformly continuous. Now, if we assume for the sake of contradiction that there exists a positive $\varepsilon>0$ and a sequence $y_k\rightarrow\infty$ with $\Phi'(y_k)\geq\varepsilon$, uniform continuity for $\Phi'$ and $\Gamma'$ implies that $\Phi'\geq\frac{\varepsilon}{2}$ on fixed neighborhoods of the points $y_k$, each of which has the same size. However, this scenario cannot occur, since this would force $\Phi$ to gain a fixed positive amount on infinitely many disjoint intervals, even though it is bounded and nondecreasing. This shows that $\displaystyle \lim_{
\substack{y\rightarrow+\infty}
}\Phi'(y)=0$. We can similarly prove that $\displaystyle \lim_{
\substack{y\rightarrow+\infty}
}\Gamma'(y)=0$. 
}
 If $g\ell^2>0$, then $\Gamma''(y)=2\Gamma\Phi^2\to2g\ell^2>0$, contradicting $\Gamma'(y)\to0$.  Thus either
$g=0$ or $\ell=0$.
The alternative $\ell=0$ is impossible.  Since $\Phi$ is nondecreasing and nonnegative, $\ell=0$ would imply
$\Phi\equiv0$.  Then the second interface equation gives $\Gamma''=0$, so $\Gamma$ is affine.  The left normalization
$\Gamma(y)+\sqrt\beta\,y\to0$ as $y\to-\infty$ forces $\Gamma(y)=-\sqrt\beta\,y$, which is negative for large positive
$y$, contradicting $\Gamma\ge0$.  Hence $\ell>0$, and therefore $g=0$.

Finally, if $0<\ell<1$, then the right-hand side of the $\Phi$ equation satisfies
\[
        \red{\lim_{\substack{y\rightarrow\infty}}}\Gamma(y)^2\Phi(y)+\beta\Phi(y)(\Phi(y)^2-1)
        \red{=} \beta\ell(\ell^2-1)<0 .
\]
Thus $\Phi''$ is bounded above by a negative constant on a tail, contradicting $\Phi'(y)\to0$ and $\Phi'\ge0$.
Consequently $\ell=1$.  We have proved $\Phi(y)\to1$ and $\Gamma(y)\to0$. \end{proof}  

We can now prove the first, qualitative convergence to the CHMO layer on fixed windows. Note that \eqref{eq:vortexlimit} is much stronger than local compactness of $ (\Phi_n,\Gamma_n)$. The key to the full convergence is the uniqueness of the CHMO layer up to translations. 

\begin{Lemma}\label{lem:qualitativeBoundary}
Let $0<\beta<4$ and assume the coarse localization of
Proposition~\ref{prop:coarseloc}. Then, for every fixed $M<\infty$,
\bel{eq:vortexlimit}
        (\Phi_n,\Gamma_n)\longrightarrow(\Phi_\beta,\Gamma_\beta)
        \qquad\hbox{in }C^1([-M,M]) .
\ee
The convergence is uniform on compact parameter intervals in the following
sequential sense.  If $K\Subset(0,4)$,
$\beta_j\in K$, $n_j\to\infty$, and one chooses an arbitrary radial
minimizer for each pair $(n_j,\beta_j)$, then
\bel{eq:uniformQualitativeBoundary}
 \left\|(\Phi_{n_j}^{(\beta_j)},\Gamma_{n_j}^{(\beta_j)})
       -(\Phi_{\beta_j},\Gamma_{\beta_j})\right\|_{C^1([-M,M])}
 \longrightarrow0 .
\ee
\end{Lemma}

\begin{proof}  Fix $M<\infty$ and choose
$L>\max\{L_0,M+2\}$.  Lemma~\ref{lem:fixedWindowMagnetic} provides uniform
bounds on the left Cauchy data $\Gamma_n(-L)$ and $\Gamma_n'(-L)$.  On
$[-L,M+1]$ the exact magnetic equation \eqref{eq:exactGamma} reads
\[
        \Gamma_n''(y)
        =-\frac1{R_n+y}\,\Gamma_n'(y)
          +\frac1{(R_n+y)^2}\,\Gamma_n(y)
          +2\Gamma_n(y)\Phi_n^2(y).
\]
Since $0\le\Phi_n\le1$ and $R_n+y\ge R_n-L$ on this interval, we have
\[
        |\Gamma_n''(y)|\le C_{L,M}\bigl(|\Gamma_n'(y)|+|\Gamma_n(y)|\bigr)
        \qquad (y\in[-L,M+1])
\]
for all large $n$.  Setting $Z_n(y):=|\Gamma_n(y)|+|\Gamma_n'(y)|$ and using
$|\Gamma_n'|\le Z_n$ and $|\Gamma_n''|\le C_{L,M} Z_n$, we obtain the
differential inequality
\[
        Z_n'(y)\le (1+C_{L,M})Z_n(y)
\]
for almost every $y\in[-L,M+1]$ (note that $Z_n$ is Lipschitz).
Gronwall's inequality and the bounds at $y=-L$ then yield uniform bounds
for $\Gamma_n$ and $\Gamma_n'$ on $[-L,M+1]$, and hence on
$[-M-1,M+1]$.
For the scalar derivative, $0\le\Phi_n\le1$ and $\Phi_n'\ge0$ imply
\[
        \int_{-M-1}^{M+1}\Phi_n'(y)\,dy
        \le1.
\]
Hence there exists $y_n\in[-M-1,-M]$ with $\Phi_n'(y_n)\le1$.
Rewrite \eqref{eq:exactPhi} as an equation for
$p_n:=\Phi_n'$ in the linear form
\[
        p_n'+(R_n+y)^{-1}p_n=f_n,
\]
where $f_n(y)=\Gamma_n(y)^2\Phi_n(y)+\beta\Phi_n(y)(\Phi_n(y)^2-1)$.
On $[-M-1,M+1]$ we have $|f_n|\le C_M$ by the already established bounds on
$\Gamma_n$ and the fact that $0\le\Phi_n\le1$.

Multiplying by the integrating factor $R_n+y$ and integrating from $y_n$ to
$y$ gives
\[
        (R_n+y)p_n(y)
        =(R_n+y_n)p_n(y_n)+\int_{y_n}^y (R_n+s)f_n(s)\,ds.
\]
Since $R_n+s\simeq R_n$ on $[-M-1,M+1]$, this yields
\[
        |p_n(y)|\le C_M\Bigl(p_n(y_n)+\int_{-M-1}^{M+1}|f_n(s)|\,ds\Bigr)
        \le C_M,
\]
and therefore $\Phi_n'$ is uniformly bounded on $[-M-1,M+1]$.  The two
equations then give uniform second-derivative bounds,
and differentiating them yields uniform third-derivative bounds on
$[-M,M]$.  By Arzel\`a--Ascoli and a diagonal argument, every subsequence
has a further subsequence converging in $C^2(I)$ for all compact intervals~$I$, hence in $C^2_{\rm loc}(\RR)$, to a global solution
of \eqref{eq:interface}.  In what follows we work with this single diagonal
limit $(\Phi,\Gamma)$.

The left affine normalization of $(\Phi,\Gamma)$ is determined by the
\emph{magnetic radius} $R_n$ (equivalently by $b_n=2n/R_n^2$). Indeed,  by Lemma~\ref{lem:fixedWindowMagnetic}, for
each fixed large $L$,
\begin{align*}
        \Gamma_n(-L)&=\red{b_nL+O_J\left(\frac{L^2}{R_n}\right)+O_J\left(\frac{e^{-cL^2}}{1+L}\right)}\\
        \Gamma_n'(-L)&=\red{-b_n+O_J\left(\frac{L}{R_n}\right)+O_J(e^{-cL^2})}.
\end{align*}
Also, \eqref{eq:movingclear}, evaluated at
$\eta=L$, gives
\[
        \Phi_n(-L)+|\Phi_n'(-L)|\le C e^{-cL^2}+o_n(1)
\]
for fixed $L$ (after reducing $c$).  Passing first to the compact
subsequential limit (with $L$ fixed), using $b_n\to\sqrt\beta$, and only
then sending $L\to\infty$ yields
\[
        \Gamma(y)+\sqrt\beta\,y\to0,\qquad
        \Gamma'(y)+\sqrt\beta\to0,\qquad
        \Phi(y)\to0
        \qquad (y\to-\infty).
\] 
\red{Lemma~\ref{lem:rightEndpointCompact} readily implies the conditions at $\Phi(y)\to 1$ and $\Gamma(y)\to 0$ at the superconducting end. Given that the solutions to the interface equations are known to be unique, we infer that every limit point is equal to $(\Phi_\beta,\Gamma_\beta)$. This proves qualitative convergence for fixed $\beta$. In order to prove the uniformity statement, we assume for the sake of contradiction that \eqref{eq:uniformQualitativeBoundary} is false, which means that there exists $\varepsilon>0$ and sequences $n_j\to\infty$ and $\beta_j\in K$ for which the norm displayed there is at least $\varepsilon$. Upon passing to a subsequence, we may assume that $\beta_j\to\beta_\ast$. Proposition~\ref{prop:coarseloc}, together with Lemmas~\ref{lem:clearingout} and ~\ref{lem:fixedWindowMagnetic} imply the existence of a $C^1([-M,M])$ limit which is a solution of the interface problem at $\beta_\ast$ satisfying the normalized endpoint conditions. By CHMO uniqueness, it must coincide with $(\Phi_{\beta_\ast},\Gamma_{\beta_\ast})$.}  On the other hand,
Lemma~\ref{lem:interfaceUniformTails} (and its compact-interval continuity
conclusion) gives
\[
 (\Phi_{\beta_j},\Gamma_{\beta_j})
 \longrightarrow(\Phi_{\beta_*},\Gamma_{\beta_*})
 \quad\hbox{in }C^1([-M,M]).
\]
This contradicts the choice of the sequence and proves
\eqref{eq:uniformQualitativeBoundary}. \end{proof}  

\section{The Higgs equation in the core--boundary overlap}\label{sec:trueoverlap}

\red{Throughout this section, $J$ is going to be a fixed compact subinterval of $\displaystyle (0,4)$. Unless otherwise noted, all of the constants will also depend on $J$.} Set
\bel{eq:psidef}
       \psi_n(u)=u^{1/2}\phi_n(u).
\ee
Then the Higgs (or scalar) equation in \eqref{eq:vortex} becomes
\bel{eq:psi}
       \psi_n''(u)=\red{(}\mathcal Q_n(u)+N_n(u)\red{)}\psi_n(u),
\ee
where
\bel{eq:QN}
       \mathcal Q_n(u)=\frac{n^2a_n(u)^2}{u^2}-\beta-\frac{1}{4u^2},\qquad
       N_n(u)=\beta\phi_n(u)^2 .
\ee
We only use this scalar equation where $N_n$ is small.  Fix two exponents
\bel{eq:twoOverlapExponents}
       0<\alpha<\widehat\alpha<\red{1},\qquad
       M_n=R_n^\alpha,\qquad T_n=R_n^{\widehat\alpha}.
\ee
These two scales arise as follows: we construct a WKB basis on the larger interval $[L,T_n]$ and use regularity at $\eta=T_n$ to show that the coefficient of the growing solution is negligible. On the smaller window $[L,M_n]$, which we refer to as the~\emph{left overlap} region, the growing mode can then be ignored. 
In fact, we will see that for 
\bel{eq:leftoverlap}
       L\le \eta:=-y=R_n-u\le M_n,
\ee
 both descriptions of the vortex are simultaneously accurate: the core
estimate gives $\Phi_n^-(\eta)\ll1$ (so $N_n=\beta\phi_n^2$ is negligible in
\eqref{eq:psi}\red{; a definition of $\Phi_n^-$ is briefly recalled below}), while the boundary-layer variables lie in the left tail regime with
$y=-\eta\to-\infty$ and $\Gamma_n^-(\eta)=b_n\eta+O_J(\eta^2/R_n)\sim\sqrt\beta\,\eta$ on $[L,M_n]$.  Throughout this section $L$ is fixed first, then
$n\to\infty$, and only afterwards do we send $L\to\infty$. \red{We also note that in Proposition \ref{prop:coreWKB}, we shall need to impose the stronger condition $0<\alpha<\widehat{\alpha}<\frac{1}{3}$ in order to control a resulting exponential action, and that it is the only result in this section for which we have to do so.}

To analyze the Higgs equation \eqref{eq:psi} on this overlap we
construct a fundamental system of two linearly
independent solutions on the larger interval $\eta\in[L,T_n]$: a recessive
(decaying) branch on the one hand,  and a dominant (growing) branch on the other hand. These are obtained by a WKB 
construction.  The coefficient of the dominant branch is ruled out by
imposing regularity of the original vortex at the origin. This condition is
conveniently evaluated at the interior point $\eta=T_n$.  The property  $M_n/T_n\to0$ separates the dominant and
recessive branches since \red{they} are not of comparable size at $\eta=T_n$. This is important on the smaller interval
$[L,M_n]$ where our asymptotic formula will be used. 

\red{We recall the definitions}
\[
       \Phi_n^-(\eta)=\phi_n(R_n-\eta),\qquad
       \Gamma_n^-(\eta)= \frac{n}{R_n-\eta}a_n(R_n-\eta),
\]
and
\[
       \Psi_n(\eta):=(R_n-\eta)^{1/2}\Phi_n^-(\eta).
\]
Then
\bel{eq:trueOverlapPsi}
       \Psi_n''(\eta)
       =\red{(}p_n(\eta)^2+\beta\Phi_n^-(\eta)^2\red{)}\Psi_n(\eta),
\ee
where
\bel{eq:trueOverlapP}
       p_n(\eta)^2=(\Gamma_n^-(\eta))^2-\beta-\frac{1}{4(R_n-\eta)^2}\red{,}
\ee
\red{with $p_n(\eta)$ chosen so that it is nonnegative. As long as $L$ is sufficiently large, $\eta\geq L$ ensures that we may simply define $p_n(\eta)$ as the square root of the right-hand side of \eqref{eq:trueOverlapP}.}

The scalar equation in the left overlap region is controlled by the following symbol estimates.

\begin{Lemma}\label{lem:trueOverlapEstimates}
Fix $0<\widehat\alpha<\red{1}$.  \red{For $L$ sufficiently large depending on $J$ and all sufficiently large $n$,} uniformly for
$L\le\eta\le T_n=R_n^{\widehat\alpha}$,
\begin{equation}\label{eq:pnTrueOverlap}
    \begin{aligned}
       p_n(\eta)&=\red{b_n\eta-\frac{\beta}{2b_n\eta}
       +O_J(\eta^{-3})+O_J(\eta^2/R_n)
       }\\
\red{p'_n(\eta)}&\red{=b_n+\frac{\beta}{2b_n\eta^2}
       +O_J(\eta^{-4})+O_J(\eta/R_n)
       }\\
\red{p''_n(\eta)}&\red{=-\frac{\beta}{b_n\eta^3}
       +O_J(\eta^{-5})+O_J(1/R_n)
       },
\end{aligned}   
\end{equation}
and
\bel{eq:overlapintegral}
       \int_L^{T_n}
       \left(
       \frac{|p_n''\red{(\eta)}|}{p_n^2\red{(\eta)}}+\frac{|p_n'\red{(\eta)}|^2}{p_n^3\red{(\eta)}}
       +\frac{\beta(\Phi_n^-\red{(\eta)})^2}{p_n\red{(\eta)}}
       \right)d\eta
       \Le \red{C_{J}L^{-2}}.
\ee
In particular $p_n>0$ on the overlap for $L$ large and $n$ large.
\end{Lemma}

\begin{proof}  This is a direct consequence of Lemma~\ref{lem:clearingout}.
From \eqref{eq:movingGamma} we have, uniformly on $L\le\eta\le T_n$,
\[
 \Gamma_n^-(\eta)=\red{b_n\eta+O_J(\eta^2/R_n)
  +O_J\!\left(\frac{e^{-\red{c_J}\eta^2}}{1+\eta}\right)
  },
\]
and $b_n\to\sqrt\beta>0$.  Substituting this into \eqref{eq:trueOverlapP} and
expanding the square root yields \eqref{eq:pnTrueOverlap} once $L$ is large.
Differentiating \eqref{eq:trueOverlapP} we obtain that 
\[
        2p_n p_n'(\eta)
        =2\Gamma_n^-(\eta)\partial_\eta\Gamma_n^-(\eta)
          -\frac{1}{2}(R_n-\eta)^{-3},
\]
and hence
\[
        p_n'(\eta)
        =\frac{\Gamma_n^-\partial_\eta\Gamma_n^-(\eta)}{p_n(\eta)}
        +O\!\left(\frac{(R_n-\eta)^{-3}}{p_n(\eta)}\right).
\]
Using \eqref{eq:movingGammaDer} and $p_n\simeq b_n\eta$ gives the stated
\red{asymptotic expansion} for $p_n'$.  Differentiating once more gives
\[
        2p_n'(\eta)^2+2p_n p_n''(\eta)
        =2(\partial_\eta\Gamma_n^-(\eta))^2
          +2\Gamma_n^-\partial_\eta^2\Gamma_n^-(\eta)
          -\frac{3}{2}(R_n-\eta)^{-4}.
\]
The term $\partial_\eta^2\Gamma_n^-$ is controlled by the exact identity
\eqref{eq:GammaEtaEquation} and the bound \eqref{eq:movingGammaSecond},
which yields the stated \red{asymptotic expansion} for $p_n''$.

Finally, the leading terms $p_n(\eta)=b_n\eta-\beta/(2b_n\eta)+O(\eta^{-3})$
imply $|p_n''\red{(\eta)}|/p_n^2\red{(\eta)}\lesssim \eta^{-5}$ and $|p_n'|^2/p_n^3\lesssim
\eta^{-3}$, giving the contribution $O(L^{-2})$ to
\eqref{eq:overlapintegral}.  The  error term \red{$O_J(\eta^2/R_n)$ yields the
expression  $O_J(1/(\eta^2R_n))$ in the first fraction and $O_J(1/(\eta R_n^2))$ for the second. The resulting contributions are $\frac{1}{LR_n}$ and $\frac{\log\left(\frac{T_n}{L}\right)}{R_n^2}$, up to a constant $C_J$, which can both be absorbed into $L^{-2}$}.  The nonlinear term
$\beta(\Phi_n^-)^2\red{(\eta)}/p_n$ is controlled by \eqref{eq:movingclear}, giving the
contribution $\red{e^{-c_JL^2}}$.  \end{proof}  

These symbol bounds will lead to an exact WKB basis in the finite-overlap region. 
For further background on perturbative analysis by Liouville--Green, rigorous WKB,  and
Langer transforms, see \cite{CDST,CST,SchlagPointwise}. We now carry out the 
construction of the exact WKB basis with the precise normalizations and bounds needed for the later
projection and matching arguments. 

\begin{Lemma}\label{lem:finiteWeberFrame}
Fix $0<\widehat\alpha<\red{1}$ and $L$ \red{sufficiently large, depending on $J$}.  Let
\[
        S_n(\eta)=\int_L^\eta p_n(s)\,ds .
\]
There are exact solutions $d_{n,L}$ and $g_{n,L}$ of
\[
       F''\red{(\eta)}=\red{(}p_n(\eta)^2+\beta\Phi_n^-(\eta)^2\red{)}F\red{(\eta)},
       \qquad L\le\eta\le T_n,
\]
normalized by the Wronskian condition $\calW[d_{n,L},g_{n,L}]=1$, where
$\calW[f,g]=fg'-f'g$, such that
\[
       d_{n,L}(\eta)=p_n(\eta)^{-1/2}e^{-S_n(\eta)}(1+e_{n,L}^d(\eta)),
\]
\[
       g_{n,L}(\eta)=\frac{1}{2}p_n(\eta)^{-1/2}e^{S_n(\eta)}(1+e_{n,L}^g(\eta)),
\]
In fact,
\bel{eq:weberFrameError}
       \sup_{L\le\eta\le T_n}
       \left(|e_{n,L}^d(\eta)|+|e_{n,L}^g(\eta)|\right)
       \Le \red{C_{J}L^{-2}}.
\ee
The same relative estimates \red{involving $e_{n,L}^d$ and $e_{n,L}^g$} hold after applying $p_n^{-1}\partial_\eta$ to
these two exact solutions. 
Every solution of \eqref{eq:trueOverlapPsi} is uniquely expressible as
\[
       \Psi=A_{n,L}d_{n,L}+B_{n,L}g_{n,L}.
\]
\end{Lemma}

\begin{proof}   Put
\[
        E_n(\eta)=\beta(\Phi_n^-(\eta))^2-p_n(\eta)^{1/2}(p_n(\eta)^{-1/2})'' .
\]
A direct computation gives
\[
        p_n^{1/2}(p_n^{-1/2})''=
        \frac{3}{4}\frac{(p_n')^2}{p_n^2}-\frac{1}{2}\frac{p_n''}{p_n},
\]
so that
\[
        \frac{|E_n|}{p_n}
        \le \frac{\beta(\Phi_n^-)^2}{p_n}
        +C\left(\frac{|p_n''|}{p_n^2}+\frac{|p_n'|^2}{p_n^3}\right).
\]
Therefore the integral bound \eqref{eq:overlapintegral} in
Lemma~\ref{lem:trueOverlapEstimates} implies
\bel{eq:WeberFrameEnorm}
        \int_L^{T_n}\frac{|E_n(s)|}{p_n(s)}\,ds
        \le \delta_{n,L}:=
        \red{C_{J}L^{-2}}.
\ee
For the decaying branch put $D_n=p_n^{-1/2}e^{-S_n}$ and write $F=D_n(1+\omega)$\red{, where $F$ is a solution of 
\begin{align*}
 F''(\eta)=\red{(}p_n(\eta)^2+\beta\Phi_n^-(\eta)^2\red{)}F(\eta)   .
\end{align*}}
Then
\[
        D_n''=\red{(}p_n^2+p_n^{1/2}(p_n^{-1/2})''\red{)}D_n,
\]
so the difference between the true coefficient $p_n^2+\beta(\Phi_n^-)^2$ and the coefficient solved by $D_n$ is exactly
$E_n$.  Hence
\[
       (D_n^2\omega')'
       =D_n^2 E_n(1+\omega).
\]
We impose the  normalization
\[
        \omega(T_n)=0,\qquad D_n(T_n)^2\omega'(T_n)=0 .
\]
Thus
\bel{eq:omegaTerminalVolterra}
        \omega(\eta)=
        \int_\eta^{T_n}
        \left(\int_\eta^s \frac{D_n(s)^2}{D_n(t)^2}\,dt\right)
        E_n(s)(1+\omega(s))\,ds .
\ee
Since
 \begin{align*} 0\leq\int_\eta^s \frac{D_n(s)^2}{D_n(t)^2}\,dt=p_n(s)^{-1}e^{-2S_n(s)}\int_\eta^sp_n(t)e^{2S_n(t)}\,dt=\frac{e^{2S_n(s)}-e^{2S_n(\eta)}}{2p_n(s)e^{2S_n(s)}}\leq\frac{1}{2p_n(s)}, 
\end{align*}
\eqref{eq:WeberFrameEnorm} gives a contraction on $L^\infty[L,T_n]$ and yields the displayed bound for
$e_{n,L}^d=\omega$.  The first differentiated estimate can be read directly from the once-integrated equation:
\[
        \omega'(\eta)=-D_n(\eta)^{-2}
        \int_\eta^{T_n}D_n(s)^2E_n(s)(1+\omega(s))\,ds .
\]
The monotonicity of the action gives
\[
        \frac{D_n(s)^2}{D_n(\eta)^2}
        \le C\frac{p_n(\eta)}{p_n(s)}
        \exp\left(-2\int_\eta^sp_n(t)\,dt\right),
        \qquad s\ge\eta.
\]
Since the exponential factor is at most $1$, the identity for $\omega'$ yields
\[
        p_n(\eta)^{-1}|\omega'(\eta)|
        \le C\int_\eta^{T_n}\frac{|E_n(s)|}{p_n(s)}\,ds
        \le C\int_L^{T_n}\frac{|E_n(s)|}{p_n(s)}\,ds
        \le C\delta_{n,L}.
\]
This is the asserted estimate after applying $p_n^{-1}\partial_\eta$.  

For the growing branch put $G_n=(1/2)p_n^{-1/2}e^{S_n}$ and write $F=G_n(1+\omega_g)$.  This time the Volterra equation
is initialized at $\eta=L$ \red{ (with analogous initial conditions)}:
\[
        \omega_g(\eta)=
        \int_L^\eta
        \left(\int_s^\eta \frac{G_n(s)^2}{G_n(t)^2}\,dt\right)
        E_n(s)(1+\omega_g(s))\,ds ,
\]
and the kernel obeys the same $C/p_n(s)$ bound.  This gives $e_{n,L}^g$ and its differentiated estimate.  The factor
$p_n^{-1}\partial_\eta$ estimate follows here from the explicit identity
\[
        \omega_g'(\eta)=G_n(\eta)^{-2}
        \int_L^\eta G_n(s)^2E_n(s)(1+\omega_g(s))\,ds
\]
and the analogous action bound.  The factor
$1/2$ is the WKB normalization for which the unperturbed Wronskian of $D_n$ and $G_n$ is one.  The Volterra corrections
change the Wronskian to a constant $c_{n,L}=1+O(\delta_{n,L})$.  We keep the recessive branch fixed and replace the
preliminary growing branch by $c_{n,L}^{-1}$ times that branch, so that
$\calW[d_{n,L},g_{n,L}]=1$ exactly.  Since
$c_{n,L}^{-1}=1+O(\delta_{n,L})$, this does not change any of the estimates. \end{proof}  

When needed, a second normalized derivative follows from the differential equation for $\omega$ and the symbol bounds in
Lemma~\ref{lem:trueOverlapEstimates}.
The next lemma shows that regularity at the origin strongly suppresses the growing  component in this basis.

\begin{Lemma}\label{lem:recessiveSelection}
Let $M_n,T_n$ be as in \eqref{eq:twoOverlapExponents}, and write
\[
        \Psi_n=A_{n,L}d_{n,L}+B_{n,L}g_{n,L}
\]
with respect to the WKB basis of Lemma~\ref{lem:finiteWeberFrame} on
$[L,T_n]$.
Then, 
\bel{eq:growingAbs}
       |B_{n,L}|g_{n,L}(L)
       \Le \red{C_{J}}
       e^{\left(-\int_L^{T_n}p_n(s)\,ds\right)}
       \red{e^{-\red{c_J}T_n^2}}.
\ee
If, for this fixed $L$, the normalized recessive coefficient satisfies\footnote{\label{fn:AnLnorm}The factor $R_n^{-1/2}$ comes from the left-endpoint normalization: recall $\Psi_n\red{(\eta)}=(R_n-\eta)^{1/2}\Phi_n^-\red{(\eta)}$.  At $\eta=L$ we have $(R_n-L)^{1/2}\sim R_n^{1/2}$, while $d_{n,L}(L)$ is $\simeq_{\red{J,L}} 1$ in the fixed-$L$ WKB normalization.  Thus $R_n^{-1/2}A_{n,L}$ is the scaling that yields a finite, nontrivial limit for the recessive coefficient as $n\to\infty$. This hypothesis is verified below in
Lemma~\ref{lem:leftProjection}.}
$R_n^{-1/2}A_{n,L}\to A_{\beta,L}\ne0$, \red{uniformly for $\beta\in J$, with $\displaystyle \inf_{\substack{\beta\in J}}|A_{\beta,L}|>0$,} then
\bel{eq:growingRelative}
 \sup_{L\le\eta\le M_n}
 \frac{|B_{n,L}|g_{n,L}(\eta)}{|A_{n,L}|d_{n,L}(\eta)}
 \le \red{C_{J,L}}\red{e^{-c_JT_n^2+C_JM_n^2}}.
\ee
\end{Lemma}

\begin{proof}  Since $\calW[d_{n,L},g_{n,L}]=1$,
\[
        B_{n,L}=\red{\calW[d_{n,L},\Psi_n]}
\]
for every $\eta$ in the overlap.  We evaluate this Wronskian at the right endpoint
$\eta=T_n$ of the interval $[L,T_n]$.  Here we use \eqref{eq:movingclear} with exponent
$\alpha=\widehat\alpha$ (so that it applies up to $\eta=T_n=R_n^{\widehat\alpha}$), and we \red{get that}
\bel{eq:PsistarBoundDetailed}
       |\Psi_n(T_n)|+p_n(T_n)^{-1}|\Psi_n'(T_n)|
       \Le \red{C_{J}e^{-\red{c_J}T_n^2}}.
\ee
The factor $R_n^C$ records only polynomial losses coming from $(R_n-\eta)^{1/2}$, from
$p_n(T_n)\simeq R_n^{\widehat\alpha}$, and from one differentiation.  The
recessive basis element satisfies
\[
       |d_{n,L}(T_n)|+p_n(T_n)^{-1}|d_{n,L}'(T_n)|
       \Le C_J p_n(T_n)^{-1/2}
       \exp\!\left(-\int_L^{T_n}p_n(s)\,ds\right),
\]
and $g_{n,L}(L)\simeq p_n(L)^{-1/2}$.  The powers of $p_n(L)$ and
$p_n(T_n)$ are polynomial in $R_n$, so the Wronskian identity gives
\eqref{eq:growingAbs}. 
Under the limit \blue{hypothesis} on $A_{n,L}$, for each fixed $L$ there is $c_L>0$ such that
\bel{eq:fixedLlowerRecessive}
        |A_{n,L}|d_{n,L}(L)\Ge c_L R_n^{1/2}
\ee
for all large $n$.     For $\eta\le M_n$, the basis estimates give
\[
 \frac{g_{n,L}(\eta)d_{n,L}(L)}{g_{n,L}(L)d_{n,L}(\eta)}
 \le C R_n^C\exp\left(2\int_L^\eta p_n(s)\,ds\right)
 \le \red{C_J}  \red{e^{c_JM_n^2}}.
\]
Combining this inequality, \eqref{eq:growingAbs}, and
\eqref{eq:fixedLlowerRecessive} yields
\[
 \sup_{L\le\eta\le M_n}
 \frac{|B_{n,L}|g_{n,L}(\eta)}{|A_{n,L}|d_{n,L}(\eta)}
 \le \red{C_{J,L}}
 e^{-\int_L^{T_n}p_n+2\int_L^{M_n}p_n}
 \red{e^{-\red{c_J}T_n^2}}.
\]
\blue{Since by Lemma~\ref{lem:trueOverlapEstimates} } $p_n(s)\ge \red{c_J}s$ for $s\ge L$, while $p_n(s)\le \red{C_J}s$ for $L\le s\le M_n$, this is bounded by the right side of
\eqref{eq:growingRelative}. \end{proof}  

The resulting left-overlap branch has the following WKB form (recall $\eta=-y=R_n-u$).

\begin{Proposition}\label{prop:coreWKB}
Fix $0<\alpha<\widehat\alpha<\frac{1}{3}$ and $L\gg1$\red{, depending on $J$}.  Assume\footnotemark[\getrefnumber{fn:AnLnorm}] that
$R_n^{-1/2}A_{n,L}\to A_{\beta,L}\ne0$ for the recessive coefficient in
the WKB basis, uniformly for $\beta\in J$, with
$\inf_{\beta\in J}|A_{\beta,L}|>0$.  
Then, uniformly for
\[
        L\le -y\le M_n=R_n^\alpha,
        \qquad u=R_n+y,
\]
the Higgs profile $\phi_n$ satisfies
\bel{eq:coreWKB}
\phi_n(R_n+y)
  =\wt C_{n,L}\,|y|^{(\beta/b_n-1)/2}
        \exp\left(-\frac{b_n}{2}y^2\right)
        \left(1+r_n^{\rm c}(y)\right),
\ee
where
\bel{eq:coreWKBcoefficient}
 \wt C_{n,L}
 =R_n^{-1/2}A_{n,L}b_n^{-1/2}
   e^{b_nL^2/2}L^{-\beta/(2b_n)} .
\ee
\red{We also have} 
\bel{eq:coreWKBerr}
        |r_n^{\rm c}(y)|
        \Le \red{C_{J}L^{-2}}+o_{n,L}(1)
\ee
as $n\to\infty$, for fixed $L$\red{, along with} 
 the differentiated bound
\[
 |r_n^{\rm c}(y)|+p_n(-y)^{-1}|\partial_y r_n^{\rm c}(y)|
 \le \red{C_{J}L^{-2}}+o_{n,L}(1).
\]
\end{Proposition}

\begin{proof}
Lemma~\ref{lem:finiteWeberFrame} gives the exact two-mode representation
\red{
\begin{align*}
    \Psi_n(\eta)&=A_{n,L}d_{n,L}(\eta)+B_{n,L}g_{n,L}(\eta),
\end{align*}
where
\[
       d_{n,L}(\eta)=p_n(\eta)^{-1/2}e^{-S_n(\eta)}(1+e_{n,L}^d(\eta)),
\]
and
\[
       g_{n,L}(\eta)=\frac{1}{2}p_n(\eta)^{-1/2}e^{S_n(\eta)}(1+e_{n,L}^g(\eta)),
\]
The errors $e_{n,L}^d$ and $e_{n,L}^g$ satisfy the bounds
\begin{align*}
       \sup_{L\le\eta\le T_n}
       \left(|e_{n,L}^d(\eta)|+|p_n^{-1}(\eta)\partial_\eta e_{n,L}^d(\eta)|+|e_{n,L}^g(\eta)|+|p_n^{-1}(\eta)\partial_\eta e_{n,L}^g(\eta)|\right)
       \Le \red{C_{J}L^{-2}}.
\end{align*} 
}
\red{By} Lemma~\ref{lem:recessiveSelection} \red{, we have
\begin{align*}
    \sup_{L\le\eta\le M_n}
 \frac{|B_{n,L}|g_{n,L}(\eta)}{|A_{n,L}|d_{n,L}(\eta)}
 \le \red{C_{J,L}}\red{e^{-c_JT_n^2+C_JM_n^2}}\leq C_JL^{-2},
\end{align*}
after increasing $C_J>0$ if necessary. This removes the growing branch in relative size, in the sense that its contribution can be absorbed into $r_n^c$ after the change of variable $y=-\eta$.

}
The differentiated estimate is obtained by differentiating the Volterra
corrections and the prefactors in the reduction from the exact fundamental
system to \eqref{eq:coreWKB}.  Lemma~\ref{lem:trueOverlapEstimates} and
Lemma~\ref{lem:recessiveSelection} bound each resulting term by the same right-hand side.
\red{Indeed, we have 
\begin{align*}
p_n(\eta)^{-1}\partial_\eta\frac{B_{n,L}g_{n,L}(\eta)}{A_{n,L}d_{n,L}(\eta)}&=\frac{p_n(\eta)^{-1}B_{n,L}\partial_\eta g_{n,L}(\eta)}{A_{n,L}d_{n,L}(\eta)}-\frac{p_n(\eta)^{-1}\partial_\eta d_{n,L}(\eta) B_{n,L} g_{n,L}(\eta)}{A_{n,L}d^2_{n,L}(\eta)}
\end{align*}
We shall analyze the first fraction. We have
\begin{align*}
  \frac{p_n(\eta)^{-1}B_{n,L}\partial_\eta g_{n,L}(\eta)}{A_{n,L}d_{n,L}(\eta)}&=\frac{B_{n,L}g_{n,L}(\eta)\left(1-\frac{p'_n(\eta)}{2p^2_n(\eta)}+\frac{p_n(\eta)^{-1}\partial_\eta e^g_{n,L}(\eta)}{1+e^g_{n,L}(\eta)}\right)}{A_{n,L}d_{n,L}(\eta)}  
\end{align*}
We now immediately get that 
\begin{align*}
  & \sup_{L\le\eta\le M_n}\left|\frac{p_n(\eta)^{-1}B_{n,L}\partial_\eta g_{n,L}(\eta)}{A_{n,L}d_{n,L}(\eta)}\right| \\ 
 &\leq \sup_{L\le\eta\le M_n}\frac{|B_{n,L}|g_{n,L}(\eta)}{|A_{n,L}|d_{n,L}(\eta)}\sup_{L\le\eta\le M_n}\left|1-\frac{p'_n(\eta)}{2p^2_n(\eta)}+\frac{p_n(\eta)^{-1}\partial_\eta e^g_{n,L}(\eta)}{1+e^g_{n,L}(\eta)}\right|
 \le C_{J}L^{-2}. 
\end{align*}
The second fraction can be analyzed similarly, once we write
\begin{align*}
 \partial_\eta d_{n,L}(\eta)&=d_{n,L}(\eta)\left(-p_n(\eta)-\frac{p'_n(\eta)}{2p_n(\eta)}+\frac{\partial_\eta e^d_{n,L}(\eta)}{1+e^d_{n,L}(\eta)}\right) .  
\end{align*}
}
\red{From} \eqref{eq:pnTrueOverlap}\red{, we have}
\[
\begin{split}
       \int_L^\eta p_n(s)\,ds
       &=\red{\frac{b_n}{2}(\eta^2-L^2)-\frac{\beta}{2b_n}\log\frac{\eta}{L}
         +O_J(L^{-2})+O_J(\eta^3/R_n)}.
\end{split}
\]
\red{The last} term tends to zero since $\eta\leq R_n^\alpha$, and $\displaystyle \alpha<\frac{1}{3}$. \red{We now plug this expansion into the contribution $A_{n,L}d_{n,L}(\eta)$ which corresponds to the decaying branch, and has turned out to be the principal one. We get that
\begin{align*}
A_{n,L}d_{n,L}(\eta)&=A_{n,L}p_n(\eta)^{-1/2}e^{-\frac{b_n}{2}(\eta^2-L^2)+\frac{\beta}{2b_n}\log\frac{\eta}{L}
         } (1+O_J(L^{-2})+o_{n,L}(1))\\
         &=A_{n,L}p_n(\eta)^{-1/2}e^{-\frac{b_n}{2}(\eta^2-L^2)
         }\eta^{\frac{\beta}{2b_n}}L^{-\frac{\beta}{2b_n}} (1+r^c_n(\eta))
\end{align*}
}
The WKB prefactors contribute $p_n(\eta)^{-1/2}\simeq b_n^{-1/2}\eta^{-1/2}$ and
$(R_n-\eta)^{-1/2}=R_n^{-1/2}(1+O(R_n^{\alpha-1}))$ on the asserted interval.
\red{Since $\Phi_n^{-}(\eta)=(R_n-\eta)^{-1/2}\Psi_n(\eta)$, making the change of variable $y=-\eta$} (and absorbing the resulting $1+o(1)$ factors into
$r_n^{\rm c}$) yields \red{\eqref{eq:coreWKB} and }\eqref{eq:coreWKBcoefficient}.
\end{proof}

On any subregion of the window   for which
\[
       |b_n-\sqrt\beta|\,y^2\to0,
\]
this implies
\[
       \phi_n(R_n+y)
       =\wt C_{n,L}|y|^{(\sqrt\beta-1)/2}
       e^{-\sqrt\beta y^2/2}
       \left(1+O_J(L^{-2})+o_{n,L}(1)\right).
\]

\section{The Weber tail of the interface}\label{sec:webertail}

We now derive the  asymptotic behavior of the vortices to the left of the interface.  The leading order will be given by parabolic-cylinder functions followed by a perturbative analysis of the 
Weber equation (for the Weber equation and these special functions cf.~\cite[\S12.1--\S12.2]{DLMF}). This is related to the normal-form mechanism of~\cite{CPS}.  In order to keep the present proof
self-contained, we include a  Weber matching Lemma~\ref{lem:webermatching} and prove
it directly by a Volterra argument. We begin with the CHMO interface itself, which does not depend on~$n$.  In view of the asymptotics of the magnetic field on the left, we write for $y<0$, 
\[
        \Gamma_\beta(y)=-\sqrt\beta\,y+g_\beta(y).
\]
Then
\begin{align}
 \Phi_\beta''\red{(y)}&=(\beta y^2-\beta)\Phi_\beta\red{(y)}+\rho_\beta(y)\Phi_\beta\red{(y)},\label{eq:PhiLeft}\\
 g_\beta''\red{(y)}&=2\Gamma_\beta\red{(y)}\Phi_\beta^2\red{(y)}.\nonumber
\end{align}
where
\begin{equation}\label{eq:rho}
\rho_\beta(y)=-2\sqrt\beta\,y\,g_\beta(y)+g_\beta(y)^2+\beta\Phi_\beta(y)^2 .
\end{equation}
The magnetic correction satisfies the Volterra equation
\bel{eq:gVolterra}
        g_\beta(y)=2\int_{-\infty}^y (y-s)\Gamma_\beta(s)\Phi_\beta(s)^2\dd s.
\ee
We now pass to the standard parabolic-cylinder scaling  by setting
\[
        x=\sqrt2\,\beta^{1/4}y,\qquad
        U_\beta(x)=\Phi_\beta\left(\frac{x}{\sqrt2\,\beta^{1/4}}\right)
\]
Then
\begin{equation}\label{eq:UbetaODE}
        U_\beta''(x)=\left(\frac{x^2}{4}-\frac{\sqrt\beta}{2}\right)U_\beta(x)
                  +\wt\rho_\beta(x)U_\beta(x),
\end{equation}
with 
\begin{equation}\label{eq:wtrho}
 \wt\rho_\beta(x)
 =(2\sqrt\beta)^{-1}
 \rho_\beta\left(x/(\sqrt2\beta^{1/4})\right).
\end{equation}
We now turn to the Weber matching lemma used for the left interface tail.

\begin{Lemma}\label{lem:webermatching}
Let $\nu>0$ and set
\[
        V_\nu(x)=\frac{x^2}{4}-\frac{\nu}{2},\qquad
        a_\nu=\frac{\nu-1}{2},\qquad
        m_\nu(x)=\abs x^{a_\nu}e^{-x^2/4},\quad x<0 .
\]
There is a unique positive solution $D_\nu^-$ of
\bel{eq:modelWeberLemma}
        D''=V_\nu D
\ee
on $(-\infty,-X_\nu]$ with the normalization
\[
        \lim_{x\to-\infty}\frac{D_\nu^-(x)}{m_\nu(x)}=1.
\]
It satisfies
\bel{eq:modelWeberAsymp}
        D_\nu^-(x)=m_\nu(x)\left(1+O_\nu(\abs x^{-2})\right),\qquad
        \frac{D_\nu^-{}'(x)}{D_\nu^-(x)}
        =\frac{a_\nu}{x}-\frac{x}{2}+O_\nu(\abs x^{-3})
\ee
as $x\to-\infty$.
Next, suppose $q$ is continuous on $(-\infty,-X]$, with
$X\ge X_\nu$, and
\bel{eq:qWeberNorm}
        \int_{-\infty}^{-X}\la t\ra^{-1}\abs{q(t)}\,\dd t<\infty .
\ee
Then the perturbed equation
\bel{eq:pertWeberLemma}
        U''=(V_\nu+q)U
\ee
has a one-dimensional space of solutions recessive at $-\infty$.  Every
nonzero such solution can be written uniquely, with
$\omega(x)\to0$ as $x\to-\infty$, in the form
\[
        U(x)=A\,D_\nu^-(x)(1+\omega(x)),
\]
where
\bel{eq:omegaWeberBound}
        \abs{\omega(x)}
        \Le C_\nu
        \int_{-\infty}^x\la t\ra^{-1}\abs{q(t)}\,\dd t\,
        \exp\left(C_\nu\int_{-\infty}^x
        \la t\ra^{-1}\abs{q(t)}\,\dd t\right).
\ee
In particular, if $|q(t)|\le C_qe^{-ct^2}$ for some $c,C_q>0$, then
every nonzero recessive solution satisfies
\bel{eq:pertWeberGaussian}
        U(x)=A\,\abs x^{(\nu-1)/2}e^{-x^2/4}
        \left(1+O_{\nu,c,C_q}(\abs x^{-2})\right)
\ee
as $x\to-\infty$, and
\bel{eq:pertWeberLogDerivative}
        \frac{U'(x)}{U(x)}
        =\frac{a_\nu}{x}-\frac{x}{2}
         +O_{\nu,c,C_q}(\abs x^{-3}).
\ee
Here \emph{recessive} means that $U/D_\nu^-$ has a finite limit at
$-\infty$.  If $U_-$ denotes the normalized perturbed solution
$U_-/D_\nu^-\to1$, then, sufficiently far to the left, $U_->0$, and
reduction of order gives the perturbed dominant companion
\bel{eq:pertWeberDominant}
        U_+(x)=U_-(x)\int_{x_0}^xU_-(s)^{-2}\,\dd s,
        \qquad W(U_-,U_+)=1.
\ee
If $|q(t)|\le C_qe^{-ct^2}$, then
\bel{eq:pertWeberDominantAsymptotic}
        U_+(x)
        =-\abs x^{-(\nu+1)/2}e^{x^2/4}
         \left(1+O_{\nu,c,C_q}(\abs x^{-2})\right),
        \qquad x\to-\infty .
\ee
Thus $(U_-,U_+)$ is a recessive--dominant basis on a sufficiently
far-left half-line.
\end{Lemma}

\begin{proof}
We first construct the model solution.  The elementary comparison
function $m=m_\nu$ satisfies
\[
        \frac{m''\red{(x)}}{m\red{(x)}}
        =V_\nu(x)+\frac{a_\nu(a_\nu-1)}{x^2}.
\]
Writing $D\red{(x)}=m\red{(x)}(1+\theta\red{(x)})$, equation \eqref{eq:modelWeberLemma} is
equivalent to
\bel{eq:thetaModel}
        (m^2\red{(x)}\theta'\red{(x)})'
        =-\frac{a_\nu(a_\nu-1)}{x^2}m^2(1+\theta).
\ee
Imposing the recessive condition at $-\infty$ yields the Volterra equation
\bel{eq:thetaVolterra}
        \theta(x)=-a_\nu(a_\nu-1)
        \int_{-\infty}^xK_\nu^0(x,t)\,
        \frac{1+\theta(t)}{t^2}\,\dd t,
\ee
where
\[
        K_\nu^0(x,t)
        =m_\nu(t)^2\int_t^xm_\nu(s)^{-2}\,\dd s,
        \qquad t\le x<0.
\]
\red{We have
\begin{align*}
 \left(-e^{\frac{s^2}{2}}(-s)^{-\nu}\right)'&=e^{\frac{s^2}{2}}(-s)^{1-\nu}\left(1-\frac{\nu}{s^2}\right).
\end{align*}
Thus, if $X_\nu>0$ is sufficiently large, then for every $s\leq-X_\nu$,
\begin{align*}
 \left(-e^{\frac{s^2}{2}}(-s)^{-\nu}\right)'&\geq \frac{1}{2}e^{\frac{s^2}{2}}(-s)^{1-\nu}=\frac{m_\nu(s)^{-2}}{2}.
\end{align*}
In this case, for $t\leq x\leq-X_\nu$,
\begin{align*}
\int_t^xm_\nu(s)^{-2}\,ds\leq 2\int_t^x\left(-e^{\frac{s^2}{2}}(-s)^{-\nu}\right)'\,ds=2e^{\frac{t^2}{2}}(-t)^{-\nu}-2e^{\frac{x^2}{2}}(-x)^{-\nu}\leq 2e^{\frac{t^2}{2}}(-t)^{-\nu}
\end{align*}
This immediately implies that
}
\bel{eq:K0WeberBound}
        0\le K_\nu^0(x,t)\Le\red{2m_\nu(t)^2e^{\frac{t^2}{2}}(-t)^{-\nu}=2} |t|^{-1},
        \qquad t\le x\le-X_\nu .
\ee
Consequently, the Volterra operator in \eqref{eq:thetaVolterra} maps
\[
 \mathcal X
 :=\left\{\theta:
 \sup_{x\le-X_\nu}|x|^2|\theta(x)|<\infty\right\}
\]
into itself and is a contraction after increasing $X_\nu$ so that the
associated Lipschitz constant is less than one.  This proves existence
and uniqueness of the normalized model solution and gives
$\theta(x)=O_\nu(|x|^{-2})$, which is the first part of
\eqref{eq:modelWeberAsymp}.  Differentiating
\eqref{eq:thetaVolterra}, or equivalently integrating
\eqref{eq:thetaModel} once, gives
$\theta'(x)=O_\nu(|x|^{-3})$ and hence the logarithmic-derivative
estimate in \eqref{eq:modelWeberAsymp}.

We next perturb around $D\red{(x)}=D_\nu^-\red{(x)}$.  Let $U\red{(x)}=D\red{(x)}(1+\omega\red{(x)})$.  Since $D$
solves \eqref{eq:modelWeberLemma}, equation
\eqref{eq:pertWeberLemma} becomes
\bel{eq:omegaPertODE}
        (D^2\red{(x)}\omega'\red{(x)})'=q\red{(x)}D^2\red{(x)}(1+\omega\red{(x)}).
\ee
The recessive condition at $-\infty$ gives
\bel{eq:omegaPertVolterra}
        \omega(x)=\int_{-\infty}^x
        K_\nu(x,t)q(t)(1+\omega(t))\,\dd t,
\ee
with
\[
        K_\nu(x,t)
        =D_\nu^-(t)^2\int_t^xD_\nu^-(s)^{-2}\,\dd s,
        \qquad t\le x.
\]
By \eqref{eq:modelWeberAsymp}, the same endpoint estimate as above gives
\bel{eq:KWeberGeneralBound}
        \abs{K_\nu(x,t)}\Le C_\nu\la t\ra^{-1},
        \qquad t\le x\le-X .
\ee
The Volterra equation \eqref{eq:omegaPertVolterra} is solved by iteration
on every interval $(-\infty,-X]$ on which \eqref{eq:qWeberNorm} holds.
The usual Volterra majorant gives \eqref{eq:omegaWeberBound}.  If
$q(t)=O(e^{-ct^2})$, then the integral in
\eqref{eq:omegaWeberBound} is $O(x^{-2}e^{-cx^2})$, and
\eqref{eq:pertWeberGaussian} follows from
\eqref{eq:modelWeberAsymp}.  One integration of
\eqref{eq:omegaPertODE} gives
\[
 \omega'(x)=D_\nu^-(x)^{-2}
 \int_{-\infty}^xD_\nu^-(t)^2q(t)(1+\omega(t))\,\dd t .
\]
Together with \eqref{eq:modelWeberAsymp}, this proves
\eqref{eq:pertWeberLogDerivative}.

It remains to identify the dominant companion.  Under the Gaussian
hypothesis on $q$, the preceding construction gives
\[
 U_-(x)=|x|^{a_\nu}e^{-x^2/4}
        \left(1+O_{\nu,c,C_q}(|x|^{-2})\right).
\]
The elementary  estimate
\[
 \int^Tt^{-2a_\nu}e^{t^2/2}\,\dd t
 =T^{-2a_\nu-1}e^{T^2/2}
  \left(1+O_\nu(T^{-2})\right)
\]
therefore yields
\[
 \int_{x_0}^xU_-(s)^{-2}\,\dd s
 =-|x|^{-2a_\nu-1}e^{x^2/2}
  \left(1+O_{\nu,c,C_q}(|x|^{-2})\right).
\]
Multiplication by $U_-$ proves
\eqref{eq:pertWeberDominantAsymptotic} \red{(after  renormalizing by a $1+o(1)$ factor, if necessary)}.  Since
$W(U_-,U_+)=1$, these solutions form a basis, and the recessive
subspace is one-dimensional.
\end{proof}

In our applications of this lemma to the CHMO interface, we shall
exploit the following continuity property.

\begin{Remark}
The construction is uniform when $\nu$ ranges in a compact subset of
$(0,\infty)$ and the quantities in \eqref{eq:qWeberNorm} are uniformly
bounded.  If, in addition, $q_\nu$ depends continuously on $\nu$ in that
weighted integral norm, then the normalizations
$D_\nu^-/m_\nu\to1$ and $U_\nu/D_\nu^-\to1$ determine solutions whose
values and first derivatives at each fixed point depend continuously on
$\nu$.  Their logarithmic derivatives have the same continuity at every
fixed point where the normalized solution is nonzero.
\end{Remark}

\blue{We apply Lemma~\ref{lem:webermatching} to~\eqref{eq:UbetaODE} } with $\nu=\sqrt\beta$ and
$q=\wt\rho_\beta$, see~\eqref{eq:wtrho}.  It remains only to check that
the perturbation is integrable at $-\infty$.  By
Lemma~\ref{lem:interfaceUniformTails} (with $k=0,1,2$), for $y$
sufficiently negative,
\begin{equation}\label{eq:coreInterfaceGaussianBarrier}
        0<\Phi_\beta(y)\Le C_\beta e^{-c_\beta y^2},
        \qquad
        g_\beta^{(k)}(y)
        =\partial_y^k\bigl(\Gamma_\beta(y)+\sqrt\beta\,y\bigr)
        =O_\beta(e^{-c_\beta y^2}),\quad k=0,1,2.
\end{equation}
In particular, the definition of $\rho_\beta$ in \eqref{eq:rho} implies
$\rho_\beta(y)=O_\beta(e^{-c_\beta y^2})$ as $y\to-\infty$, and hence
by \eqref{eq:wtrho},
\[
        \wt\rho_\beta(x)=O_\beta(e^{-c_\beta x^2}),
        \qquad x\to-\infty.
\]
Let $U_{\beta,-}$ be the normalized perturbed recessive solution
supplied by Lemma~\ref{lem:webermatching}, and let $U_{\beta,+}$ be \blue{another solution that has 
Wronskian equal to~$1$ with $U_{\beta,-}$}.  Since these solutions form a basis,
\[
        U_\beta=A_\beta U_{\beta,-}+B_\beta U_{\beta,+}
\]
\blue{for large negative $x$.}  The preceding Gaussian bound for
$\Phi_\beta$ \eqref{eq:coreInterfaceGaussianBarrier}, after the change of variables
$x=\sqrt2\,\beta^{1/4}y$, gives
\[
        0<U_\beta(x)\le C_\beta e^{-c_\beta x^2}.
\]
On the other hand, \eqref{eq:pertWeberDominantAsymptotic} shows that
$|U_{\beta,+}(x)|\to\infty$ as $x\to-\infty$.  Hence $B_\beta=0$.
Since $U_\beta>0$ and $U_{\beta,-}>0$ sufficiently far to the left,
$A_\beta>0$.

Lemma~\ref{lem:webermatching} therefore yields
\bel{eq:lefttail}
        \Phi_\beta(y)=C_\beta\abs y^{(\sqrt\beta-1)/2}
        e^{-\sqrt\beta y^2/2}
        \left(1+O(\abs y^{-2})\right),
        \qquad y\to-\infty,
\ee
with $C_\beta>0$.  Returning once more to \eqref{eq:gVolterra} gives
$g_\beta^{(k)}(y)=O(e^{-c_\beta y^2})$ for every fixed $k$.

Combining the Weber normal form of Lemma~\ref{lem:webermatching} with the previous two sections allows us to compare vortex profiles $\Phi_n$ and the interface at a fixed large distance $L$ to the left of the interface. We first apply WKB to the interface equations on the left.

\begin{Definition}\label{def:leftWeberFrame}
Fix $L>0$.  Rewriting the CHMO scalar equation in \eqref{eq:interface} at $y=-\eta$ gives
\[
        F''(\eta)=Q_\beta(\eta)F(\eta),
        \qquad
        Q_\beta(\eta):=\Gamma_\beta(-\eta)^2-\beta+\beta\Phi_\beta(-\eta)^2,
        \qquad \eta\ge L.
\]
\red{For $L$ sufficiently large, $\displaystyle \Gamma_\beta(-\eta)^2-\beta\geq 0$, so we may d}efine the WKB momentum and action based at $L$ by
\[
        p_\beta(\eta):=\sqrt{\Gamma_\beta(-\eta)^2-\beta},
        \qquad
        S_\beta(\eta):=\int_L^\eta p_\beta(s)\,\dd s,
\]
and the corresponding comparison branches
\[
        D_{\beta,L}^{\rm WKB}(\eta):=p_\beta(\eta)^{-1/2}e^{-S_\beta(\eta)},
        \qquad
        G_{\beta,L}^{\rm WKB}(\eta):=\frac{1}{2}p_\beta(\eta)^{-1/2}e^{S_\beta(\eta)}.
\]
The \emph{recessive} branch $d_{\beta,L}$ is the unique solution of $F''=Q_\beta F$ on $[L,\infty)$ such that
\[
        \frac{d_{\beta,L}(\eta)}{D_{\beta,L}^{\rm WKB}(\eta)}\longrightarrow 1,
        \qquad \eta\to\infty.
\]
Define the (growing) companion branch by reduction of order,
\[
        g_{\beta,L}(\eta):=d_{\beta,L}(\eta)\int_{L}^{\eta} d_{\beta,L}(s)^{-2}\,\dd s,
\]
so that $W(d_{\beta,L},g_{\beta,L})=1$.
\end{Definition}

The following lemma proves that \blue{in the large  $n$ limit}, the Cauchy data of $R_n^{-\frac12}\Psi_n$ at $\eta=L$ are asymptotically aligned with the recessive (decaying) branch $\Phi_\beta^-=A_{\beta,L}d_{\beta,L}$ of the limiting equation.

\begin{Lemma}\label{lem:leftProjection}
Let $L>0$ be large\red{, depending on $J$}.  Let $d_{n,L},g_{n,L}$ be the finite Weber basis from
Lemma~\ref{lem:finiteWeberFrame}, and let $d_{\beta,L}$ be the recessive
branch from Definition~\ref{def:leftWeberFrame}.  For fixed $L$,
\[
 (d_{n,L}(L),d'_{n,L}(L))
 \longrightarrow(d_{\beta,L}(L),d'_{\beta,L}(L)),
\]
\red{uniformly for $\beta\in J$.}
Write
\[
        \Psi_n\red{(\eta)}=A_{n,L}d_{n,L}\red{(\eta)}+B_{n,L}g_{n,L}\red{(\eta)},
        \qquad
        \Phi_\beta^-\red{(\eta)}:=\Phi_\beta(-\eta)=A_{\beta,L}d_{\beta,L}\red{(\eta)}.
\]
Then\footnote{The factor $R_n^{-1/2}$ comes from
$\Psi_n\red{(\eta)}=(R_n-\eta)^{1/2}\Phi_n^-\red{(\eta)}$.}
\[
 R_n^{-1/2}A_{n,L}d_{n,L}(L)
 \longrightarrow A_{\beta,L}d_{\beta,L}(L)
 \qquad(n\to\infty),
\]
\red{uniformly for $\beta\in J$.}
Consequently
\[
        R_n^{-1/2}A_{n,L}\longrightarrow A_{\beta,L},
\]
\red{uniformly for $\beta\in J$.}
Also,
\[
 R_n^{-1/2}|B_{n,L}|\bigl(|g_{n,L}(L)|+|g'_{n,L}(L)|\bigr)\longrightarrow0,
\]
\red{uniformly for $\beta\in J$.}
In particular,
\[
 \bigl(R_n^{-1/2}\Psi_n(L),\,R_n^{-1/2}\partial_\eta\Psi_n(L)\bigr)
 \longrightarrow
 \bigl(\Phi_\beta^-(L),\,\partial_\eta\Phi_\beta^-(L)\bigr),
\]
\red{uniformly for $\beta\in J$.}
Finally,
\[
 A_{\beta,L}d_{\beta,L}(L)
 =C_\beta L^{(\sqrt\beta-1)/2}
 e^{-\sqrt\beta L^2/2}\big[1+\red{O_J}(L^{-2})\big],
\]
with the corresponding logarithmic-derivative formula.  In particular,
$A_{\beta,L}\ne0$ for all sufficiently large fixed $L$.
\end{Lemma}

\begin{proof}
\red{By Lemma \ref{lem:qualitativeBoundary}, o}n every fixed interval $[L,M]$, the coefficients of the finite \blue{equation~\eqref{eq:exactPhi}, \eqref{eq:exactGamma} 
converge to those of the limiting equation~\eqref{eq:interface}.  }The localized form of
\eqref{eq:WeberFrameEnorm} gives
\[
 \int_M^{T_n}\frac{|E_n(s)|}{p_n(s)}\,ds
 \le \red{C_JM^{-2}}.
\]
The limiting tail satisfies the analogous bound
\begin{equation}\label{eq:limitingTailBound}
 \int_M^{\infty}\frac{|E_\beta(s)|}{p_\beta(s)}\,ds
 \le \red{C_JM^{-2}}\red{,}
\end{equation}
\red{where
\begin{align*}
    E_\beta(\eta)=\beta(\Phi_\beta(-\eta))^2-p_\beta(\eta)^{1/2}(p_\beta(\eta)^{-1/2})''.
\end{align*}}
Letting first $n\to\infty$ and then
$M\to\infty$ in the   Volterra equation~\eqref{eq:omegaTerminalVolterra} proves
\[
 (d_{n,L}(L),d'_{n,L}(L))
 \longrightarrow(d_{\beta,L}(L),d'_{\beta,L}(L)).
\]
Lemma~\ref{lem:qualitativeBoundary} gives
\[
 \bigl(R_n^{-1/2}\Psi_n(L),
       R_n^{-1/2}\partial_\eta\Psi_n(L)\bigr)
 \longrightarrow
 \bigl(\Phi_\beta^-(L),\partial_\eta\Phi_\beta^-(L)\bigr).
\]
On the other hand, \eqref{eq:growingAbs}, the differentiated WKB estimate,
and $\int_L^{T_n}p_n(s)\,ds\ge \red{c_J}T_n^2$ give
\[
 R_n^{-1/2}|B_{n,L}|
 \bigl(|g_{n,L}(L)|+|g'_{n,L}(L)|\bigr)\longrightarrow0.
\]
Since $\Psi_n(L)=A_{n,L}d_{n,L}(L)+B_{n,L}g_{n,L}(L)$ and the preceding estimate shows the
$g$--component is negligible at $\eta=L$, the claimed projection onto the
$d_{n,L}$ component follows.  Division by
$d_{n,L}(L)\to d_{\beta,L}(L)>0$ gives the coefficient limit.

Finally, $\Phi_\beta^-$ is recessive, so
$\Phi_\beta^-=A_{\beta,L}d_{\beta,L}$.  The last assertions follow from
\eqref{eq:lefttail} and Lemma~\ref{lem:webermatching}.
By Lemma~\ref{lem:interfaceUniformTails} and dominated
convergence, the Weber perturbations depend continuously on $\beta\in J$
in the weighted integral norm of \eqref{eq:qWeberNorm}.
The normalized Weber and WKB constructions therefore make $C_\beta$
and $A_{\beta,L}=\Phi_\beta(-L)/d_{\beta,L}(L)$ positive continuous
functions of $\beta$.  \green{It immediately follows that}
\[
 \inf_{\beta\in J}C_\beta>0,\qquad
 \inf_{\beta\in J}A_{\beta,L}>0
\]
for each fixed sufficiently large $L$.
\end{proof}

This finite-window matching identifies the left coefficient and controls the far-left remainder. 
\red{We extend the Higgs field $\phi_n$ across the origin as follows 
\bel{eq:leftZeroExtension}
 \widetilde\Phi_n^-(\eta)=
 \begin{cases}
   \phi_n(R_n-\eta),&0\le\eta<R_n,\\
   0,&\eta\ge R_n.
 \end{cases}
\ee
For $\displaystyle n\geq 2$, this is an $H^2_{\rm loc}$ extension since $\phi_n(0)=\phi_n'(0)=0$, due to the Frobenius asymptotics.  For a scalar function $U$ on $[M,\infty)$ define the graph norm 
\bel{eq:XminusNorm}
 \|U\|_{\mathcal X_-(M)}
 :=\|(1+\eta^2)U\|_{L_\eta^2(M,\infty)}
   +\|U''\|_{L_\eta^2(M,\infty)}.
\ee}
\begin{Corollary}\label{cor:leftCoefficientMatching}
Fix $L>0$ sufficiently large\red{, depending on $J$}.  Lemma~\ref{lem:leftProjection} verifies the convergence hypothesis in Proposition~\ref{prop:coreWKB}.  The coefficients in~\eqref{eq:coreWKBcoefficient} and~\eqref{eq:lefttail} then satisfy
\bel{eq:Cmatch}
        \lim_{L\to\infty}\limsup_{n\to\infty}
        \red{\sup_{\substack{\beta\in J}}}\left|\frac{\wt C_{n,L}}{C_\beta}-1\right|=0.
\ee
\red{We also have}
\bel{eq:leftTailPrebootstrap}
        \lim_{n\to\infty}\red{\sup_{\substack{\beta\in J}}}
        \norm{\widetilde\Phi_n^-(\eta)-\Phi_\beta(-\eta)}
        _{\mathcal X_-(L)}
        =0. 
\ee
\end{Corollary}

\begin{proof}
\red{Fix $\displaystyle 0<\alpha<\widehat{\alpha}<\frac{1}{3}$.} Lemma~\ref{lem:recessiveSelection} implies that, when we match at $\eta=L$, the contribution of the growing mode is negligible.  Thus the matching coefficient is determined by the recessive projection, whose limit is computed in Lemma~\ref{lem:leftProjection}.
Recall from \eqref{eq:coreWKBcoefficient} that
\[
\wt C_{n,L}=\Bigl(R_n^{-1/2}A_{n,L}\Bigr)\, b_n^{-1/2}\,e^{b_nL^2/2}\,L^{-\beta/(2b_n)}.
\]
On the limiting CHMO side, the identity $A_{\beta,L}d_{\beta,L}(L)=\Phi_\beta(-L)$\red{,} the asymptotic expansion in Lemma~\ref{lem:leftProjection} (with $b=\sqrt\beta$)\red{, along with the fact that $\displaystyle \lim_{\substack{n\rightarrow\infty}}d_{n,L}(L)=d_{\beta,L}(L)$} give
\[
 A_{\beta,L}b^{-1/2}e^{bL^2/2}L^{-\beta/(2b)}
 =C_\beta\red{(}1+\red{O_J(L^{-2})}\red{)}.
\]
To compare $\wt C_{n,L}$ with $C_\beta$, we compute
\[
\frac{\wt C_{n,L}}{C_\beta}
=\underbrace{\frac{R_n^{-1/2}A_{n,L}}{A_{\beta,L}}}_{\to 1}
\underbrace{\left(\frac{b_n}{b}\right)^{-1/2}e^{(b_n-b)L^2/2}L^{-\frac\beta2\left(\frac1{b_n}-\frac1b\right)}}_{\to 1\ \text{since }b_n\to b}
\underbrace{\frac{A_{\beta,L}b^{-1/2}e^{bL^2/2}L^{-\beta/(2b)}}{C_\beta}}_{=\,1+\red{O_J(L^{-2})}}.
\]
Since $L$ is fixed, the middle factor converges to $1$ as $n\to\infty$, and the first factor converges to $1$ by $R_n^{-1/2}A_{n,L}\to A_{\beta,L}$.  Therefore
\[
 \limsup_{n\to\infty}\red{\sup_{\substack{\beta\in J}}}\left|\frac{\wt C_{n,L}}{C_\beta}-1\right|
 \le \red{C_JL^{-2}},
\]
whence \eqref{eq:Cmatch}. 
It remains to prove \eqref{eq:leftTailPrebootstrap}.  \red{We claim that for $\displaystyle \eta\in[R_n^\alpha,R_n)$,
\begin{align*}
  \Phi^{-}_n(\eta)+\frac{|\partial_\eta\Phi^{-}_n(\eta)|}{1+\eta}\leq C_Je^{-c_J\eta^2}.  
\end{align*}
We first prove it for $u=R_n-\eta\geq 1$, which corresponds to $\eta\in[R_n^\alpha,R_n-1]$.
From \eqref{eq:LangerPsiGaussian} and the uniformity argument succeeding it, we recall that 
\begin{align*}
        0\le \Psi_n(\eta)\le C_J R_n^{1/2}e^{-c_J\eta^2},
        \qquad R_n^\alpha\leq\eta<R_n,
\end{align*}
where $\displaystyle \Psi_n(\eta)=(R_n-\eta)^{\frac{1}{2}}\Phi_n^{-}(\eta)$.
This immediately leads to the bound
\begin{align*}
   \Phi^{-}_n(\eta)\leq C_J R_n^{\frac{1}{2}}u^{-\frac{1}{2}}e^{-c_J\eta^2}\leq C_J e^{-c_J\eta^2},
\end{align*}
upon shrinking $c_J$.

We now turn our attention to the bound for $\partial_\eta\Phi^{-}_n(\eta)$. In the equation $\Psi_n''=\mathcal V_n\Psi_n$, where
\begin{align*}
 \mathcal V_n(\eta)&:=\Gamma_n^-(\eta)^2-\beta-\frac{1}{4(R_n-\eta)^2}+\beta\Phi_n^-(\eta)^2,   
\end{align*}
the coefficient $\mathcal{V}_n$ is therefore bounded by
\begin{align*}
    |\mathcal{V}_n(\eta)|&\leq \Gamma^{-}_n(\eta)^2+2\beta\leq 2\beta+\frac{n^2}{u^2}\leq C_J\frac{n^2}{u^2} 
\end{align*}
for $\eta\in[R_n^\alpha,R_n-1]$ (here we have used the bounds $0<\Phi_n<1$ and $0<a_n<1$).  
Let
\[
        I_\eta=\left[\eta-\frac{c_0}{1+\eta},\eta+\frac{c_0}{1+\eta}\right]
        \cap [R_n^\alpha,R_n-1]
\]
with some $c_0>0$ fixed and small. Applying  the bound  \eqref{eq:LangerPsiGaussian} on this interval and the equation
$\Psi_n''=\mathcal V_n\Psi_n$ yields 
\[
        \red{\sup_{\red{\lambda\in}I_\eta}|\Psi_n\red{(\lambda)}|
        \le C_J R_n^{1/2}e^{-\red{c_J}\eta^2}},\qquad
        \red{\sup_{\red{\lambda\in}I_\eta}|\Psi_n''\red{(\lambda)}|
        \le C_J \frac{R_n^{1/2}n^2}{u^2}e^{-\red{c_J}\eta^2}}.
\]
We apply
\[
        |f'(x)|\le 2\rho^{-1}\sup_{|s-x|\le\rho}|f(s)|
              +\rho\sup_{|s-x|\le\rho}|f''(s)|
\]
with $\rho=c_1(J)u/n$, where $c_1(J)>0$ is chosen so that
$\rho\le c_0/(1+\eta)$.  Such a choice is uniform because
$u(1+\eta)/n\le C_J$\green{, which is a direct consequence of} \eqref{eq:RtwosidedCoarse}.
Using the corresponding one-sided version when $\eta$ is within $\rho$ of an endpoint
of $[R_n^\alpha,R_n-1]$, we obtain
\[
        |\Psi_n'(\eta)|\le
        \red{C_J \frac{nR_n^{1/2}}{u}e^{-\red{c_J}\eta^2}}\leq \red{C_J e^{-\red{c_J}\eta^2}},
\]
once we shrink $c_J>0$, since $\eta\geq R_n^\alpha$ and $u\geq 1$.
When differentiating
$\Phi_n^-\red{(\eta)}=(R_n-\eta)^{-1/2}\Psi^{-}_n\red{(\eta)}$, we get 
\begin{align*}
 \partial_\eta\Phi_n^{-}\red{(\eta)}=(R_n-\eta)^{-1/2}\partial_\eta\Psi_n\red{(\eta)}+\frac{1}{2} (R_n-\eta)^{-3/2}\Psi_n\red{(\eta)},  
\end{align*}
hence
\begin{align*}
  |\partial_\eta\Phi_n^{-}\red{(\eta)}|&\leq C_Je^{-c_J\eta^2},  
\end{align*}
which leads to the desired bound.
We now analyze the case $u=R_n-\eta\in(0,1]$, which corresponds to $\eta\in[R_n-1,R_n)$.
Plugging $\eta=R_n-1$ into \eqref{eq:LangerPsiGaussian} gives
\begin{align*}
    \phi_n(1)=\Phi_n^{-}(R_n-1)\leq C_JR_n^{\frac{1}{2}}e^{-c_J(R_n-1)^2}\leq C_Je^{-c_JR_n^2},
\end{align*}
after shrinking $c_J>0$ if necessary.
From relation \eqref{eq:normalizedOriginPhi}, for $m=\left\lfloor\frac{n}{3}\right\rfloor$, we have
\begin{align*}
    \Phi_n^{-}(\eta)=\phi_n(u)\leq\phi_n(1)u^m\leq C_Je^{-c_JR_n^2}\leq C_Je^{-c_J\eta^2},
\end{align*}
after shrinking $c_J>0$ if necessary (here we have also used the fact that $\eta\simeq R_n$ when $R_n-1\leq\eta<R_n$).

From the identity $\displaystyle(u\phi_n'(u))'=uQ_n(u)\phi_n(u)$, where $\displaystyle Q_n(u)=\frac{n^2a_n^2(u)}{u^2}+\beta(\phi_n^2(u)-1)$, we have
\begin{align*}
|u\phi_n'(u)|&=\left|\int_0^usQ_n(s)\phi_n(s)\,ds\right|\leq C_J\int_0^us\frac{n^2}{s^2}\phi_n(1)s^m\,ds\leq nC_J \phi_n(1)u^m\leq nC_Je^{-c_JR_n^2}u^m
\end{align*}
Thus,
\begin{align*}
 |\partial_\eta\Phi^{-}_n(\eta)|=|\phi_n'(u)|\leq C_Jne^{-c_JR_n^2}u^{m-1}\leq C_Je^{-c_J\eta^2},   
\end{align*}
once we further shrink $c_J>0$, using the fact that $\eta\simeq R_n$ in this regime (we have also used the weaker bounds $0<a_n(u)<1$ and $0<\phi_n(u)<1$). This settles the proof in this case. We note that we have actually obtained the stronger bound
\begin{align*}
  \Phi^{-}_n(\eta)+|\partial_\eta\Phi^{-}_n(\eta)|\leq C_Je^{-c_J\eta^2}.  
\end{align*}
}
 Analogously, the CHMO tail satisfies $\Phi_\beta(-\eta)=O(e^{-c\eta^2})$ with the same differentiated bounds.  Since $\mathcal X_-(\cdot)$ in \eqref{eq:XminusNorm} has only polynomial weights, the $L^2$ parts on $[R_n^\alpha,\infty)$ hold.

For the $\|\cdot\|_{L^2}$ bound on the second derivative, note that on $1\le u\le R_n-R_n^\alpha$ (equivalently $R_n^\alpha\le\eta\le R_n-1$) the exact equation
\[
        \partial_\eta^2\Phi_n^-
        =\frac{1}{R_n-\eta}\partial_\eta\Phi_n^-
        +\red{(}(\Gamma_n^-)^2-\beta+\beta(\Phi_n^-)^2\red{)}\Phi_n^-
\]
has coefficients with at most polynomial growth, which is absorbed by the \red{Gaussian} bound. On $0<u\le1$, \red{we have already seen that}
\[
        u|\phi_n'(u)|\le \red{C_J}n\phi_n(1)u^m,
\]
and therefore $|\phi_n'(u)|\le \red{C_J}n\phi_n(1)u^{m-1}$.  The ODE then gives
$|\phi_n''(u)|\le u^{-1}|\phi_n'(u)|+|Q_n(u)|\phi_n(u)\le \red{C_J}n^2\phi_n(1)u^{m-2}$.  Hence, for large $n$,
\[
        \int_0^1\left(
        |\phi_n''(u)|^2+u^{-2}|\phi_n'(u)|^2
        +\frac{n^4}{u^4}a_n(u)^4\phi_n(u)^2\right)\,du
        \Le \red{C_J} n^4\phi_n(1)^2.
\]
Since $\phi_n(1)\le \red{C_Je^{-cR_n^{2\alpha}}}$ and $n=O_J(R_n^2)$, \red{this shows} that the contribution of $\eta\ge R_n^\alpha$ to \eqref{eq:leftTailPrebootstrap} vanishes \red{as $n\rightarrow\infty$}.

Finally, to upgrade this to \eqref{eq:leftTailPrebootstrap}, fix $M>L$.  On $[L,M]$, Lemma~\ref{lem:qualitativeBoundary} gives $C^1$ convergence of the scalar profiles, and the exact scalar equations give $C^2$ convergence. Since the weights in \eqref{eq:XminusNorm} are bounded on $[L,M]$, this is convergence in the norm of~$\mathcal X_-(L)$ restricted to $[L,M]$. 
 On $[M,R_n^\alpha]$, the Gaussian barrier (and its differentiated bounds) yields a majorant of the form
\[
        \red{C_J(1+\eta)^N e^{-c\eta^2}},
\]
for some fixed $N$, in every component of the $\mathcal X_-(L)$ norm, while $\Phi_\beta(-\eta)$ contributes \red{a similar majorant}.  \red{These majorants have} squared integral $\red{O_J}(e^{-cM^2})$. 

Combining the three regions, taking $n\to\infty$ with $L$ and $M$
fixed, and then letting $M\to\infty$ proves
\eqref{eq:leftTailPrebootstrap}.
\end{proof}

In every subsequent whole-line graph-norm statement, the scalar radial
profile on $y\le-R_n$ is understood in the zero-extension sense of
\eqref{eq:leftZeroExtension}.  No such extension is made for the magnetic
variable $\Gamma_n$, which is singular at the origin and is used only on
its physical domain.

\section{Right tail of the interface}

Let
\bel{eq:etaInterface}
        \eta_\beta(y)=1-\Phi_\beta(y).
\ee
\red{The purpose of this section is to show that on the superconducting side, the interface admits a tail expansion into decaying exponentials. This is given by the following}

\begin{Proposition}\label{prop:righttail}
For every $0<\beta<4$ there are constants $G_\beta,H_\beta>0$ such that 
\bel{eq:rightSub}
       \Gamma_\beta(y)=G_\beta e^{-\sqrt2y}(1+o(1)),\qquad
       \eta_\beta(y)=H_\beta e^{-\sqrt{2\beta}y}(1+o(1))
\ee holds \red{as $\displaystyle y\rightarrow\infty$}.
The same asymptotic formulae hold after any fixed number of differentiations.
On every $J\Subset(0,4)$, the coefficients $G_\beta,H_\beta$ depend continuously on $\beta$, and there exists $\delta_J>0$ such that in \eqref{eq:rightSub} the $o(1)$ remainders, together with any fixed number of their derivatives, are uniform for $\beta\in J$ after replacing $o(1)$ by $O_J(e^{-\delta_Jy})$.
\end{Proposition}

\begin{proof}
The equations for $(\Gamma_\beta,\eta_\beta)$ are
\begin{align}
 \Gamma_\beta''-2\Gamma_\beta
       &=-4\Gamma_\beta\eta_\beta
          +2\Gamma_\beta\eta_\beta^2,
          \label{eq:GammaRight}\\
 \eta_\beta''-2\beta\eta_\beta
       &=-\Gamma_\beta^2+\Gamma_\beta^2\eta_\beta
          -3\beta\eta_\beta^2+\beta\eta_\beta^3.
          \label{eq:EtaRight}
\end{align}
We use the following elementary fact,  in the
form needed below.  Suppose that, for some $a>0$,
$x,x'=O(e^{-ay})$ and
\bel{eq:scalarAsymptoticIntegration}
        x''-\mu^2x=q(y)x+r(y),
\ee
where $\mu>0$,
\[
 |q^{(j)}(y)|\le C_j e^{-cy},\qquad j=0,1,
\]
and either $r=0$, or
\[
 r(y)=r_0e^{-\sigma y}
      +O(e^{-(\sigma+c_0)y}),\qquad \sigma>\mu,
\]
with the corresponding estimate after one derivative.  Then there are
$H\in\RR$ and $\delta>0$ such that
\bel{eq:scalarFactExpansion}
 x(y)=H e^{-\mu y}
      +\frac{r_0}{\sigma^2-\mu^2}e^{-\sigma y}
      +O(e^{-(\mu+\delta)y}).
\ee
When $r=0$, the second term is absent.  The same estimates hold after one
derivative.

\noindent For completeness, we include the proof. Set
\[
        z_+=x'+\mu x,\qquad z_-=x'-\mu x.
\]
Then
\[
        z_+'-\mu z_+=qx+r,\qquad
        z_-'+\mu z_-=qx+r,
\]
so that
\bel{eq:zplusTerminal}
        z_+(y)
        =-\int_y^\infty e^{\mu(y-s)}(qx+r)(s)\,\dd s.
\ee
Variation of constants first improves the preliminary exponential decay
from exponent $a$ to
\[
        \min\{\mu,a+c,\sigma\}.
\]
Repeating this finitely many times gives
$x,x'=O(e^{-(\mu-\varepsilon)y})$ for every sufficiently small
$\varepsilon>0$. The loss of $\varepsilon$ occurs only in this preliminary weighted
bootstrap.  Choose
\[
        0<\varepsilon<\min\{c,\sigma-\mu\},
\]
with the second restriction omitted when $r=0$.  Then
\[
        qx+r=O(e^{-(\mu+\delta_{\rm tail})y})
\]
for some $\delta_{\rm tail}>0$.  The first-order formula \eqref{eq:zplusTerminal} therefore
recovers the endpoint and yields
\[
        x(y)=H e^{-\mu y}+O(e^{-(\mu+\delta_{\rm tail})y}).
\] Hence $qx+r$ decays strictly faster than $e^{-\mu y}$.
The second first-order equation then shows that $e^{\mu y}z_-(y)$ has a
finite limit.  Substituting
$x=(z_+-z_-)/(2\mu)$ and integrating the leading term
$r_0e^{-\sigma y}$ explicitly gives \eqref{eq:scalarFactExpansion}, since
\[
        (\partial_y^2-\mu^2)e^{-\sigma y}
        =(\sigma^2-\mu^2)e^{-\sigma y}.
\]
The differentiated estimate follows from the two first-order equations. 
Lemma~\ref{lem:interfaceUniformTails} supplies the preliminary exponential
bounds required in this fact.  Apply it first to
\eqref{eq:GammaRight} with
\[
        \mu=\sqrt2,\qquad
        q=-4\eta_\beta+2\eta_\beta^2,\qquad r=0.
\]
We obtain
\[
 \Gamma_\beta(y)
 =G_\beta e^{-\sqrt2y}
  +O(e^{-(\sqrt2+\delta)y}).
\]
To prove that the coefficient is strictly positive, define
\[
 \mathfrak g_\beta(y)
 =\frac{e^{\sqrt2y}}{2\sqrt2}
       \red{(}\sqrt2\,\Gamma_\beta(y)-\Gamma_\beta'(y)\red{)}.
\]
Equation~\eqref{eq:GammaRight} gives
\[
 \mathfrak g_\beta'(y)
 =\frac{e^{\sqrt2y}}{\sqrt2}
       \Gamma_\beta(y)\eta_\beta(y)\red{(}2-\eta_\beta(y)\red{)}>0.
\]
The derivative is integrable by the preceding expansion and the
preliminary decay of $\eta_\beta$, and
\[
 G_\beta
 =\lim_{y\to\infty}\mathfrak g_\beta(y)
 =\mathfrak g_\beta(Y)
  +\int_Y^\infty\mathfrak g_\beta'(s)\,\dd s .
\]
Since $\Gamma_\beta>0$ and $\Gamma_\beta'<0$, one has
$\mathfrak g_\beta(Y)>0$, and therefore $G_\beta>0$.
Rewrite~\eqref{eq:EtaRight} as
\[
 \eta_\beta''-2\beta\eta_\beta
 =q_\eta(y)\eta_\beta-\Gamma_\beta^2,
 \qquad
 q_\eta=\Gamma_\beta^2-3\beta\eta_\beta+\beta\eta_\beta^2.
\]
The coefficient $q_\eta$ is exponentially small, while
\[
 -\Gamma_\beta^2
 =-G_\beta^2e^{-2\sqrt2y}
  +O(e^{-(2\sqrt2+\delta)y}).
\]
Put
\[
        \mu=\sqrt{2\beta},\qquad \sigma=2\sqrt2.
\]
Since $0<\beta<4$, one has $\sigma>\mu$.  The scalar
asymptotic-integration result therefore gives
\[
 \eta_\beta(y)
 =H_\beta e^{-\sqrt{2\beta}y}
 -\frac{G_\beta^2}{2(4-\beta)}e^{-2\sqrt2y}
 +O(e^{-(\sqrt{2\beta}+\delta_\beta)y}),
\]
where $\delta_\beta$ is chosen smaller than
$2\sqrt2-\sqrt{2\beta}$.
To prove that $H_\beta>0$, we define
\[
 \mathfrak h_\beta(y)
 =\frac{e^{\mu y}}{2\mu}
       \red{(}\mu\eta_\beta(y)-\eta_\beta'(y)\red{)}.
\]
The exact Higgs equation yields
\[
 \mathfrak h_\beta'(y)
 =\frac{e^{\mu y}}{2\mu}
 \left[
 \Gamma_\beta(y)^2\red{(}1-\eta_\beta(y)\red{)}
 +\beta\eta_\beta(y)^2\red{(}3-\eta_\beta(y)\red{)}
 \right]>0.
\]
The right side is integrable because
\[
        2\sqrt2-\mu>0,\qquad \mu>0.
\]
Consequently,
\[
 H_\beta
 =\lim_{y\to\infty}\mathfrak h_\beta(y)
 =\mathfrak h_\beta(Y)
  +\int_Y^\infty\mathfrak h_\beta'(s)\,\dd s .
\]
Since $\eta_\beta>0$ and
$\eta_\beta'=-\Phi_\beta'<0$, one has
$\mathfrak h_\beta(Y)>0$.  Hence $H_\beta>0$, and the two additive
expansions become
\[
 \Gamma_\beta(y)
 =G_\beta e^{-\sqrt2y}(1+o(1)),\qquad
 \eta_\beta(y)
 =H_\beta e^{-\sqrt{2\beta}y}(1+o(1)).
\]
Equivalently, the tails satisfy the  Volterra equations
\begin{align}
 \Gamma_\beta(y)&=G_\beta e^{-\sqrt2 y}
 +\int_y^\infty \calG_{\sqrt2}(y,s)
      [-4\Gamma_\beta\eta_\beta
       +2\Gamma_\beta\eta_\beta^2](s)\,\dd s,
      \label{eq:GammaVolterra}\\
 \eta_\beta(y)&=H_\beta e^{-\sqrt{2\beta}y}
 -\int_y^\infty \calG_{\sqrt{2\beta}}(y,s)
      [\Gamma_\beta^2-\Gamma_\beta^2\eta_\beta
       +3\beta\eta_\beta^2-\beta\eta_\beta^3](s)\,\dd s,
      \label{eq:EtaVolterra}
\end{align}
where
\[
    \calG_\mu(y,s)=\frac{\sinh(\mu(s-y))}{\mu}{\bf 1}_{s\ge y}.
\]
All sources in these equations decay strictly faster than the corresponding
homogeneous mode.  In particular, if
$|f(s)|\le M e^{-\sigma s}$ with $\sigma>\mu$, then, for $j=0,1$,
\bel{eq:rightGreenConvolution}
 \left|\partial_y^j
 \int_y^\infty \calG_\mu(y,s)f(s)\,\dd s\right|
 \le \frac{C_{j,\mu}M}{\sigma-\mu}e^{-\sigma y}.
\ee
The first two differentiated estimates follow from these formulae.
Higher derivatives follow inductively by differentiating
\eqref{eq:GammaRight}, \eqref{eq:EtaRight}.

Finally, let $J\Subset(0,4)$.  Lemma~\ref{lem:interfaceUniformTails}
makes all preliminary estimates uniform for $\beta\in J$, while
\[
 \inf_{\beta\in J}\sqrt{2\beta}>0,\qquad
 \inf_{\beta\in J}
       \{2\sqrt2-\sqrt{2\beta}\}>0
\]
give a common positive exponent $\delta_J$.  The integral formulae for
$G_\beta$ and $H_\beta$ above, together with dominated convergence, show
that both coefficients depend continuously on $\beta$.  Since they are
positive, compactness of $J$ gives
\[
        \inf_{\beta\in J}G_\beta>0,\qquad
        \inf_{\beta\in J}H_\beta>0.
\]
The additive remainder estimates can therefore be divided by the leading
coefficients, giving the asserted uniform relative
$O_J(e^{-\delta_Jy})$ bounds.  Differentiating the equations inductively
gives the same uniform conclusion after any fixed number of
differentiations.
\end{proof}

\section{Exterior modified Bessel problem}\label{sec:exterior}

In the exterior region (to the right of the interface) set $\eta_n(u):=1-\phi_n(u)$.  The exact equations are
\begin{align}
 a_n''-\frac{1}{u}a_n'-2a_n
     &=-4a_n\eta_n+2a_n\eta_n^2,\label{eq:aExt}\\
 \eta_n''+\frac{1}{u}\eta_n'-2\beta\eta_n
     &=-\frac{n^2}{u^2}a_n^2+\frac{n^2}{u^2}a_n^2\eta_n
       -3\beta\eta_n^2+\beta\eta_n^3.\label{eq:etaExt}
\end{align}
The goal of this section is to solve this system by perturbing the decaying modified Bessel functions which solve the associated homogeneous equations.  
Here, we only require $\beta$ to lie in a compact subinterval of $(0,4)$. \red{This will provide us with asymptotics for the right tail of $(a_n,\eta_n)$, using Bessel functions as a proxy.  
We introduce the following notations:
\bel{eq:mMuBetaDef}
 m:=\sqrt{2},\qquad \mu_\beta:=\sqrt{2\beta}.
\ee
}
The next lemma turns the  decay of $(\Gamma_n,\Hdef_n)$ \red{ where $\Hdef_n(y):=1-\Phi_n(y)=\eta_n(R_n+y)$ (with $u=R_n+y$)} into  exponential bounds (including two derivatives) that place the profiles in the weighted norm $\mathcal N_Y$ used for the Bessel--Volterra fixed point in Proposition~\ref{prop:exteriorBessel}. See  Definition~\ref{def:PhiGamma} for the centered variables. 

\begin{Lemma}\label{lem:exteriorEntry}
Let $J\Subset(0,4)$.  Choose $\rho>0$ so that
\[
 0<4\rho<\inf_{\beta\in J}
 \min\{\red{\mu_\beta},\,2\red{m-\mu_\beta}\},
\]
and put
\[
 \lambda_G:=\red{m}-\rho,\qquad
 \lambda_H:=\red{\mu_\beta}-\rho.
\]
Then $\lambda_G,\lambda_H>0$ and $2\lambda_G>\lambda_H$.
There are $K_0$ and $C$, uniform for $\beta\in J$, such that for
$K\ge K_0$ and all sufficiently large $n$,
\bel{eq:exteriorEntry}
 \sum_{j=0}^2\sup_{y\ge K}
 \left\{e^{\lambda_G(y-K)}|\partial_y^j\Gamma_n(y)|
       +e^{\lambda_H(y-K)}|\partial_y^j\Hdef_n(y)|\right\}\le C\red{.}
\ee
\end{Lemma}

\begin{proof}  Uniform decay of the interface to the right and 
\eqref{eq:uniformQualitativeBoundary} of
Lemma~\ref{lem:qualitativeBoundary} allow $K_0$ to be chosen so that
$\Phi_n(K)\ge1-\epsilon_0$ for $K\ge K_0$ and large $n$, uniformly for
$\beta\in J$.  Monotonicity gives the same lower bound for every $y\ge K$.
Writing $u=R_n+y$, the exact equations become
\[
 \mathcal L_{G,n}\Gamma_n=0,
 \qquad
 \mathcal L_{H,n}\Hdef_n=\Gamma_n^2\Phi_n,
\]
where
\[
 \mathcal L_{G,n}=-\partial_y^2-u^{-1}\partial_y+u^{-2}+2\Phi_n^2\red{(y)},
 \qquad
 \mathcal L_{H,n}=-\partial_y^2-u^{-1}\partial_y
                    +\beta\Phi_n\red{(y)}(1+\Phi_n\red{(y)}).
\]
After decreasing $\epsilon_0$ \red{if necessary}, the functions
$e^{-\lambda_G(y-K)}$ and $e^{-\lambda_H(y-K)}$ are strict positive
supersolutions for the corresponding homogeneous operators, with a uniform positive margin:
\[
 \mathcal L_{G,n}e^{-\lambda_G(y-K)}
 =\red{(}2\Phi_n^2-\lambda_G^2+\lambda_G/u+u^{-2}\red{)}
      e^{-\lambda_G(y-K)},
\]
and similarly
$\mathcal L_{H,n}e^{-\lambda_H(y-K)}=\red{(}\beta\Phi_n(1+\Phi_n)-\lambda_H^2+\lambda_H/u\red{)}e^{-\lambda_H(y-K)}$.

Choose the multiplicative constant in the first supersolution larger than
$\Gamma_n(K)$.  On $[K,T]$ \green{we} add a positive constant $\varepsilon$ to the
 supersolution and choose $T$ so that $\Gamma_n(T)<\varepsilon$.  \green{We now apply the}
  maximum principle, \green{let $T$ tend to infinity} \blue{along a sequence of such endpoints}
  and then \green{let $\varepsilon$ tend to zero. We deduce that}
\[
        0<\Gamma_n(y)\le C e^{-\lambda_G(y-K)}.
\]
Since $2\lambda_G>\lambda_H$, the source $\Gamma_n^2\Phi_n$ in the second equation satisfies
\[
        0\le\Gamma_n^2\Phi_n\le C e^{-2\lambda_G(y-K)}
        \le C e^{-\lambda_H(y-K)}.
\]
Increasing the coefficient of the second exponential supersolution makes
its image under $\mathcal L_{H,n}$ dominate this \blue{forcing term}.  The same
comparison argument gives
\[
        0<\Hdef_n(y)\le C e^{-\lambda_H(y-K)}.
\]
Finally, on every unit interval contained in $[K,\infty)$, variation of
constants for the two scalar equations gives the elementary interior
estimate
\[
 |x'(y)|+|x''(y)|
 \le C\left(\sup_{|s-y|\le1}|x(s)|
       +\sup_{|s-y|\le1}|f(s)|\right)
\]
for $x''+p x'+q x=f$ with the present uniformly bounded coefficients.
Applying it first to $\Gamma_n$ and then to $\Hdef_n$ proves the derivative
bounds in \eqref{eq:exteriorEntry}. \end{proof}

Solutions to the radial exterior problem \eqref{eq:aExt}, \eqref{eq:etaExt} are well-approximated by modified Bessel functions. To quantify that, we work in a shifted, exponentially weighted $C^2$--type norm: for $Y>0$ and a pair $(a,\eta)$ on $[R_n+Y,\infty)$ set
\[
        \mathcal N_Y(a,\eta)
        :=\sum_{j=0}^2\sup_{y\ge Y}
        \left(e^{\red{m}y}\abs{\partial_y^j\!\left(\frac{n}{R_n+y}a(R_n+y)\right)}
        +e^{\red{\mu_\beta}y}\abs{\partial_y^j\eta(R_n+y)}\right),
\]
where $u=R_n+y$. Abusing notation a little, we denote the corresponding space by~$\mathcal N_Y$. \blue{ The following proposition reduces controlling the decay of $a_n,\eta_n$ to the following pair of Volterra equations associated to the equations \eqref{eq:aExt} and \eqref{eq:etaExt}:
\begin{align}
 a_n(u)&=A_n k_1(u)
 +\int_u^\infty \calG_1(u,s)[-4a_n\eta_n+2a_n\eta_n^2](s)\dd s,\label{eq:aVolterra}\\
 \eta_n(u)&=B_n k_{0,\beta}(u)
 -\int_u^\infty \calG_{0,\beta}(u,s)
  \left[\frac{n^2}{s^2}a_n^2-\frac{n^2}{s^2}a_n^2\eta_n
       +3\beta\eta_n^2-\beta\eta_n^3\right](s)\dd s
\label{eq:etaVolterra}
\end{align}
 in the space $\mathcal N_Y$ 
for $\beta\in J$. For the Green kernels see~\eqref{eq:G1}, \eqref{eq:G0}. We refer to $A_n, B_n$ as the coefficients of the leading Bessel functions by which we parameterize the space of $a_n,\eta_n$.  These are large quantities, which are renormalized as follows:
\bel{eq:Ahat}
         A_n:=\left(\frac{2\red{m}R_n}{\pi}\right)^{1/2}
        \frac{1}{n}e^{\red{m}R_n}\widehat A_n 
\ee
and
\bel{eq:Bhat}
         B_n:=\left(\frac{2\red{\mu_\beta}R_n}{\pi}\right)^{1/2}
        e^{\red{\mu_\beta}R_n}\widehat B_n .
\ee
It will turn out that $\widehat A_n, \widehat B_n$ are of unit size. }

\begin{Proposition}\label{prop:exteriorBessel}
Let $J\Subset(0,4)$.  Uniformly for $\beta\in J$ and for sufficiently large \red{$Y$ and $n$}, for prescribed   $\red{\widehat A_n},\red{\widehat B_n}$ in a fixed bounded set, 
the Volterra equations \eqref{eq:aVolterra}, \eqref{eq:etaVolterra} define a contraction on a ball in $\mathcal N_Y$,  and hence yield a unique exterior solution of \eqref{eq:aExt}, \eqref{eq:etaExt} on $u\ge R_n+Y$ with finite $\mathcal N_Y$ norm.
\blue{For the true vortex, the tail belongs to this two-parameter family, and its uniquely determined coefficients $A_n, B_n$ are positive. }
\end{Proposition}

\begin{proof} 
The fundamental systems of the linearized equations \red{\eqref{eq:aExt}, \eqref{eq:etaExt}}  are
\bel{eq:Besselmodes}
        k_1(u)=u\K_1(\red{m}u),\qquad i_1(u)=u\I_1(\red{m}u),
\ee
and, for all $\beta>0$,
\bel{eq:Bessel0modes}
        k_{0,\beta}(u)=\K_0(\red{\mu_\beta}u),\qquad
        i_{0,\beta}(u)=\I_0(\red{\mu_\beta}u).
\ee
The associated forward Green kernels are
\begin{align}
 \calG_1(u,s)&=\frac{k_1(u)i_1(s)-i_1(u)k_1(s)}{\calW[k_1,i_1](s)}
       {\bf 1}_{s\ge u},\label{eq:G1}\\
 \calG_{0,\beta}(u,s)&=\frac{k_{0,\beta}(u)i_{0,\beta}(s)-i_{0,\beta}(u)k_{0,\beta}(s)}{\calW[k_{0,\beta},i_{0,\beta}](s)}
       {\bf 1}_{s\ge u}.\label{eq:G0}
\end{align}
Recall the convention $\calW[f,g]=fg'-f'g$.  The standard identity
$\calW[\K_\nu(z),\I_\nu(z)]=z^{-1}$ gives exactly
\bel{eq:BesselWronskians}
        \calW[k_1,i_1](s)=s,\qquad
        \calW[k_{0,\beta},i_{0,\beta}](s)=s^{-1}.
\ee
 We need to solve the Volterra equations  
\eqref{eq:aVolterra} and \eqref{eq:etaVolterra}
 in the space $\mathcal N_Y$ 
for $\beta\in J$.
Standard global bounds for $I_\nu,K_\nu$ (see \cite[\S10.37 and \S10.40]{DLMF}) yield, for $0\le j\le1$ and $s\ge u$,
\[
\begin{split}
 |\partial_u^j\calG_1(u,s)|
 &\le C_J\left(\frac{s}{u}\right)^{1/2}
       e^{\red{m}(s-u)},\\
 |\partial_u^j\calG_{0,\beta}(u,s)|
 &\le C_J\left(\frac{s}{u}\right)^{1/2}
       e^{\red{\mu_\beta}(s-u)}.
\end{split}
\]
Retain the choice of $\rho$ from Lemma~\ref{lem:exteriorEntry}, namely
\bel{eq:rightSpectralGap}
 0<4\rho<\inf_{\beta\in J}
       \min\{\red{\mu_\beta},2\red{m}-\red{\mu_\beta}\}.
\ee
Substitution
of $u=R_n+y$, $s=R_n+t$, followed by integration against
$e^{-\rho t}$, implies that 
\begin{equation}\label{eq:exteriorGreenSchur}
\begin{aligned}
 &\sum_{j=0}^2\sup_{y\ge Y}e^{\red{m}y}
 \left|\partial_y^j\frac{n}{R_n+y}
 \int_{R_n+y}^\infty\calG_1(R_n+y,s)F(s)\,ds\right|\\
 & \le C_J
 \sup_{t\ge Y}e^{(\red{m}+\rho)t}
 \left|\frac{n}{R_n+t}F(R_n+t)\right|,
\end{aligned}
\end{equation}
and, analogously,
\bel{eq:exteriorGreenSchurHiggs}
\begin{aligned}
 &\sum_{j=0}^2\sup_{y\ge Y}e^{\red{\mu_\beta}y}
 \left|\partial_y^j
 \int_{R_n+y}^\infty\calG_{0,\beta}(R_n+y,s)F(s)\,ds\right|\\
 &\le C_J
 \sup_{t\ge Y}e^{(\red{\mu_\beta}+\rho)t}
 \left|F(R_n+t)\right|.
\end{aligned}
\ee
with the value of $\rho$ fixed in \eqref{eq:rightSpectralGap}.  For $j=2$ one uses the equation
after proving the first two bounds.   Fix a radius $\mathfrak R>0$ and consider the closed ball
\[
 \mathcal B_{\mathfrak R}
 :=
 \left\{(a_n,\eta_n)\in\mathcal N_Y:
        \mathcal N_Y(a_n,\eta_n)\le\mathfrak R\right\}.
\]
For a pair $(a_n,\eta_n)\in\mathcal B_{\mathfrak R}$, \red{set} $\displaystyle \nu_n:=\mathcal N_Y(a_n,\eta_n)$, and, for $s=R_n+t$, define
\[
 \Gamma_n(t):=\frac{n}{R_n+t}a_n(R_n+t).
\]
The definition of $\mathcal N_Y$ gives
\[
 |\Gamma_n(t)|\le \nu_n e^{-m t},
 \qquad
 |\eta_n(R_n+t)|\le \nu_n e^{-\mu_\beta t},
 \qquad t\ge Y.
\]
Let
\[
 Q_{a,n}(s):=-4a_n(s)\eta_n(s)+2a_n(s)\eta_n(s)^2
\]
denote the nonlinear source in \eqref{eq:aVolterra}.  Although this
source is written in terms of $a_n$, the normalization required by the
magnetic component of $\mathcal N_Y$ is supplied by
\eqref{eq:exteriorGreenSchur}, since
\[
 \frac{n}{R_n+t}Q_{a,n}(R_n+t)
 =
 -4\Gamma_n(t)\eta_n(R_n+t)
 +2\Gamma_n(t)\eta_n(R_n+t)^2.
\]
It follows that
\begin{align*}
 &\sup_{t\ge Y}e^{(m+\rho)t}
 \left|\frac{n}{R_n+t}Q_{a,n}(R_n+t)\right|\le
 4\nu_n^2 e^{-(\mu_\beta-\rho)Y}
 +2\nu_n^3 e^{-(2\mu_\beta-\rho)Y}.
\end{align*}
Likewise, let
\[
 Q_{\eta,n}(s)
 :=
 \frac{n^2}{s^2}a_n(s)^2
 -\frac{n^2}{s^2}a_n(s)^2\eta_n(s)
 +3\beta\eta_n(s)^2-\beta\eta_n(s)^3
\]
denote the expression inside the integral in
\eqref{eq:etaVolterra}.  As for the normalized magnetic field $\Gamma_n$,
\[
 Q_{\eta,n}(R_n+t)
 =
 \Gamma_n(t)^2\red{(}1-\eta_n(R_n+t)\red{)}
 +3\beta\eta_n(R_n+t)^2
 -\beta\eta_n(R_n+t)^3.
\]
Consequently,
\begin{align*}
 &\sup_{t\ge Y}e^{(\mu_\beta+\rho)t}
       |Q_{\eta,n}(R_n+t)|\\
 &\qquad\le
 \nu_n^2\red{(}
 e^{-(2m-\mu_\beta-\rho)Y}
 +3\beta e^{-(\mu_\beta-\rho)Y}
 \red{)}
 +\nu_n^3\red{(}
 e^{-(2m-\rho)Y}
 +\beta e^{-(2\mu_\beta-\rho)Y}
 \red{)}.
\end{align*}
By \eqref{eq:rightSpectralGap},
\[
 c_J:=
 \inf_{\beta\in J}
 \min\{\mu_\beta-\rho,\,2m-\mu_\beta-\rho\}>0.
\]
All the remaining exponents in the preceding estimates are bounded
below by $c_J$.  By Schur's test, 
\[
 \mathcal N_Y\bigl(\text{Volterra nonlinear part}\bigr)
 \le
 C_J e^{-c_JY}\bigl(\nu_n^2+\nu_n^3\bigr).
\]
In particular, on $\mathcal B_{\mathfrak R}$,
\[
 \mathcal N_Y\bigl(\text{Volterra nonlinear part}\bigr)
 \le
 C_J e^{-c_JY}
 \bigl(\mathfrak R^2+\mathfrak R^3\bigr).
\]
If $(\widetilde a_n,\widetilde\eta_n)\in\mathcal B_{\mathfrak R}$ is
a second pair, the corresponding polynomial difference estimates give
\begin{align*}
 &\mathcal N_Y\left(
 \mathcal T_{\mathrm{nl}}(a_n,\eta_n)
 -\mathcal T_{\mathrm{nl}}(\widetilde a_n,\widetilde\eta_n)
 \right)\leq 
 C_J e^{-c_JY}
 \bigl(\mathfrak R+\mathfrak R^2\bigr)
 \mathcal N_Y\left(
 a_n-\widetilde a_n,\eta_n-\widetilde\eta_n
 \right)\red{,}
\end{align*}
\red{where $\displaystyle \mathcal T_{\mathrm{nl}}(a_n,\eta_n)$ and $\displaystyle\mathcal T_{\mathrm{nl}}(\widetilde a_n,\widetilde\eta_n)$ are the Volterra nonlinear terms corresponding to the two pairs.}
\red{We c}hoose $\mathfrak R$ large enough to contain the linear Bessel term
associated with the coefficients $A_n, B_n$.  \red{Subsequently, we}
choose $Y$ sufficiently large so that
\[
 \red{C_J e^{-c_JY}
 \bigl(\mathfrak R+\mathfrak R^2\bigr)<\frac{1}{2}}.
\]
The Volterra map then sends $\mathcal B_{\mathfrak R}$ into itself and
is a contraction there.  The differentiated kernel estimates give the
bounds with one derivative, and the bounds for the second derivatives
then follow from \eqref{eq:aExt} and \eqref{eq:etaExt}.
The same argument implies uniqueness among all exterior solutions with the prescribed coefficients and finite $\mathcal N_Y$ norm, not only among solutions in the ball $\mathcal B_{\mathfrak R}$.  Indeed, given two such solutions, restrict them to $[R_n+Y',\infty)$ and choose a ball containing both  tails.  Since their weighted norms are finite, the
nonlinear Lipschitz constant on this ball is
$O(e^{-c_JY'})$ and can  therefore be made less than~$1$.  The  Volterra equations imply that the two solutions
coincide on this tail, and uniqueness for the regular ODE initial-value
problem extends equality to $[R_n+Y,\infty)$.  

We now verify that the physical radial tail belongs to this family without
assuming weighted bounds.  Put
\[
 F_{a,n}=-4a_n\eta_n+2a_n\eta_n^2,
 \qquad
 F_{\eta,n}=-\frac{n^2}{u^2}a_n^2
 +\frac{n^2}{u^2}a_n^2\eta_n-3\beta\eta_n^2+\beta\eta_n^3.
\]
Lemma~\ref{lem:exteriorEntry}\red{, the standard global bounds for $I_\nu,K_\nu$ from \cite[\S10.37 and \S10.40]{DLMF}, along with} the gaps in
\eqref{eq:rightSpectralGap} imply
\begin{equation}\label{Source Integral Convergence} 
 \frac{F_{a,n}(u)i_1(u)}{u}\in L^1(R_n+Y,\infty),
 \qquad
 uF_{\eta,n}(u)i_{0,\beta}(u)\in L^1(R_n+Y,\infty)\red{,}
\end{equation}
\red{for every $Y>0$.}
The point is that the first source has rate at least
$\lambda_G+\lambda_H>\red{m}$, while every term in the second has rate at
least $\min\{2\lambda_G,2\lambda_H\}>\red{\mu_\beta}$.  We have the Wronskian identities
\bel{eq:exteriorWronskianLimits}
 \left(\frac{\calW[a_n,i_1](u)}{u}\right)'
 =-\frac{F_{a,n}(u)i_1(u)}{u},
 \qquad
 \left(u\calW[\eta_n,i_{0,\beta}](u)\right)'
 =-uF_{\eta,n}(u)i_{0,\beta}(u).
\ee
Both sources are strictly negative:
\[
 F_{a,n}=-2a_n\eta_n(2-\eta_n)<0,\qquad
 F_{\eta,n}=-\frac{n^2}{u^2}a_n^2(1-\eta_n)
             -\beta\eta_n^2(3-\eta_n)<0.
\]
Consequently the two quantities in
\eqref{eq:exteriorWronskianLimits} are strictly increasing, and the limits
\bel{eq:ABWronskianDefinition}
 A_n=\lim_{u\to\infty}\frac{\calW[a_n,i_1](u)}{u},
 \qquad
 B_n=\lim_{u\to\infty}u\calW[\eta_n,i_{0,\beta}](u)
\ee
exist.  At every finite point they are
already positive because $a_n,\eta_n>0$, $a_n',\eta_n'<0$, and both
$i_1,i_{0,\beta}$ and their derivatives are positive.  Strict monotonicity
therefore gives $A_n,B_n>0$.

Fix $Y_0\ge K_0$, with $K_0$ as in
Lemma~\ref{lem:exteriorEntry}, and set
\[
        u_0:=R_n+Y_0.
\]
Integrating \eqref{eq:exteriorWronskianLimits} from $u_0$ to infinity
gives the exact relations
\bel{eq:exteriorAmplitudeIntegralFormulas}
\begin{split}
 A_n
 &=
 \frac{\calW[a_n,i_1](u_0)}{u_0}
 -\int_{u_0}^{\infty}\frac{F_{a,n}(s)i_1(s)}{s}\,\dd s,\\
 B_n
 &=
 u_0\calW[\eta_n,i_{0,\beta}](u_0)
 -\int_{u_0}^{\infty}sF_{\eta,n}(s)i_{0,\beta}(s)\,\dd s.
\end{split}
\ee
\red{By \eqref{Source Integral Convergence}, b}oth integrals converge absolutely.
The global Bessel bounds and \eqref{eq:exteriorEntry}, evaluated at the
fixed point $Y_0$, imply  (recalling that $m=\sqrt2$ and $\mu_\beta=\sqrt{2\beta}$)
\[
 \frac{n}{R_n^{1/2}}e^{-m R_n}
 \left|\frac{\calW[a_n,i_1](u_0)}{u_0}\right|
 +
 \frac{1}{R_n^{1/2}}e^{-\mu_\beta R_n}
 \left|u_0\calW[\eta_n,i_{0,\beta}](u_0)\right|
 \le C_{J,Y_0}.
\]
From the same bounds applied to the two integrals in
\eqref{eq:exteriorAmplitudeIntegralFormulas} we infer that
\[
\begin{split}
 &\frac{n}{R_n^{1/2}}e^{-m R_n}
 \int_{u_0}^{\infty}
       \frac{|F_{a,n}(s)|i_1(s)}{s}\,\dd s\\
 &\qquad+
 \frac{1}{R_n^{1/2}}e^{-\mu_\beta R_n}
 \int_{u_0}^{\infty}
       s|F_{\eta,n}(s)|i_{0,\beta}(s)\,\dd s
 \le C_{J,Y_0}.
\end{split}
\]
Consequently,
\bel{eq:physicalAmplitudeUniformBound}
 \red{|\widehat A_n|}+\red{|\widehat B_n|}\simeq_J\frac{n}{R_n^{1/2}}e^{-m R_n}|A_n|
 +
 \frac{1}{R_n^{1/2}}e^{-\mu_\beta R_n}|B_n|
 \le C_{J,Y_0}.
\ee
We next derive the  Volterra formulas without assuming that
$(a_n,\eta_n)$ already belongs to $\mathcal N_Y$.  For $u_0\le u<T$,
variation of constants leads to
\bel{eq:finiteExteriorVariation}
\begin{split}
 a_n(u)
 &=
 A_n(T)k_1(u)+C_n(T)i_1(u)
 +\int_u^T\calG_1(u,s)F_{a,n}(s)\,\dd s,\\
 \eta_n(u)
 &=
 B_n(T)k_{0,\beta}(u)+D_n(T)i_{0,\beta}(u)
 +\int_u^T\calG_{0,\beta}(u,s)F_{\eta,n}(s)\,\dd s,
\end{split}
\ee
where
\[
\begin{split}
 A_n(T)&=\frac{\calW[a_n,i_1](T)}{T},\qquad
 C_n(T)=\frac{\calW[k_1,a_n](T)}{T},\\
 B_n(T)&=T\calW[\eta_n,i_{0,\beta}](T),\qquad
 D_n(T)=T\calW[k_{0,\beta},\eta_n](T).
\end{split}
\]
By \eqref{eq:ABWronskianDefinition}, \red{as $\displaystyle T\rightarrow\infty$,}
\[
        A_n(T)\longrightarrow A_n,
        \qquad
        B_n(T)\longrightarrow B_n.
\]
The derivative bounds in Lemma~\ref{lem:exteriorEntry}, together with
the decaying Bessel estimates for $k_1$ and $k_{0,\beta}$, give
\[
        C_n(T)\longrightarrow0,
        \qquad
        D_n(T)\longrightarrow0\red{,}
\]
\red{as $\displaystyle T\rightarrow\infty$.}
Since the source integrals converge absolutely, we may let
$T\to\infty$ in \eqref{eq:finiteExteriorVariation}.  We obtain
\bel{eq:physicalTerminalVolterra}
\begin{split}
 a_n(u)
 &=
 A_nk_1(u)
 +\int_u^\infty\calG_1(u,s)F_{a,n}(s)\,\dd s,\\
 \eta_n(u)
 &=
 B_nk_{0,\beta}(u)
 +\int_u^\infty\calG_{0,\beta}(u,s)F_{\eta,n}(s)\,\dd s.
\end{split}
\ee
These are precisely \eqref{eq:aVolterra} and
\eqref{eq:etaVolterra}, since
\[
 F_{\eta,n}
 =
 -\red{\left(
 \frac{n^2}{u^2}a_n^2
 -\frac{n^2}{u^2}a_n^2\eta_n
 +3\beta\eta_n^2-\beta\eta_n^3
 \right)}.
\]
It remains to improve the preliminary decay from
Lemma~\ref{lem:exteriorEntry} to the sharp weights in $\mathcal N_Y$.
For $s=R_n+t$, write
\[
        \Gamma_n(t)=\frac{n}{R_n+t}a_n(R_n+t).
\]
The forcing terms satisfy the exact identities
\[
\begin{split}
 \frac{n}{R_n+t}F_{a,n}(R_n+t)
 &=-2\Gamma_n(t)\eta_n(R_n+t)
       \red{(}2-\eta_n(R_n+t)\red{)},\\
 F_{\eta,n}(R_n+t)
 &=-\Gamma_n(t)^2\red{(}1-\eta_n(R_n+t)\red{)}\\
 &\quad
   -\beta\eta_n(R_n+t)^2
       \red{(}3-\eta_n(R_n+t)\red{)}.
\end{split}
\]
By Lemma~\ref{lem:exteriorEntry},
\[
 |\Gamma_n(t)|\le C e^{-\lambda_Gt},
 \qquad
 |\eta_n(R_n+t)|\le C e^{-\lambda_Ht},
\]
after changing the constant to absorb the fixed shift \red{$Y_0$}.  Note that
\[
\begin{split}
 \lambda_G+\lambda_H-(m+2\rho)
 &=\mu_\beta-4\rho>0,\\
 2\lambda_G-(\mu_\beta+2\rho)
 &=2m-\mu_\beta-4\rho>0,\\
 2\lambda_H-(\mu_\beta+2\rho)
 &=\mu_\beta-4\rho>0
\end{split}
\]
by \eqref{eq:rightSpectralGap}.  Therefore
\bel{eq:physicalExteriorSourceImprovement}
 \left|\frac{n}{R_n+t}F_{a,n}(R_n+t)\right|
 \le C_J e^{-(m+2\rho)t},
 \qquad
 |F_{\eta,n}(R_n+t)|
 \le C_J e^{-(\mu_\beta+2\rho)t}.
\ee
Apply the weighted  estimates \eqref{eq:exteriorGreenSchur} and \eqref{eq:exteriorGreenSchurHiggs} to
\eqref{eq:physicalTerminalVolterra}.  The additional factor
$e^{-2\rho t}$ in
\eqref{eq:physicalExteriorSourceImprovement} yields
\bel{eq:physicalExteriorVolterraCorrection}
 \mathcal N_Y\left(
 a_n-A_nk_1,\,
 \eta_n-B_nk_{0,\beta}
 \right)
 \le \red{C_{J,Y_0}} e^{-\rho \red{Y_0}}.
\ee
The linear Bessel pair
\[
        (A_nk_1,B_nk_{0,\beta})
\]
has finite $\mathcal N_Y$ norm, with its shifted coefficients controlled
by \eqref{eq:physicalAmplitudeUniformBound}.  Hence
\[
        \mathcal N_Y(a_n,\eta_n)<\infty.
\]
Thus the physical radial tail is a fixed point of the Volterra map in
the ball $\mathcal B_{\mathfrak R}$, once $\mathfrak R$ is chosen to
contain the linear term and the bound in
\eqref{eq:physicalExteriorVolterraCorrection}.  Uniqueness of the
contraction identifies it with the corresponding member of the
two-parameter stable family.  The coefficients equal the  uniquely defined
Wronskian limits in \eqref{eq:ABWronskianDefinition}.
\end{proof}

\red{For the true vortex, we can identify the coefficients $\red{\widehat A_n}$ and $\red{\widehat B_n}$ in the limit $n\to\infty$.}  
\red{We e}valuate \eqref{eq:aVolterra} at $u=R_n+Y$, multiply by
$n/(R_n+Y)$, and use \eqref{eq:exteriorGreenSchur}.  The Bessel
asymptotic at this fixed shifted point gives, uniformly in large $n$,
\bel{eq:fixedYGaugeAmplitude}
 \left|\widehat A_n-e^{\red{m}Y}\Gamma_n(Y)\right|
 \le \red{C_{J}} e^{-\red{c_J}Y}+\frac{\red{C_{J,Y}}}{R_n}.
\ee
\red{for some $c_J>0$.}

The constant $\red{C_{J}} e^{-\red{c_J}Y}$ is the nonlinear Volterra tail, and $\frac{\red{C_{J,Y}}}{R_n}$
is the radial Bessel correction. \red{We also note that the first constant $C_J$ is independent of $Y$, due to the nonincreasing nature of the $\mathcal{N}_Y$ norm with respect to $Y$.}  Then
\eqref{eq:uniformQualitativeBoundary} allows $n\to\infty$ at fixed $Y$,
uniformly for $\beta\in J$, while  
Proposition~\ref{prop:righttail} further permits the limit $Y\to\infty$.  Hence
\bel{eq:AhatLimit}
        \widehat A_n\longrightarrow G_\beta.
\ee
\red{as $n\rightarrow \infty$.}
The same evaluation of \eqref{eq:etaVolterra} gives
\[
 \left|\widehat B_n-e^{\red{\mu_\beta}Y}\eta_n(R_n+Y)\right|
 \le \red{C_{J}} e^{-\red{c_J}Y}+\frac{\red{C_{J,Y}}}{R_n}.
\]
Taking the same iterated limits and using \eqref{eq:rightSub} yields
$\widehat B_n\to H_\beta$.
In the overlap region
\bel{eq:rightOverlap}
       1\ll y\ll R_n,\qquad u=R_n+y,
\ee
the standard Bessel asymptotics \cite{DLMF,Olver} give
\begin{align}
 k_1(R_n+y)
 &=\left(\frac{\pi R_n}{2\red{m}}\right)^{1/2}
        e^{-\red{m}R_n}e^{-\red{m}y}(1+O(R_n^{-1}+\abs y/R_n)),\label{eq:k1Overlap}\\
 k_{0,\beta}(R_n+y)
 &=\left(\frac{\pi}{2\red{\mu_\beta}R_n}\right)^{1/2}
        e^{-\red{\mu_\beta}R_n}e^{-\red{\mu_\beta}y}(1+O(R_n^{-1}+\abs y/R_n)).
\label{eq:k0Overlap}
\end{align}
Combining \eqref{eq:AhatLimit} with \eqref{eq:Ahat} yields
\bel{eq:An}
        A_n=
        \left(\frac{2\red{m}}{\pi}\right)^{1/2}\frac{R_n^{1/2}}{n}\,
        G_\beta e^{\red{m}R_n}(1+o(1)).
\ee
Using  Proposition~\ref{prop:coarseloc}, this may be written as
\bel{eq:An2}
        A_n=
        2\pi^{-1/2}\beta^{-1/8}G_\beta\,
        n^{-3/4}e^{\red{m}R_n}(1+o(1)).
\ee
Similarly, \eqref{eq:Bhat} and $\widehat B_n\to H_\beta$ give
\bel{eq:Bn}
        B_n=
        \left(\frac{2\red{\mu_\beta}R_n}{\pi}\right)^{1/2}
        H_\beta e^{\red{\mu_\beta}R_n}(1+o(1)),
\ee
as desired.
The powers of $R_n$ in \eqref{eq:An} come from two sources: the asymptotic equality 
\[
 uK_1(\red{m}u)\sim
 \left(\frac{\pi u}{2\red{m}}\right)^{1/2}e^{-\red{m}u},
\]
and the conversion $\Gamma_n(u)=na_n(u)/u$.

\medskip
The following technical lemma will be used to compare suitably normalized modified Bessel functions on the one hand with pure exponentials on the other hand. We need to be careful about the interplay between large $R=R_n$ and $y\in [K,\infty)$\red{, where $K$ is a fixed parameter}. 
The lemma gives a global comparison that retains
the required $O_J(R^{-1})$ accuracy at $y=K$. \red{We begin with  the following}
\red{
\begin{Definition}
    Fix $K\geq 1$. For $R>0$ and $y\ge K$, define the weight
\bel{eq:weightRK}
        \omega_{R,K}(y)
        :=
        \min\left\{\frac{1+y-K}{R+K},1\right\}.
\ee
For either of the two pairs
\[
        (q,\nu)=(m,1),
        \qquad
        (q,\nu)=(\mu_\beta,0),
\]
define the normalized recessive Bessel function
\bel{eq:normalizedRadialBesselMode}
 b_{q,\nu,R}(y)
 :=
 \left(\frac{\pi}{2q}\right)^{-1/2}
 R^{1/2}e^{qR}K_\nu(q(R+y)).
\ee
The first pair corresponds to the gauge variable $\Gamma$, and the
second to the Higgs defect $\eta$.
We refer to the operators 
\[
\begin{split}
 L_{q,\nu,R}
 &:=
 \partial_y^2+\frac{1}{R+y}\partial_y
 -\frac{\nu^2}{(R+y)^2}-q^2,\\
 L_{q,\infty}
 &:=
 \partial_y^2-q^2
\end{split}
\]
as radial and planar\footnote{It is gotten from the radial operator by sending $R\to\infty$ for fixed $y$.} operators, respectively. 
For $K\le y\le s$, define their forward Green kernels by
\bel{eq:radialPlanarTailKernels}
\begin{split}
 \mathcal G_{q,\nu,R}(y,s)
 &:=
 (R+s)\Bigl[
 K_\nu(q(R+y))I_\nu(q(R+s))\\
 &\hspace{37mm}
 -I_\nu(q(R+y))K_\nu(q(R+s))
 \Bigr],\\
 \mathcal G_{q,\infty}(y,s)
 &:=
 \frac{\sinh(q(s-y))}{q},
\end{split}
\ee
and set
\bel{eq:radialPlanarTailOperators}
\begin{split}
 (\mathcal V_{q,\nu,R}F)(y)
 &:=
 \int_y^\infty\mathcal G_{q,\nu,R}(y,s)F(s)\,\dd s,\\
 (\mathcal V_{q,\infty}F)(y)
 &:=
 \int_y^\infty\mathcal G_{q,\infty}(y,s)F(s)\,\dd s\red{,}
\end{split}
\ee
\red{where $\displaystyle F\in C([K,\infty))$.}
\end{Definition}
Our technical result is the following
}

\begin{Lemma}\label{lem:saturatingBesselComparison}
Let $J\Subset(0,4)$ and fix $K\ge1$.  
There is $R_0=R_0(K,J)$ such that, for $R\ge R_0$,
$\beta\in J$, and $0\le j\le2$,
\bel{eq:saturatingModeComparison}
 e^{qy}
 \left|
 \partial_y^j\red{(}b_{q,\nu,R}(y)-e^{-qy}\red{)}
 \right|
 \le C_{K,J}\omega_{R,K}(y).
\ee
Choose $\delta>0$ so that
\bel{eq:radialPlanarTailGap}
 0<4\delta<
 \inf_{\beta\in J}
 \min\{\mu_\beta,\,2m-\mu_\beta\}.
\ee
\red{Let $F\in C([K,\infty))$. }If \red{for some constant $\displaystyle M>0$}
\[
        |F(s)|\le M e^{-(q+4\delta)s},
        \qquad s\ge K,
\]
then, for $0\le j\le2$,
\bel{eq:saturatingGreenComparison}
 \left|
 \partial_y^j
 \{(\mathcal V_{q,\nu,R}-\mathcal V_{q,\infty})F\}(y)
 \right|
 \le
 C_{K,J,\delta}M
 \omega_{R,K}(y)e^{-(q+2\delta)y}.
\ee
In particular, since
\[
        \omega_{R,K}(K)=\frac{1}{R+K},
\]
both the normalized radial recessive mode $b_{q,\nu,R}$ and the Volterra operator $\mathcal V_{q,\nu,R}$ differ from their planar limits by $O_{K,J}(R^{-1})$ \red{and $\displaystyle O_{K,J,\delta}(MR^{-1})$, respectively} at $y=K$, uniformly through second $y$-derivatives.
\end{Lemma}

\begin{proof}
 Uniformly for
$\nu\in\{0,1\}$ and for $q$ in the compact range determined by $J$,
the standard large-argument expansion gives for large $z>0$
\[
 K_\nu(z)
 =
 \left(\frac{\pi}{2z}\right)^{1/2}
 e^{-z}\red{(}1+r_\nu(z)\red{)},
 \qquad
 |r_\nu(z)|+z|r_\nu'(z)|\le \frac{C_J}{z}.
\]
Substituting $z=q(R+y)$ into
\eqref{eq:normalizedRadialBesselMode} yields
\bel{eq:normalizedModeExpansion}
 b_{q,\nu,R}(y)
 =
 e^{-qy}p_R(y)\red{(}1+\rho_{q,\nu,R}(y)\red{)},
 \qquad
 p_R(y):=\left(\frac{R}{R+y}\right)^{1/2},
\ee
where
\bel{eq:normalizedModeRemainder}
 |\rho_{q,\nu,R}(y)|
 +(R+y)|\partial_y\rho_{q,\nu,R}(y)|
 \le \frac{C_J}{R+y}.
\ee
For fixed $y$ and $R\to\infty$,
\[
 p_R(y)
 =
 \left(1+\frac{y}{R}\right)^{-1/2}
 =
 1-\frac{y}{2R}+O_y(R^{-2}),
\]
while
\[
        \frac{1}{R+y}=\frac{1}{R}+O_y(R^{-2}).
\]
Thus \eqref{eq:normalizedModeExpansion} gives, for fixed $y$,
\[
 b_{q,\nu,R}(y)-e^{-qy}=O_{y,J}(R^{-1}),
\]
and the same conclusion holds after one derivative.
This fixed-$y$ estimate cannot be uniform on $[K,\infty)$.  Indeed,
for fixed $R$,
\[
        \lim_{y\to\infty}p_R(y)=0,
\]
and therefore
\[
        \lim_{y\to\infty}
        e^{qy}\red{(}b_{q,\nu,R}(y)-e^{-qy}\red{)}=-1.
\]
The function $\omega_{R,K}$ records this loss of uniform relative
closeness.  Directly from its definition and from the formula for $p_R$,
\bel{eq:modeWeightElementary}
 |p_R(y)-1|+|p_R'(y)|+\frac{1}{R+y}
 \le C_K\omega_{R,K}(y),
 \qquad y\ge K.
\ee
For $y-K\le R$, this follows by expanding $p_R$ and using
\[
        \omega_{R,K}(y)=\frac{1+y-K}{R+K}.
\]
For $y-K\ge R$, the right side is bounded below by a positive constant
depending only on $K$, whereas $0<p_R\le1$ and
$|p_R'|\le C(R+y)^{-1}$.  
Combining
\eqref{eq:normalizedModeExpansion}--\eqref{eq:modeWeightElementary}
proves \eqref{eq:saturatingModeComparison} for $j=0,1$.
For the second derivative, use
\[
 L_{q,\nu,R}b_{q,\nu,R}=0,
 \qquad
 L_{q,\infty}e^{-qy}=0.
\]
Writing $d_R=b_{q,\nu,R}-e^{-qy}$ gives
\[
 d_R''
 =
 q^2d_R
 -\frac{1}{R+y}b_{q,\nu,R}'
 +\frac{\nu^2}{(R+y)^2}b_{q,\nu,R}.
\]
The already proved estimates and
\eqref{eq:modeWeightElementary} give
\eqref{eq:saturatingModeComparison} for $j=2$.

We next compare the Volterra operators.  The exact Wronskian
identity
\[
 \mathcal W\!\left[
 K_\nu(q(R+\cdot)),I_\nu(q(R+\cdot))
 \right](s)
 =\frac{1}{R+s}
\]
shows that \eqref{eq:radialPlanarTailKernels} is the Green kernel for
$L_{q,\nu,R}$.  Similarly,
\[
 \mathcal W[e^{-q\cdot},e^{q\cdot}]=2q
\]
gives the planar kernel
$\mathcal G_{q,\infty}(y,s)=q^{-1}\sinh(q(s-y))$.
Set
\[
        u:=R+y,\qquad v:=R+s,\qquad d:=s-y\ge0.
\]
The large-argument expansions of $I_\nu$ and $K_\nu$, together with
their first derivatives, imply
\bel{eq:radialKernelLeadingForm}
 \mathcal G_{q,\nu,R}(y,s)
 =
 \left(\frac{v}{u}\right)^{1/2}
 \frac{\sinh(qd)}{q}
 +\mathcal E_{q,\nu,R}(y,s),
\ee
where, for $j=0,1$,
\bel{eq:radialKernelRemainder}
 |\partial_y^j\mathcal E_{q,\nu,R}(y,s)|
 \le
 C_J e^{qd}
 \left(\frac{v}{u}\right)^{1/2}
 \left(\frac{1}{u}+\frac{1}{v}\right).
\ee
An elementary division into the two cases
$\omega_{R,K}(s)<1$ and $\omega_{R,K}(s)=1$ shows that 
\[
\begin{split}
 &\left|
 \left(\frac{v}{u}\right)^{1/2}-1
 \right|
 +
 \left(\frac{v}{u}\right)^{1/2}
 \left(\frac{1}{u}+\frac{1}{v}\right)\le
 C_K(1+d)^2\omega_{R,K}(s).
\end{split}
\]
Using this estimate in
\eqref{eq:radialKernelLeadingForm}, \eqref{eq:radialKernelRemainder}
yields
\bel{eq:radialPlanarKernelDifference}
 \left|
 \partial_y^j
 \{\mathcal G_{q,\nu,R}
       -\mathcal G_{q,\infty}\}(y,s)
 \right|
 \le
 C_{K,J}(1+s-y)^2e^{q(s-y)}
 \omega_{R,K}(s),
 \qquad j=0,1.
\ee
The comparison function satisfies
\bel{eq:comparisonWeightTransport}
 \omega_{R,K}(s)
 \le
 (1+s-y)\omega_{R,K}(y),
 \qquad K\le y\le s.
\ee
Indeed, if $x=1+y-K\ge1$ and $d=s-y$, then
\[
 \min\left\{\frac{x+d}{R+K},1\right\}
 \le
 (1+d)\min\left\{\frac{x}{R+K},1\right\}.
\]
Combining \eqref{eq:radialPlanarKernelDifference},
\eqref{eq:comparisonWeightTransport}, and the assumed bound on $F$
implies that, for $j=0,1$,
\[
\begin{split}
 &\left|
 \partial_y^j
 \{(\mathcal V_{q,\nu,R}-\mathcal V_{q,\infty})F\}(y)
 \right|\\
 &\quad\le
 C_{K,J}M\omega_{R,K}(y)
 \int_y^\infty
 (1+s-y)^3e^{q(s-y)}e^{-(q+4\delta)s}\,\dd s\\
 &\quad=
 C_{K,J}M\omega_{R,K}(y)e^{-(q+4\delta)y}
 \int_0^\infty(1+t)^3e^{-4\delta t}\,\dd t\\
 &\quad\le
 C_{K,J,\delta}M
 \omega_{R,K}(y)e^{-(q+2\delta)y}.
\end{split}
\]
This proves \eqref{eq:saturatingGreenComparison} for $j=0,1$.
Finally, let
\[
        w_R:=\mathcal V_{q,\nu,R}F,
        \qquad
        w_\infty:=\mathcal V_{q,\infty}F.
\]
Then
\[
        L_{q,\nu,R}w_R=F,
        \qquad
        L_{q,\infty}w_\infty=F,
\]
and hence
\[
 (w_R-w_\infty)''
 =
 q^2(w_R-w_\infty)
 -\frac{1}{R+y}w_R'
 +\frac{\nu^2}{(R+y)^2}w_R.
\]
The individual Green-kernel bounds give
\[
        |w_R(y)|+|w_R'(y)|
        \le C_{J,\delta}M e^{-(q+4\delta)y}.
\]
Since
\[
        \frac{1}{R+y}\le\omega_{R,K}(y),
\]
the preceding equation and the estimates already proved for
$j=0,1$ yield \eqref{eq:saturatingGreenComparison} for~$j=2$.
\end{proof} 

At the right endpoint of the fixed interval $[-K,K]$, not every vector in
$\RR^4$ can occur as the Cauchy data of a solution which decays as
$y\to+\infty$.  Such a solution is determined by its two recessive
coefficients, one for the magnetic potential  and one for the Higgs field.  We
therefore parameterize the admissible Cauchy data at $y=K$ by these two
coefficients and compare the resulting radial and planar trace maps.
Put $ U=(\Gamma,\eta)$ where $\eta=1-\Phi$, with $m$ and $\mu_\beta$ as in~\eqref{eq:mMuBetaDef}.  On $[K,\infty)$ let
\[
\begin{split}
 \norm{U}_{\mathcal X_{\beta,K}}
 &:=
 \sum_{j=0}^2\sup_{y\ge K}
 \left(
 e^{my}\abs{\partial_y^j\Gamma(y)}
 +
 e^{\mu_\beta y}\abs{\partial_y^j\eta(y)}
 \right),\\
 \mathcal X_{\beta,K}
 &:=
 \left\{
 U\in C^2([K,\infty);\RR^2):
 \norm{U}_{\mathcal X_{\beta,K}}<\infty
 \right\}.
\end{split}
\]
The nonlinear terms in the gauge and Higgs equations are
\bel{eq:rightTailNonlinearity}
 \mathcal Q_\beta(U)
 :=
 \begin{pmatrix}
 -4\Gamma\eta+2\Gamma\eta^2\\
 -\Gamma^2+\Gamma^2\eta-3\beta\eta^2+\beta\eta^3
 \end{pmatrix}.
\ee
Using the Volterra operators of
Lemma~\ref{lem:saturatingBesselComparison}, define
\[
\begin{split}
 \mathbf V_{\beta,\infty}(F_G,F_H)
 &:=
 \left(
 \mathcal V_{m,\infty}F_G,\,
 \mathcal V_{\mu_\beta,\infty}F_H
 \right),\\
 \mathbf V_{\beta,R}(F_G,F_H)
 &:=
 \left(
 \mathcal V_{m,1,R}F_G,\,
 \mathcal V_{\mu_\beta,0,R}F_H
 \right).
\end{split}
\]
Thus $\mathbf V_{\beta,\infty}$ is the Volterra operator for the
two planar equations, whereas $\mathbf V_{\beta,R}$ is its radial
counterpart.
For $z=(z_G,z_H)\in\RR^2$, define
\bel{eq:rightTailHomogeneousTerms}
\begin{split}
 H_{\beta,z}^{\infty}(y)
 &:=
 \begin{pmatrix}
 (G_\beta+z_G)e^{-my}\\
 (H_\beta+z_H)e^{-\mu_\beta y}
 \end{pmatrix},\\
 H_{\beta,z}^{R}(y)
 &:=
 \begin{pmatrix}
 (G_\beta+z_G)b_{m,1,R}(y)\\
 (H_\beta+z_H)b_{\mu_\beta,0,R}(y)
 \end{pmatrix},
\end{split}
\ee
where the normalized radial modes $b_{q,\nu,R}$ are defined in
\eqref{eq:normalizedRadialBesselMode}.  Thus $z_G$ and $z_H$ measure
the deviations of the two recessive coefficients from the CHMO coefficients
$G_\beta$ and $H_\beta$.

For $U=(\Gamma,\eta)$, define its Cauchy trace at $y=K$, written in the
original variables $(\Phi,\Gamma)$, by
\bel{eq:rightTailTrace}
 \operatorname{Tr}_K U
 :=
 \bigl(
 1-\eta(K),\,
 \Gamma(K),\,
 -\eta'(K),\,
 \Gamma'(K)
 \bigr)\in\RR^4.
\ee
\red{We also define
\bel{eq:rightTailTraceMaps}
 \Theta_{\beta,K}(z):=\operatorname{Tr}_K U_{\beta,z},
 \qquad
 \Theta_{n,K}(z):=\operatorname{Tr}_K U_{n,z}.
\ee}

\begin{Lemma}\label{lem:rightTailManifold}
Let $J\Subset(0,4)$.  There are $K_0<\infty$, $\rho>0$, and
$\mathfrak R<\infty$ with the following property.  Fix
$\beta\in J$ and $K\ge K_0$.  Then, for every
$z\in B_\rho(0)$, the planar  equation
\bel{eq:planarTailFixedPoint}
 U_{\beta,z}
 =
 H_{\beta,z}^{\infty}
 +\mathbf V_{\beta,\infty}
       \mathcal Q_\beta(U_{\beta,z})
\ee
has a unique fixed point in
\[
 \left\{
 U\in\mathcal X_{\beta,K}:
 \norm{U}_{\mathcal X_{\beta,K}}\le\mathfrak R
 \right\}.
\]
For all sufficiently large $n$, the radial equation
\bel{eq:radialTailFixedPoint}
 U_{n,z}
 =
 H_{\beta,z}^{R_n}
 +\mathbf V_{\beta,R_n}
       \mathcal Q_\beta(U_{n,z})
\ee
has a unique fixed point in the same ball.  The maps
\[
        z\longmapsto U_{\beta,z},
        \qquad
        z\longmapsto U_{n,z}
\]
are of class $C^2$ with values in $\mathcal X_{\beta,K}$.
Then, \red{for the traces $\Theta_{n,K}$ and $\Theta_{\beta,K}$, we have}
\bel{eq:rightTailManifoldC2}
 \norm{\Theta_{n,K}-\Theta_{\beta,K}}_{C^2(B_\rho(0))}
 \le \frac{C_{K,J}}{R_n}.
\ee
At $z=0$, the planar fixed point is the CHMO tail:
\bel{eq:rightTailBaseTrace}
 U_{\beta,0}=(\Gamma_\beta,\eta_\beta),
 \qquad
 \Theta_{\beta,K}(0)
 =
 \bigl(
 \Phi_\beta(K),\Gamma_\beta(K),
 \Phi_\beta'(K),\Gamma_\beta'(K)
 \bigr),
\ee
where $\eta_\beta=1-\Phi_\beta$.  Also,
\bel{eq:rightTailTangentPlane}
 D\Theta_{\beta,K}(0)\RR^2=E_\beta^+(K),
\ee
where $E_\beta^+(K)$ is the plane of Cauchy data at $y=K$ of the
solutions of the interface linearization which decay at $+\infty$, as
constructed in Lemma~\ref{lem:rightstable}. After decreasing $\rho$, if necessary, one has
\bel{eq:rightTailParameterControl}
 |z-\widetilde z|
 \le
 C_{K,J}
 \left|
 \Theta_{\#,K}(z)-\Theta_{\#,K}(\widetilde z)
 \right|,
 \qquad z,\widetilde z\in B_\rho(0),
\ee
both for $\Theta_{\#,K}=\Theta_{\beta,K}$ and, for all sufficiently
large $n$, for $\Theta_{\#,K}=\Theta_{n,K}$.  In particular,
\[
 \mathcal S_{\beta,K}^+
 :=
 \Theta_{\beta,K}(B_\rho(0)),
 \qquad
 \mathcal S_{n,K}^+
 :=
 \Theta_{n,K}(B_\rho(0))
\]
are $C^2$ embedded two-dimensional submanifolds of $\RR^4$.  They are
precisely the local sets of traces at $y=K$ obtained from the two
families of decaying tails constructed above.
Finally, set
\[
 z_n:=
 \bigl(
 \widehat A_n-G_\beta,\,
 \widehat B_n-H_\beta
 \bigr),
\]
where $\widehat A_n$ and $\widehat B_n$ are as in~\eqref{eq:Ahat}, \eqref{eq:Bhat}.
Then $z_n\to0$, and the physical radial solution satisfies
\bel{eq:physicalRightTailTrace}
 \bigl(
 \Phi_n(K),\Gamma_n(K),
 \Phi_n'(K),\Gamma_n'(K)
 \bigr)
 =
 \Theta_{n,K}(z_n).
\ee
\end{Lemma}

\begin{proof}
\red{We initially fix $\rho>0$ and c}hoose
\[
 \sigma_J:=
 \frac{1}{\red{8}}\inf_{\beta\in J}
 \min\{\mu_\beta,\,2m-\mu_\beta\}>0.
\]
The positivity of $\sigma_J$ is where  
$J\Subset(0,4)$ is used.
We first construct the fixed points.  If
\[
 \norm{U}_{\mathcal X_{\beta,K}},
 \norm{\widetilde U}_{\mathcal X_{\beta,K}}
 \le\mathfrak R,
\]
then the polynomial form of $\mathcal Q_\beta$ gives
\[
\begin{split}
 \abs{
 \mathcal Q_{\beta,G}(U)(y)
 -
 \mathcal Q_{\beta,G}(\widetilde U)(y)}
 &\le
 C_{J,\mathfrak R}
 e^{-(m+4\sigma_J)y}
 \norm{U-\widetilde U}_{\mathcal X_{\beta,K}},\\
 \abs{
 \mathcal Q_{\beta,H}(U)(y)
 -
 \mathcal Q_{\beta,H}(\widetilde U)(y)}
 &\le
 C_{J,\mathfrak R}
 e^{-(\mu_\beta+4\sigma_J)y}
 \norm{U-\widetilde U}_{\mathcal X_{\beta,K}}.
\end{split}
\]
Indeed, every term in the first difference contains the factor
$\Gamma\eta$, while the slowest terms in the second difference have
the rates $2m$ and $2\mu_\beta$.  The definitions of $\sigma_J$ and
$\mathcal X_{\beta,K}$ therefore give the desired bounds.
The \red{same} planar and radial Volterra estimates \red{used in the proof of Lemma \ref{lem:saturatingBesselComparison}} now imply
\bel{eq:rightTailContractionEstimate}
 \red{\norm{
 \mathbf V_{\beta,\#}
 \left(
 \mathcal Q_\beta(U)-\mathcal Q_\beta(\widetilde U)
 \right)
 }_{\mathcal X_{\beta,K}}}
 \le
 \kappa_K
 \norm{U-\widetilde U}_{\mathcal X_{\beta,K}},
 \qquad
 \kappa_K\le
 C_{J,\mathfrak R}e^{-2\sigma_JK},
\ee
where $\mathbf V_{\beta,\#}$ denotes either
$\mathbf V_{\beta,\infty}$ or $\mathbf V_{\beta,R_n}$.
The normalized homogeneous terms satisfy
\[
 \sup_{\beta\in J}
 \sup_{z\in B_\rho(0)}
 \left(
 \norm{H_{\beta,z}^{\infty}}_{\mathcal X_{\beta,K}}
 +
 \norm{H_{\beta,z}^{R_n}}_{\mathcal X_{\beta,K}}
 \right)
 \le C_J
\]
for all sufficiently large $n$.  Choose $\mathfrak R>2C_J$ and then
increase $K_0$ so that $\kappa_K\le1/2$ and the two fixed-point maps
send the closed ball of radius $\mathfrak R$ into itself.  The
contraction mapping theorem proves existence and uniqueness in this ball.
These maps are affine in $z$ and polynomial in $U$.
Since their derivatives with respect to $U$ have norm at most
$\kappa_K<1$, we obtain the following fixed points depending on parameters 
\[
        U_{\beta,\cdot},U_{n,\cdot}
        \in C^2\bigl(B_\rho(0);\mathcal X_{\beta,K}\bigr).
\]
We next compare the radial and planar fixed points.  Write
\[
        \omega_n(y):=\omega_{R_n,K}(y)
\]
and, for $W=(W_G,W_H)$, define
\bel{eq:rightTailComparisonNorm}
 \norm{W}_{\mathcal X_{\beta,K}^{(\omega_n)}}
 :=
 \sum_{j=0}^2\sup_{y\ge K}
 \frac{1}{\omega_n(y)}
 \left(
 e^{my}\abs{\partial_y^jW_G(y)}
 +
 e^{\mu_\beta y}\abs{\partial_y^jW_H(y)}
 \right).
\ee
Subtracting \eqref{eq:planarTailFixedPoint} from
\eqref{eq:radialTailFixedPoint} gives
\bel{eq:rightTailFixedPointDifference}
\begin{split}
 U_{n,z}-U_{\beta,z}
 &=
 H_{\beta,z}^{R_n}-H_{\beta,z}^{\infty}\\
 &\quad+
 \left(
 \mathbf V_{\beta,R_n}-\mathbf V_{\beta,\infty}
 \right)
 \mathcal Q_\beta(U_{\beta,z})\\
 &\quad+
 \mathbf V_{\beta,R_n}
 \red{\left(
 \mathcal Q_\beta(U_{n,z})
 -
 \mathcal Q_\beta(U_{\beta,z})
 \right)}.
\end{split}
\ee
Lemma~\ref{lem:saturatingBesselComparison} implies that
\bel{eq:rightTailLinearComparison}
 \sup_{z\in B_\rho(0)}
 \left[
 \norm{
 H_{\beta,z}^{R_n}-H_{\beta,z}^{\infty}
 }_{\mathcal X_{\beta,K}^{(\omega_n)}}
 +
 \norm{
 \left(
 \mathbf V_{\beta,R_n}-\mathbf V_{\beta,\infty}
 \right)
 \mathcal Q_\beta(U_{\beta,z})
 }_{\mathcal X_{\beta,K}^{(\omega_n)}}
 \right]
 \le C_{K,J}.
\ee
The same bounds on the forcing terms used above, together with
\[
        \omega_n(s)
        \le (1+s-y)\omega_n(y),
        \qquad K\le y\le s,
\]
give
\bel{eq:rightTailWeightedLipschitz}
 \norm{
 \mathbf V_{\beta,R_n}
 \red{\left(
 \mathcal Q_\beta(U)-\mathcal Q_\beta(\widetilde U)
 \right)}
 }_{\mathcal X_{\beta,K}^{(\omega_n)}}
 \le
 \kappa_K
 \norm{
 U-\widetilde U
 }_{\mathcal X_{\beta,K}^{(\omega_n)}}.
\ee
Consequently, \eqref{eq:rightTailFixedPointDifference} and
$\kappa_K\le1/2$ imply
\bel{eq:rightTailWeightedComparison}
 \sup_{z\in B_\rho(0)}
 \norm{
 U_{n,z}-U_{\beta,z}
 }_{\mathcal X_{\beta,K}^{(\omega_n)}}
 \le C_{K,J}.
\ee
It remains to compare the first two parameter derivatives.  Set
\[
 A_{\#,z}
 :=
 \mathbf V_{\beta,\#}
 D\mathcal Q_\beta(U_{\#,z}),
 \qquad
 \#\in\{\infty,R_n\}.
\]
By \eqref{eq:rightTailContractionEstimate},
\[
        \norm{A_{\#,z}}\le\kappa_K\le\frac{1}{2},
        \qquad
        \norm{(I-A_{\#,z})^{-1}}\le2.
\]
Because the homogeneous terms are affine in $z$, differentiation of
the fixed-point equations gives, for $k,\ell\in\{G,H\}$,
\bel{eq:rightTailFirstParameterDerivative}
 (I-A_{\#,z})\partial_{z_\ell}U_{\#,z}
 =
 \partial_{z_\ell}H_{\beta,z}^{\#},
\ee
and
\bel{eq:rightTailSecondParameterDerivative}
\begin{split}
 (I-A_{\#,z})
 \partial_{z_k}\partial_{z_\ell}U_{\#,z}
 &=
 \mathbf V_{\beta,\#}
 D^2\mathcal Q_\beta(U_{\#,z})
 \left[
 \partial_{z_k}U_{\#,z},
 \partial_{z_\ell}U_{\#,z}
 \right].
\end{split}
\ee
\red{Here, by a slight abuse of notation, we have made the convention that $U_{\infty,z}=U_{\beta,z}$.}

\red{We s}ubtract the radial and planar versions of these identities.
\red{We first focus on the first order estimates. Thus, 
\begin{align*}
  \partial_{z_l}(U_{n,z}-U_{\beta,z})&= \partial_{z_l}(H_{\beta,z}^{R_n}-H_{\beta,z}^{\infty})\\
 &\quad+
 \partial_{z_l}\left(\left(
 \mathbf V_{\beta,R_n}-\mathbf V_{\beta,\infty}
 \right)
 \mathcal Q_\beta(U_{\beta,z})\right)\\
 &\quad+
 \partial_{z_l}(\mathbf V_{\beta,R_n}
 \red{\left(
 \mathcal Q_\beta(U_{n,z})
 -
 \mathcal Q_\beta(U_{\beta,z})
 \right))} 
\end{align*}
Equivalently,
\begin{equation}\label{eq:FirstOrderExpansion}
\begin{aligned}
  (I-A_{R_n,z})(\partial_{z_l}(U_{n,z}-U_{\beta,z}))&= \partial_{z_l}(H_{\beta,z}^{R_n}-H_{\beta,z}^{\infty})\\
 &\quad+
 \left(
 \mathbf V_{\beta,R_n}-\mathbf V_{\beta,\infty}
 \right)
 D\mathcal Q_\beta(U_{\beta,z})(\partial_{z_l}U_{\beta,z})\\
 &\quad+
 \mathbf V_{\beta,R_n}
 \red{\left(
 D\mathcal Q_\beta(U_{n,z})
 -
 D\mathcal Q_\beta(U_{\beta,z})
 \right)(\partial_{z_l}U_{\beta,z})} 
 \end{aligned}
\end{equation}
Repeating the proof of \eqref{eq:rightTailWeightedLipschitz} for the linearized version of the equations gives
\begin{align*}
    \|A_{R_n,z}\|_{\mathcal{X}^{(\omega_n)}_{\beta,K}\rightarrow \mathcal{X}^{(\omega_n)}_{\beta,K}}\leq \kappa_K,
\end{align*}
which for $K$ sufficiently large, leads to
\begin{equation}\label{eq:CrudeResolvent}
    \|(I-A_{R_n,z})^{-1}\|_{\mathcal{X}^{(\omega_n)}_{\beta,K}\rightarrow \mathcal{X}^{(\omega_n)}_{\beta,K}}\leq 2.
\end{equation}
Together with equations \eqref{eq:rightTailFirstParameterDerivative} and \eqref{eq:rightTailSecondParameterDerivative}, these lead to
\begin{equation}\label{eq:CrudeDerivativeBounds}
  \sup_{\substack{z\in B_\rho(0)}}\sum_{\substack{\#\in\{\infty,R_n\}}}\sum_{\substack{1\leq|\alpha|\leq 2}}\|\partial_z^\alpha U_{\#,z}\|_{\mathcal{X}_{\beta,K}}\leq C_{K,J}.  
\end{equation}
From \eqref{eq:saturatingModeComparison}, we have
\begin{align*}
  \|\partial_{z_l}(H_{\beta,z}^{R_n}-H_{\beta,z}^{\infty})\|_{\mathcal{X}^{(\omega_n)}_{\beta,K}} &\leq C_{K,J}, 
\end{align*}
which deals with the first term in \eqref{eq:FirstOrderExpansion}.

For the second term, we use the polynomial structure of $\mathcal{Q}_\beta$ and \eqref{eq:CrudeDerivativeBounds}, combined with Lemma~\ref{lem:saturatingBesselComparison}. For the third term, we use a similar argument to the one from the proof of \eqref{eq:rightTailWeightedLipschitz}, along with \eqref{eq:rightTailWeightedComparison} and \eqref{eq:CrudeDerivativeBounds}.

Finally, the previous resolvent bound \eqref{eq:CrudeResolvent} leads to 
\begin{align*}
 \sup_{z\in B_\rho(0)}
 \sum_{|\alpha|=1}
 \norm{
 \partial_z^\alpha
 \red{(}U_{n,z}-U_{\beta,z}\red{)}
 }_{\mathcal X_{\beta,K}^{(\omega_n)}}
 \le C_{K,J}.
\end{align*}
The second order bounds follow from the equations, along with the previous first order estimate and \eqref{eq:rightTailWeightedComparison}. They all lead to
}
\bel{eq:rightTailWeightedC2Comparison}
 \sup_{z\in B_\rho(0)}
 \sum_{1\le|\alpha|\le2}
 \norm{
 \partial_z^\alpha
 \red{(}U_{n,z}-U_{\beta,z}\red{)}
 }_{\mathcal X_{\beta,K}^{(\omega_n)}}
 \le C_{K,J}.
\ee
Together with \eqref{eq:rightTailWeightedComparison}, this proves the
same bound for all $|\alpha|\le2$.
Since
\[
        \omega_n(K)=\frac{1}{R_n+K},
\]
evaluation at $y=K$ in
\eqref{eq:rightTailWeightedComparison} and
\eqref{eq:rightTailWeightedC2Comparison} gives
\[
 \norm{
 \Theta_{n,K}-\Theta_{\beta,K}
 }_{C^2(B_\rho(0))}
 \le \frac{C_{K,J}}{R_n+K}
 \le \frac{C_{K,J}}{R_n}.
\]
This proves \eqref{eq:rightTailManifoldC2}.
For $z=0$, equation \eqref{eq:planarTailFixedPoint} is the system
\eqref{eq:GammaVolterra}--\eqref{eq:EtaVolterra}.  The CHMO tail belongs
to the  ball containing the fixed point when $K$ is sufficiently large.  Uniqueness
therefore gives
\[
        U_{\beta,0}=(\Gamma_\beta,\eta_\beta),
\]
and hence \eqref{eq:rightTailBaseTrace}.
Let
\[
 W_G:=
 \left.\partial_{z_G}U_{\beta,z}\right|_{z=0},
 \qquad
 W_H:=
 \left.\partial_{z_H}U_{\beta,z}\right|_{z=0}.
\]
Differentiating \eqref{eq:planarTailFixedPoint} shows that $W_G$ and
$W_H$ solve the interface equations linearized around
$(\Gamma_\beta,\eta_\beta)$.
Their respective coefficient vectors in these variables are
$(1,0)$ and $(0,1)$.  Passing to the variations
$(\xi,\psi)=(-W_\eta,W_\Gamma)$ \green{immediately yields}
\[
 D\Theta_{\beta,K}(0)(1,0)^t=Z_{G,\beta}(K),\qquad
 D\Theta_{\beta,K}(0)(0,1)^t=-Z_{H,\beta}(K).
\]
The two columns are independent and span $E_\beta^+(K)$ by
Lemma~\ref{lem:rightstable}.  This proves
\eqref{eq:rightTailTangentPlane}.

\red{
We define the coordinate
projection $P:\RR^4\to\RR^2$ given by $\displaystyle P(x_1,x_2,x_3,x_4)=(x_2,-x_1)$. We also know that $D\Theta_{\beta,K}(0)$ has rank two, and we claim that 
\[
        M_{\beta,K}:=S_{\beta,K}P D\Theta_{\beta,K}(0)
\]
is invertible, where $\displaystyle S_{\beta,K}=\begin{pmatrix}e^{mK} & 0\\0 & e^{\mu_\beta K}\end{pmatrix}$.
To see that, let
\[
 W_{G,z}:=
 \partial_{z_G}U_{\beta,z},
 \qquad
 W_{H,z}:=
 \partial_{z_H}U_{\beta,z}.
\]
From \eqref{eq:rightTailFirstParameterDerivative} we get that
\begin{align*}
 (I-A_{\infty,z})\partial_{z_G}U_{\beta,z}
 &=
 \partial_{z_G}H_{\beta,z}^{\infty}=(e^{-my},0)\\
 (I-A_{\infty,z})\partial_{z_H}U_{\infty,z}
 &=
 \partial_{z_H}H_{\beta,z}^{\infty}=(0,e^{-\mu_\beta y}).
 \end{align*}
 Thus,
 \begin{align*}
 (I-A_{\infty,z})W_{G,z}
 &=
 (e^{-my},0)\\
 (I-A_{\infty,z})W_{H,z}
 &=
 (0,e^{-\mu_\beta y}).
 \end{align*}
 From \eqref{eq:rightTailContractionEstimate}, we have
 \begin{align*}
\|A_{\infty,z}\|_{\mathcal X_{\beta,K}\rightarrow \mathcal X_{\beta,K}}\leq\kappa_K\leq \frac{1}{2},
 \end{align*}
 for every $z\in B_\rho(0)$. In particular, it follows that
 \begin{align*}
     \|W_{G,z}-(e^{-my},0)\|_{\mathcal X_{\beta,K}}+\|W_{H,z}-(0,e^{-\mu_\beta y})\|_{\mathcal X_{\beta,K}}\leq \frac{C_J\kappa_K}{1-\kappa_K},
 \end{align*}
 where $\kappa_K\leq C_{J}e^{-2\sigma_JK}$,
 hence for sufficiently large $K_0$,
 \begin{align*}
     \sup_{\substack{z\in B_\rho(0)}}\|S_{\beta,K}PD\Theta_{\beta,K}(z)-I\|\leq\frac{1}{8}.
 \end{align*}
 These imply that
 \begin{align*}
     \|M_{\beta,K}-I\|&\leq\frac{1}{8}\\
     \|M^{-1}_{\beta,K}\|&\leq\frac{8}{7}\\
     \sup_{\substack{z\in B_\rho(0)}}\|S_{\beta,K}PD\Theta_{\beta,K}(z)-M_{\beta,K}\|&\leq\frac{1}{4}.
 \end{align*}
 For $z,\widetilde z\in B_\rho(0)$, the fundamental theorem of calculus
then gives
\[
\begin{split}
 S_{\beta,K}P\Theta_{\beta,K}(z)-S_{\beta,K}P\Theta_{\beta,K}(\widetilde z)
 &=
 M_{\beta,K}(z-\widetilde z)\\
 \quad+&
 \int_0^1
 \red{\left(S_{\beta,K}
 P D\Theta_{\beta,K}
 \bigl(\widetilde z+t(z-\widetilde z)\bigr)
 -
 M_{\beta,K}
 \right)}
 (z-\widetilde z)\,\dd t .
\end{split}
\]
It follows that
\[
 |z-\widetilde z|
 \le
 2\norm{M_{\beta,K}^{-1}}
 \left|
 S_{\beta,K}P(\Theta_{\beta,K}(z)-\Theta_{\beta,K}(\widetilde z))
 \right|\leq C_{K,J}|\Theta_{\beta,K}(z)-\Theta_{\beta,K}(\widetilde z)|.
\]
The $C^1$ estimate \eqref{eq:rightTailManifoldC2} gives the same
inequality, with a possibly larger constant, for $\Theta_{n,K}$ when
$n$ is sufficiently large.}
This proves
\eqref{eq:rightTailParameterControl}.  It also shows that the two
trace maps are injective immersions with continuous local inverses, and 
their images are therefore $C^2$ embedded two-dimensional
submanifolds of $\RR^4$.
It remains to identify the trace of the vortex.  From
\eqref{eq:Ahat} and \eqref{eq:Bhat}, together with the definition of
$b_{q,\nu,R}$, one has the exact identities
\[
\begin{split}
 \frac{n}{R_n+y}A_n k_1(R_n+y)
 &=
 \widehat A_n\,b_{m,1,R_n}(y),\\
 B_n k_{0,\beta}(R_n+y)
 &=
 \widehat B_n\,b_{\mu_\beta,0,R_n}(y).
\end{split}
\]
Thus the normalized recessive coefficients of the tail of the actual vortex
are $\widehat A_n$ and $\widehat B_n$.  Define
\[
        z_n=
        \bigl(
        \widehat A_n-G_\beta,\,
        \widehat B_n-H_\beta
        \bigr).
\]
By \eqref{eq:AhatLimit} \red{and the similar statement $\displaystyle \widehat{B}_n\rightarrow H_\beta$ succeeding it},
\[
        z_n\longrightarrow0.
\]
Proposition~\ref{prop:exteriorBessel} shows that the vortex tail
belongs to the ball from above and satisfies
\eqref{eq:radialTailFixedPoint} with parameter $z_n$.  By uniqueness,
\[
        (\Gamma_n,\eta_n)=U_{n,z_n}
        \qquad\text{on }[K,\infty).
\]
Applying \eqref{eq:rightTailTrace} proves
\eqref{eq:physicalRightTailTrace}.
\end{proof}

\section{The virial identity and preliminary localization}
\label{sec:preliminaryLocalization}

The quantitative comparison of the radial and planar profiles requires the
absolute localization
\bel{eq:SLhyp}
        R_n=R_n^{(0)}+O_\beta(1),
        \qquad
        b_n=\sqrt\beta+O_\beta(R_n^{-1}).
\ee
We prove this estimate before beginning the quantitative comparison.  The approach we follow in this section is natural because $R_n$ describes the scale of the vortex, and the virial identity is precisely the stationarity condition associated with changing that scale. The
argument is the first of two uses of the exact virial identity.  In this
first pass the splitting point $M$ is fixed.  The core estimates from
Section~\ref{sec:locRn}, qualitative convergence on fixed intervals, and
uniform decay on the right determine the reduced radius equation up to an
error which vanishes after first letting $n\to\infty$ and then
$M\to\infty$.  

After \eqref{eq:SLhyp} has been established,
Sections~\ref{sec:quantitativeBoundary} and~\ref{sec:leftMatching}
produce $O(R_n^{-1})$ comparisons on fixed interface intervals and in
the two adjacent tails.  We then return to the virial identity in
Section~\ref{sec:refinedEnergy}.  Fixing $C_{\log}$ sufficiently large,
we divide the radial energy and virial integrals at
\[
        u=R_n-M_n,
        \qquad
        M_n=C_{\log}\sqrt{\log R_n}.
\]
Since $M_n\to\infty$ while $M_n/R_n\to0$, this cutoff lies in the
overlap between the core and the interface.  \green{The terms involving this artificial cutoff cancel out when the integrals over the corresponding regions are recombined, and in this way we end up with an error that is uniformly bounded in $n$. This second use of the virial identity provides us with the desired constant correction to the radius, whereas the energy identity allows us to infer the asymptotic expansion of the tension.}  

Throughout this section, $\beta$ ranges in a fixed compact interval
$J\Subset(0,4)$.  The same parameter range will be used in the
quantitative argument, by Proposition~\ref{prop:nearBPS}.

\subsection{Exact scaling and reduced densities}

The following virial, or Pohozaev, identity is obtained from the stationarity
of the radial energy under dilations.  We include its short derivation because
it is the starting point for both localization arguments.

\begin{Lemma}\label{lem:scalingVirial}
For every radial minimizer, we have $\calM_n=\calV_n$, 
\[
        \red{\text{ where\ \  }}
        \calM_n=\int_0^\infty \frac{n^2}{2u}(a_n')^2\,du
        \quad \red{\text{and\ \  }}
        \calV_n=\int_0^\infty \frac{\beta}{2}u(1-\phi_n^2)^2\,du .
\]
\end{Lemma}

\begin{proof}
For $\lambda>0$ set
\[
        \phi_{n,\lambda}(u)=\phi_n(u/\lambda),
        \qquad
        a_{n,\lambda}(u)=a_n(u/\lambda).
\]
The boundary conditions are unchanged, so
$(\phi_{n,\lambda},a_{n,\lambda})\in\mathscr A_n$.  After the change of
variables $u=\lambda v$, the four terms in \eqref{eq:energy} become,
respectively,
\[
 \lambda^{-2}\calM_n,\qquad
 \int_0^\infty v(\phi_n'(v))^2\,dv,\qquad
 \int_0^\infty \frac{n^2}{v}a_n(v)^2\phi_n(v)^2\,dv,\qquad
 \lambda^2\calV_n.
\]
The function
$\lambda\mapsto\calT_n[\phi_{n,\lambda},a_{n,\lambda}]$ has a minimum at
$\lambda=1$.  Differentiation at that point gives
$-2\calM_n+2\calV_n=0$.
\end{proof}

For the translated profiles define
\bel{eq:hnjn}
\begin{split}
 h_n(y)&=\Phi_n'(y)^2+\frac{\beta}{2}(1-\Phi_n(y)^2)^2,\\
 j_n(y)&=\Gamma_n(y)^2\Phi_n(y)^2
             +\frac{\beta}{2}(1-\Phi_n(y)^2)^2.
\end{split}
\ee
and let $h_\beta,j_\beta$ denote the same expressions with
$(\Phi_n,\Gamma_n)$ replaced by $(\Phi_\beta,\Gamma_\beta)$.
\red{By using} \eqref{eq:magneticCouplingIdentity}, we may remove the
terms involving $a_n'$ and $a_n\phi_n$ \red{from the energy}, and we denote the resulting density by $h_n$.
If \red{we use} \eqref{eq:magneticCouplingIdentity} \red{together} with the virial identity \red{instead},
\red{we are naturally led to consider} a second density function, \red{which we shall denote by} $j_n$.  

\begin{Lemma}\label{lem:exactReducedDensities}
The exact energy and scaling identities take the forms
\bel{eq:exactReducedEnergy}
        \calT_n=\frac{n^2}{R_n^2}
        +\int_{-R_n}^{\infty}(R_n+y)h_n(y)\,dy
\ee
and
\bel{eq:exactReducedVirialDensity}
        0=-\frac{n^2}{R_n^2}
        +\int_{-R_n}^{\infty}(R_n+y)j_n(y)\,dy .
\ee
In addition,
\bel{eq:hjSurface}
 \int_{\mathbb R}\left(h_\beta-\frac{\beta}{2}{\bf1}_{(-\infty,0)}\right)dy
 =s_\beta,
 \qquad
 2\int_{\mathbb R}\left(j_\beta-\frac{\beta}{2}{\bf1}_{(-\infty,0)}\right)dy
 =s_\beta .
\ee
Each  integrand in \eqref{eq:hjSurface}, as well as its product
with $y$, belongs to $L^1(\mathbb R)$.
\end{Lemma}

\begin{proof}
The magnetic equation and the boundary conditions give
\eqref{eq:magneticCouplingIdentity}, namely
\[
 \int_0^\infty
 \frac{n^2}{2u}(a_n')^2+\frac{n^2}{u}a_n^2\phi_n^2 du
 =\frac{n^2}{R_n^2}.
\]
After the change of variables $u=R_n+y$, the two remaining terms in
\eqref{eq:energy} give \eqref{eq:exactReducedEnergy}.  The same change of
variables gives
\[
 \int_0^\infty \frac{n^2}{u}a_n^2\phi_n^2\,du
 =\int_{-R_n}^{\infty}
   (R_n+y)\Gamma_n(y)^2\Phi_n(y)^2\,dy.
\]
It follows from the preceding identity and Lemma~\ref{lem:scalingVirial}
that
\[
 \calV_n=\frac{n^2}{R_n^2}
 -\int_{-R_n}^{\infty}
   (R_n+y)\Gamma_n(y)^2\Phi_n(y)^2\,dy.
\]
Substituting the definition of $\calV_n$  proves
\eqref{eq:exactReducedVirialDensity}.
For the interface orbit, the Hamiltonian identity \eqref{eq:hamzero} gives
\[
        j_\beta=\Phi_\beta'^2+\frac{1}{2}\Gamma_\beta'^2,
        \qquad
        e_\beta=2j_\beta .
\]
The magnetic interface equation \eqref{eq:interface} gives
\[
 \frac{1}{2}(\Gamma_\beta\Gamma_\beta')'
 =\frac{1}{2}\Gamma_\beta'^2+\Gamma_\beta^2\Phi_\beta^2,
\]
and hence
\[
        h_\beta=e_\beta
        -\frac{1}{2}(\Gamma_\beta\Gamma_\beta')'.
\]
By Lemma~\ref{lem:interfaceUniformTails},
\[
 \Gamma_\beta(y)\Gamma_\beta'(y)
 =\beta y+O_\beta(e^{-c_\beta y^2})
 \qquad (y\to-\infty),
\]
whereas $\Gamma_\beta\Gamma_\beta'\to0$ at $+\infty$.  Therefore
\[
 \lim_{L\to\infty}
 \left(
 \frac{1}{2}\Gamma_\beta(-L)\Gamma_\beta'(-L)
 +\frac{\beta}{2}L
 \right)=0.
\]
Integrating $h_\beta=e_\beta-\tfrac12(\Gamma_\beta\Gamma_\beta')'$ on $(-L,\infty)$ and subtracting $\tfrac\beta2L$ gives
\[
 \int_{-L}^{\infty}h_\beta\,dy-\tfrac\beta2L
 =\int_{-L}^{\infty}e_\beta\,dy-\tfrac\beta2L
 -\tfrac12\lim_{y\to+\infty}\Gamma_\beta(y)\Gamma_\beta'(y)+\tfrac12\Gamma_\beta(-L)\Gamma_\beta'(-L).
\]
Since $\Gamma_\beta\Gamma_\beta'\to0$ at $+\infty$ and
$$
\lim_{L\to\infty}\Big(\tfrac12\Gamma_\beta(-L)\Gamma_\beta'(-L)+\tfrac\beta2L\Big)=0,
$$
taking $L\to\infty$ yields
$$
\int_{\RR}\Big(h_\beta-\tfrac\beta2\mathbf 1_{(-\infty,0)}\Big)dy
=\lim_{L\to\infty}\Big[\int_{-L}^{\infty}e_\beta\,dy-\beta L\Big]=s_\beta
$$
by \eqref{eq:sbeta}.  The second identity in \eqref{eq:hjSurface} follows from $e_\beta=2j_\beta$.  The Gaussian estimates at
$-\infty$ and the exponential estimates at $+\infty$ in
Lemma~\ref{lem:interfaceUniformTails} also prove the two first-moment
statements.
\end{proof}

\subsection{The core contribution}

We next estimate the part of the reduced densities on
$0\le u\le R_n-M$.  To avoid confusing the radial and translated
variables, write
\[
        \gamma_n(u):=\frac{n}{u}a_n(u),\qquad u>0.
\]
Thus $\gamma_n(R_n+y)=\Gamma_n(y)$. The following lemma establishes a Caccioppoli estimate on the scalar field. For the preliminary localization argument below, the parameter $M$ will
eventually be kept fixed.  We nevertheless prove the following core
estimate uniformly up to
\[
        M=C_{\log}\sqrt{\log R_n},
\]
because this larger range is needed in the refined calculation of
Section~\ref{sec:refinedEnergy}.  At that scale a Gaussian remainder,
after multiplication by the radial factor $R_n$, satisfies
\[
        R_ne^{-cM^2}
        \le R_n^{1-cC_{\log}^2}.
\]
Thus, by choosing $C_{\log}$ sufficiently large, the Gaussian error becomes
bounded.  Sections~\ref{sec:quantitativeBoundary}
and~\ref{sec:leftMatching} provide the corresponding quantitative
 estimates on the boundary layer up to the same scale.

\begin{Lemma}\label{lem:coreCaccioppoli}
\red{Let $\displaystyle\alpha\in(0,1)$.} There is $\red{c_J}>0$ \red{ and $\displaystyle L_0(J)<\infty$}, independent of $C_{\log}$, such that for every fixed
$C_{\log}<\infty$ there \red{are} $\red{C(C_{\log})=C_J}(C_{\log})<\infty$ for which, \red{for every $\beta\in J$, } uniformly for
\[
        1\le M\le C_{\log}\sqrt{\log R_n},
\]
one has
\bel{eq:coreCaccioppoli}
 \int_0^{R_n-M}
 u\left((\phi_n')^2+\gamma_n^2\phi_n^2+\phi_n^2\right)du
 \le C(C_{\log})+C(C_{\log})R_n e^{-cM^2},
\ee
\red{for all $n$ sufficiently large.}
\end{Lemma}

\begin{proof}
\red{We first note that if we pick} $L_0$ large, \red{ we have the following coercivity estimate for the coefficient of $\phi_n$ in the Higgs equation of \eqref{eq:vortex}}
\bel{eq:coreCoercivity}
 \gamma_n(u)^2-\beta(1-\phi_n(u)^2)
 \ge \frac12(1+\gamma_n(u)^2),
 \qquad 0<u\le R_n-L_0. 
\ee
Indeed, \eqref{eq:anLowerBound} yields a pointwise lower bound $a_n(u)\ge 1-u^2/R_n^2$ for
$0<u\le R_n-L_0$, hence
\[\gamma_n(u)=(n/u)a_n(u)\ge (n/u)(1-u^2/R_n^2)\ge (2n/R_n^2)(R_n-u)\ge c_0(R_n-u)\] after renaming constants (using $u\le R_n$ and $b_n=2n/R_n^2\to\sqrt\beta$). In other words, 
\bel{eq:coreMagneticLowerBound}
        \gamma_n(u)\ge c_0(R_n-u)\red{\geq c_0L_0}
        \qquad (0<u\le R_n-L_0).
\ee
\red{We now have 
\begin{align*}
    \gamma_n(u)^2-\beta(1-\phi_n(u)^2)\geq \frac{1+\gamma_n(u)^2}{2}+\frac{c_0^2L_0^2-1}{2}-\beta.
\end{align*}}
After increasing $L_0$, this implies \eqref{eq:coreCoercivity}.
Assume first that $M\ge2L_0$.  Choose
$\zeta\in C^\infty([0,\infty))$ such that
\[
\begin{gathered}
 \zeta=1\quad\hbox{on }[0,R_n-M],\qquad
 \zeta=0\quad\hbox{on }[R_n-M/2,\infty),\\
 |\zeta'|\le \frac{C}{M}.
\end{gathered}
\]
Write the $\phi_n$ equation \eqref{eq:vortex} in the form
\[
 -\phi_n''-\frac{1}{u}\phi_n'
 +(\gamma_n^2-\beta(1-\phi_n^2))\phi_n=0.
\]
Multiply this equation by $u\zeta^2\phi_n$ and integrate.  The boundary
term at the origin vanishes because
$\phi_n(u)=O(u^n)$ and $\phi_n'(u)=O(u^{n-1})$.  The other boundary term
vanishes because $\zeta$ has compact support.  Thus we obtain the Caccioppoli identity
\bel{eq:coreCaccioppoliIdentity}
\begin{split}
 &\int_0^\infty u\zeta^2(\phi_n')^2\,du
 +\int_0^\infty
   u\zeta^2(\gamma_n^2-\beta(1-\phi_n^2))\phi_n^2\,du\\
 &\qquad
 =-2\int_0^\infty u\zeta\zeta'\phi_n\phi_n'\,du\\
 &\qquad
 \le \frac{1}{2}\int_0^\infty u\zeta^2(\phi_n')^2\,du
 +2\int_0^\infty u|\zeta'|^2\phi_n^2\,du .
\end{split}
\ee
 On the support of $\zeta'$ one has
\[
        \frac{M}{2}\le \eta:=R_n-u\le M.
\]
Since $M\ge2L_0$ and $C_{\log}$ is fixed,
\[
        L_0\le\eta\le M
        \le C_{\log}\sqrt{\log R_n}
        \le R_n^\alpha
\]
for all sufficiently large $n$.  Lemma~\ref{lem:clearingout} therefore
applies throughout the support of $\zeta'$ and gives, after decreasing
$c>0$ if necessary,
\[
        \phi_n(u)^2
        =\bigl(\Phi_n^-(\eta)\bigr)^2
        \le C e^{-cM^2}.
\]
Using $u\le R_n$, $|\zeta'|\le C/M$, and the fact that the support of
$\zeta'$ has length at most $M$, we obtain
\[
\begin{split}
 \int_0^\infty u|\zeta'|^2\phi_n^2\,du
 &\le
 \frac{C}{M^2}
 \int_{R_n-M}^{R_n-M/2}u\phi_n(u)^2\,du\\
 &\le
 \frac{CR_n}{M}e^{-cM^2}
 \le CR_n e^{-cM^2},
\end{split}
\]
where the last inequality uses $M\ge1$.  Moreover,
$\operatorname{supp}\zeta\subset[0,R_n-M/2]$ and
$M/2\ge L_0$, so \eqref{eq:coreCoercivity} applies on the support of
$\zeta$.  Substitution of that inequality into
\eqref{eq:coreCaccioppoliIdentity} and absorption of the first term on its right-hand side prove
\eqref{eq:coreCaccioppoli} when $M\ge2L_0$.

If $1\le M<2L_0$, \red{we} apply the estimate already proved with
$M=2L_0$.  It remains to bound
\[
 \mathcal I_{n,M}
 :=
 \int_{R_n-2L_0}^{R_n-M}
 u\left(
 (\phi_n')^2+\gamma_n^2\phi_n^2+\phi_n^2
 \right)\,du .
\]
Set $y=u-R_n$.  The interval of integration then becomes
\[
        -2L_0\le y\le-M\le-1,
\]
which is contained in the fixed compact interval $[-2L_0,-1]$.
By the qualitative boundary convergence
\eqref{eq:vortexlimit},
\[
 \sup_{-2L_0\le y\le-1}
 \left\{
 |\Phi_n(y)|+|\Phi_n'(y)|+|\Gamma_n(y)|
 \right\}
 \le C_{L_0,\beta}
\]
for all sufficiently large $n$.  Since
\[
 \phi_n(R_n+y)=\Phi_n(y),\qquad
 \phi_n'(R_n+y)=\Phi_n'(y),\qquad
 \gamma_n(R_n+y)=\Gamma_n(y),
\]
we obtain
\[
\begin{split}
 \mathcal I_{n,M}
 &=
 \int_{-2L_0}^{-M}
 (R_n+y)
 \left\{
 |\Phi_n'(y)|^2
 +|\Gamma_n(y)|^2|\Phi_n(y)|^2
 +|\Phi_n(y)|^2
 \right\}\,dy\\
 &\le C_{L_0,\beta}R_n.
\end{split}
\]
Finally, because $1\le M<2L_0$, we have $
        e^{-cM^2}\ge e^{-4cL_0^2}.$ 
\red{Once we enlarge} the constant by $e^{4cL_0^2}$, \red{we infer that}
\[
        \mathcal I_{n,M}
        \le \red{C_{L_0,J}}R_ne^{-cM^2}.
\]
Combining this estimate with the result for $M=2L_0$ proves
\eqref{eq:coreCaccioppoli} throughout the range
$1\le M<2L_0$ and completes the proof.
\end{proof}

We now return to the exact reduced energy and virial identities
\eqref{eq:exactReducedEnergy} and
\eqref{eq:exactReducedVirialDensity}.  Their integrals will be divided at
$y=-M$.  In the core region $-R_n\le y\le-M$, the Higgs field is small,
so both reduced densities have the same leading term, namely the constant
$\beta/2$.  Indeed,
\[
\begin{split}
 h_n(y)-\frac{\beta}{2}
 &=
 \Phi_n'(y)^2
 +\frac{\beta}{2}
 \left((1-\Phi_n(y)^2)^2-1\right),\\
 j_n(y)-\frac{\beta}{2}
 &=
 \Gamma_n(y)^2\Phi_n(y)^2
 +\frac{\beta}{2}
 \left((1-\Phi_n(y)^2)^2-1\right).
\end{split}
\]
Since $0\le\Phi_n\le1$,
\[
 \left|(1-\Phi_n^2)^2-1\right|
 \le 2\Phi_n^2.
\]
\red{This shows that upon applying}  Lemma~\ref{lem:coreCaccioppoli}\red{, we obtain uniform control in $M$ for} each term in these two differences after multipl\red{ying them} by the radial weight $R_n+y$.  The contribution \red{arising from} the common constant
term is explicit:
\[
 \frac{\beta}{2}
 \int_{-R_n}^{-M}(R_n+y)\,dy
 =
 \frac{\beta}{4}(R_n-M)^2.
\]
The next lemma records the resulting expansion of the core contributions
to both exact identities.

\begin{Corollary}\label{cor:bulkExpansion}
The exponent $c>0$ may be chosen independently of $C_{\log}$.  For every
fixed $C_{\log}<\infty$, uniformly for
$1\le M\le C_{\log}\sqrt{\log R_n}$,
\bel{eq:hcoreExpansion}
 \int_{-R_n}^{-M}(R_n+y)h_n(y)\,dy
 =\frac{\beta}{4}(R_n-M)^2+O(1)+O(R_ne^{-cM^2}),
\ee
and
\bel{eq:jcoreExpansion}
 \int_{-R_n}^{-M}(R_n+y)j_n(y)\,dy
 =\frac{\beta}{4}(R_n-M)^2+O(1)+O(R_ne^{-cM^2}).
\ee
\end{Corollary}

\begin{proof}
Set $u=R_n+y$.  Lemma~\ref{lem:coreCaccioppoli} controls the scalar
kinetic term and the magnetic coupling term:
\[
 \int_0^{R_n-M}u(\phi_n')^2\,du
 +\int_0^{R_n-M}u\gamma_n^2\phi_n^2\,du
 \le C+C R_ne^{-cM^2}.
\]
Also,
\[
 \left|(1-\phi_n^2)^2-1\right|
 =|{-2\phi_n^2+\phi_n^4}|
 \le2\phi_n^2.
\]
Consequently
\[
\begin{split}
 \int_0^{R_n-M}\frac{\beta}{2}u(1-\phi_n^2)^2\,du
 &=\frac{\beta}{2}\int_0^{R_n-M}u\,du
 +O\left(\int_0^{R_n-M}u\phi_n^2\,du\right)\\
 &=\frac{\beta}{4}(R_n-M)^2
 +O(1)+O(R_ne^{-cM^2}).
\end{split}
\]
The definitions of $h_n$ and $j_n$ now give
\eqref{eq:hcoreExpansion} and \eqref{eq:jcoreExpansion}.
\end{proof}

\subsection{Boundary and exterior contributions at qualitative accuracy}

The next comparison uses only qualitative convergence on fixed intervals
and the uniform exponential decay already proved on the right.  In
particular, it does not use the quantitative fixed-window or tail estimates
proved in the subsequent sections.

\begin{Lemma}\label{lem:preliminaryBoundaryExterior}
\red{Let $J$ be} a compact subinterval of $(0,4)$.  \red{Uniformly for $\beta \in J$, f}or each fixed $M\ge1$
there are explicitly defined numbers
$r_{n,M}^h,r_{n,M}^j,E_{n,M}^h,E_{n,M}^j$ and numbers
$\varepsilon_M>0$, with $\varepsilon_M\to0$ as $M\to\infty$, such that
\[
\begin{split}
 \int_{-M}^{\infty}(R_n+y)h_n(y)\,dy
 &=R_n\int_{-M}^{\infty}h_\beta(y)\,dy
   +r_{n,M}^h+E_{n,M}^h,\\
 \int_{-M}^{\infty}(R_n+y)j_n(y)\,dy
 &=R_n\int_{-M}^{\infty}j_\beta(y)\,dy
   +r_{n,M}^j+E_{n,M}^j,
\end{split}
\]
where, for each fixed $M$,
\[
 \frac{r_{n,M}^h}{R_n}\longrightarrow0,\qquad
 \frac{r_{n,M}^j}{R_n}\longrightarrow0,
\]
and, for all sufficiently large $n$,
\[
        |E_{n,M}^h|+|E_{n,M}^j|
        \le R_n\varepsilon_M .
\]
\end{Lemma}

\begin{proof}
There are $K_0,c,C>0$ such that
Lemma~\ref{lem:exteriorEntry} and Proposition~\ref{prop:righttail} imply
\bel{eq:preliminaryRightDensityDecay}
 |h_n(y)|+|j_n(y)|+|h_\beta(y)|+|j_\beta(y)|
 \le Ce^{-cy},\qquad y\ge K_0,
\ee
uniformly for all sufficiently large $n$.
For fixed $M$, set
\[
        K_M\red{:}=\max\{K_0,M+1\}.
\]
For $d=h,j$, define
\bel{eq:preliminaryRemainders}
\begin{split}
 r_{n,M}^d
 &:=\int_{-M}^{K_M}
 \left((R_n+y)d_n(y)-R_nd_\beta(y)\right)dy,\\
 E_{n,M}^d
 &:=\int_{K_M}^{\infty}
 \left((R_n+y)d_n(y)-R_nd_\beta(y)\right)dy .
\end{split}
\ee
These definitions give the two asserted decompositions identically.
On the fixed interval $[-M,K_M]$,
Lemma~\ref{lem:qualitativeBoundary} implies
$d_n\to d_\beta$ uniformly.  Hence
\[
 \frac{r_{n,M}^d}{R_n}
 =\int_{-M}^{K_M}(d_n-d_\beta)\,dy
 +\frac{1}{R_n}\int_{-M}^{K_M}y\,d_n(y)\,dy
 \longrightarrow0 .
\]
Using \eqref{eq:preliminaryRightDensityDecay} in the second line of
\eqref{eq:preliminaryRemainders}, we obtain
\[
 |E_{n,M}^h|+|E_{n,M}^j|
 \le C R_n(1+K_M)e^{-cK_M}.
\]
Thus the conclusion holds with
$\varepsilon_M=C(1+K_M)e^{-cK_M}$, which tends to zero.
\end{proof}

\subsection{Conclusion of the first pass}

For later use, define
\bel{eq:reducedRadiusFunction}
 F_n(R)=-\frac{2n^2}{R^3}+\frac{\beta}{2}R+s_\beta,
 \qquad R>0,
\ee
and denote by $\widehat R_n$ its unique positive zero.  Indeed,
\[
 F_n'(R)=\frac{6n^2}{R^4}+\frac{\beta}{2}\ge\frac{\beta}{2}.
\]
Since $F_n(R_n^{(0)})=s_\beta$, the mean-value theorem gives
\bel{eq:reducedRootCoarse}
 |\widehat R_n-R_n^{(0)}|\le \frac{2|s_\beta|}{\beta}.
\ee

\begin{Lemma}\label{lem:preliminaryLocalization}
\red{Uniformly f}or \red{$\beta\in J\Subset (0,4)$} ,
\bel{eq:firstpassSL}
 R_n=R_n^{(0)}+O_J(1),
 \qquad
 b_n=\frac{2n}{R_n^2}
 =\sqrt\beta+O_J(R_n^{-1}).
\ee
In particular, the estimate \eqref{eq:SLhyp}  holds for the radial minimizers considered here.
\end{Lemma}

\begin{proof}
Fix $M\ge1$.  Split the exact identity
\eqref{eq:exactReducedVirialDensity} at $y=-M$.
From Corollary~\ref{cor:bulkExpansion},
\[
 2\int_{-R_n}^{-M}(R_n+y)j_n(y)\,dy
 =\frac{\beta}{2}(R_n-M)^2
 +O(1)+O(R_ne^{-cM^2}).
\]
Lemma~\ref{lem:preliminaryBoundaryExterior} gives
\[
 2\int_{-M}^{\infty}(R_n+y)j_n(y)\,dy
 =2R_n\int_{-M}^{\infty}j_\beta(y)\,dy
 +2r_{n,M}^j+2E_{n,M}^j .
\]
By Lemma~\ref{lem:interfaceUniformTails} and
\eqref{eq:hjSurface},
\[
 2\int_{-M}^{\infty}j_\beta(y)\,dy
 =\beta M+s_\beta+O(e^{-cM^2}).
\]
Substitution in \eqref{eq:exactReducedVirialDensity} cancels the two
terms of size $R_nM$ and gives
\bel{eq:firstpassvirial}
 -\frac{2n^2}{R_n^2}+\frac{\beta}{2}R_n^2+s_\beta R_n
 =-2r_{n,M}^j-2E_{n,M}^j
 +O_M(1)+O(R_ne^{-cM^2}).
\ee
Divide by $R_n$.  Coarse localization,
Proposition~\ref{prop:coarseloc}, gives $R_n\to\infty$.  Hence, for each
fixed $M$,
\[
 \limsup_{n\to\infty}|F_n(R_n)|
 \le2\varepsilon_M+Ce^{-cM^2}.
\]
Letting $M\to\infty$ yields
\bel{eq:preliminaryReducedEquation}
        F_n(R_n)\longrightarrow0.
\ee
Since $F_n'\ge\beta/2$,
\[
 |R_n-\widehat R_n|
 \le \frac{2}{\beta}|F_n(R_n)|
 \longrightarrow0.
\]
Together with \eqref{eq:reducedRootCoarse}, this proves the first estimate
in \eqref{eq:firstpassSL}.  Finally,
\[
 \left|\frac{2n}{R_n^2}-\frac{2n}{(R_n^{(0)})^2}\right|
 =\frac{2n|R_n-R_n^{(0)}|(R_n+R_n^{(0)})}{R_n^2(R_n^{(0)})^2}
 \le \frac{C(\beta)}{R_n},
\]
and $2n/(R_n^{(0)})^2=\sqrt\beta$.  This proves the second estimate.
\end{proof}

We have therefore established \eqref{eq:SLhyp}.  It will now be used to compare the finite radial
Cauchy data with the corresponding limiting data at order $R_n^{-1}$.

\section{Quantitative convergence on fixed intervals}
\label{sec:quantitativeBoundary}

Throughout this section $\beta$ ranges in a fixed compact interval
$J_0\Subset(0,4)$.  Constants written with a $\beta$ subscript may be
chosen uniformly on $J_0$.  All thresholds on $n$ have the same
uniformity. 
Lemma~\ref{lem:qualitativeBoundary} proves convergence of the translated
profiles on every fixed interval, but does not give a rate.  The preliminary
localization in Lemma~\ref{lem:preliminaryLocalization} has now established
\eqref{eq:SLhyp}.  We use it here to prove \red{that} for every fixed $M$,
\[
 \|(\Phi_n,\Gamma_n)-(\Phi_\beta,\Gamma_\beta)\|_{H^2([-M,M])}
 \le C_{M,\beta}R_n^{-1}.
\]
The proof has four steps.  First, we write the equation for the difference
of the two profiles on $[-K,K]$.  Second, Proposition~\ref{prop:nearBPS}
is converted into a linear estimate on that interval.  Third, we compare
the Cauchy data selected by regularity at the origin with the limiting line
$E_\beta^-(-K)$, and the Cauchy data selected by decay at infinity with the
limiting plane $E_\beta^+(K)$.  Finally, the linear estimate and these two
endpoint comparisons give the $O(R_n^{-1})$ rate after a quadratic term is
absorbed.  The long part of the section is the third step on the left \red{overlap}: it
transports the regularity (and thus, vanishing) condition at the origin through the normal core to $y=-K$.

\subsection{The difference equation and the finite-interval estimate}

Define the nonlinear interface operator
\bel{eq:nonlinearInterfaceMap}
 \calF_\beta(\Phi,\Gamma)=
 \begin{pmatrix}
 \Phi''-\Gamma^2\Phi-\beta\Phi(\Phi^2-1)\\
 \Gamma''-2\Gamma\Phi^2
 \end{pmatrix}.
\ee
With the operator and sign convention of \eqref{eq:Alinearization}, 
\[
        D\calF_\beta(\Phi_\beta,\Gamma_\beta)=\calA_\beta .
\]
We first record the equation for the difference on a fixed interval.

\begin{Lemma}\label{lem:boundaryResidual}
Let
\[
        V_n=(\xi_n,\psi_n)
        :=(\Phi_n-\Phi_\beta,\Gamma_n-\Gamma_\beta).
\]
On every fixed interval $[-L,L]$,
\bel{eq:boundaryDifferenceEquation}
        \calA_\beta V_n=\calR_n+\calQ_\beta^{\rm int}(V_n),
        \qquad
        \norm{\calR_n}_{H^1([-L,L])}
        \Le C_{L,\beta}R_n^{-1},
\ee
where
\bel{eq:boundaryResidualExact}
 \calR_n(y)=
 \begin{pmatrix}
 -(R_n+y)^{-1}\Phi_n'(y)\\
 -(R_n+y)^{-1}\Gamma_n'(y)
 +(R_n+y)^{-2}\Gamma_n(y)
 \end{pmatrix}
\ee
and, for $V=(\xi,\psi)$,
\bel{eq:boundaryNonlinearRemainder}
 \calQ_\beta^{\rm int}(V)=
 \begin{pmatrix}
 2\Gamma_\beta\xi\psi+\Phi_\beta\psi^2+\xi\psi^2
       +3\beta\Phi_\beta\xi^2+\beta\xi^3\\
 2\Gamma_\beta\xi^2+4\Phi_\beta\xi\psi+2\psi\xi^2
 \end{pmatrix}.
\ee
If $\norm{V}_{H^2([-L,L])},\norm{W}_{H^2([-L,L])}\le1$, then
\bel{eq:boundaryNonlinearLipschitz}
\begin{split}
 \norm{\calQ_\beta^{\rm int}(V)
       -\calQ_\beta^{\rm int}(W)}_{L^2([-L,L])}
 &\Le C_{L,\beta}
 \bigl(\norm{V}_{H^2([-L,L])}
       +\norm{W}_{H^2([-L,L])}\bigr)\\
 &\hspace{22mm}\times
 \norm{V-W}_{\red{L}^2([-L,L])}.
\end{split}
\ee
\end{Lemma}

\begin{proof}
The translated radial equations \eqref{eq:exactPhi}–\eqref{eq:exactGamma} give
\[
        \calF_\beta(\Phi_n,\Gamma_n)=\calR_n.
\]
Subtracting the interface equations~\eqref{eq:interface}, viz.\  $\calF_\beta(\Phi_\beta,\Gamma_\beta)=0$ and expanding the
polynomials gives
\[
 \calF_\beta(\Phi_\beta+\xi,\Gamma_\beta+\psi)
 =\calA_\beta(\xi,\psi)-\calQ_\beta^{\rm int}(\xi,\psi),
\]
which proves \eqref{eq:boundaryDifferenceEquation} and
\eqref{eq:boundaryNonlinearRemainder}.  On $[-L,L]$ the profiles and their
first three derivatives are bounded uniformly for large $n$, by
Lemma~\ref{lem:qualitativeBoundary} and the differential equations.
Consequently \eqref{eq:boundaryResidualExact} is
$O_{L,\beta}(R_n^{-1})$ in $H^1$.  Finally,
$H^2([-L,L])$ is an algebra and embeds continuously in $C^1([-L,L])$.
Applying these facts to the nonlinearities proves
\eqref{eq:boundaryNonlinearLipschitz}.
\end{proof}

For $V=(\xi,\psi)\in H^2([-K,K];\RR^2)$, write
\bel{eq:fixedWindowTraces}
        \mathsf C_\pm V
        =(\xi,\psi,\xi',\psi')(\pm K)\in\RR^4.
\ee
The spaces $E_\beta^-(-K)$ and $E_\beta^+(K)$ have dimensions one and two,
respectively.  Choose orthonormal coordinates on their Euclidean orthogonal
complements, and let
\[
 \mathcal R_{\beta,K}^-:\RR^4\longrightarrow\RR^3,
 \qquad
 \mathcal R_{\beta,K}^+:\RR^4\longrightarrow\RR^2
\]
be the resulting orthogonal projections.  Thus
\bel{eq:fixedWindowProjectionKernels}
 \ker\mathcal R_{\beta,K}^-=E_\beta^-(-K),
 \qquad
 \ker\mathcal R_{\beta,K}^+=E_\beta^+(K).
\ee
With this choice,
\[
 |\mathcal R_{\beta,K}^-z|
 =\operatorname{dist}(z,E_\beta^-(-K)),
 \qquad
 |\mathcal R_{\beta,K}^+z|
 =\operatorname{dist}(z,E_\beta^+(K)).
\]
The two terms in the following estimate therefore measure exactly how far
the endpoint data are from the Cauchy data allowed by the two half-line
conditions.

\begin{Lemma}\label{lem:finiteWindowLinearEstimate}
Let $\beta\in J\Subset(0,4)$, with $J$ compact, and take $K$ sufficiently large that
Lemmas~\ref{lem:leftline} and~\ref{lem:rightstable} apply.  Then
\bel{eq:finiteWindowOverdet}
\begin{split}
 \|V\|_{H^2([-K,K])}
 \le C_{K,\beta}\bigl(
 &\|\calA_\beta V\|_{L^2([-K,K])}\\
 &+|\mathcal R_{\beta,K}^-\mathsf C_-V|
 +|\mathcal R_{\beta,K}^+\mathsf C_+V|
 \bigr)
\end{split}
\ee
for every $V\in H^2([-K,K];\RR^2)$.  For fixed $J$ and $K$,
the constant can be chosen independently of $\beta\in J$.
\end{Lemma}

\begin{proof}
For every $W\in H^2([-K,K];\RR^2)$, the differential expression
\eqref{eq:Alinearization} gives
\bel{eq:fixedWindowElementaryODE}
 \|W\|_{H^2([-K,K])}
 \le C_{K,\beta}
 (\|\calA_\beta W\|_{L^2([-K,K])}
   +\|W\|_{H^1([-K,K])}).
\ee
Indeed, the two second derivatives are the components of
$\calA_\beta W$ plus bounded zeroth-order coefficients times the
components of $W$.
Suppose that \eqref{eq:finiteWindowOverdet} is false.  There would then
be a sequence $V_j$ with
\[
 \|V_j\|_{H^2([-K,K])}=1
\]
for which all three terms on the right side of
\eqref{eq:finiteWindowOverdet} tend to zero.  After passing to a
subsequence, compactness of the embedding
$H^2([-K,K])\hookrightarrow C^1([-K,K])$ gives
\[
 V_j\rightharpoonup V\quad\hbox{in }H^2([-K,K]),
 \qquad
 V_j\longrightarrow V\quad\hbox{in }C^1([-K,K]).
\]
It follows that
\[
 \calA_\beta V=0,
 \qquad
 \mathsf C_-V\in E_\beta^-(-K),
 \qquad
 \mathsf C_+V\in E_\beta^+(K).
\]
The condition at $y=-K$ extends $V$ to the normalized recessive solution
on the left half-line.  The condition at $y=K$ extends it to a decaying
solution on the right half-line.  Uniqueness for the initial-value problem
makes these extensions agree with $V$ at the two endpoints.  They therefore
form a global solution in the endpoint class of
Proposition~\ref{prop:nearBPS}, which gives $V=0$.

Consequently $V_j\to0$ in $H^1([-K,K])$.  Applying
\eqref{eq:fixedWindowElementaryODE} to $V_j$ and using
$\calA_\beta V_j\to0$ in $L^2$ now gives $V_j\to0$ in $H^2$, contradicting 
its normalization.  This proves \eqref{eq:finiteWindowOverdet}.
To obtain a uniform constant, suppose instead that the estimate
fails along \green{a sequence} $\beta_j\in J$.  \green{We now pass} to a subsequence with
$\beta_j\to\beta_*\in J$ and \green{
\[
 V_j\rightharpoonup V\quad\hbox{in }H^2([-K,K]),
 \qquad
 V_j\longrightarrow V\quad\hbox{in }C^1([-K,K]).
\]
} 
The coefficient continuity in Lemma~\ref{lem:interfaceUniformTails}
implies
\[
 \|(\calA_{\beta_j}-\calA_{\beta_*})V_j\|_{L^2([-K,K])}\longrightarrow0.
\]
The endpoint subspaces also depend continuously on $\beta$, so the
limiting traces belong to $E_{\beta_*}^-(-K)$ and $E_{\beta_*}^+(K)$.
\green{By repeating the previous} argument at $\beta_*$, \green{we reach} the same contradiction.
\green{We also note that t}he constants in \eqref{eq:fixedWindowElementaryODE} are uniform because
the coefficients are bounded on \green{the compact set} $J\times[-K,K]$.
\end{proof}

Let
\[
        V_n=(\Phi_n-\Phi_\beta,\Gamma_n-\Gamma_\beta).
\]
Lemma~\ref{lem:boundaryResidual} and
Lemma~\ref{lem:finiteWindowLinearEstimate} reduce the proof of the
quantitative estimate to controlling the two projected endpoint traces
\[
 \mathcal R_{\beta,K}^-\mathsf C_-V_n,
 \qquad
 \mathcal R_{\beta,K}^+\mathsf C_+V_n.
\]
The right endpoint comparison follows from the construction in
Section~\ref{sec:exterior}.  We next construct the corresponding finite
family at the left endpoint.

\subsection{The Cauchy data selected by regularity at the origin}\label{subsec:CauchyData}

At $y=-K$, use the variable
\[
        \eta=-y,\qquad u_n(\eta)=R_n-\eta.
\]
Thus $y\le-K$ corresponds to $\eta\ge K$.  We shall also use
\[
        \Lambda_n=C_{\log}\sqrt{\log R_n},
        \qquad
        \Lambda_n^+=2\Lambda_n.
\]
For fixed $K$ and $C_{\log}$, all sufficiently large $n$ satisfy
\[
        K<\Lambda_n<\Lambda_n^+<R_n-1.
\]
The larger point $\Lambda_n^+$ is used only in the comparison of
logarithmic derivatives.  The nonlinear problem which constructs the
left Cauchy data ends at $\Lambda_n$.

The construction has four stages.  First, we compare two distinguished
solutions: (i) the solution of the Higgs  equation in the core selected by
regularity at the vortex origin, and (ii) the recessive (decaying) solution of
the limiting Weber equation.  Second, we construct and compare the Green
operators for the finite and limiting equations on the logarithmic
interval.  Third, a contraction argument produces a one-parameter family
of nonlinear left profiles, parameterized by the scalar value at $\eta=K$. The 
tangent to its limiting Cauchy-data curve is $E_\beta^-(-K)$.  Finally,
we use regularity at $u=0$ and the decay estimates in the core to show that the
physical radial profile belongs to this family up to an error
$O_{K,\beta}(R_n^{-2})$ at $\eta=K$.

The vortex origin is $u=0$, or equivalently $\eta=R_n$, whereas the
left endpoint of the fixed interval is $\eta=K$.  Regularity at the origin
must therefore be transported through the core to $\eta=K$.

The exact gauge profile obtained by setting the scalar \red{Higgs} field equal to zero
and retaining the magnetic-radius normalization is
\bel{eq:emptyCoreGamma}
 \Gamma_n^0(\eta)
 =\frac{n}{u_n(\eta)}-\frac{n u_n(\eta)}{R_n^2}
 =b_n\eta+\frac{b_n\eta^2}{2(R_n-\eta)},
\ee
using that $b_n=2nR_n^{-2}$. 
It satisfies
\[
 (\Gamma_n^0)''
 -\frac{1}{u_n}(\Gamma_n^0)'
 -\frac{1}{u_n^2}\Gamma_n^0=0.
\]
Its difference from the affine limiting profile is
\bel{eq:exactAffineDefect}
 P_n(\eta):=\Gamma_n^0(\eta)-\sqrt\beta\,\eta
 =(b_n-\sqrt\beta)\eta
   +\frac{b_n\eta^2}{2(R_n-\eta)}.
\ee
Put
\bel{eq:emptyCoreScalarPotential}
 V_n^0(\eta)
 :=(\Gamma_n^0(\eta))^2-\beta-\frac{1}{4u_n(\eta)^2}.
\ee
Let $D_n^0$ be the solution of
\bel{eq:emptyCoreScalar}
        (D_n^0)''=V_n^0D_n^0
\ee
which has positive Frobenius coefficient in
\[
        D_n^0(\eta)\sim u_n(\eta)^{n+1/2}
        \quad\hbox{as }\eta\uparrow R_n,
\]
and normalize it by
\bel{eq:regularOriginNormalization}
        D_n^0(K)=1.
\ee
For the limiting equation, let $d_\beta$ be the positive solution of
\[
        d_\beta''(\eta)=(\beta\eta^2-\beta)d_\beta(\eta)
\]
which decays as $\eta\to\infty$, and normalize it by $d_\beta(K)=1$.

\subsubsection{The regular scalar comparison solution}

The next lemma proves, before any Green function is introduced, the
positivity and decay estimates needed below.

\begin{Lemma}\label{lem:regularOriginLogDerivative}
  Fix $C_{\log}\ge C_{\log,0}$, where
$C_{\log,0}$ is a sufficiently large constant, and take $K$ sufficiently
large.  Then $D_n^0$ is positive on $[K,R_n)$ for all sufficiently large
$n$.  There are constants $c,C>0$, independent of $n$, such that
\bel{eq:regularOriginCoefficientBounds}
\begin{split}
 c(1+\eta)^2&\le V_n^0(\eta)\le C(1+\eta)^2,\\
 |\partial_\eta^jV_n^0(\eta)|
 &\le C_j(1+\eta)^{2-j},\qquad j=1,2,
\end{split}
\ee
on $K\le\eta\le \Lambda_n^+$.
Define
\[
 \lambda_n=-\frac{(D_n^0)'}{D_n^0},
 \qquad
 \pi_n=(V_n^0)^{1/2},
 \qquad
 \lambda_\beta=-\frac{d_\beta'}{d_\beta}.
\]
Then
\bel{eq:regularOriginLogDerivativeBounds}
 \lambda_n(\eta)\ge c(1+\eta),\qquad
 |\lambda_n(\eta)-\pi_n(\eta)|
 \le \frac{C\pi_n(\eta)}{1+\eta},
\ee
and
\bel{eq:regularOriginDecayRatio}
 \frac{D_n^0(s)}{D_n^0(\eta)}
 \le C e^{-c(s^2-\eta^2)},
 \qquad K\le\eta\le s<R_n.
\ee
Also,
\bel{eq:leftChartCoefficientComparison}
 \left|V_n^0(\eta)-(\beta\eta^2-\beta)\right|
 +\frac{1}{1+\eta}
 \left|\partial_\eta(V_n^0(\eta)-(\beta\eta^2-\beta))\right|
 \le \frac{C(1+\eta)^3}{R_n}
\ee
on $[K,\Lambda_n^+]$, and
\bel{eq:regularOriginSlopeComparison}
 |\lambda_n(\eta)-\lambda_\beta(\eta)|
 \le \frac{C(1+\eta)^2}{R_n}
 +C(1+\Lambda_n^+)^C e^{-c((\Lambda_n^+)^2-\eta^2)}.
\ee
Consequently, for some integer $p_1$,
\bel{eq:regularOriginWeberComparison}
\begin{split}
 &|D_n^0(\eta)-d_\beta(\eta)|
 +(1+\eta)^{-1}|(D_n^0)'(\eta)-d_\beta'(\eta)|\\
 &\qquad\le
 C(K,\beta,C_{\log})R_n^{-1}
 (1+\eta)^{p_1}e^{-c(\eta^2-K^2)},
 \qquad K\le\eta\le \Lambda_n.
\end{split}
\ee
\end{Lemma}

\begin{proof}
Write the equation in the variable $u=R_n-\eta$.  \blue{The point $u=0$ is a regular singular point, and the condition
$D_n^0(\eta)\sim u^{n+\frac12}$ as $u\downarrow0$ selects the unique (up to a scalar multiple) solution which is regular at the origin.
In particular, this recessive solution vanishes at $u=0$ and has $\partial_uD_n^0>0$ for $u$ small.  This endpoint behavior is what later allows us
to integrate the Riccati equation from $u=0$ (equivalently $\eta=R_n$) without an uncontrolled boundary term. In fact, for}
 $K$ large,
\eqref{eq:emptyCoreGamma} gives $V_n^0>0$ on $[K,R_n)$.  As long as
$D_n^0$ is positive, its $u$-derivative is increasing.  Hence neither
$\partial_uD_n^0$ nor $D_n^0$ can have a first zero.  This proves
positivity.

The explicit formula \eqref{eq:emptyCoreGamma}, the relation
$n/R_n^2=b_n/2$, and \eqref{eq:SLhyp} give
\eqref{eq:regularOriginCoefficientBounds} and
\eqref{eq:leftChartCoefficientComparison} by direct differentiation.
Since $D_n^0>0$, its negative logarithmic derivative satisfies \blue{the Riccati equation}
\[
        \lambda_n'=\lambda_n^2-V_n^0.
\]
Put
$\rho_n=\lambda_n-\pi_n$.  Then
\bel{eq:logDerivativeErrorEquation}
        \rho_n'
        =2\pi_n\rho_n+\rho_n^2-\pi_n'.
\ee
For
\[
 \mathcal K_n(\eta,s)
 =\exp\left(-2\int_\eta^s\pi_n(t)\,dt\right),
\]
integration from $R_n-\varepsilon$, followed by
$\varepsilon\downarrow0$, gives
\bel{eq:logDerivativeIntegralEquation}
 \rho_n(\eta)=
 \int_\eta^{R_n}
 \mathcal K_n(\eta,s)(\pi_n'(s)-\rho_n(s)^2)\,ds .
\ee
We \red{briefly} justify the endpoint limit used in
this formula.  The Frobenius expansion at $u=R_n-\eta=0$ gives
\bel{eq:regularOriginEndpointAsymptotics}
\begin{split}
 \lambda_n(R_n-u)
 &= \frac{n+\frac12}{u}+O(u),\\
 \pi_n(R_n-u)
 &= \frac{\sqrt{n^2-\frac14}}{u}+O(u).
\end{split}
\ee
Consequently
\[
 \mathcal K_n(\eta,R_n-\varepsilon)
 \rho_n(R_n-\varepsilon)\longrightarrow0
 \qquad(\varepsilon\downarrow0),
\]
which proves \eqref{eq:logDerivativeIntegralEquation}.

\red{We are left with bounding the kernel. We begin with the following} estimates that are uniform across the two possible
locations of $\eta$.  On $u_n(s)\ge R_n/2$,
\[
 \pi_n(s)\simeq 1+s,\qquad 0<\pi_n'(s)\le C,
\]
whereas on $u_n(s)\le R_n/2$,
\[
 \pi_n(s)\simeq  \frac{n}{u_n(s)},\qquad
 0<\pi_n'(s)\le \frac{Cn}{u_n(s)^2}.
\]
The bounds agree at $u_n=R_n/2$.  In particular,
\bel{eq:piDerivativeRatio}
 0\le \frac{\pi_n'(s)}{\pi_n(s)^2}
 \le \frac{C}{(1+s)^2},
 \qquad K\le s<R_n.
\ee
The positivity follows from
\[
 (V_n^0)'=2\Gamma_n^0(\Gamma_n^0)'
          -\frac{1}{2u_n^3}>0
\]
after increasing $K$.
\red{We now set}
\[
 I_1(\eta)=\int_\eta^{R_n}
       \mathcal K_n(\eta,s)\pi_n'(s)\,ds,\qquad
 I_2(\eta)=\int_\eta^{R_n}
       \mathcal K_n(\eta,s)\frac{\pi_n(s)^2}{(1+s)^2}\,ds .
\]
Since
$\partial_s\mathcal K_n=-2\pi_n\mathcal K_n$,
\eqref{eq:regularOriginEndpointAsymptotics} also gives
$\mathcal K_n(\eta,s)\pi_n(s)\to0$ as $s\uparrow R_n$.
Integration by parts therefore yields
\[
 \int_\eta^{R_n}\mathcal K_n(\eta,s)\pi_n(s)^2\,ds
 =\frac{1}{2}\pi_n(\eta)+\frac{1}{2}I_1(\eta).
\]
Using \eqref{eq:piDerivativeRatio} implies, for $K$
large enough, that
\[
\begin{split}
 I_1(\eta)
 &\le \frac{C}{(1+\eta)^2}
 \left(\frac{1}{2}\pi_n(\eta)+\frac{1}{2}I_1(\eta)\right),\\
 I_2(\eta)
 &\le \frac{1}{(1+\eta)^2}
 \left(\frac{1}{2}\pi_n(\eta)+\frac{1}{2}I_1(\eta)\right).
\end{split}
\]
After absorbing the occurrence of $I_1$ on the right, we obtain
\bel{eq:logDerivativeKernelBounds}
\begin{split}
 I_1(\eta)&\le \frac{C\pi_n(\eta)}{(1+\eta)^2},\\
 I_2(\eta)&\le \frac{C\pi_n(\eta)}{(1+\eta)^2},
 \qquad K\le\eta<R_n.
\end{split}
\ee
\red{These immediately imply \eqref{eq:regularOriginLogDerivativeBounds}.}

We now verify that these estimates apply to the logarithmic derivative
of the regular solution itself.  By
\eqref{eq:regularOriginEndpointAsymptotics},
\[
 \rho_n(R_n-u)
 =\frac{n+\frac12-\sqrt{n^2-\frac14}}{u}+O(u)>0
\]
for all sufficiently small $u>0$.  At any zero of $\rho_n$, 
\eqref{eq:logDerivativeErrorEquation} and $\pi_n'>0$ give
\[
        \rho_n'=-\pi_n'<0.
\]
If $\rho_n$ were negative at some point, then, moving from that point
toward $R_n$, it would have a first zero at which it crosses from
negative to positive, and its derivative there would be nonnegative.  This
contradicts the preceding inequality.  Hence $\rho_n\ge0$ on
$[K,R_n)$.  The integral equation now yields
\[
\begin{split}
 0\le\rho_n(\eta)
 &=
 I_1(\eta)
 -\int_\eta^{R_n}
   \mathcal K_n(\eta,s)\rho_n(s)^2\,ds\\
 &\le I_1(\eta)
 \le \frac{C\pi_n(\eta)}{(1+\eta)^2}.
\end{split}
\]
This proves \eqref{eq:regularOriginLogDerivativeBounds}, with a slightly
stronger error estimate than stated.  Integrating
$$-(\log D_n^0)'=\lambda_n$$ gives
\eqref{eq:regularOriginDecayRatio}, while on $0<u_n\le1$ the same estimate
follows directly from the Frobenius power. 
It remains to compare $D_n^0$ with $d_\beta$.  Set
\[
        \Delta_n=\lambda_n-\lambda_\beta,
        \qquad
        E_n=V_n^0-(\beta\eta^2-\beta).
\]
Subtracting the two logarithmic-derivative equations and integrating
backward from $\Lambda_n^+$ gives the exact identity
\bel{eq:exactRiccatiComparison}
\begin{split}
 \Delta_n(\eta)
 &=
 e^{-\int_\eta^{\Lambda_n^+}(\lambda_n+\lambda_\beta)\,dt}
       \Delta_n(\Lambda_n^+)\\
 &\quad+
 \int_\eta^{\Lambda_n^+}
 e^{-\int_\eta^s(\lambda_n+\lambda_\beta)\,dt}
       E_n(s)\,ds .
\end{split}
\ee
The estimates already proved, together with the corresponding elementary
comparison for $d_\beta$, imply
\[
 \lambda_n(s)+\lambda_\beta(s)\ge c(1+s),
 \qquad
 |\Delta_n(\Lambda_n^+)|\le C(1+\Lambda_n^+)^C.
\]
Combining these bounds with
\eqref{eq:leftChartCoefficientComparison} proves
\eqref{eq:regularOriginSlopeComparison}.  Integrating the difference of
the logarithmic derivatives from $K$ to $\eta$, and then differentiating
once, proves \eqref{eq:regularOriginWeberComparison}.
\red{To see that, we note that the normalization conditions $\displaystyle D_n^0(K)=d_\beta(K)=1$ imply that
\begin{align*}
   \frac{D_n^0(\eta)}{d_\beta(\eta)}&=e^{-\int_K^\eta(\lambda_n(s)-\lambda_\beta(s))\,ds} .
\end{align*}
When $\eta\in[K,\Lambda_n]$, upon absorbing the Gaussian term into the first one on the right hand side, \eqref{eq:regularOriginSlopeComparison} implies that
\begin{align*}
    \left|\int_K^\eta (\lambda_n(s)-\lambda_\beta(s))\,ds\right|&\leq \frac{C(1+\eta)^3}{R_n},
\end{align*}
which can be made uniformly small for $n$ sufficiently large.

Now, this estimate along with the inequality $\displaystyle e^{-x}-1\leq C|x|$ for $x$ with $|x|$ small and the asymptotics for the Weber recessive solution $d_\beta$ imply the desired estimate for $D_n^0(\eta)-d_\beta(\eta)$.

For the derivative bound, we write 
\begin{align*}
\partial_\eta(D_n^0(\eta)-d_\beta(\eta))=-\lambda_n(\eta)(D_n^0(\eta)-d_\beta(\eta))-(\lambda_n(\eta)-\lambda_\beta(\eta))d_\beta(\eta).
\end{align*}
The desired derivative estimate follows from \eqref{eq:regularOriginLogDerivativeBounds} and \eqref{eq:regularOriginSlopeComparison}.
}

\red{Once again, we point out that we make} the comparison
on $[K,\Lambda_n]$, while the uncontrolled endpoint term originates at
$\Lambda_n^+=2\Lambda_n$, \red{and this allows us to bound its contribution as}
\[
 C(1+\Lambda_n^+)^C e^{-c((\Lambda_n^+)^2-\Lambda_n^2)}
 \le Ce^{-c\Lambda_n^2}.
\]
This term is absorbed into the displayed Gaussian weight after decreasing
its exponent.
\end{proof}

\subsubsection{Linear solution operators on the logarithmic interval}

We now construct the Green operators used for the left families.  Let
$J\Subset(0,4)$ be a fixed compact interval.  Fix the Gaussian exponent $c>0$ so that it is smaller
than the exponent in Lemma~\ref{lem:regularOriginLogDerivative} and
\[
        0<c<\frac{1}{2}\inf_{\beta\in J}\sqrt\beta,
\]
and set
\[
        w_{K,p}(\eta)
        =(1+\eta)^p e^{-c(\eta^2-K^2)}.
\]
For $I=[K,T]$, define
\bel{eq:leftWeightedSpaces}
\begin{split}
 \|f\|_{X_p(I)}
 &:=%
 \sum_{j=0}^1\sup_{\eta\in I}
 \frac{(1+\eta)^{-j}|\partial_\eta^jf(\eta)|}{w_{K,p}(\eta)},\\
 \|F\|_{Y_p(I)}
 &:=%
 \sup_{\eta\in I}
 \frac{|F(\eta)|}{(1+\eta)w_{K,p}(\eta)}.
\end{split}
\ee
For pairs, we use the sum of the two component norms.  The same notation
on $[K,\infty)$ has the evident meaning.
\red{The purpose of our next result is to study some properties of the relevant Green functions for our finite-$n$ equations and for the limiting Weber equations. We shall need the following
\begin{Definition}
 Let
\[
 d_n=D_n^0,\qquad
 \ell_n(\eta)=d_n(\eta)\int_K^\eta d_n(s)^{-2}\,ds,
\]
and define
 \bel{eq:leftScalarGreenKernel}
 G_{S,n}(\eta,s)=
 \begin{cases}
  -\ell_n(\eta)d_n(s),&\eta\le s,\\
  -d_n(\eta)\ell_n(s),&s\le\eta .
 \end{cases}
\ee   
Let $\mathcal G_{G,n}$ solve
\[
 H''-\frac{1}{u_n}H'-\frac{1}{u_n^2}H=F,
 \qquad H(T)=H'(T)=0.
\]
and $\mathcal G_{\beta,K}$ be the Green operator for
\[
 v''-(\beta\eta^2-\beta)v=F
\]
with $v(K)=0$ and decay at infinity. Then define 
\[
 (\mathcal G_{G,\infty}F)(\eta)
 =\int_\eta^\infty(s-\eta)F(s)\,ds.
\]
If $R_T$ denotes restriction to $[K,T]$ and $E_T$ extension by zero to
$[K,\infty)$, set
\[
 \widehat{\mathcal G}_{S,n}F
 :=u_n^{-1/2}\mathcal G_{S,n}(u_n^{1/2}F).
\]
\end{Definition}
}
We now identify the Green kernel associated with the mixed boundary conditions at $K$ and $T$. In addition, we record the Gaussian off-diagonal bounds for this kernel that will be used repeatedly in the remainder of our analysis of the left-hand tail of the interface.

\begin{Lemma}\label{lem:regularOriginMixedGreen}
Take $K$ and
$C_{\log}\ge C_{\log,0}$ sufficiently large, and let
$K<T\le \Lambda_n=C_{\log}\sqrt{\log R_n}$. 
Then $W(\ell_n,d_n)=-1$.  The Green function for
\[
 v''-V_n^0v=F,\qquad
 v(K)=0,\qquad W(v,\red{d_n})(T)=0
\]
is \red{$G_{S,n}$ defined in \eqref{eq:leftScalarGreenKernel}}.
\blue{Fix any real number $p_0\ge0$.  Then there are constants $c,C>0$ (depending on $p_0$ but independent of $n$) such that}
\bel{eq:leftScalarGreenKernelBound}
\begin{split}
 |G_{S,n}(\eta,s)|
 &\le \frac{C}{1+\min\{\eta,s\}}\,
       e^{-2c|\eta^2-s^2|},\\
 (1+\eta)^{-1}|\partial_\eta G_{S,n}(\eta,s)|
 &\le \frac{C}{1+\min\{\eta,s\}}\,
       e^{-2c|\eta^2-s^2|},
\end{split}
\ee
and the associated operator satisfies
\bel{eq:leftChartScalarGreen}
        \|\mathcal G_{S,n}F\|_{X_{p_0}([K,T])}
        \le \frac{C}{K}\|F\|_{Y_{p_0}([K,T])}.
\ee
\red{For $G_{G,n}$ defined above, we have}
\bel{eq:leftGaugeGreenKernel}
 (\mathcal G_{G,n}F)(\eta)
 =\int_\eta^T K_{G,n}(\eta,s)F(s)\,ds,
 \qquad
 K_{G,n}(\eta,s)
 =(s-\eta)-\frac{(s-\eta)^2}{2u_n(\eta)},
\ee
and
\bel{eq:leftChartGaugeGreen}
        \|\mathcal G_{G,n}F\|_{X_{p_0}([K,T])}
        \le \frac{C}{K}\|F\|_{Y_{p_0}([K,T])}.
\ee
\red{If the Green functions $\widehat{\mathcal G}_{S,n}$, $\mathcal G_{\beta,K}$, and $\mathcal G_{G,\infty}$, along with the operators $R_T$ and $E_T$, are as defined above, then for some $p_*\ge p_0+2$,}
\bel{eq:leftGreenOperatorComparison}
\begin{split}
 \|\widehat{\mathcal G}_{S,n}
       -R_T\mathcal G_{\beta,K}E_T\|_{Y_{p_0}\to X_{p_*}}
 &\le \frac{C(K,\beta,C_{\log})}{R_n},\\
 \|\mathcal G_{G,n}
       -R_T\mathcal G_{G,\infty}E_T\|_{Y_{p_0}\to X_{p_*}}
 &\le \frac{C(K,\beta,C_{\log})}{R_n}.
\end{split}
\ee
For each fixed $p\in\{p_0,p_*\}$, the finite and limiting scalar and
magnetic Green operators also satisfy
\bel{eq:leftGreenSameWeight}
\begin{split}
 \|\mathcal G_{S,n}F\|_{X_p([K,T])}
 +\|\widehat{\mathcal G}_{S,n}F\|_{X_p([K,T])}
 +\|R_T\mathcal G_{\beta,K}E_TF\|_{X_p([K,T])}
 &\le \frac{C_p}{K}\|F\|_{Y_p([K,T])},\\
 \|\mathcal G_{G,n}F\|_{X_p([K,T])}
 +\|R_T\mathcal G_{G,\infty}E_TF\|_{X_p([K,T])}
 &\le \frac{C_p}{K}\|F\|_{Y_p([K,T])}.
\end{split}
\ee
\end{Lemma}

\begin{proof}
Lemma~\ref{lem:regularOriginLogDerivative} gives
\[
 \frac{d_n(s)}{d_n(\eta)}
 \le Ce^{-c(s^2-\eta^2)},\qquad \eta\le s.
\]
\blue{We first note that $\ell_n$ satisfies $\ell_n(K)=0$ and $\ell_n'(K)=1$, and a direct differentiation of
$\ell_n(\eta)=d_n(\eta)\int_K^\eta d_n(s)^{-2}\,ds$ gives $W(\ell_n,d_n)=-1$.  Consequently, for each fixed $s$,
$G_{S,n}(\cdot,s)$ defined in \eqref{eq:leftScalarGreenKernel} solves
$v''-V_n^0v=0$ on $(K,s)\cup(s,T)$, satisfies $v(K)=0$ and $W(v,d_n)(T)=0$, and has the jump condition
$\partial_\eta G_{S,n}(s^+,s)-\partial_\eta G_{S,n}(s^-,s)=1$.  Thus $G_{S,n}$ is the Green kernel for the stated boundary-value problem.}

\blue{Using the ratio estimate from Lemma~\ref{lem:regularOriginLogDerivative}, we bound $\ell_n$ directly. Indeed, since}
\[
 d_n(\eta)\ell_n(\eta)=d_n(\eta)^2\int_K^\eta d_n(s)^{-2}\,ds
 =\int_K^\eta\left(\frac{d_n(\eta)}{d_n(s)}\right)^{\!2}\,ds,
\]
\blue{and for $K\le s\le\eta$ we have $(d_n(\eta)/d_n(s))\le Ce^{-c(\eta^2-s^2)}$, it follows that}
\[
 d_n(\eta)\ell_n(\eta)
 \le C\int_K^\eta e^{-2c(\eta^2-s^2)}\,ds
 \le C\int_0^\infty e^{-2c\eta\tau}\,d\tau
 \le \frac{C}{1+\eta}.
\]
In the second inequality we use
$\eta^2-s^2=(\eta+s)(\eta-s)\ge\eta(\eta-s)$ for $K\le s\le\eta$
and \green{we also} set $\tau=\eta-s$.  The \green{desired estimate follows since the }resulting integral is $(2c\eta)^{-1}$.
\blue{For the derivative, \green{we} differentiate the definition \green{and get that}}
\[
 \ell_n'(\eta)=d_n'(\eta)\int_K^\eta d_n(s)^{-2}\,ds+d_n(\eta)^{-1}.
\]
\blue{Multiplying by $d_n(\eta)$ and writing $d_n'(\eta)=-\lambda_n(\eta)d_n(\eta)$ gives}
\[
 d_n(\eta)\ell_n'(\eta)=-\lambda_n(\eta)\,d_n(\eta)\ell_n(\eta)+1.
\]
\blue{Using $|\lambda_n(\eta)|\le C(1+\eta)$ (from \eqref{eq:regularOriginLogDerivativeBounds}) together with $d_n(\eta)\ell_n(\eta)\le C(1+\eta)^{-1}$ yields $d_n(\eta)|\ell_n'(\eta)|\le C$. Hence}
\[
 \blue{\ell_n(\eta)+(1+\eta)^{-1}|\ell_n'(\eta)|
 \le \frac{C}{(1+\eta)d_n(\eta)}.}
\]
Substitution in \eqref{eq:leftScalarGreenKernel} proves
\eqref{eq:leftScalarGreenKernelBound}.  Integrating separately over
$s\ge\eta$ and $s\le\eta$, and using
\[
 \int_\eta^\infty
 e^{-c(s^2-\eta^2)}(1+s)^a\,ds
 \le C_a(1+\eta)^{a-1},
\]
proves \eqref{eq:leftChartScalarGreen}.  The same calculation is
unchanged when $p_0$ is replaced by any fixed exponent $p$.  On
$[K,\Lambda_n]$ one has
\[
 \frac{u_n(s)}{u_n(\eta)}
 =1+O\left(\frac{|s-\eta|}{R_n}\right),
 \qquad K\le \eta,s\le \Lambda_n,
\]
and the same estimate holds after one $\eta$-derivative.  Inserting the
factor $(u_n(s)/u_n(\eta))^{1/2}$ in the scalar kernel therefore preserves
the preceding weighted integral bounds.  This proves the estimate for
$\widehat{\mathcal G}_{S,n}$ as well.

The homogeneous magnetic equation has the basis $u_n,u_n^{-1}$.
Variation of constants gives \eqref{eq:leftGaugeGreenKernel}.  In
particular,
\[
 |K_{G,n}(\eta,s)|\le C(s-\eta),
 \qquad
 |\partial_\eta K_{G,n}(\eta,s)|\le C
\]
on $K\le\eta\le s\le \Lambda_n$.  \red{We also have}
\[
\begin{split}
 \int_\eta^{\Lambda_n}(s-\eta)(1+s)w_{K,p_0}(s)\,ds
 &\le \frac{Cw_{K,p_0}(\eta)}{1+\eta},\\
 \frac{1}{1+\eta}\int_\eta^{\Lambda_n}
       (1+s)w_{K,p_0}(s)\,ds
 &\le \frac{Cw_{K,p_0}(\eta)}{1+\eta}.
\end{split}
\]
Since $(1+\eta)^{-1}\le K^{-1}$, these inequalities prove
\eqref{eq:leftChartGaugeGreen}.  Replacing $p_0$ by $p_*$ in the same
two integrals proves the finite magnetic estimate at exponent $p_*$.
The limiting scalar and magnetic kernels satisfy the same estimates,
with the factors involving $u_n$ absent.  This proves
\eqref{eq:leftGreenSameWeight}.

We next compare the conditions imposed at the right endpoint of
$[K,T]$.  Let $\ell_\beta$ solve
\[
 \ell_\beta''=(\beta\eta^2-\beta)\ell_\beta,
 \qquad
 \ell_\beta(K)=0,\qquad \ell_\beta'(K)=1.
\]
Then $W(\ell_\beta,d_\beta)=-1$ and
\[
        \ell_\beta(\eta)
        =d_\beta(\eta)\int_K^\eta d_\beta(s)^{-2}\,ds .
\]
Set
\[
        \sigma_{n,T}
        =\lambda_n(T)-\lambda_\beta(T).
\]
By reduction of order and the Weber estimates
\[
        0<d_\beta(T)\ell_\beta(T)\le \frac{C}{1+T}.
\]
Together with \eqref{eq:regularOriginSlopeComparison}, this implies
\[
\begin{split}
 |\sigma_{n,T}d_\beta(T)\ell_\beta(T)|
 &\le C\left(
 \frac{1+T}{R_n}
 +\frac{(1+\Lambda_n^+)^{C}}{1+T}
 e^{-c((\Lambda_n^+)^2-T^2)}
 \right)\le \frac{1}{2}
\end{split}
\]
for all large $n$, uniformly for $K<T\le \Lambda_n$. 
 Hence there is a unique \red{solution $r_{n,T}$ of the limiting Weber equation 
 \begin{align*}
 \partial_\eta^2r_{n,T}(\eta)=(\beta\eta^2-\beta)r_{n,T}(\eta),
 \end{align*} which has the form}
\[
        r_{n,T}=d_\beta+\varepsilon_{n,T}\ell_\beta
\]
such that
\[
        -\frac{r_{n,T}'(T)}{r_{n,T}(T)}=\lambda_n(T).
\]
The two exact identities
\bel{eq:rightEndpointRankOneIdentities}
 \varepsilon_{n,T}
 =-\sigma_{n,T}d_\beta(T)r_{n,T}(T),
 \qquad
 \frac{r_{n,T}(T)}{d_\beta(T)}
 =\frac{1}{1+\sigma_{n,T}d_\beta(T)\ell_\beta(T)}
\ee
follow by solving this scalar equation.  In particular,
\bel{eq:rightEndpointComparisonSolution}
 \frac{1}{2}d_\beta(T)\le r_{n,T}(T)\le2d_\beta(T).
\ee
The solution $r_{n,T}$ has no zero on $[K,T]$.  Indeed,
\[
 \frac{r_{n,T}(\eta)}{d_\beta(\eta)}
 =1+\varepsilon_{n,T}\int_K^\eta d_\beta(s)^{-2}\,ds .
\]
If $\varepsilon_{n,T}\ge0$, the right side is at least one.  If
$\varepsilon_{n,T}<0$, it is decreasing in $\eta$ and is therefore
bounded below by its positive value at $\eta=T$, as given by
\eqref{eq:rightEndpointComparisonSolution}.

If
$\mathcal G_{\beta,K,T}^{(n)}$ is the Green operator for the limiting
Weber equation with $v(K)=0$ and $W(v,r_{n,T})(T)=0$, then its kernel
differs from that of
$R_T\mathcal G_{\beta,K}E_T$\red{\blue{, whose kernel we denote by $\displaystyle K_{\beta,T}^{(0)}(\eta,s)$,}} by the exact rank-one expression
\bel{eq:rightEndpointRankOneKernel}
 K_{\beta,T}^{(n)}(\eta,s)
 -K_{\beta,T}^{(0)}(\eta,s)
 =-\varepsilon_{n,T}\ell_\beta(\eta)\ell_\beta(s).
\ee
\red{\blue{Indeed, $K_{\beta,T}^{(n)}$ is the Green kernel for the boundary-value problem with right endpoint condition $W(v,r_{n,T})(T)=0$, so it is built from the pair $(\ell_\beta,r_{n,T})$, whereas $K_{\beta,T}^{(0)}$ corresponds to the kernel of $R_T\mathcal G_{\beta,K}E_T$, i.e., the problem with the recessive solution $d_\beta$ at $+\infty$ transported back to $T$, and is built from $(\ell_\beta,d_\beta)$. Explicitly,}}

\red{\blue{$\displaystyle K_{\beta,T}^{(n)}(\eta,s)=
 \begin{cases}
  -\ell_\beta(\eta)r_{n,T}(s),&\eta\le s,\\
  -r_{n,T}(\eta)\ell_\beta(s),&s\le\eta,
 \end{cases}
 \qquad
 K^{(0)}_{\beta,T}(\eta,s)=
 \begin{cases}
  -\ell_\beta(\eta)d_\beta(s),&\eta\le s,\\
  -d_\beta(\eta)\ell_\beta(s),&s\le\eta,
 \end{cases}$}}
\blue{and inserting $r_{n,T}=d_\beta+\varepsilon_{n,T}\ell_\beta$ yields
$K_{\beta,T}^{(n)}-K_{\beta,T}^{(0)}=-\varepsilon_{n,T}\ell_\beta\otimes\ell_\beta$.}
The Weber estimates obtained from Lemma~\ref{lem:webermatching} after
the scaling used in Section~\ref{sec:webertail} are
\[
 d_\beta(t)\simeq 
 t^{(\sqrt\beta-1)/2}e^{-\sqrt\beta t^2/2},
 \qquad
 |\ell_\beta(t)|\le C_{K,J}
 t^{-(\sqrt\beta+1)/2}e^{\sqrt\beta t^2/2}
\]
with the corresponding differentiated bounds.  They imply
\[
\begin{split}
 \|\varepsilon_{n,T}\ell_\beta\otimes\ell_\beta\|
       _{Y_{p_0}\to X_{p_*}}
 &\le
 |\sigma_{n,T}|d_\beta(T)|r_{n,T}(T)|\\
 &\quad\times
 \sup_{K\le\eta\le T}
 \frac{|\ell_\beta(\eta)|
 +(1+\eta)^{-1}|\ell_\beta'(\eta)|}{w_{K,p_*}(\eta)}\\
 &\quad\times
 \int_K^T|\ell_\beta(s)|(1+s)w_{K,p_0}(s)\,ds .
\end{split}
\]
To evaluate the last two factors, write
$a=(\sqrt\beta-1)/2$.  The supremum is bounded by
\[
 C(1+T)^{-a-1-p_*}
 e^{(\sqrt\beta/2+c)T^2},
\]
whereas the integral is bounded by
\[
 C(1+T)^{p_0-a-1}
 e^{(\sqrt\beta/2-c)T^2}.
\]
By \eqref{eq:rightEndpointComparisonSolution},
\[
 d_\beta(T)|r_{n,T}(T)|
 \le C(1+T)^{2a}e^{-\sqrt\beta T^2}.
\]
Since $p_*\ge p_0+2$, the product of these three quantities is at most
$C(1+T)^{-2}$.  The slope estimate
\eqref{eq:regularOriginSlopeComparison}, and the choice of
$C_{\log,0}$ which makes the remaining endpoint term
$O(R_n^{-1})$, now give
\[
 \|\varepsilon_{n,T}\ell_\beta\otimes\ell_\beta\|
       _{Y_{p_0}\to X_{p_*}}
 \le \frac{C(K,\beta,C_{\log})}{R_n}.
\]
Thus \eqref{eq:rightEndpointRankOneKernel} has the required operator
bound.
The same calculation with  $p_0$ in place of $p_*$
gives
\[
 \|\varepsilon_{n,T}\ell_\beta\otimes\ell_\beta\|
       _{Y_{p_0}\to X_{p_0}}
 \le \frac{C|\sigma_{n,T}|}{(1+T)^2}\le C_{K,\beta}.
\]
Together with the estimate for
$R_T\mathcal G_{\beta,K}E_T$, this proves
\bel{eq:adjustedLimitingGreenStrong}
 \|\mathcal G_{\beta,K,T}^{(n)}F\|_{X_{p_0}([K,T])}
 \le C_{K,\beta}\|F\|_{Y_{p_0}([K,T])}.
\ee
The finite scalar operator \red{$\displaystyle \mathcal G_{S,n}$} and
$\mathcal G_{\beta,K,T}^{(n)}$ have the same conditions at $K$ and $T$.
The Green-function identity
\[
 \mathcal G_{S,n}-\mathcal G_{\beta,K,T}^{(n)}
 =
 \mathcal G_{S,n}
 \bigl(V_n^0-(\beta\eta^2-\beta)\bigr)
 \mathcal G_{\beta,K,T}^{(n)}
\]
is understood as an identity on a source, with multiplication by the
coefficient difference between the two Green operators.  Its mapping
properties are \red{determined by}
\[
\begin{split}
 \mathcal G_{\beta,K,T}^{(n)}
 &:Y_{p_0}\longrightarrow X_{p_0},\\
 V_n^0-(\beta\eta^2-\beta)
 &:X_{p_0}\longrightarrow R_n^{-1}Y_{p_*},\\
 \mathcal G_{S,n}
 &:Y_{p_*}\longrightarrow X_{p_*}.
\end{split}
\]
Indeed, the middle line follows from
\eqref{eq:leftChartCoefficientComparison} because
$p_*\ge p_0+2$, the first line is
\eqref{eq:adjustedLimitingGreenStrong}, and the last follows from
\eqref{eq:leftGreenSameWeight}.  Therefore
\[
 \|\mathcal G_{S,n}
       -\mathcal G_{\beta,K,T}^{(n)}\|_{Y_{p_0}\to X_{p_*}}
 \le \frac{C(K,\beta,C_{\log})}{R_n}.
\]
The conjugation contributes the factor
\[
 \left(\frac{u_n(s)}{u_n(\eta)}\right)^{1/2},
 \qquad
 \left|
 \left(\frac{u_n(s)}{u_n(\eta)}\right)^{1/2}-1
 \right|
 \le \frac{C|s-\eta|}{R_n}.
\]
The kernel bounds remain integrable after multiplication by
$|s-\eta|$, and differentiation of the factor is also
$O(R_n^{-1})$.  This proves the scalar estimate in
\eqref{eq:leftGreenOperatorComparison}.
Finally, the exact identity
\[
        K_{G,n}(\eta,s)-(s-\eta)
        =-\frac{(s-\eta)^2}{2u_n(\eta)}
\]
and $u_n(\eta)\simeq  R_n$ on $[K,\Lambda_n]$ prove the magnetic estimate in
\eqref{eq:leftGreenOperatorComparison} by the same weighted integral
bounds used above.
\end{proof}

\subsubsection{The one-parameter family of left Cauchy data}

We now construct the nonlinear solutions whose Cauchy data furnish the
left boundary condition at $y=-K$.  For $C$ near the physical limiting
value
\[
        C_{\beta,K}^{\rm ph}:=\Phi_\beta(-K),
\]
the parameter $C$ is the prescribed scalar value at $\eta=K$.

\begin{Lemma}\label{lem:leftNonlinearChart}
 \red{Fix $p\geq 0$.} There are finite $K_0>0$,
$C_{\log,0}>0$, and $0<c_0<1$, $c>0$, uniform
for $\beta$ in any fixed compact interval $J\Subset(0,4)$, with the
following property.  Fix $K\ge K_0$, $C_{\log}\ge C_{\log,0}$, and
put $\Lambda_n=C_{\log}\sqrt{\log R_n}$.  Let
\[
        I_{\beta,K}
        =\left\{C:
        |C-C_{\beta,K}^{\rm ph}|
        \le c_0C_{\beta,K}^{\rm ph}\right\}.
\]
For every $C\in I_{\beta,K}$ there is a unique pair
\[
        (\Phi_{\beta,C}^-,\Gamma_{\beta,C}^-)
\]
in the weighted ball specified in the proof which solves the interface
equations on $[K,\infty)$, satisfies
\[
\begin{split}
 \Phi_{\beta,C}^-(K)&=C,\\
 \Gamma_{\beta,C}^-(\eta)-\sqrt\beta\,\eta&\longrightarrow0,\\
 (\Gamma_{\beta,C}^-)'(\eta)-\sqrt\beta&\longrightarrow0
\end{split}
\qquad(\eta\to\infty),
\]
and has the decaying scalar behavior at infinity.
For every sufficiently large $n$ and every $C\in I_{\beta,K}$ there is
a unique pair
\[
        (\Phi_{n,C}^-,\Gamma_{n,C}^-)
\]
in the corresponding weighted ball which solves the exact finite radial
equations on $[K,\Lambda_n]$ and satisfies
\bel{eq:finiteChartBC}
\begin{split}
 \Phi_{n,C}^-(K)&=C,\\
 W(u_n^{1/2}\Phi_{n,C}^-,D_n^0)(\Lambda_n)&=0,\\
 \Gamma_{n,C}^-(\Lambda_n)&=\Gamma_n^0(\Lambda_n),\\
 (\Gamma_{n,C}^-)'(\Lambda_n)&=(\Gamma_n^0)'(\Lambda_n).
\end{split}
\ee
Both solution maps are $C^1$ in $C$.
For the limiting family, the Cauchy-data map and its $C$ derivative
are jointly continuous in $(\beta,C)$, for $\beta\in J$ and
$C\in I_{\beta,K}$. With
\[
 m_{\beta,K}(C)
 =\frac{(\Phi_{\beta,C}^-)'(K)}{C},
 \qquad
 m_{n,K}(C)
 =\frac{(\Phi_{n,C}^-)'(K)}{C},
\]
one has
\[
        |m_{n,K}(C)-m_{\beta,K}(C)|
        \le C(K,\beta,C_{\log})R_n^{-1}.
\]
Uniformly for $C\in I_{\beta,K}$ and $K\le\eta\le \Lambda_n$,
\bel{eq:leftChartScalar}
\begin{split}
 &|\Phi_{n,C}^-(\eta)-\Phi_{\beta,C}^-(\eta)|
 +(1+\eta)^{-1}
 |(\Phi_{n,C}^-)'(\eta)-(\Phi_{\beta,C}^-)'(\eta)|\\
 &\qquad\le
 C(K,\beta,C_{\log})R_n^{-1}
 (1+\eta)^p e^{-c(\eta^2-K^2)},
\end{split}
\ee
and
\bel{eq:leftChartGauge}
\begin{split}
 &|\Gamma_{n,C}^-(\eta)-\Gamma_{\beta,C}^-(\eta)-P_n(\eta)|\\
 &\quad
 +(1+\eta)^{-1}
 |(\Gamma_{n,C}^-)'(\eta)
  -(\Gamma_{\beta,C}^-)'(\eta)-P_n'(\eta)|\\
 &\qquad\le
 C(K,\beta,C_{\log})R_n^{-1}
 (1+\eta)^p e^{-c(\eta^2-K^2)}.
\end{split}
\ee
The limiting family also satisfies
\bel{eq:leftLimitingParameterDecay}
\begin{split}
 \sum_{j=0}^1(1+\eta)^{-j}
 \bigl(
 |\partial_\eta^j\partial_C\Phi_{\beta,C}^-(\eta)|
 +|\partial_\eta^j\partial_C\Gamma_{\beta,C}^-(\eta)|
 \bigr)
 \le \red{C(K,\beta,C_{\log})}(1+\eta)^p e^{-c(\eta^2-K^2)} .
\end{split}
\ee
Define the Cauchy-data maps, in the order
$(\Phi,\Gamma,\partial_y\Phi,\partial_y\Gamma)$, by
\bel{eq:leftTraceMaps}
\begin{split}
 \Theta_{n,K}^-(C)
 &:=
 \bigl(
 \Phi_{n,C}^-(K),\Gamma_{n,C}^-(K),
 -(\Phi_{n,C}^-)'(K),-(\Gamma_{n,C}^-)'(K)
 \bigr),\\
 \Theta_{\beta,K}^-(C)
 &:=
 \bigl(
 \Phi_{\beta,C}^-(K),\Gamma_{\beta,C}^-(K),
 -(\Phi_{\beta,C}^-)'(K),-(\Gamma_{\beta,C}^-)'(K)
 \bigr).
\end{split}
\ee
Then
\bel{eq:leftTraceComparison}
 \sup_{C\in I_{\beta,K}}
 |\Theta_{n,K}^-(C)-\Theta_{\beta,K}^-(C)|
 \le C(K,\beta,C_{\log})R_n^{-1},
\ee
and
\bel{eq:leftTraceTangentLine}
 D\Theta_{\beta,K}^-(C_{\beta,K}^{\rm ph})\RR
 =E_\beta^-(-K).
\ee
\end{Lemma}
The tangent identity \eqref{eq:leftTraceTangentLine} uses the
one-dimensional endpoint space from Lemma~\ref{lem:leftline}\green{, and} it does
not require the rank assertion \green{from} Proposition~\ref{prop:nearBPS}.
\begin{proof}
We give the fixed-point equations and the estimates defining the
uniqueness class.  Write
\[
 g=\Gamma-\sqrt\beta\,\eta
\]
for the limiting problem.  Then
\bel{eq:limitTailFixedPoint}
\begin{split}
 \Phi(\eta)
 &=Cd_\beta(\eta)
 +\mathcal G_{\beta,K}
 \left(
 (2\sqrt\beta\,\eta g+g^2+\beta\Phi^2)\Phi
 \right)(\eta),\\
 g(\eta)
 &=2\int_\eta^\infty(s-\eta)
       (\sqrt\beta\,s+g(s))\Phi(s)^2\,ds .
\end{split}
\ee
For the finite problem, write
\[
        H=\Gamma-\Gamma_n^0.
\]
After conjugating the scalar equation by
$\Psi=u_n^{1/2}\Phi$, the equations are
\bel{eq:finiteTailFixedPoint}
\begin{split}
 \Phi(\eta)
 &=C\,u_n(K)^{1/2}u_n(\eta)^{-1/2}D_n^0(\eta)\\
 &\quad+
 \widehat{\mathcal G}_{S,n}
 \left(
 (2\Gamma_n^0H+H^2+\beta\Phi^2)\Phi
 \right)(\eta),\\
 H(\eta)
 &=\mathcal G_{G,n}
 \left(2(\Gamma_n^0+H)\Phi^2\right)(\eta).
\end{split}
\ee
Differentiating these equations verifies the two differential systems.
The definitions of the Green operators give precisely the boundary
conditions in the statement. \red{In what follows, we shall apply Lemma \ref{lem:regularOriginMixedGreen} with $\displaystyle p_0=p$, and let $\displaystyle p_\ast\geq p_0+2$ be a comparison exponent for which \eqref{eq:leftGreenOperatorComparison} holds.} For the limiting problem\red{, we} use
\[
 \mathfrak X_{\beta,K}
 =X_{p}([K,\infty))\times X_{p}([K,\infty)),
\]
and for the finite problem use
\[
 \mathfrak X_{n,K}
 =X_{p}([K,\Lambda_n])\times X_{p}([K,\Lambda_n]).
\]
For comparisons, use the weaker norms
\[
 \mathfrak X_{\beta,K}^*
 =X_{p_*}([K,\infty))\times X_{p_*}([K,\infty)),
 \qquad
 \mathfrak X_{n,K}^*
 =X_{p_*}([K,\Lambda_n])\times X_{p_*}([K,\Lambda_n]),
\]
where $p_*$ is the exponent in
Lemma~\ref{lem:regularOriginMixedGreen}.
Let $\mathscr P_{\beta,C}$ and $\mathscr P_{n,C}$ denote the right sides
of \eqref{eq:limitTailFixedPoint} and
\eqref{eq:finiteTailFixedPoint}, respectively.  The Green estimates of
Lemma~\ref{lem:regularOriginMixedGreen}, the bounds
\[
 |\Gamma_n^0(\eta)|+\sqrt\beta\,\eta\le C(1+\eta),
\]
and the fact that
$C_{\beta,K}^{\rm ph}$ is Gaussian-small as $K\to\infty$ give numbers
$r_K>0$ and $K_0$ such that
\bel{eq:leftFixedPointBall}
\begin{split}
 \mathscr P_{\star,C}
 \bigl(\overline B_{\mathfrak X_{\star,K}}(0,r_K)\bigr)
 &\subset
 \overline B_{\mathfrak X_{\star,K}}(0,r_K),\\
 \|\mathscr P_{\star,C}(Z)
       -\mathscr P_{\star,C}(\widetilde Z)\|_{\mathfrak X_{\star,K}}
 &\le \frac{1}{3}
 \|Z-\widetilde Z\|_{\mathfrak X_{\star,K}},
\end{split}
\ee
for $\star\in\{\beta,n\}$, all $C\in I_{\beta,K}$, and all
$Z,\widetilde Z$ in
the displayed ball.  The same calculation, with the extra fixed
polynomial factor allowed by $p_*$, also gives
\bel{eq:leftFixedPointComparisonLipschitz}
 \|\mathscr P_{\star,C}(Z)-\mathscr P_{\star,C}(\widetilde Z)\|
       _{\mathfrak X_{\star,K}^*}
 \le \frac{1}{3}
 \|Z-\widetilde Z\|_{\mathfrak X_{\star,K}^*}.
\ee
To see the first estimate directly, every scalar
source in \eqref{eq:limitTailFixedPoint} and
\eqref{eq:finiteTailFixedPoint} contains the decaying scalar factor,
whereas every magnetic source contains two scalar factors. \red{Upon applying the Green operators and also accounting for the homogeneous terms, the right hand sides of the fixed point equations turn out to be bounded by}
\red{
\begin{align*}
    C_0C^{\text{ph}}_{\beta,K}(1+c_0)(1+K)^{-p}+C_{p}\left(\frac{(1+K)^{p}r_K}{K}+\frac{(1+K)^{2p}r^2_K}{K^2}\right)r_K
\end{align*}
Indeed, for instance, the quadratic term from the scalar Higgs source in \eqref{eq:limitTailFixedPoint} satisfies, via the Green estimates in Lemma~\ref{lem:regularOriginMixedGreen}
\begin{align*}
  \|\mathcal{G}_{\beta,K}(2\sqrt{\beta}\eta g\Phi)\|_{X_{p}}&\leq \frac{C_{p}}{K}\|2\sqrt{\beta}\eta g\Phi\|_{Y_{p}} 
\end{align*}
On the other hand,
\begin{align*}
    \frac{|\eta g(\eta)\Phi(\eta)|}{(1+\eta)w_{K,p}(\eta)}&\leq C_{p}\|g\|_{X_{p}}\|\Phi\|_{X_{p}}w_{K,p}(\eta)\leq C_{p}(1+K)^{p}r_K^2,
\end{align*}
hence
\begin{align*}
     \|\mathcal{G}_{\beta,K}(2\sqrt{\beta}\eta g\Phi)\|_{X_{p}}&\leq \frac{C_{p}(1+K)^{p}r_K^2}{K}.
\end{align*}
Similar considerations for the other terms show that
\begin{align*}
    \left\|G_{\beta,K}
 \left(
 (2\sqrt\beta\,\eta g+g^2+\beta\Phi^2)\Phi
 \right)\right\|_{X_{p}}\leq C_{p}\left(\frac{(1+K)^{p}r_K}{K}+\frac{(1+K)^{2p}r^2_K}{K^2}\right)r_K.
\end{align*}
For the linear term we have
\begin{align*}
    \|Cd_\beta\|_{X_{p}}\leq C^{\text{ph}}_{\beta,K}(1+c_0)\|d_\beta\|_{X_{p}}
\end{align*}
Given the known asymptotic behavior for $d_\beta$, there exists some $c_1>0$ such that
\begin{align*}
    |d_\beta(\eta)|+\frac{|d_\beta'(\eta)}{1+\eta}\leq C_0e^{-c_1(\eta^2-K^2)}.
\end{align*}
Here, we may assume that $c_1>c$. Then,
\begin{align*}
    \|d_\beta\|_{X_{p}}\leq C_0\sup_{\substack{\eta\geq\kappa}}(1+\eta)^{-p}e^{-(c_1-c)(\eta^2-K^2)}\leq C_0(1+K)^{-p},
\end{align*}
hence
\begin{align*}
    \|Cd_\beta\|_{X_{p}}\leq c_0C^{\text{ph}}_{\beta,K}(1+c_0)(1+K)^{-p}. 
\end{align*}
Thus,
\begin{align*}
    \|\mathcal{P}_{p,C}(Z)\|_{\mathfrak{X}_{p,K}}\leq  C_0C^{\text{ph}}_{\beta,K}(1+c_0)(1+K)^{-p}+C_{p}\left(\frac{(1+K)^{p}r_K}{K}+\frac{(1+K)^{2p}r^2_K}{K^2}\right)r_K.
\end{align*}
for every $Z\in \overline{B}_{\mathfrak X_{p,K}}(0,r_K)$.
Similarly,
\begin{align*}
    \|\mathcal{P}_{\beta,C}(Z)-\mathcal{P}_{p,C}(\tilde{Z})\|_{\mathfrak{X}_{\beta,K}}\leq  C_{p}\left(\frac{(1+K)^{p}r_K}{K}+\frac{(1+K)^{2p}r^2_K}{K^2}\right)\|Z-\tilde{Z}\|_{\mathfrak{X}_{\beta,K}}.
\end{align*}
We now choose $r_K$ of the form $r_k=\epsilon(1+K)^{-p}$, with $\epsilon>0$ small enough that 
\begin{align*}
  C_{p}\left(\frac{\epsilon}{K}+\frac{\epsilon^2}{K^2}\right)\leq\frac{1}{3}\,\qquad C_0C^{ph}_{\beta,K}(1+c_0)\leq\frac{\epsilon}{3}. 
\end{align*}
Due to the asymptotic behavior of $\Phi_\beta$, this can be done for large enough $K\geq 1$.
}  The contraction theorem gives
                existence and uniqueness in the stated balls. \red{\eqref{eq:finiteTailFixedPoint} can be analyzed similarly}.

Define $\mathscr P_{\beta,C}^{\Lambda_n}$ on
$\mathfrak X_{n,K}$ by the right side of
\eqref{eq:limitTailFixedPoint}, with
$\mathcal G_{\beta,K}$ and $\mathcal G_{G,\infty}$ replaced by
$R_{\Lambda_n}\mathcal G_{\beta,K}E_{\Lambda_n}$ and
$R_{\Lambda_n}\mathcal G_{G,\infty}E_{\Lambda_n}$, respectively.  Thus only the
nonlinear sources are extended by zero; the homogeneous scalar term is
$Cd_\beta$ restricted to $[K,\Lambda_n]$.  The same estimates as in
\eqref{eq:leftFixedPointBall} give a unique fixed point
$Z_{\beta,C}^{\Lambda_n}$ in the finite-interval ball.
\red{We claim that
\[
 \left\|
 u_n(K)^{1/2}u_n(\eta)^{-1/2}D_n^0(\eta)-d_\beta(\eta)
 \right\|_{X_{p_*}([K,\Lambda_n])}
 \le \frac{C(K,\beta,C_{\log})}{R_n}
\]
Indeed, estimate \eqref{eq:regularOriginWeberComparison} from Lemma ~\ref{lem:regularOriginLogDerivative} readily gives the bound
\begin{align*}
\|D_n^0(\eta)-d_\beta(\eta)\|_{X_{p_\ast}([K,\Lambda_n])}
 &\le
 \frac{C(K,\beta,C_{\log})}{R_n}.    
\end{align*}
Given that $\displaystyle u_n(K)^{1/2}u_n(\eta)^{-1/2}=1+O\left(\frac{\eta-K}{R_n-\eta}\right)$, together with the fact that $\eta\leq \Lambda_n^+$, we can see that the latter also implies
\begin{align*}
\|u_n(K)^{1/2}u_n(\eta)^{-1/2}(D_n^0(\eta)-d_\beta(\eta))\|_{X_{p_\ast}([K,\Lambda_n])}
 &\le
 \frac{C(K,\beta,C_{\log})}{R_n}.    
\end{align*}
Finally, the asymptotics for the Weber solution $d_\beta$ also imply the bound 
\begin{align*}
\|(u_n(K)^{1/2}u_n(\eta)^{-1/2}-1)d_\beta(\eta)\|_{X_{p_\ast}([K,\Lambda_n])}
 &\le
 \frac{C(K,\beta,C_{\log})}{R_n}.    
\end{align*}
Combined, they give the previously claimed bound.
We also note that \eqref{eq:SLhyp} and \eqref{eq:exactAffineDefect} imply that
\begin{align*}
    P_n(\eta)&=\Gamma_n^0(\eta)-\sqrt{\beta}\eta=O(R_n^{-1}(1+\eta)^2).
\end{align*}
}
\red{Together with t}he coefficient estimate
\eqref{eq:leftChartCoefficientComparison}, $p_*\geq \red{p}+2$, \red{and} the Green-operator comparison\red{, this implies that}
\bel{eq:leftFixedPointMapComparison}
 \sup_{\|Z\|_{\mathfrak X_{n,K}}\le r_K}
 \|\mathscr P_{n,C}(Z)
       -\mathscr P_{\beta,C}^{\Lambda_n}(Z)\|_{\mathfrak X_{n,K}^*}
 \le \frac{C(K,\beta,C_{\log})}{R_n}
 .
\ee
The limiting nonlinear sources and their first $C$-derivatives
satisfy
\[
        |F^{(\ell)}(s)|
        \le C_{\ell,K}(1+s)^{C_1} e^{-cs^2},
        \qquad 0\le\ell\le1.
\]
\red{We also have
\begin{align*}
  Z_{\beta,C}^{\Lambda_n}
       -R_{\Lambda_n}Z_{\beta,C}&=\mathscr{P}^{\Lambda_n}_{\beta,C}(Z_{\beta,C}^{\Lambda_n})- R_{\Lambda_n}\mathscr{P}_{\beta,C}(Z_{\beta,C})\\
       &=(\mathscr{P}^{\Lambda_n}_{\beta,C}(Z_{\beta,C}^{\Lambda_n})- \mathscr{P}^{\Lambda_n}_{\beta,C}(R_{\Lambda_n}Z_{\beta,C}))+(\mathscr{P}^{\Lambda_n}_{\beta,C}(R_{\Lambda_n}Z_{\beta,C})- R_{\Lambda_n}\mathscr{P}_{\beta,C}(Z_{\beta,C}))
\end{align*}
For the first line, we have
\begin{align*}
   \|\mathscr{P}^{\Lambda_n}_{\beta,C}(Z_{\beta,C}^{\Lambda_n})- \mathscr{P}^{\Lambda_n}_{\beta,C}(R_{\Lambda_n}Z_{\beta,C})\|_{\mathfrak{X}^\ast_{n,K}}&\leq\frac{1}{3}\|Z_{\beta,C}^{\Lambda_n}-R_{\Lambda_n}Z_{\beta,C}\|_{\mathfrak{X}^\ast_{n,K} }
\end{align*}
For the second line, we can write
\begin{align*}
    \mathscr{P}^{\Lambda_n}_{\beta,C}(R_{\Lambda_n}Z_{\beta,C})- R_{\Lambda_n}\mathscr{P}_{\beta,C}(Z_{\beta,C})&=-R_{\Lambda_n}\mathbf{G}_{\beta}(\mathbf{1}_{(\Lambda_n,\infty)})(\mathcal{N}_\beta(Z_{\beta,C}))
\end{align*}
Here, $\mathcal{N}_\beta(Z_{\beta,C})$ denotes the source term in the joint differential equation solved by $Z_{\beta,C}$, constructed by adjoining the two forcing terms in the equations for the components of $Z_{\beta,C}$, and $\mathbf{G}_{\beta}$ is the joint Green function constructed likewise. 
}
The half-line Green estimates therefore give
\red{
\begin{align*}
    \|\mathscr{P}^{\Lambda_n}_{\beta,C}(R_{\Lambda_n}Z_{\beta,C})- R_{\Lambda_n}\mathscr{P}_{\beta,C}(Z_{\beta,C})\|_{\mathfrak X_{n,K}^*}
 \le C_{K,\beta}e^{-c\Lambda_n^2}.
\end{align*}
Thus,
\begin{align*}
    \|Z_{\beta,C}^{\Lambda_n}-R_{\Lambda_n}Z_{\beta,C}\|_{\mathfrak{X}^\ast_{n,K} }&\leq\frac{1}{3}\|Z_{\beta,C}^{\Lambda_n}-R_{\Lambda_n}Z_{\beta,C}\|_{\mathfrak{X}^\ast_{n,K} }+ C_{K,\beta}e^{-c\Lambda_n^2}.
\end{align*}
Upon absorbing the first term on the right hand side, this gives
}
\bel{eq:leftFixedPointTruncation}
 \|Z_{\beta,C}^{\Lambda_n}
       -R_{\Lambda_n}Z_{\beta,C}\|_{\mathfrak X_{n,K}^*}
 \le C_{K,\beta}e^{-c\Lambda_n^2}.
\ee
Increase $C_{\log,0}$ so that $e^{-c\Lambda_n^2}\le R_n^{-2}$.
\red{We have
\begin{align*}
  Z_{n,C}-Z_{\beta,C}^{\Lambda_n}&=\mathscr{P}_{n,C}(Z_{n,C})-\mathscr{P}^{\Lambda_n}_{n,C}(Z^{\Lambda_n}_{\beta,C})\\
  &=(\mathscr{P}_{n,C}(Z_{n,C})-\mathscr{P}_{n,C}(Z^{\Lambda_n}_{\beta,C}))+(\mathscr{P}_{n,C}(Z^{\Lambda_n}_{\beta,C})-\mathscr{P}^{\Lambda_n}_{n,C}(Z^{\Lambda_n}_{\beta,C}))
\end{align*}
}
Using
\eqref{eq:leftFixedPointComparisonLipschitz} \red{for the first term on the right} and
\eqref{eq:leftFixedPointMapComparison} \red{for the second one on the right},
\[
\begin{split}
 \|Z_{n,C}-Z_{\beta,C}^{\Lambda_n}\|_{\mathfrak X_{n,K}^*}
 &\le \frac{1}{3}
 \|Z_{n,C}-Z_{\beta,C}^{\Lambda_n}\|_{\mathfrak X_{n,K}^*}
 +\frac{C(K,\beta,C_{\log})}{R_n},
\end{split}
\]
\red{which, upon absorbing the first term on the right, becomes
\begin{align*}
   \|Z_{n,C}-Z_{\beta,C}^{\Lambda_n}\|_{\mathfrak X_{n,K}^*}
 &\le\frac{C(K,\beta,C_{\log})}{R_n}. 
\end{align*}
}
Combining this estimate with
\eqref{eq:leftFixedPointTruncation}\red{, using the fact that $p_\ast\geq p+2$, as well as shrinking the Gaussian exponent when necessary,} proves
\eqref{eq:leftChartScalar} and \eqref{eq:leftChartGauge}.
At $\eta=K$, the sharp localization \eqref{eq:SLhyp} and
\eqref{eq:exactAffineDefect} give
$|P_n(K)|+|P_n'(K)|\le C_{K,\beta}R_n^{-1}$. \red{The inequality $C\geq (1-c_0)C^{ph}_{\beta,K}>0$, along with \eqref{eq:leftChartScalar} evaluated at $\eta=K$ implies the bound 
\begin{align*}
     |m_{n,K}(C)-m_{\beta,K}(C)|
        \le C(K,\beta,C_{\log})R_n^{-1}.
\end{align*}
}
Evaluati\red{ng} the two
profile estimates at $K$ therefore proves
\eqref{eq:leftTraceComparison}.

The maps in \eqref{eq:limitTailFixedPoint} and
\eqref{eq:finiteTailFixedPoint} are polynomial in the profiles and affine
in $C$.  A contraction mapping argument shows that their
fixed points are $C^1$ in $C$.  Equivalently, differentiating once gives
\[
 (I-D_Z\mathscr P_{\star,C})\partial_CZ_{\star,C}
 =\partial_C\mathscr P_{\star,C}.
\]
The inverse is the Neumann series and has norm at most $3/2$.
\green{We note that these bounds are uniform f}or the limiting family.
\green{We f}ix a common Gaussian exponent $0<c'<c$.  \green{Once we increase $K_0$, t}he Green estimates in
\eqref{eq:leftGreenSameWeight} also \green{provide us with} a contraction constant
$q<1$, and an inverse bound $(1-q)^{-1}$ in this weaker norm on
the original balls.  The stronger Gaussian
bounds \green{ensure that} the tails \green{are} uniformly small in the weaker norm. \green{We also note that t}he Weber solutions and the Green kernels are continuous in $\beta$ on
bounded intervals\green{, and that t}heir uniform Gaussian bounds \green{guarantee the} continuity of
the fixed-point maps and their derivatives in a common, slightly weaker
Gaussian norm. \green{In order to see that, we can} split each integral at a fixed large point, use continuity
on the bounded part, and then \green{allow} that point \green{to} tend to infinity in the
uniform tail bound.  \green{Upon s}ubtracting the fixed-point equations, and then
the displayed differentiated equations, \green{we infer that }
$Z_{\beta,C}$ and $\partial_C Z_{\beta,C}$ \green{are jointly continuous}.  \green{Upon evaluating $Z_{\beta,C}$} 
and \green{its} first derivative at $K$\green{, we deduce that the aforementioned Cauchy-data maps are continuous}.  In particular, their $C$ derivatives have a common
modulus of continuity on \green{the compact set}
$\{(\beta,C):\beta\in J,\ C\in I_{\beta,K}\}$.
Here\green{, the positivity and continuity of the interface also imply that} $\min_{\beta\in J}\Phi_\beta(-K)>0$.
In the
limiting differentiated fixed-point equation, the
inhomogeneous scalar term is $d_\beta$\red{, while} in the second \red{one} it is zero.
All other terms contain either a parameter derivative of the solution
or a factor of $\Phi_{\beta,C}^-$.  \red{We now apply} the half-line scalar and
magnetic Green bounds in \eqref{eq:leftGreenSameWeight}, first with
$p=p_0$ and then with the fixed larger exponent $p_*$, \red{use that $p_\ast\geq p+2$, and shrink the Gaussian exponent as before, in order to infer that}
\[
\begin{split}
 &\sum_{j=0}^1(1+\eta)^{-j}
 \left(
 |\partial_\eta^j\partial_C\Phi_{\beta,C}^-(\eta)|
 +|\partial_\eta^j\partial_C
      \Gamma_{\beta,C}^-(\eta)|
 \right)\\
 &\qquad\le
 C_{K,\beta}(1+\eta)^p e^{-c(\eta^2-K^2)}.
\end{split}
\]
This is \eqref{eq:leftLimitingParameterDecay}. 
At $C=C_{\beta,K}^{\rm ph}$ the limiting solution is
\[
 \Phi_{\beta,C_{\beta,K}^{\rm ph}}^-(\eta)
 =\Phi_\beta(-\eta),
 \qquad
 \Gamma_{\beta,C_{\beta,K}^{\rm ph}}^-(\eta)
 =\Gamma_\beta(-\eta).
\]
Differentiation with respect to $C$ gives a nonzero solution of the
linearized interface equations satisfying the normalized conditions at
$y=-\infty$.  Its Cauchy data at $y=-K$ therefore span
$E_\beta^-(-K)$, which proves
\eqref{eq:leftTraceTangentLine}.
\end{proof}

\subsubsection{The trace of the physical radial solution}

We next show that the physical radial solution satisfies, up to a
negligible error, the conditions imposed at $\eta=\Lambda_n$ in
\eqref{eq:finiteChartBC}.  This conclusion is obtained from regularity at
the origin and does not use the fixed-window estimate.

\begin{Lemma}\label{lem:physicalEndpointDiscrepancy}
 \red{We recall the notations}
 \[
 \Phi_n^-(\eta)=\Phi_n(-\eta),\qquad
 \Gamma_n^-(\eta)=\Gamma_n(-\eta).
\]
\red{We also set}
\[
 \Psi_n^{\rm ph}=u_n^{1/2}\Phi_n^-,
 \qquad
 H_n^{\rm ph}=\Gamma_n^- -\Gamma_n^0.
\]
There are $\sigma<\infty$ and constants $c_0,C>0$ such that the following
holds.  Put
$T=C_{\log}\sqrt{\log R_n}$ and suppose
\bel{eq:physicalScalarTailInput}
 |\Phi_n^-(\eta)|
 +(1+\eta)^{-1}|(\Phi_n^-)'(\eta)|
 \le
 \delta_T:=\red{Ce^{-c_0T^2}},
 \qquad T\le\eta<R_n.
\ee
\red{In particular, we may assume that $\delta_T\leq 1$ for $n$ sufficiently large. }Then
\bel{eq:physicalEndpointDiscrepancy}
\begin{split}
 |H_n^{\rm ph}(T)|
 +(1+T)^{-1}|(H_n^{\rm ph})'(T)|
 &\le CR_n^\sigma\delta_T^2,\\
 \frac{|W(\Psi_n^{\rm ph},D_n^0)(T)|}{|D_n^0(T)|+(1+T)^{-1}|(D_n^0)'(T)|}
 &\le CR_n^\sigma\delta_T^2 .
\end{split}
\ee
Consequently, for every $0<c_1<c_0$, one may choose $C_{\log}$ sufficiently large so that
both quantities in \eqref{eq:physicalEndpointDiscrepancy} are bounded by
\bel{eq:physicalEndpointDiscrepancySmall}
        \red{Ce^{-c_1T^2}}.
\ee
\end{Lemma}

\begin{proof}
\red{We recall the notation}
\[
        q_n(u)=a_n(u)-\left(1-\frac{u^2}{R_n^2}\right).
\]
Then
\[
        H_n^{\rm ph}(R_n-u)=\frac{n}{u}q_n(u).
\]
Equation \eqref{eq:qeqforbootstrap}, the conditions
$q_n(0)=q_n'(0)=0$, and \eqref{eq:normalizedOriginPhi} show that, for
$0<u\le1$,
\[
 |q_n(u)|\le \frac{C\phi_n(1)^2}{m^2}u^{2m+2},
 \qquad
 |q_n'(u)|\le \frac{C\phi_n(1)^2}{m}u^{2m+1},
 \qquad m=\lfloor n/3\rfloor .
\]
It \red{also} follows that
\[
 H_n^{\rm ph}(R_n)=0,\qquad
 (H_n^{\rm ph})'(R_n)=0.
\]
These are the right-end conditions used in the following Green formula.
The exact equation for $H_n^{\rm ph}$ and the homogeneous basis
$u_n,u_n^{-1}$ give
\bel{eq:physicalGaugeEndpointFormula}
\begin{split}
 H_n^{\rm ph}(T)
 &=\int_T^{R_n}K_{G,n}(T,s)
       2\Gamma_n^-(s)(\Phi_n^-(s))^2\,ds,\\
 (H_n^{\rm ph})'(T)
 &=\int_T^{R_n}\partial_TK_{G,n}(T,s)
       2\Gamma_n^-(s)(\Phi_n^-(s))^2\,ds .
\end{split}
\ee
On $u_n(s)\ge1$ use
\[
 |K_{G,n}(T,s)|\le C(s-T),\qquad
 |\partial_TK_{G,n}(T,s)|\le C,\qquad
 \Gamma_n^-(s)\le \frac{n}{u_n(s)}.
\]
On $0<u_n(s)\le1$, \eqref{eq:normalizedOriginPhi}, with
$m=\lfloor n/3\rfloor$, \red{and \eqref{eq:physicalScalarTailInput} give}
\[
 \Phi_n^-(s)=\phi_n(u_n(s))
 \le\phi_n(1)u_n(s)^m,
 \qquad
 \phi_n(1)\le\delta_T.
\]
Notice that the small quantity here is the radial value $\phi_n(1)$,
not the shifted value $\Phi_n(1)$.  Substitution in
\eqref{eq:physicalGaugeEndpointFormula} now gives the claimed polynomial
loss explicitly.  On $1\le u_n(s)\le R_n-T$,
\[
\begin{split}
 &\int_T^{R_n-1}
 \left(|K_{G,n}(T,s)|
 +(1+T)^{-1}|\partial_TK_{G,n}(T,s)|\right)
 2\Gamma_n^-(s)(\Phi_n^-(s))^2\,ds\\
 &\qquad\le
 Cn\left(R_n+(1+T)^{-1}\right)\delta_T^2
 \int_1^{R_n-T}\frac{du}{u}
 \le CR_n^4\delta_T^2 .
\end{split}
\]
Here coarse localization gives $n\simeq  R_n^2$; the harmless logarithm
has been absorbed into the displayed fourth power.  On
$0<u_n(s)\le1$, \eqref{eq:normalizedOriginPhi} gives instead
\[
\begin{split}
 &\int_{R_n-1}^{R_n}
 \left(|K_{G,n}(T,s)|
 +(1+T)^{-1}|\partial_TK_{G,n}(T,s)|\right)
 2\Gamma_n^-(s)(\Phi_n^-(s))^2\,ds\\
 &\qquad\le
 Cn\left(R_n+(1+T)^{-1}\right)\delta_T^2
 \int_0^1u^{2m-1}\,du
 \le CR_n^4\delta_T^2 .
\end{split}
\]
Consequently
\bel{eq:physicalGaugeEndpointBound}
 |H_n^{\rm ph}(T)|
 +(1+T)^{-1}|(H_n^{\rm ph})'(T)|
 \le CR_n^4\delta_T^2 .
\ee
The same argument with $T$ replaced by any
$\tau\in[T,R_n)$ gives the corresponding uniform estimate for
$H_n^{\rm ph}(\tau)$ and its derivative.
The scalar equation is
\[
\begin{split}
 (\Psi_n^{\rm ph})''
 &=
 V_n^0\Psi_n^{\rm ph}+
 (2\Gamma_n^0H_n^{\rm ph}
 +(H_n^{\rm ph})^2+\beta(\Phi_n^-)^2)
 \Psi_n^{\rm ph}.
\end{split}
\]
The regular-origin solutions $\Psi_n^{\rm ph}$ and $D_n^0$ have the
same Frobenius exponent.  Their Wronskian therefore vanishes at
$\eta=R_n$.  With $W(f,g)=fg'-f'g$, integration  gives
\bel{eq:physicalScalarEndpointFormula}
\begin{split}
 W(\Psi_n^{\rm ph},D_n^0)(T)
 &=
 \int_T^{R_n}
 (2\Gamma_n^0H_n^{\rm ph}
 +(H_n^{\rm ph})^2+\beta(\Phi_n^-)^2)
 \Psi_n^{\rm ph}D_n^0\,ds .
\end{split}
\ee
On $T\le s\le R_n-1$,
\[
 \frac{D_n^0(s)}{D_n^0(T)}
 \le Ce^{-c(s^2-T^2)},\qquad
 |\Psi_n^{\rm ph}(s)|\le R_n^{1/2}\delta_T,
\]
and \eqref{eq:physicalGaugeEndpointBound} bounds
$H_n^{\rm ph}$.  The integral over this interval is therefore at most
\[
        \frac{CR_n^9\delta_T^2}{1+T}|D_n^0(T)|.
\]
 On $0<u_n\le1$\red{, which corresponds to $\displaystyle R_n-1\leq s<R_n$, we have}
 \[
 |\Psi_n^{\rm ph}(R_n-u)|
 \le \delta_Tu^{m+1/2},
 \]
whereas the global decay ratio
\eqref{eq:regularOriginDecayRatio} gives
\[
 \frac{D_n^0(R_n-u)}{D_n^0(T)}
 \le
 Ce^{-c((R_n-u)^2-T^2)}
 \le Ce^{-c((R_n-1)^2-T^2)}.
\]
On this interval,
\[
 |\Gamma_n^0(R_n-u)|\le \frac{Cn}{u},\qquad
 |H_n^{\rm ph}(R_n-u)|\le CR_n^4\delta_T^2.
\]
Thus the factor $u^{m+1/2}$ in $\Psi_n^{\rm ph}$ makes every term
in \eqref{eq:physicalScalarEndpointFormula} integrable at $u=0$, and
direct integration gives
\[
\begin{split}
 &\int_0^1
 \left|
 2\Gamma_n^0H_n^{\rm ph}
 +(H_n^{\rm ph})^2+\beta(\Phi_n^-)^2
 \right|
 |\Psi_n^{\rm ph}D_n^0|\,du\\
 &\qquad\le
 CR_n^9\delta_T^2
 e^{-c((R_n-1)^2-T^2)}|D_n^0(T)|
 \le \frac{CR_n^9\delta_T^2}{1+T}|D_n^0(T)|.
\end{split}
\]
This is the same bound as on $T\le s\le R_n-1$.
Finally, \eqref{eq:regularOriginLogDerivativeBounds} gives
\[
 |D_n^0(T)|+(1+T)^{-1}|(D_n^0)'(T)|
 \simeq  |D_n^0(T)|.
\]
Thus the second quantity in
\eqref{eq:physicalEndpointDiscrepancy} is bounded by
$CR_n^{9}\delta_T^2$.  One may take $\sigma=9$.  Expanding
$R_n^\sigma\delta_T^2$ and choosing $C_{\log}$ \red{as} stated \red{above}
proves \eqref{eq:physicalEndpointDiscrepancySmall}.
\end{proof}

The preceding endpoint estimates allow us to compare the physical
solution with the finite one-parameter family.

\begin{Lemma}\label{lem:physicalNoGrowingChart}
 \red{Let $\displaystyle J\Subset(0,4)$, and }fix $K\ge K_0$ and choose
$C_{\log}\ge C_{\log,0}$ sufficiently large.  
\red{Recall the notation $\displaystyle \Lambda_n=C_{\log}\sqrt{\log R_n}$. }Put
\[
       C_{n,K}^{\rm ph}\red{:}=\Phi_n(-K).
\]
Then $C_{n,K}^{\rm ph}\in I_{\beta,K}$ for all sufficiently large $n$
\red{uniformly for $\beta\in J$},
\bel{eq:leftPhysicalNoGrowingRemainder}
\begin{split}
 &\sum_{j=0}^1(1+\eta)^{-j}
 \left|
 \partial_\eta^j
 (\Phi_n^-(\eta)
   -\Phi_{n,C_{n,K}^{\rm ph}}^-(\eta))
 \right|\le
 \red{C(K,J,C_{\log})}
 \red{e^{-c\Lambda_n^2}}
\end{split}
\ee
for $K\le\eta\le \Lambda_n$.  On the same interval,
\bel{eq:leftPhysicalGaugeRemainder}
\begin{split}
 &\sum_{j=0}^1(1+\eta)^{-j}
 \left|
 \partial_\eta^j
 (\Gamma_n^-(\eta)
   -\Gamma_{n,C_{n,K}^{\rm ph}}^-(\eta))
 \right|\le
 \red{C(K,J,C_{\log})}
 \red{e^{-c\Lambda_n^2}}
\end{split}
\ee
\red{for all $\beta\in J$.}
In particular, after increasing $C_{\log,0}$, both
right sides are $O_{K,J}(R_n^{-2})$.
\end{Lemma}

\begin{proof}
Qualitative boundary convergence gives
$C_{n,K}^{\rm ph}\to C_{\beta,K}^{\rm ph}$, so the parameter belongs to
$I_{\beta,K}$ for large $n$. \red{For simplicity, redenote $T=\Lambda_n$ and recall that}
\[
 \Psi_n^{\rm ph}=u_n^{1/2}\Phi_n^-
 .
\]
\red{As in Lemma \ref{lem:physicalEndpointDiscrepancy}, we set
\[
 H_n^{\rm ph}=\Gamma_n^- -\Gamma_n^0.
\]}
The  estimates \eqref{eq:fixedclear} and
\eqref{eq:fixedgauge}, together with
\eqref{eq:normalizedOriginPhi}, imply
\eqref{eq:physicalScalarTailInput}.  Lemma~\ref{lem:physicalEndpointDiscrepancy} therefore gives
\bel{eq:physicalEndpointDataSmall}
\begin{split}
 |\alpha|+(1+T)^{-1}|\gamma|
 &\le \red{Ce^{-c_1T^2}},\\
 \frac{|\omega|}{|D_n^0(T)|+(1+T)^{-1}|(D_n^0)'(T)|}
 &\le \red{Ce^{-c_1T^2}},
\end{split}
\ee
where
\[
 \alpha=H_n^{\rm ph}(T),\qquad
 \gamma=(H_n^{\rm ph})'(T),\qquad
 \omega=W(\Psi_n^{\rm ph},D_n^0)(T).
\]
We next write explicit homogeneous corrections for these three endpoint
discrepancies.  Set
\[
 E_{n,K}(\eta)
 =D_n^0(\eta)\int_K^\eta D_n^0(s)^{-2}\,ds .
\]
Then $E_{n,K}(K)=0$ and $W(E_{n,K},D_n^0)=-1$.  Hence
\bel{eq:scalarEndpointCorrection}
        P_{S,n}(\eta)=-\omega E_{n,K}(\eta)
\ee
satisfies
\[
        P_{S,n}(K)=0,\qquad
        W(P_{S,n},D_n^0)(T)=\omega .
\]
Writing $u_T=R_n-T$, define
\bel{eq:gaugeEndpointCorrection}
\begin{split}
 A_T&=\frac{1}{2}\left(\frac{\alpha}{u_T}-\gamma\right),\\
 B_T&=\frac{1}{2}\left(u_T\alpha+u_T^2\gamma\right),\\
 P_{G,n}(\eta)&=A_Tu_n(\eta)+B_Tu_n(\eta)^{-1}.
\end{split}
\ee
Direct substitution, using $u_n'=-1$, gives
\[
        P_{G,n}(T)=\alpha,\qquad
        P_{G,n}'(T)=\gamma .
\]
Decrease, if necessary, the exponent used in the weights $w_{K,p_0}$,
and denote the resulting exponent by $c_{\rm w}$, chosen so that
$0<c_{\rm w}<c_1$.  The reduction-of-order formula and
\eqref{eq:regularOriginDecayRatio} give
\bel{eq:scalarEndpointCorrectionBound}
\begin{split}
 \|u_n^{-1/2}P_{S,n}\|_{X_{p_0}([K,T])}
 \le
 C_K(1+T)^C e^{c_{\rm w}T^2}
 \frac{|\omega|}{|D_n^0(T)|+(1+T)^{-1}|(D_n^0)'(T)|}.
\end{split}
\ee
Indeed, $E_{n,K}$ is exactly the function $\ell_n$ used in
Lemma~\ref{lem:regularOriginMixedGreen}.  The estimate
\[
 |E_{n,K}(\eta)|+(1+\eta)^{-1}|E_{n,K}'(\eta)|
 \le \frac{C}{(1+\eta)D_n^0(\eta)}
\]
was proved there.  Combining it with
\eqref{eq:regularOriginLogDerivativeBounds}, and using
$D_n^0(T)/D_n^0(\eta)\le1$, gives
\eqref{eq:scalarEndpointCorrectionBound}.
The explicit formula \eqref{eq:gaugeEndpointCorrection} and
$u_n(\eta)\simeq  R_n$ on $[K,T]$ give
\bel{eq:gaugeEndpointCorrectionBound}
 \|P_{G,n}\|_{X_{p_0}([K,T])}
 \le C_KR_n(1+T)^Ce^{c_{\rm w}T^2}
 \{|\alpha|+(1+T)^{-1}|\gamma|\}.
\ee
Since
$T=C_{\log}\sqrt{\log R_n}$ and $c_{\rm w}<c_1$, increasing
$C_{\log,0}$ \red{and using \eqref{eq:physicalEndpointDataSmall}} gives
\bel{eq:combinedEndpointCorrectionBound}
\begin{split}
 \|u_n^{-1/2}P_{S,n}\|_{X_{p_0}([K,T])}
 +\|P_{G,n}\|_{X_{p_0}([K,T])}
 \le \red{C(K,J,C_{\log})}\red{e^{-cT^2}}.
\end{split}
\ee
Let
\[
 Z_n^{\rm ph}=(\Phi_n^-,H_n^{\rm ph}),\qquad
 Z_{n,C}=
 \bigl(\Phi_{n,C}^-,
       \Gamma_{n,C}^- -\Gamma_n^0\bigr),
\]
with $C=C_{n,K}^{\rm ph}$, and set
\[
        \mathbf P_n=(u_n^{-1/2}P_{S,n},P_{G,n}).
\]
The physical pair \red{$Z_n^{ph}$} and $Z_{n,C}$ have the same scalar value at $K$.
After subtracting $\mathbf P_n$, their \red{Higgs} scalar difference \red{multiplied by $u_n^{1/2}$} has zero value at $K$
and zero Wronskian with $D_n^0$ at $T$; their magnetic difference has
zero value and derivative at $T$.  Applying the two Green operators gives
the exact identity
\bel{eq:physicalFamilyDifference}
 Z_n^{\rm ph}-Z_{n,C}
 =
 \mathbf P_n+\mathbf G_n
 (\mathbf N_n(Z_n^{\rm ph})-\mathbf N_n(Z_{n,C})),
\ee
where
\[
 \mathbf G_n
 =\operatorname{diag}(\widehat{\mathcal G}_{S,n},
                      \mathcal G_{G,n})
\]
and
\bel{eq:leftFiniteNonlinearity}
 \mathbf N_n(\Phi,H)=
 \begin{pmatrix}
 (2\Gamma_n^0H+H^2+\beta\Phi^2)\Phi\\
 2(\Gamma_n^0+H)\Phi^2
 \end{pmatrix}.
\ee
The estimates \eqref{eq:movingclear},
\eqref{eq:qSharp}, and \eqref{eq:qSharpDer} imply, for fixed large $K$,
\[
        \|Z_n^{\rm ph}\|_{\mathfrak X_{n,K}}\le r_K
\]
for all large $n$.  Thus both \red{$Z_n^{ph}$ and $Z_{n,C}$} lie in the ball
\eqref{eq:leftFixedPointBall}.  Its Lipschitz estimate and
\eqref{eq:combinedEndpointCorrectionBound} give
\[
\begin{split}
 \|Z_n^{\rm ph}-Z_{n,C}\|_{\mathfrak X_{n,K}}
 &\le \frac{1}{3}
 \|Z_n^{\rm ph}-Z_{n,C}\|_{\mathfrak X_{n,K}}+\red{C(K,J,C_{\log})}\red{e^{-cT^2}}.
\end{split}
\]
Absorbing the first term proves
\eqref{eq:leftPhysicalNoGrowingRemainder} and
\eqref{eq:leftPhysicalGaugeRemainder}.
\end{proof}

\subsection{Comparison of the two endpoint conditions}

The preceding construction gives the required estimate at $y=-K$.

\begin{Corollary}\label{cor:leftEndpointResidual}
Let $J_0\Subset(0,4)$ be a fixed compact interval and $K$ be fixed and sufficiently large\red{, and assume that $\beta\in J_0$}.  Then
there is a sequence $\varepsilon_{n,K}^-\to0$ such that
\bel{eq:leftProjectedBoundaryResidual}
 |\mathcal R_{\beta,K}^-\mathsf C_-V_n|
 \le \red{C_{K,J_0}}R_n^{-1}
      +\varepsilon_{n,K}^-|\mathsf C_-V_n|
\ee
for all sufficiently large $n$.
\end{Corollary}

\begin{proof}
Set
\[
 C_{n,K}^{\rm ph}=\Phi_n(-K),
 \qquad
 C_{\beta,K}^{\rm ph}=\Phi_\beta(-K).
\]
Lemma~\ref{lem:physicalNoGrowingChart}, with $C_{\log}$ chosen as in that
lemma and the trace definitions
\eqref{eq:leftTraceMaps} give
\bel{eq:leftPhysicalTraceExpansion}
 \mathsf C_-V_n
 =\Theta_{n,K}^-(C_{n,K}^{\rm ph})
  -\Theta_{\beta,K}^-(C_{\beta,K}^{\rm ph})
 +d_{n,K}^-,
 \qquad
 |d_{n,K}^-|\le \red{C_{K,J_0}}R_n^{-2}.
\ee
The trace comparison \eqref{eq:leftTraceComparison} gives
\[
 \left|
 \Theta_{n,K}^-(C_{n,K}^{\rm ph})
 -\Theta_{\beta,K}^-(C_{n,K}^{\rm ph})
 \right|\le \red{C_{K,J_0}}R_n^{-1}.
\]
Qualitative boundary convergence implies
$C_{n,K}^{\rm ph}\to C_{\beta,K}^{\rm ph}$.
\green{By
\eqref{eq:uniformQualitativeBoundary}, t}his convergence is uniform on $J_0$. \green{Moreover, the positive minimum property of
$\Phi_\beta(-K)$} \green{ensures that} $C_{n,K}^{\rm ph}$ \green{belongs to} $I_{\beta,K}$ for a
common sufficiently large $n$.  \green{The joint continuity property of the derivative in Lemma~\ref{lem:leftNonlinearChart} ensures that the Taylor remainder appearing below is uniformly controlled on $J_0$.} Since
$\Theta_{\beta,K}^-$ is differentiable at
$C_{\beta,K}^{\rm ph}$, there are vectors $q_{n,K}^-$ satisfying
\[
\begin{split}
 \Theta_{\beta,K}^-(C_{n,K}^{\rm ph})
 -\Theta_{\beta,K}^-(C_{\beta,K}^{\rm ph})
 &=D\Theta_{\beta,K}^-(C_{\beta,K}^{\rm ph})
   (C_{n,K}^{\rm ph}-C_{\beta,K}^{\rm ph})+q_{n,K}^-,\\
 |q_{n,K}^-|
 &\le \varepsilon_{n,K}^-
 |C_{n,K}^{\rm ph}-C_{\beta,K}^{\rm ph}|,
 \qquad \varepsilon_{n,K}^-\longrightarrow0.
\end{split}
\]
The linear term lies in $E_\beta^-(-K)$ by
\eqref{eq:leftTraceTangentLine}.  It is therefore annihilated by
$\mathcal R_{\beta,K}^-$.  Finally,
\[
 |C_{n,K}^{\rm ph}-C_{\beta,K}^{\rm ph}|
 =|\Phi_n(-K)-\Phi_\beta(-K)|
 \le |\mathsf C_-V_n|.
\]
Applying $\mathcal R_{\beta,K}^-$ to the preceding expansion proves
\eqref{eq:leftProjectedBoundaryResidual}.
\end{proof}

The corresponding right endpoint estimate uses only the family already
constructed in Lemma~\ref{lem:rightTailManifold}.

\begin{Corollary}\label{cor:rightEndpointResidual}
Let $K$ be fixed and sufficiently large.  Then
there is a sequence $\varepsilon_{n,K}^+\to0$ such that
\bel{eq:rightProjectedBoundaryResidual}
 |\mathcal R_{\beta,K}^+\mathsf C_+V_n|
 \le C_{K,\beta}R_n^{-1}
      +\varepsilon_{n,K}^+|\mathsf C_+V_n|
\ee
for all sufficiently large $n$.
\end{Corollary}

\begin{proof}
Lemma~\ref{lem:rightTailManifold} gives $z_n\to0$\red{, where $z_n=(\widehat{A}_n-G_\beta,\widehat{B}_n-H_\beta)$,} and
\[
 \mathsf C_+V_n
 =\Theta_{n,K}(z_n)-\Theta_{\beta,K}(0).
\]
\eqref{eq:rightTailManifoldC2} gives
\[
 \sup_{z\in B_\rho(0)}
 |\Theta_{n,K}(z)-\Theta_{\beta,K}(z)|
 \le C_{K,\beta}R_n^{-1}.
\]
The derivative $D\Theta_{\beta,K}(0)$ is injective by
\eqref{eq:rightTailTangentPlane}.  Hence, after reducing $\rho$ if
necessary, the limiting map satisfies
\[
 |z|\le C_{K,\beta}
 |\Theta_{\beta,K}(z)-\Theta_{\beta,K}(0)|,
 \qquad z\in B_\rho(0).
\]
The preceding two displays and the identity for $\mathsf C_+V_n$ imply
\[
 |z_n|\le C_{K,\beta}(|\mathsf C_+V_n|+R_n^{-1}).
\]
Since $z_n\to0$ and $\Theta_{\beta,K}$ is differentiable at zero,
\[
 \Theta_{\beta,K}(z_n)-\Theta_{\beta,K}(0)
 =D\Theta_{\beta,K}(0)z_n+q_{n,K}^+,
 \qquad
 |q_{n,K}^+|\le o(1)|z_n|.
\]
The $C^2$ bounds in Lemma~\ref{lem:rightTailManifold} \green{imply that}
$|q_{n,K}^+|\le C_{K,J_0}|z_n|^2$.
\green{By the amplitude convergence proved in Proposition~\ref{prop:exteriorBessel}, we deduce that t}he convergence $z_n\to0$ is uniform on $J_0$.
\green{This shows that} the $o(1)$ \green{error} here and the sequence in
\eqref{eq:rightProjectedBoundaryResidual} may be chosen uniformly.
The linear term lies in $E_\beta^+(K)$ by
\eqref{eq:rightTailTangentPlane} and is annihilated by
$\mathcal R_{\beta,K}^+$.  Combining the last three displays proves
\eqref{eq:rightProjectedBoundaryResidual}.
\end{proof}

Combining the two corollaries, for $\beta\in J_0$, we have
for a sequence $\varepsilon_{n,K}\to0$,
\bel{eq:finiteWindowBoundaryResidual}
\begin{split}
 &|\mathcal R_{\beta,K}^-\mathsf C_-V_n|
 +|\mathcal R_{\beta,K}^+\mathsf C_+V_n|\le
 C_{K,\beta}R_n^{-1}
 +\varepsilon_{n,K}
 \bigl(|\mathsf C_-V_n|+|\mathsf C_+V_n|\bigr).
\end{split}
\ee

\subsection{Conclusion of the fixed-interval argument}

\begin{Lemma}\label{lem:fixedWindowQuantRate}
Let $J_0\Subset(0,4)$ be a compact interval and $\beta\in J_0$.   For every fixed
$M<\infty$,
\[
 \norm{(\Phi_n,\Gamma_n)-(\Phi_\beta,\Gamma_\beta)}
       _{H^2([-M,M])}
 \Le C_{M,\beta}R_n^{-1}
\]
for all sufficiently large $n$.
\end{Lemma}

\begin{proof}
Fix $K>M+4$ large enough for the two endpoint constructions, and apply
Lemma~\ref{lem:finiteWindowLinearEstimate} to $V_n$ on $[-K,K]$.
First note that Lemma~\ref{lem:qualitativeBoundary} gives
$V_n\to0$ in $C^1([-K,K])$.  Equations
\eqref{eq:boundaryDifferenceEquation} and
\eqref{eq:boundaryNonlinearRemainder} then imply
$\calA_\beta V_n\to0$ in $L^2([-K,K])$.  The elementary estimate
\eqref{eq:fixedWindowElementaryODE} therefore gives
\bel{eq:fixedWindowH2Smallness}
        \|V_n\|_{H^2([-K,K])}\longrightarrow0.
\ee
\green{By \eqref{eq:uniformQualitativeBoundary}  and the uniform coefficient bounds on $[-K,K]$, a}ll three convergences are uniform for $\beta\in J_0$ and
for the choice of \green{the} radial minimizer.
We can consequently use \eqref{eq:boundaryNonlinearLipschitz} with
$W=0$.  Together with \eqref{eq:boundaryDifferenceEquation}, it gives
\[
 \|\calA_\beta V_n\|_{L^2([-K,K])}
 \le C_{K,\beta}R_n^{-1}
      +C_{K,\beta}\|V_n\|_{H^2([-K,K])}^2.
\]
The trace theorem gives
\[
 |\mathsf C_-V_n|+|\mathsf C_+V_n|
 \le C_K\|V_n\|_{H^2([-K,K])}.
\]
Substitution of these two estimates and
\eqref{eq:finiteWindowBoundaryResidual} into
\eqref{eq:finiteWindowOverdet} yields, after incorporating the uniform inverse and trace constants
into $\varepsilon_{n,K}$,
\[
 \|V_n\|_{H^2([-K,K])}
 \le C_{K,\beta}R_n^{-1}
      +\varepsilon_{n,K}\|V_n\|_{H^2([-K,K])}
      +C_{K,\beta}\|V_n\|_{H^2([-K,K])}^2.
\]
The constants in this inequality may be replaced by a common
$C_{K,J_0}$.  \green{Once we fix} $K$, \green{we} choose $n$ \green{large enough} that
$\varepsilon_{n,K}\le1/4$ and
$C_{K,J_0}\|V_n\|_{H^2([-K,K])}\le1/4$, uniformly on $J_0$.
\green{Absorbing these} last two terms to the left gives the claimed bound, with
at most twice the preceding constant. Hence
\[
        \|V_n\|_{H^2([-K,K])}
        \le C_{K,\beta}R_n^{-1}.
\]
Restricting to $[-M,M]$ completes the proof.
\end{proof}

We may now state the quantitative boundary-layer result.

\begin{Theorem}\label{thm:boundaryconvergence}
Let $J_0\Subset(0,4)$ be a compact interval and $\beta\in J_0$. Then, for every $M<\infty$,
\bel{eq:boundaryH2conv}
 \norm{(\Phi_n,\Gamma_n)-(\Phi_\beta,\Gamma_\beta)}
       _{H^2([-M,M])}
 \Le C_{M,\beta}R_n^{-1}
\ee
for all sufficiently large $n$.  Consequently,
\bel{eq:boundaryEnergyConv}
 \int_{-M}^M e[\Phi_n,\Gamma_n](y)\,dy
 =
 \int_{-M}^M e_\beta(y)\,dy
 +O_{M,\beta}(R_n^{-1}),
\ee
where
\[
 e[\Phi,\Gamma]
 =\Phi'^2+\frac{1}{2}\Gamma'^2+\Gamma^2\Phi^2
  +\frac{\beta}{2}(1-\Phi^2)^2.
\]
\end{Theorem}

\begin{proof}
The estimate \eqref{eq:boundaryH2conv} is
Lemma~\ref{lem:fixedWindowQuantRate}.  Since
$H^2([-M,M])$ embeds continuously in $C^1([-M,M])$, each term in the
polynomial density $e[\Phi,\Gamma]$ is locally Lipschitz as a function of
$(\Phi,\Gamma)$ in the $H^2$ norm.  Integrating the resulting pointwise
estimate proves \eqref{eq:boundaryEnergyConv}.
\end{proof}

\section{Quantitative estimates in the left and right tails}
\label{sec:leftMatching}

The  rate of convergence established in Theorem~\ref{thm:boundaryconvergence} on a fixed interval 
controls the translated profiles only while $y$ remains bounded.  We now need to 
obtain analogous control outside of a fixed interval, in other words, in the tails.  Three distinct conclusions are
required later.  First, the Weber formula in the core--interface overlap  has to be
uniform on an interval whose length tends to infinity.  Second, the
$O(R_n^{-1})$ estimate at $y=-K$ will  be propagated through the part of
the left tail extending to $O(\!\sqrt{\log R_n})$,  as this is the range used in
the refined energy calculation.  Third, the estimate at $y=K$ must be
propagated through the entire right tail.

\red{We encounter two} fixed large parameters \red{in the analysis}, \red{each of which plays a different role}.  The
\red{parameter} $L$ \red{specifies} the matching point in the core to interface overlap.  In a
statement involving $L$, one first lets $n\to\infty$ and then
$L\to\infty$.  The \red{parameter} $K$ determines the endpoints from which the
quantitative boundary estimate is propagated to \red{both the normal and superconducting ends}\red{, and i}t remains
fixed as $n\to\infty$.  \red{The parameters} $L$ and $K$ \red{are completely independent of each other}. Throughout this section, we fix an arbitrary compact interval $J_0\Subset(0,4)$.

\subsection{Uniform Weber matching in the left overlap}\label{subsec:leftMatching}

Fix
\[
        0<\alpha<\widehat\alpha<\frac13
\]
as in \eqref{eq:twoOverlapExponents}, and fix $L$ sufficiently large.  Let
$I_n\subset[L,R_n^\alpha]$ be any sequence of intervals such that
\bel{eq:leftMatchingUniformity}
 \sup_{\eta\in I_n}
 \left\{
 |b_n-\sqrt\beta|\,\eta^2+\frac{\eta^3}{R_n}
 \right\}\longrightarrow0.
\ee
Then Proposition~\ref{prop:coreWKB}, its differentiated remainder estimate,
and the coefficient estimate in the proof of
Corollary~\ref{cor:leftCoefficientMatching} give
\bel{eq:leftMatchingOnIntervals}
\begin{split}
 \phi_n(R_n-\eta)
 &=
 C_\beta\eta^{(\sqrt\beta-1)/2}
 e^{-\sqrt\beta\,\eta^2/2}
 \left(1+\varepsilon_{n,L}(\eta)\right),
 \qquad \eta\in I_n,
\end{split}
\ee
where
\[
 \sup_{\eta\in I_n}
 \left\{
 |\varepsilon_{n,L}(\eta)|
 +(1+\eta)^{-1}|\varepsilon_{n,L}'(\eta)|
 \right\}
 \le C_{J_0}L^{-2}+o_{n,J_0,L}(1).
\]
Indeed, the first term in \eqref{eq:leftMatchingUniformity} controls the
difference between
\[
 \eta^{(\beta/b_n-1)/2}e^{-b_n\eta^2/2}
 \quad\hbox{and}\quad
 \eta^{(\sqrt\beta-1)/2}e^{-\sqrt\beta\,\eta^2/2}.
\]
The second term controls the $O(\eta^3/R_n)$ error obtained by integrating
the radial correction in the action.  On the larger interval
$[L,T_n]$, used to construct the WKB basis, this same contribution is
$O(R_n^{3\widehat\alpha-1})=o(1)$.  This is the reason for the restriction
$\widehat\alpha<1/3$.

Under the localization estimate \eqref{eq:SLhyp}, condition
\eqref{eq:leftMatchingUniformity} holds on the entire left overlap.  In fact,
uniformly for $L\le\eta\le R_n^\alpha$,
\[
 |b_n-\sqrt\beta|\,\eta^2=O_{J_0}(R_n^{2\alpha-1}),
 \qquad
 \frac{\eta^3}{R_n}=O(R_n^{3\alpha-1}),
\]
and both quantities tend to zero.  Consequently,
\[
 \lim_{L\to\infty}\limsup_{n\to\infty}
 \sup_{L\le\eta\le R_n^\alpha}
 \left|
 \frac{\phi_n(R_n-\eta)}
 {C_\beta\eta^{(\sqrt\beta-1)/2}
  e^{-\sqrt\beta\,\eta^2/2}}
 -1
 \right|=0.
\]

\subsection{Propagation through the left tail}\label{subsec:leftTail}

We next propagate the $O(R_n^{-1})$ estimate at $y=-K$ into the
logarithmically growing interval needed in Section~\ref{sec:refinedEnergy}.
\red{We recall the notations}
\[
 \Phi_n^-(\eta):=\Phi_n(-\eta)=\phi_n(R_n-\eta),
 \qquad
 \Gamma_n^-(\eta):=\Gamma_n(-\eta),
 \qquad \eta\ge K.
\]
The proof compares the vortex tail with the CHMO tail through the
one-parameter families constructed in
Lemma~\ref{lem:leftNonlinearChart}.  Three errors occur: the difference
between the physical solution and the member of the finite radial family
having the same value at $\eta=K$; the difference between the finite and
limiting families at the same parameter; and the change in the limiting
family caused by the $O(R_n^{-1})$ displacement of that parameter.  The
next lemma estimates these three terms separately.

\begin{Lemma}\label{lem:leftQuantTailPropagation}
Assume $\beta\in J_0$.  There \red{exist 
$K_0<\infty$ and} $c>0$ with the following property.  For every fixed
$C_{\log}<\infty$, every $K\ge K_0$ \red{and all sufficiently large $n$}, one has
\bel{eq:leftQuantitativeScalarTail}
\begin{split}
 &\sum_{j=0}^1(1+\eta)^{-j}
 \left|
 \partial_\eta^j
 \left(\Phi_n^-(\eta)-\Phi_\beta(-\eta)\right)
 \right|\le
 \red{C(K,J_0,C_{\log})R_n^{-1}
 (1+\eta)^2 e^{-c\eta^2}}
\end{split}
\ee
for
\[
        K\le\eta\le C_{\log}\sqrt{\log R_n}.
\]
Let $P_n$ be the explicit error defined in
\eqref{eq:exactAffineDefect}.  On the same interval,
\bel{eq:leftQuantitativeGaugeTail}
\begin{split}
 &\sum_{j=0}^1(1+\eta)^{-j}
 \left|
 \partial_\eta^j
 \left(
 \Gamma_n^-(\eta)-\Gamma_\beta(-\eta)-P_n(\eta)
 \right)
 \right|\le
 \red{C(K,J_0,C_{\log})R_n^{-1}
 (1+\eta)^2 e^{-c\eta^2}}.
\end{split}
\ee
In particular,
\bel{eq:leftQuantitativeGaugeAbsolute}
\begin{split}
 &\sum_{j=0}^1(1+\eta)^{-j}
 \left|
 \partial_\eta^j
 \left(\Gamma_n^-(\eta)-\Gamma_\beta(-\eta)\right)
 \right| \le
 \red{C(K,J_0,C_{\log})R_n^{-1}(1+\eta)^2}.
\end{split}
\ee
\end{Lemma}

\begin{proof}
Put
\[
 C_{n,K}^{\rm ph}:=\Phi_n^-(K),
 \qquad
 C_{\beta,K}^{\rm ph}:=\Phi_\beta(-K).
\]
The fixed-interval estimate in
Lemma~\ref{lem:fixedWindowQuantRate}, together with
$\partial_\eta\Phi_n^-(K)=-\Phi_n'(-K)$, gives
\bel{eq:leftPhysicalTraceRate}
\begin{split}
 &|C_{n,K}^{\rm ph}-C_{\beta,K}^{\rm ph}|
 +\left|
 \partial_\eta\Phi_n^-(K)+\Phi_\beta'(-K)
 \right|
 \le C_{K,J_0}R_n^{-1}.
\end{split}
\ee
In particular, $C_{n,K}^{\rm ph}\in I_{\beta,K}$ for all sufficiently
large $n$.
  Choose
\[
 \widetilde C_{\log}\ge
 \max\{C_{\log},C_{\log,0}\}
\]
large enough that \red{both Lemmas \ref{lem:leftNonlinearChart} and \ref{lem:physicalNoGrowingChart} are applicable}.  \red{We note that the}
constant $c$ \red{will be independent of this choice}.  Construct the finite and limiting
one-parameter families of Lemma~\ref{lem:leftNonlinearChart} on
$[K,\widetilde\Lambda_n]$, and restrict their estimates to
$[K,C_{\log}\sqrt{\log R_n}]$.

We first prove the scalar estimate.  Add and subtract the members of the
finite and limiting families with parameter $C_{n,K}^{\rm ph}$:
\[
\begin{split}
 \Phi_n^-(\eta)-\Phi_\beta(-\eta)
 &=
 \underbrace{
 \Phi_n^-(\eta)-\Phi_{n,C_{n,K}^{\rm ph}}^-(\eta)
 }_{\text{physical solution minus finite family}}\\
 &\quad+
 \underbrace{
 \Phi_{n,C_{n,K}^{\rm ph}}^-(\eta)
 -\Phi_{\beta,C_{n,K}^{\rm ph}}^-(\eta)
 }_{\text{finite family minus limiting family}}\\
 &\quad+
 \underbrace{
 \Phi_{\beta,C_{n,K}^{\rm ph}}^-(\eta)
 -\Phi_{\beta,C_{\beta,K}^{\rm ph}}^-(\eta)
 }_{\text{change of parameter in the limiting family}}.
\end{split}
\]
Here
\[
        \Phi_{\beta,C_{\beta,K}^{\rm ph}}^-(\eta)
        =\Phi_\beta(-\eta).
\]
Lemma~\ref{lem:physicalNoGrowingChart}, specifically
\eqref{eq:leftPhysicalNoGrowingRemainder} , bounds the first term and its
first derivative by 
\red{\begin{align*}
C(K,J_0,C_{\log})
 (1+\eta)^2 e^{-a\tilde{\Lambda}^2_n},
 \end{align*}
where $a>0$ and $\tilde{\Lambda}_n=\tilde{C}_{\log}\sqrt{\log R_n}$. Upon shrinking $c>0$ if necessary, we may also assume that $c<a$, and we also require $\displaystyle a\tilde{C}_{\log}^2>1+cC^2_{\log}$. The first term and its derivative will then be bounded by
\begin{align*}
C(K,J_0,C_{\log})
 (1+\eta)^2 e^{-a\tilde{\Lambda}^2_n}&=C(K,J_0,C_{\log})
 (1+\eta)^2 e^{-a\tilde{C}^2_{\log}\log R_n}\\&<C(K,J_0,C_{\log})
 (1+\eta)^2 e^{-(cC^2_{\log}+1)\log R_n}\\
 &\leq C(K,J_0,C_{\log})
 (1+\eta)^2R_n^{-1} e^{-c\eta^2}
.\end{align*}
}
Estimate
\eqref{eq:leftChartScalar} \red{with $\displaystyle p=2$} bounds the second term and its first derivative by
\[
 C(K,\red{J_0},C_{\log})
  R_n^{-1}(1+\eta)^2e^{-c(\eta^2-K^2)}.
\]
For the third term, the fundamental theorem of calculus gives
\[
\begin{split}
 &\Phi_{\beta,C_{n,K}^{\rm ph}}^-(\eta)
 -\Phi_{\beta,C_{\beta,K}^{\rm ph}}^-(\eta)=
 (C_{n,K}^{\rm ph}-C_{\beta,K}^{\rm ph})
 \int_0^1
 \partial_C
 \Phi_{\beta,
 C_{\beta,K}^{\rm ph}
 +t(C_{n,K}^{\rm ph}-C_{\beta,K}^{\rm ph})}^-(\eta)
 \,\dd t.
\end{split}
\]
Combining \eqref{eq:leftPhysicalTraceRate} with the parameter-derivative
estimate \eqref{eq:leftLimitingParameterDecay} bounds this term and its
first derivative by
\[
 C_{K,\red{J_0}}R_n^{-1}(1+\eta)^2
 e^{-c(\eta^2-K^2)}.
\]
Absorbing the fixed factor $e^{cK^2}$ into the constant and adding the
three estimates proves \eqref{eq:leftQuantitativeScalarTail}. 
The bounds on the magnetic part follow from 
\[
\begin{split}
 \Gamma_n^-\red{(\eta)}-\Gamma_\beta(-\red{\eta})-P_n\red{(\eta)}
 &=
 \underbrace{
 \Gamma_n^-\red{(\eta)}-\Gamma_{n,C_{n,K}^{\rm ph}}^-\red{(\eta)}
 }_{\text{physical solution minus finite family}}\\
 &\quad+
 \underbrace{
 \Gamma_{n,C_{n,K}^{\rm ph}}^-\red{(\eta)}
 -\Gamma_{\beta,C_{n,K}^{\rm ph}}^-\red{(\eta)}-P_n\red{(\eta)}
 }_{\text{finite family minus limiting family}}\\
 &\quad+
 \underbrace{
 \Gamma_{\beta,C_{n,K}^{\rm ph}}^-\red{(\eta)}
 -\Gamma_{\beta,C_{\beta,K}^{\rm ph}}^-\red{(\eta)}
 }_{\text{change of parameter in the limiting family}}.
\end{split}
\]
The three terms, including their first derivatives, are controlled,
respectively, by \eqref{eq:leftPhysicalGaugeRemainder},
\eqref{eq:leftChartGauge}, and the $C$-derivative estimate
\eqref{eq:leftLimitingParameterDecay} together with
\eqref{eq:leftPhysicalTraceRate}.  Their sum proves
\eqref{eq:leftQuantitativeGaugeTail}.

Finally, \eqref{eq:exactAffineDefect} and \eqref{eq:SLhyp} imply, for
$\eta\le C_{\log}\sqrt{\log R_n}$,
\[
 |P_n(\eta)|
 +(1+\eta)^{-1}|P_n'(\eta)|
 \le C_{\red{J_0},C_{\log}}R_n^{-1}(1+\eta)^2.
\]
Adding this \red{finite-radius profile} correction to
\eqref{eq:leftQuantitativeGaugeTail} proves
\eqref{eq:leftQuantitativeGaugeAbsolute}.
\end{proof}

\subsection{Propagation through the right tail}\label{subsec:rightTail}

Recall the shifted Higgs defect
\[
\red{\eta_\beta}(y):=1-\Phi_\beta(y).
\]
\red{We also introduce $\displaystyle \eta^+_n(y):=1-\Phi_n(y)$, and redenote}
\[
 \red{U_n:=(\Gamma_n,\eta^+_n)},
 \qquad
 \red{U_\beta:=(\Gamma_\beta,\eta_\beta)}.
\]
Lemma~\ref{lem:rightTailManifold} constructed the two-parameter families of
decaying radial and planar tails and their Cauchy data maps
$\Theta_{n,K}$ and $\Theta_{\beta,K}$ \blue{(which we also refer to as the traces)}.  We first use the fixed-interval
estimate to control the two amplitude parameters.  We then compare the
radial and planar tails at the same parameter and finally account for the
parameter displacement.

\begin{Lemma}\label{lem:rightQuantTailPropagation}
Assume $\beta\in J_0$.  There are
$K_1<\infty$ and $c>0$ such that, for every $K\ge K_1$, there is
a constant $C_{K,\red{J_0}}$ \red{such that for all sufficiently large $n$,}
\[
\begin{split}
 &\sum_{j=0}^2
 |\partial_y^j(\Gamma_n(y)-\Gamma_\beta(y))|
 +\sum_{j=0}^2
 |\partial_y^j(\red{\eta^+_n}(y)-\red{\eta}_\beta(y))|\le C_{K,\red{J_0}}R_n^{-1}e^{-cy},
 \qquad y\ge K.
\end{split}
\]
\end{Lemma}

\begin{proof}
Let
\[
 z_n=
 \bigl(\widehat A_n-G_\beta,\widehat B_n-H_\beta\bigr)
\]
be the parameter of the vortex tail, see Lemma~\ref{lem:rightTailManifold}.  By
\eqref{eq:physicalRightTailTrace},
\[
 \operatorname{Tr}_K U_n=\Theta_{n,K}(z_n),
 \qquad
 \operatorname{Tr}_K U_\beta=\Theta_{\beta,K}(0).
\]
The bound \eqref{eq:boundaryH2conv} on any fixed interval and \eqref{eq:rightTailManifoldC2} imply that
\[
\begin{split}
 |\Theta_{n,K}(z_n)-\Theta_{\beta,K}(0)|
 &\le C_{K,\red{J_0}}R_n^{-1},\\
 |\Theta_{n,K}(0)-\Theta_{\beta,K}(0)|
 &\le C_{K,\red{J_0}}R_n^{-1}.
\end{split}
\]
Hence \eqref{eq:rightTailParameterControl}, applied to the  trace map,
yields
\bel{eq:rightAmplitudeRate}
        |z_n|\le C_{K,\red{J_0}}R_n^{-1}.
\ee
The physical and limiting tails satisfy
\[
 U_n=U_{n,z_n},
 \qquad
 U_\beta=U_{\beta,0}.
\]
\blue{Here $U_{n,z}$ and $U_{\beta,z}$ are the fixed points introduced first in
Lemma~\ref{lem:rightTailManifold}, defined by the radial and planar tail
fixed-point equations \eqref{eq:radialTailFixedPoint} and
\eqref{eq:planarTailFixedPoint}, respectively; moreover,
$U_{\beta,0}$ is identified in \eqref{eq:rightTailBaseTrace}.
We write}
\bel{eq:rightTailTwoComparisons}
 U_n-U_\beta
 =(U_{n,z_n}-U_{\beta,z_n})
  +(U_{\beta,z_n}-U_{\beta,0}).
\ee
The first term compares the radial and planar equations at the same
parameters.  For $j=0,1$, estimate
\eqref{eq:rightTailWeightedComparison} gives
\[
 |\partial_y^j(U_{n,z_n}(y)-U_{\beta,z_n}(y))|
 \le
 C_{K,\red{J_0}}\omega_{R_n,K}(y)
 \left(e^{-\red{m}y}+e^{-\red{\mu_\beta}y}\right).
\]
The second term in \eqref{eq:rightTailTwoComparisons} is caused only by the
change in the parameters.  The derivative equation
\eqref{eq:rightTailFirstParameterDerivative} and the contraction estimate
\eqref{eq:rightTailContractionEstimate} give uniform parameter
derivatives for the limiting family.  Therefore
\[
 |\partial_y^j(U_{\beta,z_n}(y)-U_{\beta,0}(y))|
 \le
 C_{K,\red{J_0}}|z_n|
 \left(e^{-\red{m}y}+e^{-\red{\mu_\beta}y}\right),
 \qquad j=0,1.
\]
\red{We next} convert the comparison involving the saturating weight into
an absolute $O(R_n^{-1})$ estimate.  Let
$\lambda_G=\red{m}-\rho$ and
$\lambda_H=\red{\mu_\beta}-\rho$ be the exponents from
Lemma~\ref{lem:exteriorEntry}, and choose
\[
        0<c<\red{\inf_{\beta\in J_0}}\min\{\lambda_G,\lambda_H\}.
\]
\blue{Let $q\in\{\red{m},\red{\mu_\beta}\}$ (see \eqref{eq:mMuBetaDef}).  Since
$\lambda_G=\red{m}-\rho<\red{m}$ and $\lambda_H=\red{\mu_\beta}-\rho<\red{\mu_\beta}$,
our choice $c<\red{\inf_{\beta\in J_0}}\min\{\lambda_G,\lambda_H\}$ implies $q>c$.
We now claim that}
\bel{eq:saturatingToAbsolute}
 \omega_{R_n,K}(y)
 \left(e^{-\red{m}y}+e^{-\red{\mu_\beta}y}\right)
 \le C_{K,\red{J_0}}R_n^{-1}e^{-cy},
 \qquad y\ge K.
\ee
Indeed, fix $q\in\{\red{m},\red{\mu_\beta}\}$ and assume $y-K\le R_n$.  Then
\[
 \omega_{R_n,K}(y)e^{-qy}
 \le
 \frac{1+y-K}{R_n+K}e^{-(q-c)\red{(y-K)}}\red{e^{-(q-c)K}}e^{-cy}
 \le C_{K,\red{J_0}}R_n^{-1}e^{-cy}.
\]
If $y-K\ge R_n$, then
\[
 \omega_{R_n,K}(y)e^{-qy}
 \le e^{-(q-c)y}e^{-cy}
 \le C_{K,\red{J_0}}R_n^{-1}e^{-cy},
\]
where the final inequality follows from
$e^{-(q-c)R_n}\le \red{C_{J_0,c}}R_n^{-1}$.
Combining \eqref{eq:rightAmplitudeRate},
\eqref{eq:rightTailTwoComparisons}, and
\eqref{eq:saturatingToAbsolute} proves the assertion for $j=0,1$.

For the second derivatives, put
\[
 \Delta_G:=\Gamma_n-\Gamma_\beta,
 \qquad
 \Delta_{\red{\eta}}:=\red{\eta^+_n}-\red{\eta}_\beta,
 \qquad
 u:=R_n+y.
\]
Subtracting the radial and planar equations gives
\bel{eq:rightTailSecondDerivativeDifference}
\begin{split}
 \Delta_G''
 & =
 2\Delta_G
 +\left(
 \mathcal Q_{\beta,G}(U_n)
 -\mathcal Q_{\beta,G}(U_\beta)
 \right)
 -\frac{\Gamma_n'}{u}
 +\frac{\Gamma_n}{u^2},\\
 \Delta_{\red{\eta}}''
 & =
 2\beta\Delta_{\red{\eta}}
 +\left(
 \mathcal Q_{\beta,\red{\eta}}(U_n)
 -\mathcal Q_{\beta,\red{\eta}}(U_\beta)
 \right)
 -\frac{\red{(\eta^+_n)'}}{u}.
\end{split}
\ee
\red{Here, $\mathcal{Q}_{\beta,\eta}:=Q_{\beta,H}$, where the latter is defined in Section \ref{sec:exterior}.} The polynomial map $\mathcal Q_\beta$ is uniformly Lipschitz on the
bounded set containing the two tails, so its difference is bounded by
\[
 C_{J_0}(|\Delta_G|+|\Delta_{\red{\eta}}|).
\]
Moreover, $u^{-1}\le R_n^{-1}$, and
Lemma~\ref{lem:exteriorEntry} bounds
$\Gamma_n,\Gamma_n'$, and $\red{(\eta^+_n)'}$ exponentially.  Inserting the
already established estimates for $j=0,1$ into
\eqref{eq:rightTailSecondDerivativeDifference} gives
\[
 |\Delta_G''(y)|+|\Delta_{\red{\eta}}''(y)|
 \le C_{K,\red{J_0}}R_n^{-1}e^{-cy}.
\]
This completes the proof.
\end{proof}

\section{Refined energy and radius asymptotics}
\label{sec:refinedEnergy}

The   argument in Section~\ref{sec:preliminaryLocalization}
established \eqref{eq:SLhyp} but retained only an error which is small
relative to $R_n$.  Theorem~\ref{thm:boundaryconvergence} and the two tail
estimates in Section~\ref{sec:leftMatching} now permit a second pass with a
splitting point of size $\sqrt{\log R_n}$.  The core and boundary errors are
then bounded, and the terms depending on the artificial splitting point
cancel explicitly.  This yields both the reduced energy formula and the
reduced first-variation identity.  The latter determines the constant term
in the radius asymptotic.

\subsection{Boundary and exterior contributions with bounded error}

\blue{Using the crucial expansions~\eqref{eq:SLhyp}, the quantitative boundary analysis in
Section~\ref{sec:quantitativeBoundary} and the left/right matching estimates in
Section~\ref{sec:leftMatching} provide uniform control of the boundary and exterior
contributions when we split at $y=-M$ with $M\lesssim\sqrt{\log R_n}$.  The next lemma
packages these bounds in the precise form needed for the refined energy expansion.}

\begin{Lemma}\label{lem:boundaryExteriorEnergy}
Let $\beta\in J_0\Subset(0,4)$, with $J_0$ a fixed compact interval, and fix some finite $C_{\log}>0$.  For $d=h,j$ and \red{$n$ sufficiently large, uniformly for}
$1\le M\le C_{\log}\sqrt{\log R_n}$, define
\bel{eq:refinedBoundaryRemainder}
 \mathfrak R_{n,M}^d
 :=\int_{-M}^{\infty}(R_n+y)
       (d_n(y)-d_\beta(y))\,dy .
\ee
Then
\[
        |\mathfrak R_{n,M}^h|
        +|\mathfrak R_{n,M}^j|\le C_{\beta,C_{\log}}.
\]
Consequently,
\[
\begin{split}
 \int_{-M}^{\infty}(R_n+y)h_n(y)\,dy
 &=R_n\int_{-M}^{\infty}h_\beta(y)\,dy
   +\int_{-M}^{\infty}y h_\beta(y)\,dy+O(1),\\
 \int_{-M}^{\infty}(R_n+y)j_n(y)\,dy
 &=R_n\int_{-M}^{\infty}j_\beta(y)\,dy
   +\int_{-M}^{\infty}y j_\beta(y)\,dy+O(1),
\end{split}
\]
uniformly in the stated range of $M$, where $h_n,j_n$ are the densities defined in \eqref{eq:hnjn} and $h_\beta,j_\beta$ denote the corresponding expressions with $(\Phi_n,\Gamma_n)$ replaced by $(\Phi_\beta,\Gamma_\beta)$.
\end{Lemma}

\begin{proof}
Fix $K$ sufficiently large for both
Lemmas~\ref{lem:leftQuantTailPropagation}
and~\ref{lem:rightQuantTailPropagation}.  If $1\le M<K$, split
\eqref{eq:refinedBoundaryRemainder} at $y=K$. \red{The fixed window-estimate on the interval $[-K,K]$ implies the one for the integral corresponding to $[-M,K]$, while Lemma~\ref{lem:rightQuantTailPropagation} controls the integral on $[K,\infty)$. These directly imply the desired bound.} We may therefore assume \red{that} $M\ge K$
and split
\eqref{eq:refinedBoundaryRemainder} into the intervals
$[-M,-K]$, $[-K,K]$, and $[K,\infty)$.

On $[-K,K]$, Theorem~\ref{thm:boundaryconvergence} gives
\[
 \|(\Phi_n,\Gamma_n)-(\Phi_\beta,\Gamma_\beta)\|_{H^2([-K,K])}
 \le \frac{C_{K,\beta}}{R_n}.
\]
Sobolev embedding and the polynomial definitions of
$h$ and $j$ therefore imply
\[
 \int_{-K}^{K}(R_n+|y|)
 (|h_n\red{(y)}-h_\beta\red{(y)}|+|j_n\red{(y)}-j_\beta\red{(y)}|)\,dy\le C_{K,\beta}.
\]
On the left interval write $\eta=-y$.  The following algebraic identities
make clear where the Gaussian factor enters:
\bel{eq:leftDensityDifferences}
\begin{split}
 h_n-h_\beta
 &=(\Phi_n'-\Phi_\beta')(\Phi_n'+\Phi_\beta')+\frac{\beta}{2}
 \left((1-\Phi_n^2)^2-(1-\Phi_\beta^2)^2\right),\\
 j_n-j_\beta
 &=(\Gamma_n-\Gamma_\beta)(\Gamma_n+\Gamma_\beta)\Phi_n^2+\Gamma_\beta^2(\Phi_n-\Phi_\beta)(\Phi_n+\Phi_\beta)\\
 &\quad+\frac{\beta}{2}
 \left((1-\Phi_n^2)^2-(1-\Phi_\beta^2)^2\right).
\end{split}
\ee
Here all functions are evaluated at $y=-\eta$.  Lemma
\ref{lem:leftQuantTailPropagation} gives
\[
\begin{split}
 |\Phi_n(-\eta)-\Phi_\beta(-\eta)|
 &+(1+\eta)^{-1}
 |\Phi_n'(-\eta)-\Phi_\beta'(-\eta)|\le \red{\frac{C(K,\beta,C_{\log})}{R_n}(1+\eta)^2 e^{-c\eta^2}}.
\end{split}
\]
The absolute magnetic estimate
\eqref{eq:leftQuantitativeGaugeAbsolute} gives
\[
 |\Gamma_n(-\eta)-\Gamma_\beta(-\eta)|
 \le \red{\frac{C(K,\beta,C_{\log})(1+\eta)^2}{R_n}}
.
\]
\red{We fix $\alpha\in(0,1)$. For $n$ sufficiently large, we have $R_n^\alpha>C_{\log}\sqrt{\log R_n}$. }By Lemma~\ref{lem:interfaceUniformTails} and
Lemma~\ref{lem:clearingout},
\[
\begin{split}
 |\Phi_n(-\eta)|+|\Phi_\beta(-\eta)|
 &\le \red{C(1+\eta)^2 e^{-c\eta^2}},\\
 |\Phi_n'(-\eta)|+|\Phi_\beta'(-\eta)|
 &\le \red{C(1+\eta)^{3}e^{-c\eta^2}},\\
 |\Gamma_n(-\eta)|+|\Gamma_\beta(-\eta)|
 &\le C(1+\eta).
\end{split}
\]
Substitution in \eqref{eq:leftDensityDifferences} yields,
\[
 |h_n(-\eta)-h_\beta(-\eta)|
 +|j_n(-\eta)-j_\beta(-\eta)|
 \le \red{\frac{C(K,\beta,C_{\log})}{R_n}(1+\eta)^7 e^{-c\eta^2}}
\]
Since $R_n-\eta\le R_n$ and
$M\le C_{\log}\sqrt{\log R_n}$, integration over $K\le\eta\le M$
gives a bound independent of $n$ and $M$.
On $[K,\infty)$, Lemma~\ref{lem:rightQuantTailPropagation} gives
\[
 |h_n(y)-h_\beta(y)|+|j_n(y)-j_\beta(y)|
 \le \frac{C_{K,\beta}}{R_n}e^{-cy}.
\]
Therefore
\[
 \int_K^\infty(R_n+y)
 (|h_n-h_\beta|+|j_n-j_\beta|)\,dy
 \le C_{K,\beta}.
\]
Adding the estimates on the three intervals proves the bound for
\eqref{eq:refinedBoundaryRemainder}.  The two displayed expansions follow
from the exact identity
\[
 \int_{-M}^{\infty}(R_n+y)d_n\,dy
 =R_n\int_{-M}^{\infty}d_\beta\,dy
 +\int_{-M}^{\infty}y d_\beta\,dy
 +\mathfrak R_{n,M}^d 
\]
as desired.
\end{proof}
\subsection{Cancellation of the splitting point}
\blue{In order to estimate the energy and virial identities we split the $y$--integrals at $-M$, so that $M$ separates the interface  from the bulk. \red{The next lemma shows that once we add the leading terms $\frac{\beta}{4}(R-M)^2$ and $\frac{\beta}{2}(R-M)^2$ to the integral expressions involving $h_\beta$ and $j_\beta$, respectively, the resulting expressions are independent of $M$, modulo exponentially small errors arising from the normal end. This allows us to recover the universal surface contribution $s_\beta R$ along with a remainder of size $O_J(1)$, which will in turn enable us to choose $M$ conveniently later on.}}

\begin{Lemma}\label{lem:cutoffcancel}
Let \red{$\displaystyle J\Subset(0,4)$, and}
\[
 \rho_h(y)=h_\beta(y)-\frac{\beta}{2}{\bf1}_{(-\infty,0)}(y),
 \qquad
 \rho_j(y)=j_\beta(y)-\frac{\beta}{2}{\bf1}_{(-\infty,0)}(y).
\]
There exists $c=c_J>0$ such that, for every integer $n\ge1$, every
$R>0$, and every $M\ge1$,
\[
\begin{split}
 &\frac{n^2}{R^2}
 +\frac{\beta}{4}(R-M)^2
 +R\int_{-M}^{\infty}h_\beta(y)\,dy
 +\int_{-M}^{\infty}y h_\beta(y)\,dy\\
 &\qquad
 =\frac{n^2}{R^2}+\frac{\beta}{4}R^2+s_\beta R
 +O_J(1)+O_J(Re^{-cM^2}),
\end{split}
\]
and
\[
\begin{split}
 &-\frac{2n^2}{R^2}
 +\frac{\beta}{2}(R-M)^2
 +2R\int_{-M}^{\infty}j_\beta(y)\,dy
 +2\int_{-M}^{\infty}y j_\beta(y)\,dy\\
 &\qquad
 =-\frac{2n^2}{R^2}+\frac{\beta}{2}R^2+s_\beta R
 +O_J(1)+O_J(Re^{-cM^2})\red{,}
\end{split}
\]
\red{uniformly in $\beta\in J$.}
The constants in the $O_J(1)$ terms are independent of $R$ and $M$.
\end{Lemma}

\begin{proof}
By Lemma~\ref{lem:exactReducedDensities},
\[
 \rho_h,\rho_j,y\rho_h,y\rho_j\in L^1(\mathbb R),
 \qquad
 \int_{\mathbb R}\rho_h=s_\beta,
 \qquad
 2\int_{\mathbb R}\rho_j=s_\beta.
\]
The Gaussian estimates in
Lemma~\ref{lem:interfaceUniformTails} for the left tail imply that, for some
$N<\infty$ and $c_0>0$,
\[
 |\rho_h(y)|+|\rho_j(y)|
 \le C(1+|y|)^N e^{-c_0y^2},
 \qquad y\le-1.
\]
Choose $0<c<c_0$.  The polynomial factor is absorbed by decreasing the
Gaussian exponent:
\[
\begin{split}
 \int_{-\infty}^{-M}
 (|\rho_h(y)|+|\rho_j(y)|)\,dy
 &\le C_Je^{-cM^2},\\
 \int_{-\infty}^{-M}
 |y|(|\rho_h(y)|+|\rho_j(y)|)\,dy
 &\le C_Je^{-cM^2}.
\end{split}
\]
It follows that
\[
\begin{split}
 \int_{-M}^{\infty}h_\beta(y)\,dy
 &=\frac{\beta}{2}M+s_\beta+O_J(e^{-cM^2}),\\
 2\int_{-M}^{\infty}j_\beta(y)\,dy
 &=\beta M+s_\beta+O_J(e^{-cM^2}),
\end{split}
\]
and
\[
\begin{split}
 \int_{-M}^{\infty}y h_\beta(y)\,dy
 &=-\frac{\beta}{4}M^2+O_J(1)+\red{O_J(e^{-cM^2})},\\
 2\int_{-M}^{\infty}y j_\beta(y)\,dy
 &=-\frac{\beta}{2}M^2+O_J(1)+\red{O_J(e^{-cM^2})}.
\end{split}
\]
The polynomial terms cancel exactly:
\[
 \frac{\beta}{4}(R-M)^2
 +R\frac{\beta}{2}M-\frac{\beta}{4}M^2
 =\frac{\beta}{4}R^2,
\]
\[
 \frac{\beta}{2}(R-M)^2
 +R\beta M-\frac{\beta}{2}M^2
 =\frac{\beta}{2}R^2.
\]
Since \red{$e^{-cM^2}$} is bounded for $M\ge1$, the two conclusions follow.
\end{proof}

\subsection{The reduced energy and first-variation identities}
\blue{Lemma~\ref{lem:cutoffcancel} allows us to split the identities \eqref{eq:exactReducedEnergy}--\eqref{eq:exactReducedVirialDensity} at $-M$ without incurring more than an $O_J(1)$ error, provided $M$ grows slowly with $R_n$.  Choosing a logarithmic cutoff makes the exponentially small remainders comparable to the  $O_J(1)$ errors coming from Theorem~\ref{thm:boundaryconvergence}. We can then extract the asymptotic expansion of $\calT_n$ and, \red{together with} the virial identity, \red{this allows us to prove }a \red{property of} the optimal radius $R_n$.}
\begin{Proposition}\label{prop:energyreduction}
For $\beta$ in the range of Theorem~\ref{thm:main},
\bel{eq:energyReduced}
 \calT_n=\frac{n^2}{R_n^2}
 +\frac{\beta}{4}R_n^2+s_\beta R_n+O(1).
\ee
In addition,
\bel{eq:reducedVirial}
 R_n\left(
 -\frac{2n^2}{R_n^3}+\frac{\beta}{2}R_n+s_\beta
 \right)=O(1).
\ee
\end{Proposition}

\begin{proof}
We invoke
 Theorem~\ref{thm:boundaryconvergence} and
Lemmas~\ref{lem:leftQuantTailPropagation}
and~\ref{lem:rightQuantTailPropagation}.
Choose $C_{\log}$ so large that
$cC_{\log}^2>1$, where $c$ is smaller than the Gaussian rates in Corollary~\ref{cor:bulkExpansion} and Lemma~\ref{lem:cutoffcancel}, and set
\[
        M_n^{\rm log}=C_{\log}\sqrt{\log R_n}.
\]
Then $R_ne^{-c(M_n^{\rm log})^2}=O(1)$.  Split
\eqref{eq:exactReducedEnergy} at $y=-M_n^{\rm log}$.  Corollary~
\ref{cor:bulkExpansion} gives
\[
 \int_{-R_n}^{-M_n^{\rm log}}(R_n+y)h_n(y)\,dy
 =\frac{\beta}{4}(R_n-M_n^{\rm log})^2+O(1),
\]
and Lemma~\ref{lem:boundaryExteriorEnergy} gives
\[
\begin{split}
 \int_{-M_n^{\rm log}}^{\infty}(R_n+y)h_n(y)\,dy
 &=R_n\int_{-M_n^{\rm log}}^{\infty}h_\beta(y)\,dy+\int_{-M_n^{\rm log}}^{\infty}y h_\beta(y)\,dy+O(1).
\end{split}
\]
The first identity in Lemma~\ref{lem:cutoffcancel}, with $R=R_n$,
now proves \eqref{eq:energyReduced}.

Apply the same decomposition to
\eqref{eq:exactReducedVirialDensity}.  Corollary~
\ref{cor:bulkExpansion} and Lemma~\ref{lem:boundaryExteriorEnergy} give
\[
\begin{split}
 0={}&-\frac{2n^2}{R_n^2}
 +\frac{\beta}{2}(R_n-M_n^{\rm log})^2\\
 &+2R_n\int_{-M_n^{\rm log}}^{\infty}j_\beta(y)\,dy
 +2\int_{-M_n^{\rm log}}^{\infty}y j_\beta(y)\,dy+O(1).
\end{split}
\]
The second identity in Lemma~\ref{lem:cutoffcancel} reduces this to
\[
 0=-\frac{2n^2}{R_n^2}
 +\frac{\beta}{2}R_n^2+s_\beta R_n+O(1),
\]
which is equivalent to \eqref{eq:reducedVirial}.
\end{proof}

\begin{Corollary}\label{cor:sharploc}
For $\beta$ in the range of Theorem~\ref{thm:main},
\[
 |R_n-\widehat R_n|\le \frac{C(\beta)}{R_n},
\]
where $\widehat R_n$ is the zero of $F_n$ defined in
\eqref{eq:reducedRadiusFunction}.
\end{Corollary}

\begin{proof}
Divide \eqref{eq:reducedVirial} by $R_n$.  Then
$F_n(R_n)=O(R_n^{-1})$.  Since $F_n(\widehat R_n)=0$ and
$F_n'\ge\beta/2$, the mean-value theorem gives
\[
        |R_n-\widehat R_n|
        \le \frac{2}{\beta}|F_n(R_n)|
        \le \frac{C(\beta)}{R_n}\red{,}
\]
\red{as claimed.}
\end{proof}
\subsection{Proofs of the two main theorems}

\begin{proof}[Proof of Theorem~\ref{thm:main}]
The radius is determined by the reduced first-variation identity, not by
the energy expansion alone.  By \eqref{eq:reducedVirial},
\[
        F_n(R_n)=O(R_n^{-1}).
\]

Write $\mu_n=R_n-R_n^{(0)}$.  \red{Lemma~\ref{lem:preliminaryLocalization}} gives
$\mu_n=O_\beta(1)$.  Since
\[
 F_n(R_n^{(0)})=s_\beta,\qquad
 F_n'(R_n^{(0)})=2\beta,\qquad
 \sup_{|R-R_n^{(0)}|\le C(\beta)}|F_n''(R)|
 \le \frac{C(\beta)}{R_n^{(0)}},
\]
Taylor's theorem gives
\[
 F_n(R_n)
 =s_\beta+2\beta\mu_n
 +O_\beta\left(\frac{\mu_n^2}{R_n^{(0)}}\right).
\]
Therefore
\bel{eq:mu}
        \mu_n=-\frac{s_\beta}{2\beta}
        +O_\beta((R_n^{(0)})^{-1}),
\ee
which is \eqref{eq:Rasymp}.

The function
\[
        R\longmapsto \frac{n^2}{R^2}+\frac{\beta}{4}R^2
\]
has vanishing derivative at $R=R_n^{(0)}$ and a uniformly bounded second
derivative on bounded neighborhoods of that point.  Hence\red{, by Taylor's theorem,}
\[
 \frac{n^2}{R_n^2}+\frac{\beta}{4}R_n^2
 =\sqrt\beta\,n+O_\beta(1).
\]
Also,
\[
 s_\beta R_n
 =\sqrt2\,s_\beta\beta^{-1/4}n^{1/2}+O_\beta(1).
\]
Substitution in \eqref{eq:energyReduced} proves \eqref{eq:Tasymp}.  At
$\beta=1$, the calculation following \eqref{eq:BPSsurface} gives
$s_1=0$, while the Bogomolny completion \eqref{eq:BPScompletion} gives
$\calT_n=n$ for every $n$.
For $\beta$ in a fixed compact interval $J\Subset(0,4)$,
the surface energy is bounded by Lemma~\ref{lem:interfaceUniformTails},
and \eqref{eq:firstpassSL} gives $\mu_n=O_J(1)$.
\green{It follows that t}he preceding Taylor bounds are uniform on $J$.
\green{Moreover, t}he endpoint, fixed-interval, and density estimates used in
Proposition~\ref{prop:energyreduction} are uniform there as well.
This proves the asserted uniformity of both remainders, for every
choice of radial minimizers.
\end{proof}

\medskip

\begin{proof}[Proof of Theorem~\ref{thm:asymptoticShape}]
The preliminary localization lemma verifies \eqref{eq:SLhyp}.
Theorem~\ref{thm:boundaryconvergence} therefore gives
\eqref{eq:shapeH2}.
We prove the level-set statement uniformly in
$\theta\in[\varepsilon,1-\varepsilon]$.  Set
\[
 I_{\varepsilon,\beta}
 =[y_{\beta,\varepsilon/2},y_{\beta,1-\varepsilon/2}],
 \qquad
 m_{\varepsilon,\beta}
 =\min_{y\in I_{\varepsilon,\beta}}\Phi_\beta'(y)>0.
\]
The positivity follows from \eqref{eq:interface-monotone}.  By
\eqref{eq:shapeH2} and the embedding
$H^2(I_{\varepsilon,\beta})\hookrightarrow C^1(I_{\varepsilon,\beta})$,
\[
 \|\Phi_n-\Phi_\beta\|_{C^1(I_{\varepsilon,\beta})}
 \le \frac{C_{\varepsilon,\beta}}{R_n}.
\]
For all large $n$, the values of $\Phi_n$ at the two endpoints of
$I_{\varepsilon,\beta}$ lie below $\varepsilon$ and above
$1-\varepsilon$, respectively.  Monotonicity therefore places
$x_{n,\theta}:=u_{n,\theta}-R_n$ in
$I_{\varepsilon,\beta}$ for every
$\theta\in[\varepsilon,1-\varepsilon]$.  Since
$\Phi_n(x_{n,\theta})=\Phi_\beta(y_{\beta,\theta})=\theta$,
\[
 |\Phi_\beta(x_{n,\theta})-\Phi_\beta(y_{\beta,\theta})|
 \le \frac{C_{\varepsilon,\beta}}{R_n}.
\]
The mean-value theorem and the definition of
$m_{\varepsilon,\beta}$ give
\[
 |x_{n,\theta}-y_{\beta,\theta}|
 \le \frac{C_{\varepsilon,\beta}}{m_{\varepsilon,\beta}R_n},
\]
uniformly in $\theta$.  This is \eqref{eq:shapeLevels};
subtracting the two level-radius formulae gives \eqref{eq:shapeWidth}.

It remains to verify the left overlap.  Fix
$0<\alpha<\widehat\alpha<1/3$.  Proposition~\ref{prop:coreWKB},
the error estimate \eqref{eq:coreWKBerr}, and
Corollary~\ref{cor:leftCoefficientMatching} give the finite Weber
asymptotic with coefficient converging to $C_\beta$.

Finally, \eqref{eq:Rasymp} gives
$b_n-\sqrt\beta=O_\beta(R_n^{-1})$, and therefore
\[
 |b_n-\sqrt\beta|\,y^2
 \le C(\beta) R_n^{2\alpha-1}\longrightarrow0,
 \qquad
 \left|\frac{\beta}{b_n}-\sqrt\beta\right|\log(2+|y|)
 \longrightarrow0
\]
uniformly on that interval.  Thus the Gaussian exponent, the algebraic
power, and the coefficient in \eqref{eq:coreWKB} converge uniformly to
the corresponding quantities in \eqref{eq:shapeLeftOverlap}.  Taking
first $n\to\infty$ with $L$ fixed and then $L\to\infty$ proves
\eqref{eq:shapeLeftOverlap}.

The  estimate \eqref{eq:shapeRightComparison} for the right tail is
Lemma~\ref{lem:rightQuantTailPropagation}.  The first formula in
\eqref{eq:shapeInterfaceTails} is \eqref{eq:lefttail}; the two formulas at
$+\infty$ follow from Proposition~\ref{prop:righttail} and
\eqref{eq:rightSub}.
\end{proof}

\bigskip 

\bigskip

\section*{Notation}

\begin{longtable}{@{}>{$}l<{$}p{0.65\textwidth}@{}}
\toprule
\textbf{Symbol} & \textbf{Meaning}\\
\midrule
\endhead
J,\ J_0 & Fixed compact subintervals of $(0,4)$ on which the parameter estimates are uniform.\\
\beta & Coupling parameter in the (quartic) Abelian Higgs or Ginzburg--Landau model.\\
n & Vortex number (magnetic flux and topological degree).\\
\calT_n & (Dimensionless) string tension / energy functional; see \eqref{eq:energy}.\\
R_n & Magnetic radius; see \eqref{eq:frob} and Proposition~\ref{prop:coarseloc}.\\
R_n^{(0)} & Leading-order radius $R_n^{(0)}=\sqrt2\,\beta^{-1/4}n^{1/2}$; see \eqref{eq:coarseloc}.\\
F_n(R),\ \widehat{R}_n & Reduced radius function and its unique positive zero; see \eqref{eq:reducedRadiusFunction}.\\ 
u=evr,\ y=u-R_n,\ \eta=-y & Dimensionless radial variable and the two centered interface coordinates; see \eqref{eq:energy}, \eqref{eq:boundaryvars} and \eqref{eq:PhiGamma_def}.\\
b_n=\frac{2n}{R_n^2} & Effective magnetic slope associated with the magnetic radius $R_n$, which has the property that $b_n\rightarrow\sqrt{\beta}$; see \eqref{eq:bncoarse}.\\
m:=\sqrt{2},\mu_\beta=\sqrt{2\beta} & Gauge and Higgs decay rates on the superconducting side; see \eqref{eq:mMuBetaDef}\\
\phi_n(u),\ a_n(u) & Radial vortex profiles (Higgs and gauge); see \eqref{eq:vortex}--\eqref{eq:bc}.\\
\Phi_n(y),\ \Gamma_n(y) & Shifted profiles in $y=u-R_n$; see Definition~\ref{def:PhiGamma} and \eqref{eq:boundaryvars}.\\
\Phi_n^-(\eta),\ \Gamma_n^-(\eta) & Shifted profiles in $\eta=R_n-u$; see \eqref{eq:PhiGamma_def}.\\
\eta_n(u),\ \eta_n^+(y),\ \eta_\beta(y) & Radial and centered Higgs defects; see the beginning of section \ref{sec:exterior}, subsection \ref{subsec:rightTail}, and \eqref{eq:etaInterface}.\\
\Phi_\beta,\ \Gamma_\beta & CHMO interface orbit; see \eqref{eq:interface}--\eqref{eq:interfaceright} and \eqref{eq:CHMOscaling}.\\
\Psi,\ A & CHMO variables for the interface orbit; see \eqref{eq:CHMOeq}--\eqref{eq:CHMObc}.\\
e_\beta(y) & Interface energy density along $(\Phi_\beta,\Gamma_\beta)$; see \eqref{eq:edens}.\\
s_\beta & Renormalized surface energy; see \eqref{eq:sbeta}.\\
\sigma_\beta & Rescaled surface energy $\sigma_\beta=\sqrt2\,s_\beta$; see \eqref{eq:sigma}.\\
W[f,g]=fg'-f'g & Wronskian convention used in the Weber and modified-Bessel constructions.\\
\alpha,\widehat{\alpha} & Auxiliary overlap exponents, with restrictions stated locally in the results in which they are needed. In the argument from the proof of Theorem \ref{thm:asymptoticShape}, they are chosen so that $0<\alpha<\widehat{\alpha}<\frac{1}{3}$; see \ref{subsec:leftMatching} and \eqref{eq:twoOverlapExponents}.\\
d_{n,L},\ g_{n,L} & Exact recessive and dominant finite Weber branches, normalized by $W[d_{n,L},g_{n,L}]=1$; see Lemma \ref{lem:finiteWeberFrame}.\\
d_{\beta,L},\ g_{\beta,L} & Corresponding exact recessive and dominant branches for the limiting interface; see Definition \ref{def:leftWeberFrame}.\\
C_\beta & Coefficient in the $y\to-\infty$ asymptotics of $(\Phi_\beta,\Gamma_\beta)$; see \eqref{eq:lefttail}.\\
\widetilde{C}_{n,L} & Finite Weber coefficient in the context of the asymptotics in the overlap region; see \eqref{eq:coreWKBcoefficient} , with $\displaystyle \widetilde{C}_{n,L}\rightarrow C_\beta$ in the sense of \eqref{eq:Cmatch}.\\
G_\beta,\ H_\beta & Coefficients in the $y\to+\infty$ asymptotics of $\Gamma_\beta$ and $\eta_\beta$; see \eqref{eq:rightSub}.\\
\mathfrak g_\beta,\ \mathfrak h_\beta & Auxiliary quantities used to prove $G_\beta,H_\beta>0$ (see the proof following \eqref{eq:rightSub}).\\
D_{\beta,L}^{\rm WKB},\ G_{\beta,L}^{\rm WKB} & Recessive and dominant WKB Weber comparison branches for the limiting interface scalar equation; see Definition \ref{def:leftWeberFrame}.\\
\mathfrak B_n(u) & Magnetic field $\mathfrak B_n(u)=-(n/u)a_n'(u)$; see \eqref{eq:magneticField} and \eqref{eq:magneticCouplingIdentity}.\\
\calM_n,\ \calV_n & Magnetic and potential terms in the Pohozaev or virial identity; see Lemma~\ref{lem:scalingVirial}.\\
h_n,\ j_n,\ h_\beta,\ j_\beta & Radial reduced densities and their limiting interface counterparts; see \eqref{eq:hnjn}.\\
A_n,\ B_n & Exterior modified-Bessel coefficients; see \eqref{eq:aVolterra}--\eqref{eq:etaVolterra}.\\
\widehat A_n,\ \widehat B_n & Renormalized coefficients corresponding to $A_n,B_n$; see \eqref{eq:Ahat}--\eqref{eq:Bhat}.\\
\operatorname{Tr}_K & Trace operator at $y=K$ (Cauchy data); see \eqref{eq:rightTailTrace}.\\
\Theta_{\beta,K},\ \Theta_{n,K} & Trace maps $\Theta_{\beta,K}(z)=\operatorname{Tr}_K U_{\beta,z}$ and $\Theta_{n,K}(z)=\operatorname{Tr}_K U_{n,z}$; see \eqref{eq:rightTailTraceMaps}.\\
\Lambda_n:=C_{\log}\sqrt{\log R_n},\ \Lambda_n^+:=2\Lambda_n & Logarithmic matching scales used in the left\red{/normal} endpoint construction; see the beginning of subsection \ref{subsec:CauchyData}\\
\calA_\beta & Linearized interface operator at $(\Phi_\beta,\Gamma_\beta)$; see \eqref{eq:Alinearization}.\\
E_\beta^+(Y) & Two-dimensional stable plane of Cauchy data at $y=Y$ (solutions decaying as $y\to+\infty$); see Lemma~\ref{lem:rightstable}.\\
E_\beta^-(-Y) & Corresponding line of Cauchy-data at $y=-Y$ generated by solutions satisfying \eqref{eq:normalized}; see Lemma \ref{lem:leftline}.\\
Z_{H,\beta},\ Z_{G,\beta} & Distinguished basis of decaying linearized modes spanning $E_\beta^+(Y)$; see Lemma~\ref{lem:rightstable}.\\
P_n(\eta):=\Gamma^0_n(\eta)-\sqrt{\beta}\eta & Explicit finite-radius correction to the affine limiting gauge profile; see \eqref{eq:exactAffineDefect}.\\
\bottomrule
\end{longtable}

\end{document}